\documentclass{amsart}

\usepackage{amssymb,euscript,amsmath, mathrsfs,tensor}
\usepackage[dvips]{graphicx}
\usepackage[dvips]{color}

\usepackage[usenames,dvipsnames]{xcolor}
\usepackage[colorlinks]{hyperref}
\hypersetup{colorlinks=true,citecolor=NavyBlue,linkcolor=Red,urlcolor=Green}

\newcounter{ENUM}

\input xy
\xyoption{all}
\CompileMatrices

\def\presuper#1#2%
  {\mathop{}%
   \mathopen{\vphantom{#2}}^{#1}%
   \kern-\scriptspace%
   #2}  

\def\<{\langle}
\def\>{\rangle}
\def\0{{{\bf 0}}}

\def\OO{{\mathcal O}}

\def\CF{{\mathcal F}}

\def\CN{{\mathcal N}}

\def\CM{{\mathcal M}}

\def\FF{{\mathbb F}}
\def\GG{{\mathbb G}}
\def\HH{{\widetilde H}{}{}}

\def\LG{{\tensor*[^L]G{}}}

\def\QQ{{\mathbb Q}}

\def\tX{{\tilde X}}

\newcommand{\Lie}{\operatorname{Lie}}

\newcommand{\Ad}{\operatorname{Ad}}

\newcommand{\SL}{\operatorname{SL}}

\newcommand{\Sp}{\operatorname{Sp}}
\newcommand{\SO}{\operatorname{SO}}

\newcommand{\sss}{\mbox{\rm \tiny ss}}

\newcommand{\et}{\mbox{\rm \tiny \'et}}

\def\ZZ{{\mathbb Z}}
\def\NN{{\mathbb N}}

\def\ad{\operatorname{ad}}

\def\Hom{\operatorname{Hom}}

\def\Fr{\operatorname{Fr}}

\def\End{\operatorname{End}}
\def\Aut{\operatorname{Aut}}

\def\GL{\operatorname{GL}}

\def\tr{\operatorname{tr}}
\def\Spec{\operatorname{Spec}}
\def\Pic{\operatorname{Pic}}

\def\id{\operatorname{id}}

\def\univ{\operatorname{univ}}

\def\im{\operatorname{im}}

\def\varphiL{\tensor*[^L]{\varphi}{}}
\def\To#1{\buildrel\hbox{\tiny{$#1$}}\over\longrightarrow}
\def\ZM{{\mathbb{Z}}}
\def\CM{{\mathbb{C}}}

\def\uZ{{\underline Z}}
\def\uH{{\underline H}}
\def\uSigma{{\underline\Sigma}}
\def\HG{{\hat G}}
\def\oFl{\bar{\mathbb{F}}_{\ell}}

\def\To#1{\buildrel\hbox{\tiny{$#1$}}\over\longrightarrow}
\def\ZM{{\mathbb{Z}}}

\def\uZ{{\underline Z}}
\def\uSigma{{\underline\Sigma}}
\def\varphiL{{\tensor*[^L]\varphi{}}}
\def\phiL{{^{L}\phi}}

\def\Sym{{\operatorname{Sym}}}
\def\uZ{{\underline{Z}}}
\def\HG{{\hat G}}
\def\HH{{\hat H}}
\def\HT{{\hat T}}
\def\HB{{\hat B}}

\def\To#1{\buildrel\hbox{\tiny{$#1$}}\over\longrightarrow}
\def\ZM{{\mathbb{Z}}}

\def\CM{{\mathbb{C}}}
\def\FM{{\mathbb{F}}}
\def\uZ{{\underline Z}}
\def\uH{{\underline H}}
\def\uSigma{{\underline\Sigma}}
\def\HG{{\hat G}}
\def\oFl{\bar{\mathbb{F}}_{\ell}}
\def\varphiL{{\tensor*[^L]\varphi{}}}
\def\phiL{{^{L}\phi}}
\def\phit{\phi}
\def\alphat{\alpha}

\newcommand{\margh}[1]{}

\newtheorem{thm}{Theorem}[section]
\newtheorem{prop}[thm]{Proposition}
\newtheorem{lemma}[thm]{Lemma}
\newtheorem{cor}[thm]{Corollary}
\newtheorem{conj}[thm]{Conjecture}

\theoremstyle{definition}
\newtheorem{defn}[thm]{Definition}

\newtheorem{ex}[thm]{Example}

\newtheorem{notn}[thm]{Notation}
\newtheorem{rem}[thm]{Remark}

\numberwithin{equation}{section}

\def\To#1{\buildrel\hbox{\tiny{$#1$}}\over\longrightarrow}
\def\ZM{{\mathbb{Z}}}
\def\NM{{\mathbb{N}}}
\def\uZ{{\underline Z}}
\def\uH{{\underline H}}
\def\uSigma{{\underline\Sigma}}
\def\HG{{\hat G}}
\def\oFl{\bar{\mathbb{F}}_{\ell}}

\usepackage{stmaryrd,subfiles}
\usepackage{faktor}
\usepackage{xfrac}
\usepackage{hyperref}

\title{Moduli of Langlands parameters}
\author{Jean-Fran\c cois Dat}
\address{Jean-Fran\c cois Dat, Institut de Math\'ematiques de Jussieu,
  Sorbonne Universit\'e, Universit\'e de Paris, CNRS 4, place Jussieu, 75252, Paris, France.}
\email{jean-francois.dat@imj-prg.fr }
\author{David Helm}
\address{David Helm, Department of Mathematics, Imperial College, London, SW7 2AZ, United Kingdom.}
\email{d.helm@imperial.ac.uk }
\author{Robert Kurinczuk}
\address{Robert Kurinczuk, School of Mathematics and Statistics, University of Sheffield, Sheffield, S3 7RH, United Kingdom.}
\email{robkurinczuk@gmail.com}
\author{Gilbert Moss}
\address{Gil Moss, Department of Mathematics and Statistics, 5752 Neville Hall, Room 237, The University of Maine, Orono, ME 04469, USA.}
\email{gilbert.moss@maine.edu}


\begin{document}

\begin{abstract}
Let $F$ be a non-archimedean local field of residue characteristic
$p$, let ${\hat G}$ be a split reductive group scheme over $\ZZ[\frac{1}{p}]$ with an action of $W_F$, and let $\LG$ denote the semidirect product
${\hat G} \rtimes W_F$.  We construct a moduli space of Langlands parameters $W_F \rightarrow \LG$, and show that it is locally of finite type and flat over $\ZZ[\frac{1}{p}]$, and
that it is a reduced local complete intersection.
We give parameterizations of the connected components and the irreducible components
  of the geometric fibers of this space, and parameterizations of the
  connected components of the total space over
  $\overline{\ZZ}[\frac{1}{p}]$ (under mild hypotheses) and over $\overline{\ZZ}_{\ell}$ for
  $\ell\neq p$. In each case, we show precisely how each connected
  component identifies with the ``principal'' connected component
  attached to a  smaller split reductive group scheme.
Finally, we study 
the GIT quotient of this space by ${\hat G}$ and give a description of its fibers up to homeomorphism, and a complete description of
its ring of functions after inverting an explicit finite set of primes depending only on~$\LG$.

\end{abstract}
\maketitle

\setcounter{tocdepth}{1}
\tableofcontents

\section{Introduction and main results}

\subsection{Introduction}
Let $F$ be a local field with residue characteristic $p$, 
and $G$ a quasi-split connected reductive group over $F$.  Let $\ell$ be a prime different from $p$.
A {\em Langlands parameter} for $G$ is a continuous $L$-homomorphism $W_F \rightarrow \LG(\overline{\QQ}_{\ell})$; that is,
an $\ell$-adically continuous homomorphism from the Weil group $W_F$ to the group of $\overline{\QQ}_{\ell}$-points of the Langlands dual
group $\LG := \hat G \rtimes W_F$ of $G$, such that the composition with the natural map 
$\LG(\overline{\QQ}_{\ell}) \rightarrow W_F$
is the identity. 

When $G$ is the general linear group $\GL_n$, then $\LG$ is simply the product $\GL_n \times W_F$, and a Langlands parameter for
$G$ is simply a continuous representation: $W_F \rightarrow \GL_n(\overline{\QQ}_{\ell})$.  Such representations vary nicely in
algebraic families; in particular, given a continuous representation $\overline{\varphi}: W_F \rightarrow \GL_n(\overline{\FF}_{\ell})$, 
we can associate to it the {\em universal framed deformation ring} $R_{\overline{\varphi}}^{\Box}$. 
a complete Noetherian local $W(\overline{\FF}_{\ell})$-algebra
that admits a continouous representation $\varphi^{\univ}: W_F \rightarrow \GL_n(R_{\overline{\varphi}}^{\Box})$ such that the pair
$(R_{\overline{\varphi}}^{\Box},\varphi^{\univ})$ is universal for pairs $(R,\varphi)$, where $R$ is a complete Noetherian local
$W(\overline{\FF}_{\ell})$-algebra and $\varphi: W_F \rightarrow \GL_n(R)$ is a lift of $\overline{\varphi}$.  

Given the importance of such deformation spaces in the Langlands program, it is natural to attempt to construct corresponding 
``universal deformation spaces'' for Langlands parameters attached to groups $G$ other than $\GL_n$.  Indeed, Bellovin and Gee~\cite{BG}
and Booher and Patrikis~\cite{BP19} independently study a closely related problem.  Specifically, (cf.~\cite{BP19}, Section 2) define
an {\em $\LG$-Weil-Deligne representation} over a $W(\overline{\FF}_{\ell})$-algebra $A$ to be a triple $(D_A,r,N)$, where
$D_A$ is an $\LG$-bundle over $\Spec A$, $r: W_F \rightarrow \Aut_{\LG}(D_A)$ is a homomorphism {\em with open kernel}, and $N$ is
a nilpotent element of the Lie algebra of $\Aut_{\LG}(D_A)$ such that $\operatorname{Ad}_r(w) N = \lvert w \rvert N$ for all $w \in W_F$.
Both Bellovin-Gee and Booher-Patrikis construct moduli spaces of such $\LG$-Weil-Deligne representations, that are schemes 
locally of finite type over $W(\overline{\FF}_{\ell})$, and show that their general fibers are generically smooth and equidimensional
of dimension equal to the dimension of $\LG$.

When $A$ is complete local, and $\ell$ is invertible in $A$, Grothendieck's monodromy theorem gives a natural bijection between $\LG$-Weil-Deligne 
representations with values in $A$ and Langlands parameters with values in $A$, so the results of Bellovin-Gee and Booher-Patrikis in
some sense give a solution to the problem of finding universal families for Langlands parameters over $G$.  Their method
relies heavily on the exponential and logarithm maps, which have denominators, and also involves division by the order of the image of an element of inertia.
There is thus reason to question whether
a naive extension of these constructions to situations where $\ell$ is not invertible gives the ``right'' objects, particularly
if the prime $\ell$ is small enough to divide one of these denominators.
For instance, when $G = \GL_2$, and $\ell$ divides $q^2 - 1$ (where $q$ denotes the order of the residue field of $F$),
the analogue of the spaces constructed by Bellovin-Gee
and Booher-Patrikis fails to be flat over $W(\overline{\FF}_{\ell})$.  Since universal framed deformation rings are known to be flat over
$W(\overline{\FF}_{\ell})$, this means that when $G = \GL_2$, naive generalization of the constructions of Bellovin-Gee and Booher-Patrikis fails to recover
the existing theory in such characteristics.  It is reasonable to expect that this failure of flatness persists for more complicated groups.
Such a failure makes these spaces unsuitable for formulating
analogues of Shotton's ``$\ell \neq p$ Breuil-Mezard'' results for $\GL_n$ \cite{MR3769675}.
We refer the reader to section~\ref{subsection:bellovin-gee} for further discussion of this point.

In light of these issues, it is tempting to look at alternative characterizations of Langlands parameters over fields of characteristic zero,
in the hope that they suggest better behaved moduli problems.  There are (at least) three definitions of a ``Langlands parameter
over $\overline{\QQ}_{\ell}$'' common in the literature:
\begin{enumerate}
\item pairs $(r,N)$, where $r: W_F \rightarrow \LG(\overline{\QQ}_{\ell})$ is an $L$-homomorphism with open kernel and $N \in \Lie({\hat G}_{\overline{\QQ}_{\ell}})$
a nilpotent element, such that $\operatorname{Ad}_r(w) = \lvert w\rvert N$,
\item maps $W_F \times \SL_2(\overline{\QQ}_{\ell}) \rightarrow \LG(\overline{\QQ}_{\ell})$ whose restriction to the first factor is an $L$-homomorphism with open kernel
and whose restriction to the second factor is algebraic, and
\item $L$-homomorphisms $\varphiL: W_F \rightarrow \LG(\overline{\QQ}_{\ell})$ that are $\ell$-adically continuous.
\end{enumerate}

The first of these definitions generalizes in an obvious way to coefficients in an arbitrary $W(\overline{\FF}_{\ell})$-algebra $R$, and considering the associated
moduli problem leads to the schemes considered by Bellovin-Gee and Booher-Patrikis.  The second likewise generalizes to such algebras $R$, but the associated
moduli space is much less well-behaved.  For instance, the moduli space of unramified pairs $(r,N)$ as in (1) is connected over $\overline{\QQ}_{\ell}$, 
whereas the space of unramified maps
$W_F \times \SL_2 \rightarrow \LG$ as in (2) is, over $\overline{\QQ}_{\ell}$, a disjoint union over the set of conjugacy classes of unipotent elements
$u \in {\hat G}(\overline{\QQ}_{\ell})$, of the loci where the image of the matrix $\left(\begin{smallmatrix} 1 & 1\\0 & 1\end{smallmatrix}\right)$ in the $\SL_2$ factor is conjugate to $u$.

It is therefore tempting to try to construct a moduli space of $\ell$-adically continuous $L$-homomorphisms from $W_F$ to $\LG$ as in (3).  The notion of
$\ell$-adic continuity for $L$-homomorphisms valued in $\LG(\overline{\QQ}_{\ell})$ generalizes naturally to complete local rings of residue characteristic $\ell$;
this is sufficient for a well-behaved deformation theory but is insufficient to obtain a moduli space that is locally of finite type.  In order to obtain
such a space, one would need a broader notion of $\ell$-adic continuity.

Our approach to this question is inspired by previous work of the second author in~\cite{curtis}.  That paper introduces
a notion of $\ell$-adic continuity for maps $W_F \rightarrow \GL_n(R)$ that makes sense for arbitrary $W(\overline{\FF}_{\ell})$-algebras $R$,
and constructs universal families of such representations over a suitable $W(\overline{\FF}_{\ell})$-scheme, which we will denote here by $X_n$.  
(This notation differs from that of~\cite{curtis}, where what we call the scheme $X_n$ only appears implicitly, as the disjoint union of the 
schemes denoted $X_{q,n}^{\nu}$).
As with the constructions
of Bellovin-Gee and Booher-Patrikis, the scheme $X_n$ is locally of finite type over $W(\overline{\FF}_{\ell})$, but unlike their construction,
the completion of the local ring of $X_n$ at any $\overline{\FF}_{\ell}$-point of $X_n$, corresponding to a map $\overline{\varphi}: W_F \rightarrow \GL_n(\overline{\FF}_{\ell})$,
is the universal framed deformation ring $R^{\Box}_{\overline{\varphi}}$.  In other words, $X_n$ is a locally of finite-type $W(\overline{\FF}_{\ell})$-scheme
that ``interpolates'' the universal framed deformation rings of all $n$-dimensional mod $\ell$ representations of $W_F$.  

The schemes $X_n$ constructed in~\cite{curtis} play a central role in the formulation and proof of the
``local Langlands correspondence in families'' for the group $\GL_n$, now proven by two of the authors in ~\cite{HM18}.  (These results,
in turn, imply the existence of the families conjectured by Emerton and the second author in~\cite{EH14}.) 
In particular, the subring of functions on $X_n$ that are invariant under the conjugation action on Langlands parameters is naturally
isomorphic to the center of the category of smooth $W(\overline{\FF}_{\ell})[\GL_n(F)]$-modules.  Morally, this means that aspects
of the geometry of $X_n$ are reflected in the representation theory of $\GL_n(F)$.  For instance, the connected components of $X_n$
correspond to the ``blocks'' of the category of smooth $W(\overline{\FF}_{\ell})[\GL_n(F)]$-modules.

In this paper our first objective is to generalize the construction of~\cite{curtis} to the setting of Langlands parameters for arbitrary quasi-split,
connected reductive groups, with an eye towards formulating a conjectural analogue of the local Langlands correspondence in families
for such groups.  In a departure from previous work on the subject, we work over the base ring $\ZZ[\frac{1}{p}]$ rather than
over a ring of Witt vectors; this introduces some technical complexity but gives us the smallest possible base ring
for such a correspondence.   (In particular this allows us to study chains of congruences of Langlands parameters modulo several different primes.)  We refer the reader to the next subsection for precise definitions.

Second, we aim to understand the geometry of these moduli spaces of Langlands parameters.  Several natural questions arise.  
It turns out that, as in the setting of local deformation theory of Galois representations, the spaces we obtain have a
quite tractable local structure: they are reduced local complete intersections that are flat over $\Spec \ZZ[\frac{1}{p}]$,
of  dimension $\dim {G}$.  Moreover, we give descriptions of the connected components of these moduli spaces, both over
algebraically closed fields of arbitrary characteristic $\ell \neq p$, and (conjecturally) over $\overline{\ZZ}[\frac{1}{p}]$.

Finally, we study the rings of functions on these moduli spaces that are invariant under ${\hat G}$-conjugacy (or, equivalently,
the GIT quotient of the moduli space of Langlands parameters by the conjugation action of ${\hat G}$.) 
As in the case of $\GL_n$, the ring of such functions is in general quite complicated, and does not admit an explicit description.
(In particular, the corresponding GIT quotients are very far from being normal.)
Nonetheless, we show that after inverting an explicit finite set of primes (depending only on $G$), the GIT quotients are quite nice;
indeed, they are disjoint unions of quotients of tori by finite group actions.  Over the complex numbers these connected components
coincide with varieties studied by Haines \cite{haines}.

\subsection{The moduli space of Langlands parameters} 
We now describe in detail the moduli problem that we study.
Following \cite{curtis}, the approach we take is to ``discretize'' the tame inertia
group.  Fix an arithmetic Frobenius element  $\Fr$ in $W_F$ and a pro-generator $s$ of the tame inertia group $I_F/P_F$.
These satisfy the relation $\Fr s \Fr^{-1} = s^q$.
We then consider the subgroup $\langle\Fr,s\rangle =
s^{\ZM[\frac 1q]}\rtimes\Fr^{\ZM}$ of $W_{F}/P_{F}$, we denote by $W_{F}^{0}$ its inverse
image in $W_{F}$, and we endow it with the topology that extends the profinite
topology of $P_{F}$ and induces the discrete topology on $\langle\Fr,s\rangle$.
Note that (in contrast to the subgroup $W_F$ of $G_F$), the subgroup $W_F^0$ of $W_F$ very much depends
on the choices of $\Fr$ and $s$.

Although
the topology on $W_F^0$ is finer than the one induced from $W_{F}$,  the relation
$\Fr s\Fr^{-1}=s^{q}$ implies that a morphism
$W_{F}^{0}\rightarrow\LG(\overline{\mathbb{Q}}_{\ell})$ is continuous if and only if it is continuous for the  topology
induced from $W_{F}$. 
It follows  that restriction to $W_{F}^{0}$ induces a bijection between objects of type (3)
and the following objects :
\begin{enumerate} \setcounter{enumi}{3}
\item continuous morphisms $\varphiL :\, W_{F}^{0}\rightarrow\LG(\overline{\mathbb{Q}}_{\ell})$ (with either the discrete
  or the natural topology on $\LG(\overline{\mathbb{Q}}_{\ell})$).
\end{enumerate}
These objects are now easy to define over any $\ZM_{\ell}$-algebra $R$ since only the
discrete topology of $\LG(R)$ is needed. Indeed, they are also defined for any
$\ZM[\frac 1p]$-algebra and their moduli space over $\ZM[\frac 1p]$ is already interesting.

We therefore consider the following setting:
\begin{itemize}
\item $\HG$ is a split reductive group scheme over $\ZM[\frac 1p]$ endowed
with a finite action of the absolute Galois group $G_{F}$ (we do not
assume that $G_{F}$ preserves a pinning).
\item $W_{F}^{0}$ is the inverse image in $W_{F}$ of the subgroup $s^{\ZM[\frac{1}{q}]}\rtimes
  \Fr^{\ZM}$ of $W_{F}/P_{F}$, which depends on the choice of a generator $s$ of the
  tame inertia group $I_{F}/P_{F}$ and a lift of Frobenius.   
\item 
$(P_{F}^{e})_{e\in\mathbb{N}}$ is a decreasing sequence of open subgroups of $P_{F}$ that are
normal in $W_{F}$ and whose intersection is $\{1\}$.
\end{itemize}


Note that for any $\ZM[\frac{1}{p}]$-algebra $R$, there is a natural bijection between the
continuous $L$-homomorphisms $\varphiL: W_F^{0} \rightarrow \LG(R)$ (with respect to the
discrete topology on $\LG(R)$) and the set of continuous $1$-cocycles
$Z^1(W_F^0,\HG(R))$ on $W_F^0$ with values in $\HG(R)$.  If, given $\varphiL$, we denote by
$\varphi$ the corresponding cocycle, then this bijection is characterized by the identity
$\varphiL(w) = (\varphi(w), w)$ for all $w \in W_F$.

Since the cocycles we consider are continuous with respect to the discrete topology, we have
$Z^{1}(W_{F}^{0},\HG(R))=\bigcup_{e\in\mathbb{N}}Z^{1}(W_{F}^{0}/P_{F}^{e},\HG(R))$.
It is easy to see that the functor $R\mapsto Z^{1}(W_{F}^{0}/P_{F}^{e},\HG(R))$ on
$\ZM[\frac 1p]$-algebras is
represented by an affine scheme of finite presentation over $\ZM[\frac 1p]$, that we denote
 by $\uZ^{1}(W_{F}^{0}/P_{F}^{e},\HG)$. It follows that the functor
$R\mapsto Z^{1}(W_{F}^{0},\HG(R))$ is represented by a scheme $\uZ^{1}(W_{F}^{0},\HG)$,
in which each $\uZ^{1}(W_{F}^{0}/P_{F}^{e},\HG)$ sits as a direct summand, and 
which is the increasing union of all these subschemes.

As a $\ZZ[\frac{1}{p}]$-scheme, the scheme $\uZ^1(W_F^0,\HG)$ depends on the choices we made
defining $W_F^0$ as a subgroup of $W_F$.  Indeed, if $W_F^{0'}$ is the subgroup arising from
a different choice $(\Fr',s')$ then there is not typically a canonical isomorphism
of $\ZZ[\frac{1}{p}]$-schemes from $\uZ^1(W_F^0,\HG)$ to $\uZ^1(W_F^{0'},\HG)$.  However,
there are canonical such isomorphisms over $\ZZ_{\ell}$ for each $\ell$
not equal to $p$ (Corollary \ref{cor:indep_moduli}).
Moreover, we show (Theorem~\ref{thm:indep_quotient}) that the GIT quotient $\uZ^1(W_F^0,\HG)\sslash {\hat G}$ is, up to {\em canonical}
isomorphism, independent of the choices defining $W_F^0$.  We further suspect, but do not prove,
that the corresponding quotient stacks are also canonically isomorphic.

Our study of $\uZ^1(W_F^0,\HG)$ relies on the restriction map 
$\uZ^1(W_F^0,\HG) \To{} \uZ^{1}(P_{F},\HG)$.
One crucial point is that the scheme $\uZ^{1}(P_{F},\HG)$ is
particularly well behaved, because $p$ is invertible in our coefficient rings. Indeed, we
prove the following result in the appendix.

\begin{prop}
The scheme $ \uZ^{1}(P_{F},\HG)$ is smooth and its base change to $\overline\ZM[\frac 1p]$ is a disjoint
union of orbit schemes. More precisely, there is a set $\Phi\subset
Z^{1}(P_{F},\HG(\overline\ZM[\frac 1p]))$ such that
\begin{enumerate}
\item $ \uZ^{1}(P_{F},\HG)_{\overline\ZM[\frac 1p]} = \coprod_{\phi\in\Phi} \HG\cdot\phi$
and $\HG\cdot\phi$ represents the sheaf-theoretic (fppf or \'etale) quotient
$\HG/C_{\HG}(\phi)$. 
 \item each centralizer $C_{\HG}(\phi)$ is smooth over
$\overline\ZM[\frac 1p]$ with split reductive neutral component and constant $\pi_{0}$.
\end{enumerate}
\end{prop}

This says in particular that 
any cocycle
 $\phi'\in Z^{1}(P_{F},\HG(R))$ is, locally for the \'etale topology on
 $R$,  $\HG$-conjugate to a locally
unique $\phi$ in $\Phi$.


Via the restriction morphism 
$\uZ^{1}(W_{F}^{0},\HG) \To{} \uZ^{1}(P_{F},\HG)$,  the proposition
induces a decomposition
$$\uZ^{1}(W_{F}^{0},\HG)_{\overline\ZM[\frac 1p]} = \coprod_{\phi\in\Phi} 
\HG \times^{C_{\HG}(\phi)} \uZ^{1}(W_{F}^{0},\HG)_{\phi}$$
where $\uZ^{1}(W_{F}^{0},\HG)_{\phi}$ is the closed subscheme of parameters
$\varphi$ such that $\varphi_{|P_{F}}=\phi$.
In Section 3 we further decompose $\uZ^{1}(W_{F}^{0},\HG)_{\phi}$ as follows.
\begin{prop} For each $\phi\in\Phi$,
  there is a finite set
  $\Phi_{\phi}\subset Z^{1}(W_{F}^{0},\HG(\overline\ZM[\frac 1p]))_{\phi}$,
  which is a singleton if $C_{\HG}(\phi)$ is connected, 
 with the following properties :
  \begin{enumerate}
  \item $\forall\tilde\varphi\in\Phi_{\phi}$, $\tilde\varphi(W_{F}^{0})$ normalizes a Borel pair in
    $C_{\HG}(\phi)^{\circ}$ 
  \item $\forall\tilde\varphi\in\Phi_{\phi}$, the map  $\eta\mapsto
    \eta\cdot\tilde\varphi$ defines a closed and open immersion
    $$\uZ^{1}_{\rm Ad_{\tilde\varphi}}(W_{F}^{0}/P_{F},C_{\HG}(\phi)^{\circ}) \hookrightarrow
    \uZ^{1}(W_{F}^{0},\HG)_{\phi}$$
  \item The collection of these maps defines an isomorphism
    \begin{equation}
\coprod_{(\phi,\tilde\varphi)} \HG\times^{C_{\HG}(\phi)_{\tilde\varphi}}\uZ^{1}_{\rm
      Ad_{\tilde\varphi}}(W_{F}^{0}/P_{F},C_{\HG}(\phi)^{\circ}) \To\sim
    \uZ^{1}(W_{F}^{0},\HG)_{\overline\ZM[\frac 1p]}\label{eq:decomp_intro}
  \end{equation}
  where $C_{\HG}(\phi)_{\tilde\varphi}$ is the open and closed
  subgroup scheme of $C_{\HG}(\phi)$ that stabilizes the image of (2).
  \end{enumerate} 
\end{prop}

These results are essentially relative versions of the constuctions
of~\cite[Section 2]{functoriality}.
We note that if $\HG$ is a classical group and $p>2$, then $C_{\HG}(\phi)$ is always connected.
Moreover, if the center of $\HG$ is smooth over $\ZZ[\frac 1p]$, then we show that ``Borel pair'' can be
replaced by ``pinning'' in (1).

In general, this result  shows that the crucial case to study is the
space of \emph{tame} parameters for a \emph{tame}  action of $W_{F}$ that preserves a
Borel pair of $\HG$.  This case is thoroughly studied in Section 2.
Using the results of that section and the above decomposition we will get
the following result, (Theorem \ref{thm:geometry})


\begin{thm}
  The scheme $\uZ^{1}(W_{F}^{0},\HG)$ is syntomic (flat and locally a complete
  intersection) over $\uZ^{1}(P_{F},\HG)$, generically smooth, of pure  absolute dimension $\dim(\HG)$.
\end{thm}

Beware that $\dim\HG=\dim G+1$ whenever $\HG$ is the Langlands dual
group of a reductive group $G$ over $F$, since the base scheme of
$\HG$ has dimension $1$.

We further conjecture that the summands appearing in the decomposition of~(\ref{eq:decomp_intro})
are connected. The last proposition reduces this conjecture to proving
that for any $\HG'$ with a tame Galois action preserving a Borel pair,
the summand  in~(\ref{eq:decomp_intro}) corresponding to {\em tame}
parameters is connected. In Theorem
\ref{thm:connected_pinning_preserved}  we  prove this
under the assumption that the action even preserves a pinning, i.e. when $\LG'$ is genuinely
the $L$-group of a tamely ramified  reductive group $G'$ over $F$.
This allows us  to deduce our conjecture in many cases. In particular,
Theorem \ref{thm:connected} asserts :

\begin{thm}
  If the center of $\HG$ is smooth, then all the summands in the decomposition of
  (\ref{eq:decomp_intro}) are connected.
\end{thm}



For $G=\GL_{n}$, where all centralizers are connected, this result says that
each $\uZ^{1}(W_{F}^{0},\HG)_{\phi}$ is connected, and
may be thought of  as the Galois
counterpart of the fact, discovered by S\'echerre and Stevens \cite{MR4000000}, that two
irreducible representations of $\GL_{n}(F)$ belong to the same endoclass if and
only if they are connected by a series of congruences at various
primes different from $p$.

The reduction to tame parameters also allows us to obtain a parameterization of the geometric irreducible components
of $\uZ^{1}(W_F^0,\HG)$, that is, the irreducible components of $\uZ^{1}(W_F^0,\HG)_L$ for an algebraically closed field
$L$ of characteristic different from $p$.  Such components are characterized by ``inertial types'' (that is, by specifing
the restriction of the parameter to the inertia subgroup of $W_F^0$), together with some extra data that accounts for disconnectedness of
centralizers.  In particular, combining Corollary~\ref{cor:tame irreducible components} with this reduction to tame parameters, we find:

\begin{thm}
For any algebraically closed field $L$ of characteristic different from $p$, there is a natural bijection between
the irreducible components of $\uZ^{1}(W_F^0,\HG)_L$ and the set of $\HG(L)$-conjugacy classes of pairs $(\xi,\overline{\CF}_0)$,
where $\xi$ is an element in the image of the restriction map $\uZ^1(W_F^0,\HG(L)) \rightarrow \uZ^1(I_F^0,\HG(L))$,
and $\overline{\CF}_0$ is an element of $\pi_0(T_{\hat G}(\xi^{\Fr}, \xi'))$.   [Here $\xi^{\Fr}$ is the conjugate of $\xi$ under
the action of $\Fr$ on ${\hat G}$, $\xi'$ is the composition of $\xi$ with the automorphism ``conjugation by $\Fr$'' of $I_F^0$,
and $T_{\hat G}(\xi^{\Fr},\xi')$ is the transporter; that is, the subgroup of ${\hat G}_L$ consisting of elements that conjugate
$\xi^{\Fr}$ to $\xi'$.]

Moreover, this bijection is characterized by the property that for a general $L$-point $\varphi$ of the irreducible component
of $\uZ^1(W_F^0, \HG)_L$ corresponding to a pair $(\xi,
\overline{\CF}_0)$, there exists a $\HG(L)$-conjugate of $\xi$ whose restriction 
to $I_F^0$ is equal to $\xi,$ and value at $\Fr$ 
lies in the component of
$T_{\hat G}(\xi^{\Fr}, \xi')$ given by $\overline{\CF}_0$.
\end{thm}

\subsection{The space of parameters over $\ZM_{\ell}$}

Let us now fix a prime number $\ell\neq p$.
For a $\ZM_{\ell}$-algebra $R$,  we say that a $\HG(R)$-valued cocycle $\varphi$ is
\emph{$\ell$-adically continuous} if there is 
some $\ell$-adically separated  ring $R_{0}$ such that $\varphi$ comes by
pushforward from some  $\HG(R_{0})$-valued cocycle $\varphi_{0}$, all of whose pushforwards
to $\HG(R_{0}/\ell^{n})$ are continuous for the topology inherited from $W_{F}$. 
It is not \emph{a priori} clear that this definition is local for any
usual topology. But the following result, extracted from Theorem \ref{thm:geometry}, shows
it is, and may justify again our approach involving the weird group $W_{F}^{0}$. 

\begin{thm}
The ring of functions $R_{\LG}^{e}$ of the affine scheme
$\uZ^{1}(W_{F}^{0}/P_{F}^{e},\HG)$ is $\ell$-adically separated and the universal cocycle
 $\varphi^{e}_{\rm univ}$  extends
uniquely to an $\ell$-adically continuous cocycle
$$ \varphi^{e}_{\ell-\rm univ}:\, W_{F}/P_{F}^{e}\To{}\HG(R_{\LG}^{e}\otimes\ZM_{\ell})$$
which is universal for $\ell$-adically continuous cocycles.  
\end{thm}

The $\ell$-adic continuity property of  $\varphi_{\ell-\rm univ}^{e}$ and the $\ell$-adic
separateness of $R_{\LG}^{e}\otimes\ZM_{\ell}$ imply that
the restriction of $\varphi_{\ell-\rm univ}^{e}$ to the prime-to-$\ell$
inertia group $I_{F}^{\ell}$ factors over a finite quotient. Since the order of this
finite quotient is invertible in $\ZM_{\ell}$,  we can use the same
strategy as before to decompose $\uZ^{1}(W_{F}^{0},\HG)_{\overline\ZM_{\ell}}$ using now
restriction of parameters to $I_{F}^{\ell}$. The upshot is a decomposition similar to (\ref{eq:decomp_intro})
\begin{equation}
 \uZ^{1}(W_{F}^{0},\HG)_{\overline\ZM_{\ell}} =
\coprod_{(\phi^{\ell},\tilde\varphi)} \HG\times^{C_{\HG}(\phi^{\ell})_{\tilde\varphi}}
\uZ^{1}_{{\rm Ad}_{\tilde\varphi}}(W_{F}^{0}/P_{F},C_{\HG}(\phi^{\ell})^{\circ})_{1_{I_{F}^{\ell}}}
\label{eq:decomp_intro_zl}
\end{equation}
Theorem \ref{thm:connectedZl} asserts that each summand of this decomposition has a geometrically connected
special fiber so, in particular, is connected.  The collection of all these connectedness
results for varying $\ell$ is used in the proof of the connectedness results over
$\overline\ZM[\frac 1p]$.

\subsection{The categorical quotient over a field} We now fix an algebraically closed
field $L$ of characteristic $\ell\neq p$
but we allow $\ell=0$. We consider the categorical quotient
$$ \uZ^{1}(W^{0}_{F},\HG)_{L}\sslash\HG_{L} = \lim_{\longrightarrow}\, \Spec( (R_{\LG}^{e}\otimes L)^{\HG_{L}}).$$
Recall that the closed points of $\uZ^{1}(W_{F}^{0},\HG)_{L}\sslash\HG_{L}$ correspond to
closed  $\HG(L)$-orbits in $Z^{1}(W_{F}^{0},\HG(L))$. A theorem of Richardson tells us that a
cocycle $\varphi$ has closed orbit if and only if its image in $\LG=\HG\rtimes W_{F}$ is
\emph{completely reducible} in the sense that whenever it is contained in a parabolic
subgroup of $\LG$, it has to be contained in some Levi subgroup of this parabolic subgroup.

When $\ell\neq 0$, we already know from (\ref{eq:decomp_intro_zl}) how to parametrize its connected
components, and we now wish to describe them explicitly, at least up to homeomorphism. 
In order to give a unified treatment including $\ell=0$,
we (re)label the connected components of $\uZ^{1}(W_{F}^{0},\HG)_{L}\sslash\HG_{L}$ by
the set $\Psi(L)$ of $\HG(L)$-conjugacy classes of pairs $(\phi,\beta)$ consisting of

\begin{itemize}
\item a completely reducible inertial cocycle $\phi\in Z^{1}(I_{F},\HG(L))$.

\item an element $\beta$ in $\{\tilde\beta\in\HG(L)\rtimes \Fr,\,
\tilde\beta \phi(i)\tilde\beta^{-1}= \phi(\Fr i \Fr^{-1})\}/C_{\HG}(\phi)^{\circ}$.
\end{itemize}

For such a pair, the centralizer $C_{\HG}(\phi(I_{F}))$ is a (possibly disconnected) reductive
algebraic group over $L$. So we fix a Borel pair $(\hat B_{\phi},\hat T_{\phi})$ in
$C_{\HG}(\phi)^{\circ}$ and we choose a lift $\tilde\beta$ of $\beta$ that normalizes this
Borel pair. The adjoint action of $\tilde\beta$ on $\hat T_{\phi}$ only depends on
$\beta$, and so does its action on the Weyl group $\Omega_{\phi}=\Omega_{\phi}^{\circ}\rtimes\pi_{0}(C_{\HG}(\phi))$.
Now for all $\hat t\in\hat T_{\phi}$ we can extend $\phi$ uniquely  to a cocycle
$\varphi_{\hat t\tilde\beta}\in \uZ^{1}(W_{F}^{0},\HG_{L})$ such that  $\varphi_{\hat t\tilde\beta}(\Fr)=\hat t\tilde\beta$.
The following result is Corollary \ref{cor:up_to_homeo}.
\begin{thm}
  The collection of maps $\hat t\mapsto \varphi_{\hat t\tilde\beta}$ define a universal homeomorphism
  $$ \coprod_{(\phi,\beta)\in\Psi(L)} (\hat T_{\phi})_{\beta}\sslash
  (\Omega_{\phi})^{\beta} \To{\approx} \uZ^{1}(W_{F}^{0},\HG_{L})\sslash
  \HG_{L},$$
 which is an  isomorphism if ${\rm char}(L)=0$.
\end{thm}

In particular, we see that each connected component of $\uZ^{1}(W_{F}^{0},\HG_{L})\sslash
  \HG_{L}$ is irreducible. 

When $L=\CM$, this allows us to compare in Section
\ref{subsec:comp-with-hain} our categorical quotient
with Haines' algebraic variety constructed in \cite{haines}.
\begin{cor}
  The scheme $\uZ^{1}(W_{F}^{0},\HG_{\CM})\sslash  \HG_{\CM}$ is canonically isomorphic
  to Haines' variety.
\end{cor}

When $L=\oFl$, we give in Theorem \ref{thm:structureGITfield} an explicit condition on $\ell$ for the
homeomorphism of the above theorem to be an isomorphism. This involves
the notion of $\LG$-banal prime that we now discuss.

\subsection{Reducedness of fibers and $\LG$-banal primes}
The obstruction to obtaining a description of the GIT quotients over $\ZZ[\frac{1}{p}]$ analogous to our description of the GIT quotients over fields
comes from non-reducedness of certain fibers of $\uZ^1(W_F, {\hat G})$.  In Theorem~\ref{thm:unobstructed_explicit} we determine an explicit finite set $S$ of
primes, depending only on $\LG$, such that the fibers of $\uZ^1(W_F, {\hat G})$ are geometrically reduced outside of $S$.

The reducedness of the fibers mod $\ell$, for $\ell$ outside $S$ implies in particular that given two distinct irreducible components of the geometric general 
fiber of $\uZ^1(W_F, {\hat G})$, their reductions mod $\ell$ remain
distinct.  Moreover, the reduction of each such component has scheme-theoretic multiplicity one.  

When $\LG$ is the $L$-group of a quasi-split connected reductive group $G$ over $F$, the philosophy underlying Shotton's ``$\ell \neq p$ Breuil-Mezard conjecture''
suggests that this ``multiplicity-preserving'' bijection between irreducible components in characteristic zero and characteristic $\ell$ should correspond,
on the representation theoretic side of the local Langlands correspondence, to a lack of congruences between distinct ``inertial types'' for $G$.
It is well-known that such congruences do not appear when the prime $\ell$ is {\em banal} for $G$; that is, when $\ell$ does not divide the pro-order of any compact open subgroup of $G$.  
We therefore call the set of primes $\ell$ outside $S$ ``$\LG$-banal'' primes, and we show that if $G$ is an unramified group over $F$ with no exceptional factors, then
the $\LG$-banal primes are precisely the primes that are banal for
$G$, see Corollary \ref{cor:banalvsbanal}.  On the other hand, for
certain exceptional groups $G$ there exist primes that are banal 
for $G$ but not $\LG$-banal.  It would be an interesting question (which we do not attempt to address in this paper) to find an explanation for this discrepancy
in terms of the representation theory of $G$.

Finally, we exploit the reducedness of fibers at primes away from $S$ to compute the GIT quotient $\uZ^1(W_F^0/P_F^e, {\hat G}) \sslash {\hat G}$ over $\overline\ZZ[\frac{1}{Mp}]$
for a suitable $M$, divisible by all $\LG$-banal primes. 
We refer to Subsection \ref{subsec:description}
for more details on the following statement,
which is essentially Theorem \ref{thm:structureGIT}.



\begin{thm} There is a set of triples  $(\phi,\tilde\beta,T_{\phi})$ consisting of a
  cocycle $\phi\in Z^{1}(I_{F},\HG(\overline\ZZ[\frac 1{pM}]))$, an element
  $\tilde\beta\in \HG(\overline\ZZ[\frac 1{pM}])\rtimes\Fr$ such that
  $\tilde\beta \phi(i)\beta^{-1}=\phi(\Fr i\Fr^{-1})$ for all $i\in I_{F}$,
  and an $\Ad_{\tilde\beta}$-stable maximal torus of
  $C_{\HG}(\phi)^{\circ}$, such that the collection of embeddings
  $T_{\phi}\hookrightarrow C_{\HG}(\phi)$ induce an isomorphism of 
$\overline\ZZ[\frac 1{pM}]$-schemes
$$\coprod_{(\phi,\beta)}
(T_{\phi})_{{\rm Ad}_{\tilde\beta}}\sslash (\Omega_{\phi})^{{\rm Ad}_{\tilde\beta}}  \To\sim
(\uZ^{1}(W_{F}^0/P_{F}^{e},\HG)\sslash\HG)_{\overline\ZZ[\frac 1{pM}]}
.
$$
\end{thm}

When $\LG$ is the Langlands dual group of an unramified
group, $M$ can be taken as the product of $\LG$-banal primes.
In general, a description of the integer $M$ can be extracted from
Proposition~\ref{prop:homeisombanal}. 

\subsection{Relation to recent work}
This paper has been a long time coming; many of the main results were already announced at the October 2019 Oberwolfach workshop ``New developments in the representation theory of $p$-adic
groups'', including in~the reports~\cite{KOberwolfach,DOberwolfach}.   The key idea of discretizing tame inertia first appeared in the 2016 arXiv version of \cite{curtis}.

At a late stage in the preparation of this paper, Xinwen Zhu (\cite{zhu}, particularly Section 3.1) independently generalized the~$\GL_n$-construction of \cite{curtis} to construct a moduli space of Langlands parameters for a general reductive group.   Zhu shows, as we do, that the spaces are flat, reduced local complete intersections, although he does not always use the same techniques.  The overlap in results between Zhu's work and our own occurs primarily with results contained in our Section 2 and Appendix A.  In particular, our global study of the connected components, including the functoriality principle identifying each connected component with the principal component of a smaller group, our parameterization of the irreducible components, our study of reducedness of the fibers, and our explicit description of the GIT quotients by ${\hat G}$ do not appear in his work. 

Even more recently, Laurent Fargues and Peter Scholze have proposed in \cite[Chapter
  VIII]{FarguesScholze} a different 
  construction of a moduli space of Langlands parameters over $\ZZ_{\ell}$, for $\ell \neq
  p$, in which the continuity constraints are dealt with via condensed
  mathematics. However, in order to study the main properties of their space and in
  particular prove flatness, reducedness and l.c.i., they revert
  to the same discretization process as ours, and their space turns out to be isomorphic
  to ours after base
  change to $\ZZ_{\ell}$. There is no further overlap with our paper,
  but they prove an additional beautiful result (under some mild hypothesis) : that the formation of the GIT quotient
  commutes with arbitrary base change.

\subsection{Acknowledgements} The authors are grateful to the organizers of the April 2018 conference on ``New developments in automorphic forms'' at the Instituto de Matematicas Universidad
de Sevilla, where many of the ideas behind this paper were first worked out.  We are also grateful to the organizers of the October 2019 Oberwolfach workshop ``New developments in the representation theory of $p$-adic
groups'' where most of the results of this paper have been announced.  We thank Jack Shotton, Stefan Patrikis, Sean Howe, Shaun Stevens, and Peter Scholze for helpful conversations on the subject of the paper.  We thank Eugen Hellmann for organizing an ``Oberseminar'' on this work, and Sean Cotner, Pol van Hoften, and Peter Schneider for their comments and corrections.  The second author was partially supported by EPSRC grant EP/M029719/1, the third author was partially supported by EPSRC grant EP/V001930/1, and the fourth author was partially supported by NSF grant DMS-200127.   Finally, we thank the referee for a tremendous list of comments and corrections, which have greatly improved the paper.

\section{The space of tame parameters}

We begin by considering moduli of tame Langlands parameters for tame groups.
Let $F$ be a non-archimedean local field of residue characteristic $p$, and let $I_F$, $P_F$ denote the inertia group and wild inertia group of
$F$, respectively.  Let $\OO$ be the ring of integers in a finite extension $K$ of $\QQ$, and
$\HG$ be a split connected reductive algebraic group over $\OO[\frac{1}{p}]$, and let $(\hat B,\hat T)$ be
a pair consisting of a Borel subgroup $\hat B$ of $\HG$ defined over $\OO[\frac{1}{p}]$ and a split maximal torus $\hat T$ of $\HG$
contained in $\hat B$.

We suppose that $\HG$ is equipped with an action of $W_F/P_F$ that preserves the pair $(\hat B,\hat T)$, and factors through a finite
quotient $W$ of $W_F/P_F$.  Regard $W$ as a constant group scheme over $\OO[\frac{1}{p}]$, and let $\LG$ denote the semidirect product 
$\HG \rtimes W$; we regard $\LG$ as a disconnected algebraic group scheme over $\OO[\frac{1}{p}]$.

\begin{rem}
Given our general motivations, the most natural setup would require
further that the action of $W_{F}$ on $\HG$ preserves a pinning of
$\HG$, so that $\LG$ would be the $L$-group of  a connected, quasi-split reductive $F$-group $G$ that 
splits over a tamely ramified extension of $F$. 
However, in the next section we will reduce the study of the space of all Langlands
parameters to the particular setup above, and at the moment we are not
able to reduce to the case where a pinning is fixed. 

On the other hand, the results of this section do not need the hypothesis
above on $W_{F}$ preserving a Borel pair of $\HG$ ; it will be useful
later when we study the GIT quotient and parametrize connected
components.
\end{rem}

Let $\Fr$ denote a lift of arithmetic Frobenius to $W_F/P_F$, and let $s$ be a topological generator of $I_F/P_F$.  We will
regard $\Fr$ and $s$ as elements of $W$.  We have $\Fr s \Fr^{-1} = s^q$ in $W_F/P_F$, where $q$ is the order of the residue field of $F$.

\subsection{Parameters, $L$-homomorphisms, and $1$-cocycles}

Recall that, in the case where $\LG$ is the $L$-group of a connected, quasi-split,
reductive $F$-group $G$, a {\em tame Langlands parameter} 
for $G$ is a continuous homomorphism $\rho: W_F/P_F \rightarrow
\LG(\overline{\QQ}_{\ell})$, whose composition with the projection  
$\LG(\overline{\QQ}_{\ell}) \rightarrow W$ is the natural quotient map $W_F \rightarrow
W$.   We will often refer to such a homomorphism as an 
{\em $L$-homomorphism}.  Note that if $\rho$ is a tame Langlands parameter,
there is a unique continuous cocycle $\rho^{\circ}$ in $Z^1(W_F/P_F, {\HG}(\overline\QQ_{\ell}))$ such that
$\rho(w) = (\rho^{\circ}(w),w)$; this gives a bijection between 
the set of $L$-homomorphisms and this set of cocycles.


Let $(W_F/P_F)^0$ denote the subgroup of $W_F/P_F$ generated by the elements $\Fr$ and $s$ that we fixed above, regarded as a {\em discrete} group.
Let $W_F^0$ be the preimage of $(W_F/P_F)^0$ in $W_F$.  (Note that both these groups depend heavily on the choices we made for $\Fr$ and $s$!)  

For any $\OO[\frac 1p]$-algebra $R$, the set of 
$L$-homomorphisms $(W_F/P_F)^0 \rightarrow \LG(R)$ is naturally in bijection with the set
of cocycles $Z^1(W_F^0/P_F, {\HG}(R))$.
\emph{Unless stated otherwise, we will denote by $\varphi$ a cocycle, and by $\varphiL$
  the associated $L$-homomorphism.} 

The group $\HG(R)$ acts by conjugation on the set of $L$-homomorphisms $(W_F/P_F)^0
\rightarrow \LG(R)$. The corresponding action on $Z^1(W_F^0/P_F, {\HG}(R))$ is sometimes called
``twisted conjugation''. We will denote by $^{g}\varphi$ the 
twisted-conjugate of the cocycle $\varphi$ by $g$. Explicitly, we have
$^{g}\varphi(w)=g\varphi(w) ({^{w}g})^{-1}$ where $^{w}g$ denotes the given action of $w$ on $g$.

\subsection{The scheme $\uZ^1(W_F^0/P_F,{\HG})$}

The functor that sends $R$
to $Z^1(W_F^0/P_F,{\HG}(R))$ is representable by an affine scheme
denoted by $\uZ^1(W_F^0/P_F,{\HG})$. Concretely, a
cocycle $\varphi$ is determined by the two elements $\varphi(\Fr)$ and $\varphi(s)$ of $\HG(R)$. 
Conversely,
a pair of elements $\CF_{0},\sigma_{0}$ arises in this way if, and only if the following
identitiy holds in $\LG(R)$
$$(\CF_0,\Fr)(\sigma_0,s)(\CF_0,\Fr)^{-1} = (\sigma_0,s)^q.$$

We may thus identify $\uZ^1(W_F^0/P_F,{\HG})$ with
the closed subscheme of $\HG \times \HG$ consisting of pairs $(\CF_0,\sigma_0) \in \HG \times \HG$
such that the above identity
holds in $\LG$.  In particular, $\uZ^1(W_F^0/P_F,{\HG})$ is affine, with coordinate ring
$R_{\LG}$, and we have a ``universal pair'' $(\CF_0,\sigma_0)$ of elements of ${\HG}(R_{\LG})$ satisfying the above identity.
The ``universal cocycle'' $\varphi_{\univ}$ on $Z^1(W_F^0/P_F,{\HG}(R_{\LG}))$ is then the unique cocycle such that
$\varphi_{\univ}(\Fr) = \CF_{0}$ and $\varphi_{\univ}(s) = \sigma_{0}$.
We will also let $\CF$ and $\sigma$ denote the universal elements $(\CF_0,\Fr)$ and
$(\sigma_0,s)$ of $\LG(R_{\LG})$, respectively, so that the universal $L$-homomorphism
$\varphiL_{\univ}$ is given by $\varphiL_{\univ}(\Fr)=\CF$ and $\varphiL_{\univ}(s)=\sigma$.

Given a $\OO[\frac{1}{p}]$-algebra $R$
and an $R$-valued point $x$ of $\uZ^1(W_F^0/P_F,{\HG})$, we will let $\CF_x$, $\sigma_x$, $(\CF_0)_x$, $(\sigma_0)_x$ $\varphi_x$ denote the 
objects obtained by base change from $\CF$, $\sigma$, $\CF_0$, $\sigma_0$, and $\varphi_{\univ}$, respectively.

Of course, the universal cocycle $\varphi_{\univ}$ cannot possibly extend in any nice way to a cocycle in $Z^1(W_F/P_F,{\HG}(R_{\LG}))$.
However, we will later show that if $v$ is any finite place of $\OO$ of residue characteristic $\ell \neq p$, then $\varphi_{\univ}$ extends naturally
to a cocycle $\varphi_{\univ,v}$ in $Z^1(W_F/P_F, {\HG}(R_{\LG,v}))$, where $R_{\LG,v}$ denotes the tensor product $R_{\LG} \otimes_{\OO} \OO_v$.  In order to prove this,
we must first understand the geometry of $\uZ^1(W_F^0/P_F,{\HG})$.

\subsection{Geometry of $\uZ^1(W_F^0/P_F,{\HG})$}

Let $L$ be an algebraically closed field over $\OO[\frac 1p]$. Denote by $\ell$ its characteristic, and consider the
fiber $\uZ^1(W_F^0/P_F, {\HG})_{L}$ of $\uZ^1(W_F^0/P_F,{\HG})$ over $\Spec L$.  
We have a map: ${\rm ev}_s: \uZ^1(W_F^0/P_F,{\HG})_{L} \rightarrow (\LG)_{L}$ that takes a
cocycle $\varphi$ to $\varphiL(s)$ or, in other words, a pair $(\CF,\sigma)$ to $\sigma$. 
Let $\xi$ be a point of $\LG(L)$ in the image of this map.  We denote
by $X_{\xi}$ the scheme-theoretic fiber of this map over $\xi$; it is a closed subscheme
of $\uZ^1(W_F^0/P_F,{\HG})_{L}.$  Similarly, denote by $X_{(\xi)}$ the locally closed 
 subscheme of $\uZ^1(W_F^0/P_F,{\HG})_{L}$ that is the preimage in
$\uZ^1(W_F^0/P_F,{\HG})_{L}$ of the $\HG(L)$-conjugacy class of $\xi$
in $\LG(L)$.  In particular, $\uZ^1(W_F^0/P_F,{\HG})(L)$
is the (set-theoretic) union  of the $X_{(\xi)}(L)$,  as $\xi$ runs over
a set of representatives for the ${\HG}(L)$-conjugacy classes of $\LG(L)$ in the image
of the map ${\rm ev}_s$. 

Let ${\HG}_{\xi}$ be the $\HG$-centralizer of $\xi$. This is a possibly non-reduced
group scheme over $\Spec L$ that acts on 
$X_{\xi}$ via $g \cdot (\CF_x,\sigma_x) = (\CF_x g^{-1},\sigma_x)$. 
Moreover, for any $L$-algebra $R$ and any two points $x = (\CF_x,\sigma_x)$ and $y = (\CF_y, \sigma_y)$ of $X_{\xi}(R)$, we have
$\sigma_x = \sigma_y = \xi$ and $\CF_x^{-1} \CF_y \in {\HG}_{\xi}(R)$.  Thus $X_{\xi}$ is
a ${\HG}_{\xi}$-torsor over $\Spec L$. 

Now fix an $L$-point $x = (\CF_x,\xi)$ in $X_{\xi}$.  We then obtain a surjective morphism
:
$$\pi_x: {\HG}_{L} \times {\HG}_{\xi} \rightarrow X_{(\xi)}$$
that sends $(g,g')$ to $(g \CF_x g' g^{-1}, g \xi g^{-1})$.  Moreover, we have an action
of ${\HG}_{\xi}$ on ${\HG}_{L} \times {\HG}_{\xi}$ given by  
$g'' \cdot (g, g') = (g (g'')^{-1}, \CF_x^{-1} g'' \CF_x g' (g'')^{-1})$.  This action
commutes with $\pi_x$ and
makes ${\HG}_{L} \times {\HG}_{\xi}$ into a ${\HG}_{\xi}$-torsor over
$X_{(\xi)}$.  
In particular,
we deduce that the reduced underlying subscheme of $X_{(\xi)}$ 
is smooth of dimension $\dim {\HG}_{L}$.


\begin{lemma} \label{lemma:jordan}
Let $x$ be an $L$-point of $\uZ^1(W_F^0/P_F,{\HG})$, and let
$$\sigma_x = \sigma^u_x\sigma^{\sss}_x$$
be the Jordan decomposition of $\sigma_x$; i.e.
$\sigma^u_x$ is a unipotent element of $\LG(L)$ and $\sigma^{\sss}_x$ is a semisimple element
that commutes with $\sigma^u_x$. 
Then the order of $\sigma^{\sss}_x$ is prime to $\ell$ and divides $e(q^{fN}-1)$, where $N$ is the order of
the Weyl group of $\HG$, $e$ is the order of $s$ in $W$, and $f$ is the order of $\Fr$ in $W$.
\end{lemma}
\begin{proof}
Let $e'$ be the prime-to-$\ell$ part of $e$ (or $e'=e$ if $\ell=0$).  The element $(\sigma^{\sss}_x)^{e'}$ is then a semisimple element 
$\sigma'_x$ of $\HG(L)$.  The element $\CF_x$ conjugates $\sigma'_x$ to its $q$th power.
Thus $\CF_x^f$ is an element of $\HG(L)$ that conjugates $\sigma'_x$ to its $q^f$th power.  Since $\sigma'_x$ is semisimple we may assume (conjugating it and $\CF_x$ appropriately)
that it lies in ${\hat T}(L)$.  
Since two elements of $\HT(L)$ that are conjugate under $\HG(L)$ are also conjugate under
the normalizer $N_{\HG}(\HT)(L)$,  there is an element $w$ of the Weyl group of $\HG$ that conjugates $\sigma'_x$ to its $q^f$th power.  Since $w^N$ is the identity
we have $\sigma'_x = (\sigma'_x)^{q^{fN}}$ and the claim follows. 
\end{proof}

\begin{cor} \label{cor:finite}
The image of
$\uZ^1(W_F^0/P_F,{\HG})(L)$ in $\LG(L)$ under the evaluation map ${\rm ev}_s$ is a union of finitely many $\HG(L)$-conjugacy classes in $\LG(L)$.
\end{cor}
\begin{proof}
Let $\sigma$ be an $L$-point in the image of ${\rm ev}_s$, and let $\sigma = \sigma^u\sigma^{\sss}$ be the Jordan decomposition of $\sigma$.  Then
$\sigma^{\sss}$ is semisimple with bounded order, so lies in one of finitely many conjugacy classes.  Moreover, if we fix $\sigma^{\sss}$, then $\sigma^u$ lies in the
centralizer ${\LG}_{\sigma^{\sss}}$ of $\sigma^{\sss}$ in $\LG$, which has reductive connected
component of identity, by
\cite[Cor. 9.4]{steinberg_endo}.
Now, two elements $\sigma,\sigma'$ with semisimple part $\sigma^{\sss}$ are $\HG(L)$-conjugate
if, and only if, their unipotent parts $\sigma^u,(\sigma')^u$ are ${\hat
  G}_{\sigma^{\sss}}(L)$-conjugate.  But there are only finitely many unipotent conjugacy 
classes in ${\LG}_{\sigma^{\sss}}(L)$ (see, for instance \cite{FG}, Corollary 2.6, for a proof
of this in positive characteristic), and therefore only 
finitely many ${\HG}_{\sigma^{\sss}}(L)$-orbits of unipotent elements of $\LG_{\sigma^{\sss}}(L)$.  The result follows.
\end{proof}

From this finiteness result we deduce that the scheme $\uZ^1(W_F^0/P_F,{\HG})_{L}$
is the (set-theoretic) union  of the subschemes $X_{(\xi)}$,  as $\xi$ runs over
a set of representatives for the ${\HG}(L)$-conjugacy classes of $\LG(L)$ in the image
of the map ${\rm ev}_{s}$.  In particular, the irreducible components of $\uZ^1(W_F^0/P_F, {\HG})_{L}$
are the closures of the connected components of the $X_{(\xi)}$.

We can use this to give a parameterization of the irreducible components of $\uZ^1(W_F^0/P_F, {\HG})_{L}$.
For any $\xi$, let $T_{\HG}({^{\Fr}\xi}, \xi^q)$ be the subscheme of ${\HG}$ consisting of elements that
conjugate ${^{\Fr}\xi}$ to $\xi^q$.  We then have:

\begin{cor} \label{cor:tame irreducible components}
For any algebraically closed field $L$ of characteristic $\ell \neq p$, the irreducible components
of $\uZ^1(W_F^0/P_F, {\HG})_L$ are in bijection with ${\HG}$-orbits of
pairs $(\xi, \overline{\mathcal{F}}_0)$, where $\xi$ is an element of ${\HG} \rtimes s$
and $\overline{\mathcal{F}}_{0}$ is an element of $\pi_0(T_{\HG}({^{\Fr}\xi}, \xi^q))$
\end{cor}
\begin{proof}
The irreducible components of $\uZ^1(W_F^0/P_F, {\HG})_L$ are in bijection with the union, over
a set of representatives $\xi$ of the ${\HG}(L)$-conjugacy classes in the image of
${\rm ev}_{s}$,
of the connected components of $X_{(\xi)}$.  It thus suffices to fix a particular $\xi$ and
show that the connected components of $X_{(\xi)}$ are in bijection with the orbits,
of the $\Fr$-twisted conjugation action of ${\HG}_{\xi}$ on $\pi_0(T_{\HG}({^{\Fr}\xi}, \xi^q))$.

Let $\tX_{(\xi)}$ be the $L$-scheme that parameterizes tuples $(\varphi, g)$, where $\varphi$ is a cocycle
in $\uZ^1(W_F^0/P_F,{\HG})_L$, and $g$ is an element of ${\HG}$ that conjugates $\varphiL(s)$
to $\xi$.  We have natural maps:
$$X_{(\xi)} \leftarrow \tX_{(\xi)} \rightarrow T_{\HG}({^{\Fr}\xi}, \xi^q),$$
where the left-hand map forgets $g$, and the right-hand map sends $(\varphi, g)$ to $g \varphi(\Fr) {^{\Fr}g^{-1}}$.

The action $h \cdot (\varphi, g) = (\varphi, hg)$ of ${\HG}_{\xi}$ on $\tX_{(\xi)}$ makes $\tX_{(\xi)}$
into a ${\HG}_{\xi}$-torsor over $X_{(\xi)}$, and thus induces a bijection of $\pi_0(X_{(\xi)})$
with $\pi_0(\tX_{(\xi)})^{\HG_{\xi}}$.  On the other hand, the action $h' \cdot (\varphi, g) = ({^{h'}\varphi}, g(h')^{-1})$
of ${\HG}_{L}$ on $\tX_{(\xi)}$ makes $\tX_{(\xi)}$ into a ${\HG}_{L}$-torsor over $T_{\HG}({^{\Fr}\xi}, \xi^q)$,
and thus induces a bijection of $\pi_0(\tX_{(\xi)})$ with $\pi_0(T_{\HG}({^{\Fr}\xi}, \xi^q))$.  The claim follows.
\end{proof}

The fact that the $X_{(\xi)}$ have dimension equal to that of $\dim \HG_{L}$ also lets us deduce:

\begin{cor} \label{cor:lci}
The scheme $\uZ^1(W_F^0/P_F,{\HG})$ is flat over $\OO[\frac{1}{p}]$
of pure absolute dimension
$\dim \HG$, and is a local complete  intersection.
\end{cor}
\begin{proof}
The scheme $\uZ^1(W_F^0/P_F,{\HG})$ is isomorphic to the fiber over the identity of the map: 
$${\HG} \times {\HG} \rightarrow {\HG}$$ 
given by $(\CF_0,\sigma_0) \mapsto (\CF_0, \Fr) (\sigma_0, s) (\CF_0, \Fr)^{-1} (\sigma_0, s)^{-q}$.  In particular its irreducible components have dimension at least $\dim {\HG}$, and
$\uZ^1(W_F^0/P_F,{\HG})$ is  
a local complete intersection if every irreducible component has
dimension exactly $\dim {\HG}$.  Suppose we have an irreducible
component $Y$ of larger dimension.   
Then for some prime $v$ of $\OO[\frac{1}{p}]$, of characteristic $\ell$, the fiber of $Y$ over $v$ has dimension greater than $\dim {\HG}-1$.  But $\uZ^1(W_F^0/P_F,{\HG})_{\overline{\FF}_{\ell}}$ is a set-theoretic
union of finitely many locally closed subschemes of dimension $\dim {\HG}_{\overline\FF_{\ell}}=\dim\HG-1$, so this is impossible.  Thus every irreducible component has dimension exactly $\dim {\HG}$,
and in particular cannot be contained in the fiber of $\uZ^1(W_F^0/P_F,{\HG})$ over $\ell$ for any prime $\ell$.  By the unmixedness theorem, every associated prime of 
$\uZ^1(W_F^0/P_F,{\HG})$ has characteristic zero, so $\uZ^1(W_F^0/P_F,{\HG})$ is flat over $\OO[\frac{1}{p}]$ as claimed.
\end{proof}

Lemma~\ref{lemma:jordan} is a pointwise result about the order of $\sigma_x^{\sss}$, but it can be turned into a global statement.  Indeed, we will say that
an $R$-point of ${\HG}$ is unipotent if the corresponding map $\Spec R \rightarrow {\HG}$ factors through the unipotent locus on ${\HG}$.  If $R$ is reduced,
one can check this pointwise on $\Spec R$.  

\begin{prop} \label{prop:global unipotent}
  There exists an integer $M$, depending only on $\LG$, such that $\sigma^M$ is a unipotent
  element of $\HG$.  When $\LG = \GL_n$, one can take
$M = q^{n!} - 1$.
\end{prop}
\begin{proof}
We first prove this when $\LG = \GL_n$.  In this case Lemma~\ref{lemma:jordan}
shows that at each geometric point $x$ of $\uZ^1(W_F^0/P_F,{\HG})$,
the expression $(\sigma_x^{\sss})^{q^{n!} - 1}$ is equal to the identity.  In particular
$\sigma^{q^{n!} - 1}$ 
is an element of $\LG(R_{\LG})$ whose specialization at every geometric point $x$ of $\uZ^1(W_F^0/P_F,{\HG})$ is unipotent.  
On the other hand, by
~\cite{curtis}, Proposition 6.2, when $\LG = \GL_n$, $\uZ^1(W_F^0/P_F,{\HG})$ is
reduced.  Hence $\sigma^{q^{n!}-1}$, seen as a morphism $ \uZ^1(W_F^0/P_F,{\HG}) \rightarrow {\HG}$, 
factors through the unipotent locus of ${\HG}$ as claimed.
When $\LG$ is arbitrary, the result follows by choosing a faithful representation $\LG \rightarrow \GL_n$, and noting that the unipotent locus on $\LG$ is
the preimage of the unipotent locus on $\GL_n$.
\end{proof}

We will see in the next section that in fact $\uZ^1(W_F^0/P_F,{\HG})$ is reduced for all $\LG$; the argument above then shows that in fact
$\sigma^{e(q^{Nf} - 1)}$ is unipotent.

\subsection{A construction of Bellovin-Gee} \label{subsection:bellovin-gee}

The scheme $\uZ^1(W_F^0/P_F,{\HG})$ is very closely related to certain affine schemes studied by Bellovin-Gee in section 2 of~\cite{BG}.  More precisely, 
for any finite Galois 
extension $L/F$ they define a scheme $Y_{L/F,\phi,\CN}$ (\cite{BG},
Definition 2.1.2) parameterizing tuples 
$(\Phi,\CN,\tau)$ where $\Phi$ is an element of $\LG$, $\CN$ is a nilpotent element of $\Lie(\HG)$, and $\tau: I_{L/F} \rightarrow \LG$
is a homomorphism, that satisfy:
\begin{enumerate}
\item $\Ad(\Phi) \CN  = q\CN$,
\item For all $w \in I_{L/F}$, $\Phi \tau(w) \Phi^{-1} = \tau(w^q)$, and
\item For all $w \in I_{L/F}$, $\Ad(\tau(w)) \CN = \CN$.
\end{enumerate}

Let $Y^{\circ}_{L/F,\phi,\CN}$ denote the closed subscheme of $Y_{L/F,\phi,\CN}$ for which the images of $\Phi$ and $\tau(s)$ under the map $\LG \rightarrow W$
are $\Fr$ and $s$, respectively.  Then $Y^{\circ}_{L/F,\phi,\CN}$ is a union of connected components of $Y_{L/F,\phi,\CN}$.

We then have:
\begin{prop} Fix $M$ such that $\sigma^M$ is unipotent, and let $L/F$ be a finite, tamely ramified Galois extension whose ramification index is divisible by $M$.
Then there is a natural isomorphism $\uZ^1(W_F^0/P_F,{\HG})_{\overline{\QQ}_{\ell}} \rightarrow (Y^{\circ}_{L/F,\phi,\CN})_{\overline{\QQ}_{\ell}}$.
\end{prop}
\begin{proof}
We give maps in both directions that are inverse to each other.  On the one hand, without any hypotheses on $L/F$, there is always a map
$Y^{\circ}_{L/F,\phi,\CN} \rightarrow \uZ^1(W_F^0/P_F,{\HG})$ over
$\overline{\QQ}_{\ell}$ that takes a triple $(\Phi,\CN,\tau)$ to the $L$-homomorphism $\varphiL$ 
defined by $\varphiL(\Fr) = \Phi$ and $\varphiL(s) = \tau(s) \exp(\CN)$.
In the other direction, given a cocycle $\varphi$ we can set 
$\Phi = \varphiL(\Fr)$, $\CN = \frac{1}{M} \log (\varphiL(s)^M)$, and let $\tau: I_F
\rightarrow \LG(\overline{\QQ}_{\ell})$ be the map taking $s^a$ to $\varphiL(s)^a \exp(-a\CN)$; 
the latter factors through $I_{L/F}$ under our ramification condition on $L$.  These two maps are clearly inverse to each other.
\end{proof}

As this isomorphism involves exponentiation, and division by $M$, it does not
extend to the special fiber modulo small primes.  In fact the space $Y^{\circ}_{L/F,\phi,\CN}$ can be quite badly behaved at small primes: for instance,
if ${\HG} = \GL_2$, and we take $L/F$ to be a finite, tamely ramified Galois extension of ramification index $M$ divisible by $q^2-1$ (so that $\sigma^M$ is unipotent), then
at any prime $\ell$ dividing $q+1$ the fiber of $Y^{\circ}_{L/F,\phi,\CN}$ has dimension five, whereas the generic fiber has dimension four.  That is, $Y^{\circ}_{L/F,\phi,\CN}$ fails
to be flat in this setting.  One could attempt to remedy this by replacing $Y^{\circ}_{L/F,\phi,\CN}$ by the closure of its generic fiber, but even then, at primes $\ell$
as above, there is not a bijection between the irreducible components of $(Y^{\circ}_{L/F,\phi,\CN})_{\overline{\FF}_{\ell}}$ and those of $\uZ^1(W_F^0/P_F,{\HG})_{\overline{\FF}_{\ell}}$.
Indeed, one can verify that the irreducible components of the latter behave in a manner consistent with the $\ell \neq p$ Breuil-Mezard conjecture of Shotton~\cite{MR3769675},
whereas those of the former do not.

Bellovin-Gee show (\cite{BG}, Theorem 2.3.6) that $Y_{L/F,\phi,\CN}$ (and hence $\uZ^1(W_F^0/P_F,{\HG})$) is generically smooth, by constructing a smooth point on each irreducible
component of $Y_{L/F,\phi,\CN}$ in characteristic zero.  We sketch their construction here (or rather, its adaptation to $\uZ^1(W_F^0/P_F,{\HG})$), both in the interests of being self-contained and
because we will need it for other purposes.

Fix a prime $\ell\neq p$ and a $\overline{\QQ}_{\ell}$ point $\xi$ of $\LG$ in the image of the map $\uZ^1(W_F^0/P_F,{\HG}) \rightarrow \LG$ taking $\varphi$ to $^{L}\varphi(s)$.
As $\uZ^1(W_F^0/P_F,{\HG})_{\overline{\QQ}_{\ell}}$ is (set-theoretically) the union of the smooth schemes $X_{(\xi)}$ for such $\xi$, it suffices to construct a smooth point of
$\uZ^1(W_F^0/P_F,{\HG})$ on each connected component of $X_{(\xi)}$.

Let $\xi = \xi^{\sss} \xi^u$ be the Jordan decomposition of $\xi$.  Since we are in
characteristic zero $\xi^u$ is a unipotent element of $\HG$, and we may consider 
its logarithm $\CN$, which is a nilpotent element of the Lie algebra of the centralizer ${\hat
  G}_{\xi^{\sss}}$ of $\xi^{\sss}$. 

Let $\lambda$ be a cocharacter of ${\HG}^{\xi^{\sss}}$ that is an associated cocharacter of $\CN$, in the sense of~\cite{BG}, section 2.3.  In particular, for all $t$
we have $\Ad(\lambda(t))\CN = t^2\CN$.  Set $\Lambda = \lambda(q^{\frac{1}{2}})$ for some square root $q^{\frac{1}{2}}$ of $q$, so that $\Ad(\Lambda) \CN = q\CN$.  Then
$\Lambda \xi^{u} \Lambda^{-1} = (\xi^{u})^q$.

Further let $H$ denote the normalizer, in $\LG$, of the subgroup of $\LG$ generated by $\xi^{\sss}$.
Let $Y$ be the set of $g \in H$ such that $g \xi^{u} g^{-1} = (\xi^{u})^q$.  Note that in
particular the map $\uZ^1(W_F^0/P_F,{\HG}) \rightarrow \LG$ that takes $\varphi$ to
$\varphiL(\Fr)$  identifies $X_{\xi}$ with a union of connected components of $Y$.

On the other hand $Y = \Lambda \cdot (H \cap \LG^{\CN})$.  By~\cite{Bel16}, Proposition 4.9, the inclusion of
$H \cap \LG^{\CN} \cap \LG^{\lambda}$ into $H \cap \LG^{\CN}$ is a bijection on connected components, and by~\cite{Bel16}, Lemma 5.3 there is a point of finite order
on each connected component of $H \cap \LG^{\CN} \cap \LG^{\lambda}$.  Thus on each connected component of $Y$ there is a point of the form $\Lambda c$, where
$c$ has finite order and commutes with $\Lambda$.  Then Bellovin and Gee show, via a cohomology calculation, that when $(\Lambda c, \xi)$ lies in $\uZ^1(W_F^0/P_F,{\HG})$
it is a smooth point of $\uZ^1(W_F^0/P_F,{\HG})_{\overline{\QQ}_{\ell}}$.  We immediately deduce:

\begin{prop} \label{prop:reduced}
The scheme $\uZ^1(W_F^0/P_F,{\HG})$ is generically smooth (and therefore reduced.)
\end{prop}
\begin{proof}
Generic smoothness is immediate since there is a point of the form $(\Lambda c, \xi)$ on every connected component of $X_{(\xi)}$ for all $\xi$.  Since
$\uZ^1(W_F^0/P_F,{\HG})$ is a local complete intersection there is no embedded locus; that is, $\uZ^1(W_F^0/P_F,{\HG})$ is reduced.
\end{proof}

\begin{rem} We will later give an argument that in fact the fibers $\uZ^1(W_F^0/P_F,{\HG})_{\overline{\FF}_{\ell}}$ are generically smooth
outside of an explicit finite set.  This argument is independent of (though partially inspired by) the above argument of Bellovin-Gee, and certainly implies the
above proposition, as well as the separatedness results below.  We include the Bellovin-Gee argument here for convenience of exposition, and because the comparison
with their construction is interesting in its own right.
\end{rem}

\begin{prop} \label{prop:integral point}
For any prime $\ell \neq p$, and any irreducible component $Y$ of $\uZ^1(W_F^{0}/P_F,{\hat
  G})_{\overline{\QQ}_{\ell}}$, there exists a $\overline{\ZZ}_{\ell}$-point of  
$\uZ^1(W_F^0/P_F,{\HG})$ on $Y$.
\end{prop}
\begin{proof}
We have shown that $Y$ contains a point of the form
$(\Lambda c,\xi)$ constructed above.  We must show that this point is conjugate to a
$\overline{\ZZ}_{\ell}$-point. Note that $\Lambda$, $c$ and $\xi$ are contained in $\LG(L)$
for some  $L\subset\overline{\QQ}_{\ell}$ finite over  ${\QQ}_{\ell}$.
Now, since $\Lambda=\lambda(q^{\frac 12})$ for some cocharacter $\lambda$ and
since $q^{\frac{1}{2}}$ is an $\ell$-unit for all $\ell \neq p$,  the element $\Lambda$ is
compact in $\LG(L)$ (i.e. the subgroup of $\LG(L)$ generated by $\Lambda$ has compact closure).
Moreover, since $c$ has finite order and commutes to $\Lambda$,
the element $\Lambda c$ is also compact.
Therefore, since
$\Lambda c$ normalizes 
the subgroup of $\LG(L)$ generated by $\xi$, and some power of $\xi$
is unipotent, the subgroup of $\LG(L)$ generated by $\Lambda c$ and
$\xi$ has compact closure. Thus it  normalizes
a facet of the semisimple building $B(\HG,L)$ and fixes its barycenter $x$.
There is
a finite extension $L'$ of $L$ such that $x$  becomes an hyperspecial
point in $B(\HG,L')$ and is conjugate to the ``canonical'' 
hyperspecial point $o$ fixed by $\HG(\mathcal{O}_{L})$ under some element
$g\in\HG(L')$.
The fixator of $o$
in $\LG(L)$ is $Z_{\HG}(L).\LG(\mathcal{O}_{L})$, hence 
the $L$-homomorphism $\varphiL:\,W_{F}^{0}/P_{F}\To{} \LG(\overline\QQ_{\ell})$  associated to the pair
$(^{g}(\Lambda c), {^{g}\xi})$ takes values in
$Z_{\HG}(\overline\QQ_{\ell}).\LG(\overline{\ZZ}_{\ell})$. Consider its composition with the
quotient map to
$(Z_{\HG}(\overline\QQ_{\ell}).\LG(\overline{\ZZ}_{\ell}))/\HG(\overline\ZZ_{\ell})=Q\rtimes
W$ with 
$Q:=(Z_{\HG}(\overline\QQ_{\ell}).\HG(\overline{\ZZ}_{\ell}))/\HG(\overline{\ZZ}_{\ell})=Z_{\HG}(\overline\QQ_{\ell})/Z_{\HG}(\overline\ZZ_{\ell})$. 
Since it  has relatively compact image and $Q$ is discrete, it factors over a
finite quotient $W'$ of $W_{F}^{0}/P_{F}$. But since $Q$ is a $\QQ$-vector
space of finite dimension, we have $H^{1}(W',Q)=\{1\}$, so the above composition is
conjugate, under some element of $q\in Q$, to the trivial $L$-homomorphism
$W_{F}^{0}\To{}Q\rtimes W$. So if $\tilde q$ is any lift of $q$ in
$Z_{\HG}(\overline\QQ_{\ell})$, the conjugate $^{\tilde q}(\varphiL)$ associated to the pair
$(^{\tilde qg}(\Lambda c), ^{\tilde qg}\xi)$ is 
$\LG(\overline\ZZ_{\ell})$-valued, as desired.
\end{proof}

\begin{cor} \label{cor:separatedness}
For any prime $\ell \neq p$, the ring $R_{\LG}$ is $\ell$-adically separated.
\end{cor}
\begin{proof}
Since $R_{\LG}$ is reduced and flat over $\OO$, we have an embedding:
$$R_{\LG} \rightarrow \prod_Y \OO_Y,$$
where $Y$ runs over the irreducible components of $R_{\LG}$.  Each $\OO_Y$ is affine, integral, and flat over $\OO$, and by Proposition~\ref{prop:integral point} contains
an integral point.  In particular $\ell$ is not invertible on $\OO_Y$.  Thus it suffices
to show that any noetherian integral flat $\ZZ_{\ell}$-algebra $A$ in which $\ell$ is not invertible
is $\ell$-adically separated.  Indeed, suppose $a$ is a nonzero element of $A$ in the intersection of the ideals generated by $\ell^i$.  Then for each $i$, there is an $a_i \in A$
such that $\ell^i a_i = a$.  Each $a_i$ is unique since $A$ is integral, so $a_{i-1} = \ell a_i$.  Since the ascending chain of ideals generated by the $a_i$ stabilizes,
we have $a_i = u a_{i-1}$ for some unit $u$ and integer $i$.  Then, as $A$ is integral, we have $u\ell = 1$, contradicting the fact that $\ell$ is not invertible in $A$.
\end{proof}

\subsection{The universal family}
Now that we have shown that $R_{\LG}$ is $\ell$-adically separated, we return to the question of extending the parameter $\varphi_{\univ}$
to an $L$-homomorphism defined on all of $W_F$.   As we have already remarked, this is
only possible after tensoring with the completed local ring $\OO_v$ for some 
finite place $v$ of $\OO$ of residue characteristic $\ell \neq p$.
The key point is the following notion of continuity, first introduced in~\cite{curtis} in the case $\LG = \GL_n$:

\begin{defn} \label{def:continuity}
Let $R$ be a Noetherian $\OO[\frac{1}{p}]$-algebra, and let $\rho: W_F \rightarrow \LG(R)$ be a group homomorphism.  We say that $\rho$ is
{\em $\ell$-adically continuous} if one of the following two conditions hold:
\begin{enumerate}
\item The ring $R$ is $\ell$-adically separated, and for each $n > 0$, the preimage of $U_n$ under $\rho$ is open in $W_F$, where $U_n$ is the kernel of the map $\LG(R) \rightarrow \LG(R/\ell^n R)$.
\item There exists a Noetherian, $\ell$-adically separated $\OO[\frac{1}{p}]$-algebra $R'$, a map $f: R' \rightarrow R$,
and an $\ell$-adically continuous map $\rho': W_F \rightarrow \LG(R')$ such that $\rho = f \circ \rho'$.
\end{enumerate}
\end{defn}

If $R$ is $\ell$-adically separated and condition (2) in the above definition holds, it is easy to check that condition (1) holds as well, so the two conditions are consistent with each other.
We will say that a cocycle  $\varphi\in Z^1(W_F,{\HG}(R))$ is $\ell$-adically
continuous if its associated $L$-homomorphism $\varphiL$ is $\ell$-adically continuous as in the above definition.

\begin{thm} \label{thm:universal family}
For each finite place $v$ of $\OO$ of residue characteristic $\ell \neq p$, there exists a
unique $\ell$-adically continuous cocycle
$$\varphi_{\univ,v}: W_F/P_F \rightarrow \HG(R_{\LG} \otimes_{\OO} \OO_v)$$
whose restriction to $(W_F/P_F)^0$ is equal to $\varphi_{\univ}$.  Moreover, if $R$ is any
Noetherian $\OO_v$-algebra, and 
$\varphi: W_F/P_F \rightarrow \HG(R)$ is an $\ell$-adically continuous cocycle,
then there is a unique map: $f: R_{\LG} \otimes_{\OO} \OO_v \rightarrow R$ such that
$\varphi = f\circ \varphi_{\univ,v}$. 
\end{thm}
\begin{proof}
When $\LG = \GL_n$, this is proved in~\cite{curtis}, Proposition 8.2; we reduce to this case.  Choose a
faithful representation $\tau: \LG \rightarrow \GL_n$ defined over $\OO[\frac 1 p]$.  Then 
$$\tau \circ \varphiL_{\univ}\in \Hom(W_F^0/P_F,\GL_n(R_{\LG}))=Z^1(W_F^0/P_F,\GL_n(R_{\LG}))$$ where $\GL_{n}$ is equipped
with the trivial action of $W_{F}$.
There is thus
a unique map $f: R_{\GL_n} \rightarrow R_{\LG}$ that takes the universal cocycle on
$\uZ^1(W_F^0/P_F,\GL_n)$ (actually a homomorphism) to $\tau \circ \varphiL_{\univ}$.
Since this universal cocycle extends to an $\ell$-adically continuous cocycle on
$W_F/P_F$, with values in $R_{\GL_n} \otimes_{\OO} \OO_v$, 
composing this extension with $f$ gives
an extension of $\tau \circ \varphiL_{\univ}$ to an $\ell$-adically continuous
homomorphism $W_F/P_F\longrightarrow \GL_n(R_{\LG} \otimes_{\OO} \OO_v)$.
Denote this homomorphism by $\varphiL_{\univ,v}$. Its restriction to $W_F^0/P_F$ factors through
$\LG(R_{\LG} \otimes_{\OO} \OO_v)$ and is equal
 to $\varphiL_{\univ}$, so it only remains to prove that $\varphiL_{\univ,v}$ factors through
 $\LG(R_{\LG} \otimes_{\OO} \OO_v)$ too. But this follows from the $\ell$-adic
 separatedness of $R_{\LG} \otimes_{\OO} \OO_v$ and the fact that for each
 $n\in\NM$, we know that the image of  $\varphiL_{\univ,v}(W_{F})$ in $\GL_{n}(R_{\LG}
 \otimes_{\OO} \OO_v/(\ell^{n}))$ coincides 
 with the image of $\varphiL_{\univ,v}(W_{F}^{0})$, which is contained in 
 $\LG(R_{\LG} \otimes_{\OO} \OO_v/(\ell^{n}))$.
 Uniqueness and the  universal property are now straightforward.
\end{proof}

In light of this, we define a ``good coefficient ring'' to be a Noetherian ring $R$ that
is an $\OO \otimes \ZZ_{\ell}$-algebra for some $\ell \neq p$, and a ``good coefficient 
field'' to be a good coefficient ring that is also a field.  Theorem~\ref{thm:universal
  family} then implies that for any good coefficient ring $R$, 
and any cocycle $\varphi^0: (W_F/P_F)^0 \rightarrow \HG(R)$, there is a unique $\ell$-adically continuous cocycle $\varphi: W_F/P_F \rightarrow \HG(R)$ extending $\varphi^0$.

In particular, if $R$ is a complete local $\OO$-algebra with maximal ideal ${\mathfrak
  m}$, of residue characteristic $\ell\neq p$,
then any $\ell$-adically continuous cocycle $\varphi: W_F/P_F \rightarrow \HG(R)$ is
clearly ${\mathfrak m}$-adically continuous.  Conversely, given 
an ${\mathfrak m}$-adically continous cocycle $\varphi: W_F/P_F \rightarrow \HG(R)$, Theorem \ref{thm:universal family} shows
that there is a unique $\ell$-adically continuous cocycle $\varphi'$ extending the restriction of $\varphi$ to $(W_F/P_F)^0$.  Then $\varphi'$ and $\varphi$ are both
${\mathfrak m}$-adically continuous and agree on $(W_F/P_F)^0$, so $\varphi$ is also $\ell$-adically continuous.  Thus the notions of $\ell$-adic and ${\mathfrak m}$-adic
continuity coincide for cocycles valued in $R$.

\section{Reduction to tame parameters}

In this section, we broaden the setting as follows. We  consider
a split reductive group scheme $\hat G$ over $\ZM[\frac 1p]$ endowed with a
finite action of $W_{F}$, but \emph{we no longer assume that this
  action is tame, nor that it stabilizes a Borel pair.}

For any $\ZM[\frac 1p]$-algebra $R$, we denote by $Z^{1}(W_{F}, \HG(R))$ the set of
$1$-cocycles which are continuous for the natural topology of the source and the discrete
topology on the target. We use similar notation for $W_{F}^{0}$ and any closed subgroup thereof.
Recall that the topology on $W_{F}^{0}$ is such that $P_{F}$, with its
natural topology, sits as a closed and open subgroup.

It will be handy to switch between $1$-cocycles and their associated $L$-morphisms. In
this regard, we usually denote by $\LG$ a group scheme of
the form $\hat G\rtimes W$ with $W$ any \emph{finite} quotient of $W_{F}$ through which the given
action on $\HG$ factors. Note that $W$ may be allowed to change according to our needs,
but we prefer to keep it finite in order to work with algebraic group schemes. For the
sake of clarity, we will most often distinguish a $1$-cocycle $\varphi$ from its
associated  $L$-homomorphism
 $\varphiL:=\varphi\rtimes\id :W_{F}\To{}\LG(R)$, although occasionally it will be 
 more handy to write $\varphi$ for the $L$-homomorphism.

\subsection{Overview} Our aim is to show how the study of moduli of  $1$-cocycles
$W_{F}^{0}\To{} \hat G$  (and subsequently, moduli of $\ell$-adically continuous $1$-cocycles 
$W_{F}\To{} \hat G$) can be reduced to the particular case considered
in the previous section, namely the case of tame $1$-cocycles valued in a
reductive group scheme with a tame Galois action that stabilizes a
Borel pair. 
The principle is very simple ; suppose $R$ is a $\ZM[\frac 1p]$-algebra
and $\varphi : W_{F}^{0}\To{} {\hat G}(R)$ is a $1$-cocycle, and denote by
$\phi : P_{F}\To{} \hat G(R)$ its restriction to $P_{F}$. 
Then the
conjugation action of $W_{F}^{0}$ on $\hat G(R)$ through $^{L}\varphi$
stabilizes the centralizer
$C_{\hat G(R)}(\phiL(P_{F}))$ and the restricted action on this subgroup factors
over  $W_{F}^{0}/P_{F}$. Denoting  by ${\rm Ad}_{\varphi}$ this action, an elementary
computation shows that the map
$\eta\mapsto \eta\cdot\varphi$ sets up a bijection
$$ Z^{1}_{{\rm Ad}_{\varphi}}(W_{F}^{0}/P_{F},C_{\hat G(R)}(\phiL(P_{F}))) \To{\sim}
\{\varphi' \in Z^{1}(W_{F}^{0},\hat G(R)),\, \varphi'_{|P_{F}}=\phi\}.$$
By Lemma \ref{hom_smooth} in the appendix, the functor on $R$-algebras $R'\mapsto C_{\hat G(R')}(\phiL(P_{F}))$ is
representable by a smooth group scheme over $R$ that we denote by
$C_{\hat G}(\phi)$. Moreover, by \cite[Thm 2.1]{PY}, its connected
geometric fibers are reductive.
Therefore, one is tempted to
see the set $Z^{1}_{{\rm Ad}_{\varphi}}(W_{F}^{0}/P_{F},C_{\hat G(R)}(\phiL(P_{F})))$ as an
instance of the type of tame parameters that were studied in the previous section.
However, making this idea work requires addressing the following issues :
\begin{itemize}
\item The group scheme $C_{\hat G}(\phi)$ may have non-connected fibers.
\item Its neutral component $C_{\hat G}(\phi)^{\circ}$ may not be split.
\item The action ${\rm Ad}_{\varphi}$ may neither be finite nor preserve a Borel pair of $C_{\hat G}(\phi)^{\circ}$. 
\end{itemize}

In order to address these issues, the first step is to find a nice set
of representatives of conjugacy classes of
continuous cocycles with source
$P_{F}$. Since we prefer to 
work with finitely presented objects, we choose a
  decreasing sequence $(P_{F}^{e})_{e\in\mathbb{N}}$ of open normal
 subgroups of $P_{F}$
 whose intersection is $\{1\}$. Then  we fix 
$e\in\NM$ such that $P_{F}^{e}$ acts trivially on $\HG$, and we restrict
attention to cocycles that are trivial on $P_{F}^{e}$. 
The following theorem follows from  Theorems \ref{representatives},
\ref{pi0fini} and  Proposition \ref{split_red_gp} in the appendix.
\begin{thm}\label{wild_type_recap}
  There is a number field $K_{e}$ and a finite  set
  $$\Phi_{e}\subset Z^{1}\left({P_{F}}/{P_{F}^{e}},\hat
    G\left(\mathcal{O}_{K_{e}}[{\textstyle\frac 1p}]\right)\right), \,\,\,\,\hbox{   such that}$$
  \begin{enumerate}
  \item For any $\mathcal{O}_{K_{e}}[\frac 1p]$-algebra $R$, any cocycle
    $\phi:P_{F}/P_{F}^{e}\To{} \hat G(R)$ is \'etale-locally $\hat G$-conjugate to a
    locally unique $\phi_{0}\in \Phi_{e}$.
  \item For any $\phi\in \Phi_{e}$, the reductive group scheme $C_{\hat G}(\phi)^{\circ}$
    is split over $\mathcal{O}_{K_{e}}[\frac 1p]$ and the component group
    $\pi_{0}(\phi):=\pi_{0}(C_{\hat   G}(\phi))$ is constant.
  \end{enumerate}
\end{thm}


\subsection{Some definitions and constructions} \label{sec:some-defin-constr}
Let $\phi\in \Phi_{e}$. For any $\mathcal{O}_{K_{e}}[\frac 1p]$-algebra $R$ we denote by
$Z^{1}(W_{F}^{0},\hat G(R))_{\phi}$ the set of $1$-cocycles
$W_{F}^{0}\To{} \hat G(R)$ that extend $\phi$. 
The functor $R\mapsto Z^{1}(W_{F}^{0},\hat G(R))_{\phi}$ is visibly representable by an
affine scheme of finite type over $\mathcal{O}_{K_{e}}[\frac 1p]$, namely a closed subscheme of
$\hat G\times\hat G$. We denote this scheme by $\uZ^{1}(W_{F}^{0},\hat G)_{\phi}$.

\begin{defn}
  An element  $\phi\in\Phi_{e}$ is called \emph{admissible} if the scheme
  $\uZ^{1}(W_{F}^{0},\hat G)_{\phi}$ is not empty.
\end{defn}

In the sequel, it will be convenient to choose our ``$L$-group''  $\LG$  in the form $\LG=\hat G\rtimes
W_{e}$ where $W_{e}$ is a finite quotient of $W_{F}$ into which $P_{F}/P_{F}^{e}$ maps
\emph{injectively}.
For example, we may choose our sequence $(P_{F}^{e})_{e}$ such that
$P_{F}^{e}=P_{F_{e}}$ for some Galois extension $F_{e}$ of $F$ and put
$W_{e}={\rm Gal}(F_{e}/F)$.
Then the $L$-homomorphism $\varphiL$ associated to $\varphi\in
Z^{1}(W_{F}^{0},\hat G(R))_{\phi}$
factors through the subgroup\footnote{Note that, despite the notation,
  this subgroup is not the centralizer of $^{L}\phi$ in $\LG$.}
$$ C_{\LG(R)}(\phi):=\left\{(g,w)\in \LG(R),\,
  (g,w)\phiL(w^{-1}pw)(g,w)^{-1}=\phiL(p), \forall p\in P_{F}\right\}.$$
Writing the functor $C_{\LG}(\phi):\,R\mapsto C_{\LG(R)}(\phi)$  on
$\mathcal{O}_{K_{e}}[\frac 1p]$-algebras 
 as a disjoint union $\bigsqcup_{w\in W_{e}} T_{\hat G}({^{w}\phi},\phi)$ of
transporters in $\hat G$ (where $^{w}\phi$ is defined by ${^{w}\phi}(p)=w(\phi(w^{-1}pw))$), we see from Lemma \ref{hom_smooth}  that this functor is represented by a smooth group
scheme that sits in an exact sequence
$$1\to C_{\hat G}(\phi)\to C_{\LG}(\phi)\to W_{e}.$$
Actually, it follows from the uniqueness of $\phi_{0}$ in i) of Theorem
\ref{wild_type_recap} that $T_{\hat G}({^{w}\phi},\phi)$ is either
empty or is a $C_{\hat G}(\phi)$-torsor for the \'etale topology. Therefore, $C_{\LG}(\phi)$ is an extension of the
\emph{constant} subgroup $W_{e,\phi}:=\{w\in W_{e},T_{\hat
  G}({^{w}\phi},\phi)\neq\emptyset\}$ of $W_{e}$ by $C_{\hat G}(\phi)$. 
Since $C_{\LG}(\phi)^{\circ}=C_{\hat G}(\phi)^{\circ}$ is a split reductive
group scheme over $\mathcal{O}_{K_{e}}[\frac 1p]$, we know by general
results \cite[Prop. 3.1.3]{conrad_luminy} that 
$$\tilde\pi_{0}(\phi):=\pi_{0}(C_{\LG}(\phi))$$ 
is a separated \'etale group scheme over $\mathcal{O}_{K_{e}}[\frac 1p]$. Since it is an
extension of $W_{e,\phi}$ by $\pi_{0}(\phi)$, we see that $\tilde\pi_{0}(\phi)$ is
actually \emph{finite} \'etale. \emph{Therefore, after maybe enlarging $K_{e}$, we may
  assume that $\tilde\pi_{0}(\phi)$ is constant over $\mathcal{O}_{K_{e}}[\frac 1p]$.}
Now, let us assume that $\phi$ is admissible. Then we have $W_{e,\phi}=W_{e}$ and
an exact sequence of abstract groups
$$1\to \pi_{0}(\phi)\to \tilde\pi_{0}(\phi)\to W_{e}\to 1.$$
Therefore, the affine scheme $\uZ^{1}(W_{F}^{0},\hat G)_{\phi}$ decomposes as a disjoint union
$$ \uZ^{1}(W_{F}^{0},\hat G)_{\phi} =\coprod_{\alpha\in \Sigma(\phi)}
\uZ^{1}(W_{F}^{0},\hat G)_{\phi,\alpha}   \hbox{, where}$$
\begin{itemize}
\item $\Sigma(\phi)$ denotes the set of homomorphisms $W_{F}\To{}\tilde\pi_{0}(\phi)$ that
  extend the map $P_{F}\To{}\tilde\pi_{0}(\phi)$ given by the
  composition of $\phiL$ with the projection to
  $\tilde\pi_{0}(\phi)$, and whose composition with
  $\tilde\pi_{0}(\phi)\To{} W_{e}$ is the natural projection $W_{F}\To{}W_{e}$. 
\item $\uZ^{1}(W_{F}^{0},\hat G)_{\phi,\alpha}(R)=Z^{1}(W_{F}^{0},\hat
  G(R))_{\phi,\alpha}$ is the subset of extensions $\varphi$ of $\phi$
  such that the composition of $\varphiL$ with the projection to $\tilde\pi_{0}(\phi)$ is $\alpha$.
\end{itemize}

\begin{defn} \label{sigma_admissible}
  We will say that $\alpha\in\Sigma(\phi)$ is \emph{admissible} if the scheme
  $\uZ^{1}(W_{F}^{0},\hat G)_{\phi,\alpha}$ is not empty.
\end{defn}

Observe that there are only finitely many admissible elements in
$\Sigma(\phi)$ since the scheme $\uZ^{1}(W_{F}^{0},\hat G)_{\phi}$ has
finitely many connected components.

We now note that two elements $\varphi,\varphi'\in
Z^{1}(W_{F}^{0},\hat G(R))_{\phi,\alpha}$ differ by a tame cocycle valued in $C_{\hat
  G}(\phi)^{\circ}$ (beware the $\circ$). More precisely, if we write $\varphi'(w)=\eta(w)
\varphi(w)$, then 
$w\mapsto \eta(w)$ belongs to $Z^{1}_{{\rm Ad}_{\varphi}}(W_{F}^{0}/P_{F},C_{\hat
  G}(\phi)^{\circ}(R))$. 
 In other words, the map $\eta\mapsto \eta\cdot\varphi$ sets up an
isomorphism of $R$-schemes 
$$ \uZ^{1}_{{\rm Ad}_{\varphi}}(W_{F}^{0}/P_{F},C_{\hat G}(\phi)^{\circ})_{R} \To{\sim}
\uZ^{1}(W_{F}^{0},\hat G)_{\phi,\alpha,R}.$$
At this point we have dealt with the first two issues mentioned in the beginning of this
section. The next result deals with the third issue and will allow us to reduce to the tame
parameters that were studied in the previous section.

\begin{thm} \label{Borel_preserving}
  There is a finite extension $K'_{e}$ of $K_{e}$ such that for any admissible
  $\phi\in\Phi_{e}$ and any admissible $\alpha\in\Sigma(\phi)$, there
  is some $\varphi_{\alpha} \in Z^{1}(W_{F}^{0},\hat G(\mathcal{O}_{K'_{e}}[\frac 1p]))_{\phi,\alpha}$ 
such that $\varphiL_{\alpha}(W_{F}^{0})$ is finite
and $\Ad_{\varphi_{\alpha}}$ preserves a Borel pair of the
split reductive group scheme  $C_{\hat G}(\phi)^{\circ}$.
\end{thm}

Fix $\phi,\alpha$ and $\varphi_{\alpha}$ as in the theorem. Since
$\varphiL_{\alpha}(W_{F}^{0})$ is finite, 
$\varphi_{\alpha}$ extends canonically to $W_{F}$ with 
$\varphiL_{\alpha}(W_{F})=\varphiL_{\alpha}(W_{F}^{0})$. So the conjugation action ${\rm
  Ad}_{\varphi_{\alpha}}$ of $W_{F}^{0}$ on  
the reductive group $C_{\hat G}(\phi)^{\circ}$ extends to a finite action of $W_{F}$, and
it has to be trivial on $P_{F}$. 
Since this action stabilizes a Borel pair, 
we see that the $\mathcal{O}_{K'_{e}}[\frac 1p]$-scheme
$\uZ^{1}_{{\rm Ad}_{\varphi}}(W_{F}^0/P_{F},C_{\hat G}(\phi)^{\circ})$
is (a base change of) an instance of those tame moduli schemes studied
in Section 2.

\begin{rem}
  It is natural to ask whether we can find $\varphi_{\alpha}$ so that
  ${\rm Ad}_{\varphi_{\alpha}}$ preserves a pinning
  of $C_{\hat G}(\phi)^{\circ}$. Our techniques can achieve this when the center of
  $C_{\hat G}(\phi)^{\circ}$ is smooth, see Remark \ref{rk_epinglage}. In Theorem
  \ref{prop:pinning_preserving}, we give a sufficient condition on $\HG$ for each $C_{\hat
    G}(\phi)^{\circ}$ to have smooth center.
\end{rem}

Before we can prove the theorem, we need some preparation.
Let us fix a Borel pair $\mathcal{B}_{\phi}=(B_{\phi},T_{\phi})$ in
$C_{\hat G}(\phi)^{\circ}$ 
and let us denote by  $\mathcal{T}_{\phi}$  the normalizer in $C_{\LG}(\phi)$ of this Borel
pair. By \cite[Prop. 2.1.2]{conrad_luminy}, this is again a smooth group scheme over $\mathcal{O}_{K_{e}}[\frac 1p]$. Since the
normalizer of a Borel pair in a connected reductive group over an algebraically closed field is the torus of the Borel
pair, we have $(\mathcal{T}_{\phi})^{\circ}=C_{\hat G}(\phi)^{\circ}\cap
\mathcal{T}_{\phi}=T_{\phi}$. Since any two Borel pairs in a connected reductive group  over an algebraically closed field
are conjugate, we also have $\pi_{0}(\mathcal{T}_{\phi})= \pi_{0}(C_{\LG}(\phi))=\tilde\pi_{0}(\phi)$. 
Moreover, since $T_{\phi}$ is abelian, the
conjugation action of $\mathcal{T}_{\phi}$ on $T_{\phi}$ factors through an action
$$ \tilde\pi_{0}(\phi)\To{}{\rm Aut}_{\mathcal{O}_{K_{e}}[\frac 1p]-gp.sch.}(T_{\phi}).$$
In particular, any section $\alpha\in\Sigma(\phi)$ provides us with an action of $W_{F}$
on the  torus $T_{\phi}$. This action has to be trivial on $P_{F}$,
since $^{L}\phi(P_{F})$ centralizes $C_{\HG}(\phi)$, so that $^{L}\phi(P_{F})\subset
\mathcal{T}_{\phi}(\mathcal{O}_{K_{e}}[\frac 1p])$ acts trivially on
$T_{\phi}$ by conjugation. Therefore, the subset
\begin{eqnarray*}
  \Sigma(W_{F}^{0},\mathcal{T}_{\phi}(R))_{\phi}
  &:=&\{\varphi\in Z^{1}(W_{F}^{0},\hat G(R))_{\phi},\,\varphiL(W_{F}^{0})\subset \mathcal{T}_{\phi}(R))\}\\
  &=&\{\varphi\in Z^{1}(W_{F}^{0},\hat G(R))_{\phi}, \,{\rm Ad}_{\varphi} \hbox{ preserves }
      \mathcal{B}_{\phi}\}
\end{eqnarray*}
decomposes as a disjoint union
$$ \Sigma(W_{F}^{0},\mathcal{T}_{\phi}(R))_{\phi} = \bigsqcup_{\alpha\in\Sigma(\phi)}
\Sigma(W_{F}^{0},\mathcal{T}_{\phi}(R))_{\phi,\alpha}$$
where $\Sigma(W_{F}^{0},\mathcal{T}_{\phi}(R))_{\phi,\alpha}$ denotes
the subset of those $\varphi\in \Sigma(W_{F}^{0},\mathcal{T}_{\phi}(R))_{\phi}$ such that
the composition $W_{F}^{0} \To{^{L}\varphi} \mathcal{T}_{\phi}(R)\To{}\tilde\pi_{0}(\phi)$ is $\alpha$.
Note that  $\Sigma(W_{F}^{0},\mathcal{T}_{\phi}(R))_{\phi,\alpha}$ is either empty or is a
principal homogeneous set under the abelian group
$Z^{1}_{\alpha}(W_{F}^0/P_{F},T_{\phi}(R))$. Varying $R$, we get a closed affine subscheme 
$\uSigma(W_{F}^{0},\mathcal{T}_{\phi})_{\phi}$ of $\uZ^{1}(W_{F}^{0},\hat G)_{\phi}$ which
decomposes as a coproduct of affine  $\mathcal{O}_{K_{e}}[\frac 1p]$-schemes
$$ \uSigma(W_{F}^{0},\mathcal{T}_{\phi})_{\phi} = \bigsqcup_{\alpha\in\Sigma(\phi)}
\uSigma(W_{F}^{0},\mathcal{T}_{\phi})_{\phi,\alpha}$$
where each $\uSigma(W_{F}^{0},\mathcal{T}_{\phi})_{\phi,\alpha}$ carries an action of the
abelian group $\mathcal{O}_{K_{e}}[\frac 1p]$-scheme
$\uZ^{1}_{\alpha}(W_{F}^0/P_{F},T_{\phi})$, and is a  
\emph{pseudo-torsor} for this action, in the sense of [The Stacks
Project, Tag 0497].

Finally, let $W$ be a finite quotient of $W_{F}$ such that $^{L}\phi$
factors over the image $P\subset W$ of $P_{F}$ in $W$ and  $\alpha$
factors over $W$.  Then the same
definitions as above provide us with a $\mathcal{O}_{K_{e}}[\frac 1p]$-scheme
$\uSigma(W,\mathcal{T}_{\phi})_{\phi,\alpha}$, 
which is a pseudo-torsor for the natural action of the group
$\mathcal{O}_{K_{e}}[\frac 1p]$-scheme
$\uZ^{1}_{\alpha}(W/P,T_{\phi})$.


\begin{thm} Suppose $\phi$ and $\alpha$ are admissible. \label{thm_reduction}
  \begin{enumerate}
  \item $\uZ^{1}_{\alpha}(W_{F}^0/P_{F},T_{\phi})$ is a diagonalisable group scheme
    over $\mathcal{O}_{K_{e}}[\frac 1p]$.
  \item  $\uSigma(W_{F}^{0},\mathcal{T}_{\phi})_{\phi,\alpha}$ is a fppf \emph{torsor} under
    $\uZ^{1}_{\alpha}(W_{F}^0/P_{F},T_{\phi})$.
  \end{enumerate}
Moreover, these two statements still hold with $W_{F}^{0}$ replaced by
a sufficiently large finite quotient $W$ as above.
\end{thm}

Before we prove this result, let us see how it implies Theorem
\ref{Borel_preserving}. The claim in Theorem \ref{Borel_preserving}
is that there exists an extension $K'_{e}$ of $K_{e}$ such that 
$\uSigma(W_{F}^{0},\mathcal{T}_{\phi})_{\phi,\alpha}$ has an
${\mathcal O}_{K'_{e}}[\frac 1p]$-point $\varphi$ with finite image.
In other words, we need to show the existence of a finite quotient $W$
of $W_{F}$ such that $\uSigma(W,\mathcal{T}_{\phi})_{\phi,\alpha}$ has an
${\mathcal O}_{K'_{e}}[\frac 1p]$-point.
So, from Theorem \ref{thm_reduction}, it suffices to show that any fppf torsor under a
diagonalisable group over $\mathcal{O}_{K_{e}}[\frac 1p]$ becomes trivial over
$\mathcal{O}_{K'_{e}}[\frac 1p]$ for some finite extension $K'_{e}$. Since a diagonalisable group
is a product of copies of $\mathbb{G}_{m}$ and $\mu_{m}$'s, we may treat each of these
groups separately. As long as $\mathbb{G}_{m}$ is concerned, since any fppf
$\mathbb{G}_{m}$-torsor is also an \'etale $\mathbb{G}_{m}$-torsor, it suffices to take $K'_{e}$
equal to the Hilbert class field $K_{e}^{h}$ of $K_{e}$. On the other hand, when base
changed to $\mathcal{O}_{K_{e}^{h}}[\frac 1p]$, a $\mu_{m}$-torsor is given as
the torsor of $m$-th roots of some element $f\in \mathcal{O}_{K_{e}^{h}}[\frac 1p]^{\times}$, because of the
exact sequence
$$  \mathcal{O}_{K_{e}}{\textstyle[\frac 1p]}^{\times}\To{(.)^{m}}\mathcal{O}_{K_{e}}{\textstyle[\frac 1p]}^{\times} \To{}
H^{1}_{fppf}(S,\mu_{m})\To{}
H^{1}_{fppf}(S,\mathbb{G}_{m})\To{(.)^{m}} H^{1}_{fppf}(S,\mathbb{G}_{m})
$$
where $S$ denotes $\Spec(\mathcal{O}_{K_{e}}[\frac 1p])$.
Thus we can take  for $K'_{e}$ a splitting field of $X^{m}-f$ over $K_{e}^{h}$ in this case.

\begin{proof}
 (1) Consider the map $Z^{1}_{\alpha}(W_{F}^0/P_{F},T_{\phi}(R))\To{}
 T_{\phi}(R)\times T_{\phi}(R)$ that sends a $1$-cocycle $\eta$ to the pair of elements
 $(\eta(\Fr),\eta(s))$. It identifies $Z^{1}_{\alpha}(W_{F}^0/P_{F},T_{\phi}(R))$
 with the subset of elements $(F,\sigma)$ in $T_{\phi}(R)\times T_{\phi}(R)$ defined by the equation
$$  F\cdot \alpha(\Fr)(\sigma)\cdot \alpha(s)^{q}(F)^{-1} = \sigma\cdot \alpha(s)(\sigma)\cdots
\alpha(s^{q-1})(\sigma).$$ 
This identifies in turn $\uZ^{1}_{\alpha}(W_{F}^0/P_{F},T_{\phi})$ with the kernel of the morphism
of group schemes $T_{\phi}\times T_{\phi} \To{} T_{\phi}$ defined by the ratio of both
sides of the equation. But a kernel of a morphism of diagonalisable groups is diagonalisable.
Further, let $W=W_{F}/W_{F'}$ be a finite quotient of $W_{F}$ for a
Galois extension $F'$ such that $^{L}\phi_{|P_{F'}}$ and
$\alpha_{|W_{F'}}$ are trivial. Then $\uZ^{1}_{\alpha}(W/P,T_{\phi})$
is the kernel of the natural restriction map
$\uZ^{1}_{\alpha}(W_{F}^0/P_{F},T_{\phi})\To{}\uZ^{1}_{\alpha}(W_{F'}^0/P_{F'},T_{\phi})$,
hence is a diagonalisable group too.

(2)  We already know that  $\uSigma(W_{F}^{0},\mathcal{T}_{\phi})_{\phi,\alpha}$ is
 finitely presented over $\mathcal{O}_{K_{e}}[\frac 1p]$, so it remains to find
a faithfully flat $\mathcal{O}_{K_{e}}[\frac 1p]$-algebra $R$ such that
$\Sigma(W_{F}^{0},\mathcal{T}_{\phi}(R))_{\phi,\alpha}$ is not empty.
We will actually exhibit an $R$ and a $\varphi\in
\Sigma(W_{F}^{0},\mathcal{T}_{\phi}(R))_{\phi,\alpha}$ such that
$^{L}\varphi$ has finite image. This will also show that the last
statement of the theorem holds for any  finite
quotient $W$ over which this $^{L}\varphi$ factors.

\medskip
\emph{Warning :} for the sake of readibility, \textbf{we will
omit the $^{L}$ from our usual notation for $L$-morphisms in the
remainder of this proof.} It should
not create any ambiguity since we will not have to consider their
associated $1$-cocycles anyway.

\medskip

\emph{Existence of a point over a closed geometric point.}
By the admissibility assumption, the scheme $\uZ^{1}(W_{F}^{0},\hat
G)_{\phi,\alpha}$ is not empty. Since it has finite 
presentation over $\mathcal{O}_{K_{e}}[\frac 1p]$, Chevalley's constructibility theorem
ensures that it has a non-empty closed fiber, which in turn ensures that it has a point
with finite residue field $k$ of characteristic  $\neq p$.
Note that the associated $L$-morphism $W_{F}^{0}\To{}\LG(k)$ has to factor over a finite
quotient of $W_{F}^{0}$, hence it is continuous for the topology of
$W_{F}^{0}$ induced by the usual topology on $W_{F}$, and the discrete topology on $\LG(k)$. 
Therefore, Proposition \ref{Borel_preserving_acf} below ensures that
$\Sigma(W_{F}^{0},\mathcal{T}_{\phi}(\bar k))_{\phi,\alpha}$ is not empty. Pick a point in
this set and let $\bar\varphi : W_{F}^{0}\To{} \mathcal{T}_{\phi}(\bar k)$  be the
$L$-morphism corresponding to this point.
Note that $\bar\varphi$ also has to factor through a finite
quotient of $W_{F}^{0}$, so it extends uniquely  to  a continuous morphism from $W_{F}$.

\medskip
\emph{Lifting this point to characteristic $0$.} Let us try to lift
$\bar\varphi$ to a Witt-vectors valued point $\varphi :
W_{F}\To{}\mathcal{T}_{\phi}({\mathcal W}_e(\bar k))$.
Here ${\mathcal W}_e(\bar k)$ is the ring of integers of the completed maximal unramified
extension of the completion of $K_{e}$ at the place given by $\OO_{K_{e}}[\frac
1p]\To{}\bar k$.
By smoothness of $\mathcal{T}_{\phi}$,
the map $\mathcal{T}_{\phi}({\mathcal W}_e(\bar k))\To{} \mathcal{T}_{\phi}(\bar k)$ is
surjective, so we may choose lifts $\tilde\varphi(w)\in
\mathcal{T}_{\phi}({\mathcal W}_e(\bar k))$ of $\bar\varphi(w)$ and we may do it in such a way
that
\begin{itemize}
\item $\tilde\varphi(w)$ only depends on $\bar\varphi(w)$ and $\tilde\varphi(w)=1$ if
  $\bar\varphi(w)=1$.
\item   $\tilde\varphi(pw)=\phi(p)\tilde\varphi(w)$ for all $w\in W_{F}$ and $p\in P_{F}$.
\end{itemize}
Note that $\tilde\varphi(w)$ belongs to the summand $T_{\HG}(\phi,{^{w}\phi})\rtimes w$, so
that we also have $\tilde\varphi(wp)=\phi(wpw^{-1})\tilde\varphi(w)=\tilde\varphi(w)\phi(p)$ for all
$w\in W_{F}$ and $p\in P_{F}$. Moreover, the automorphism
$({\rm Ad}_{\tilde\varphi(w)})_{|T_{\phi}}$ only depends on the image of
$\tilde\varphi(w)$ in $\pi_{0}(\mathcal{T}_{\phi})$, which is the same as that of
$\bar\varphi(w)$. Hence this automorphism is  the one given by the action of $\alpha(w)$.
It follows  that the map
$$c_{2}:\,(w,w')\in W_{F}\times W_{F} \mapsto
\tilde\varphi(w)\tilde\varphi(w')\tilde\varphi(ww')^{-1} \in
\ker\left(\mathcal{T}_{\phi}({\mathcal W}_e(\bar k))\To{}\mathcal{T}_{\phi}(\bar k)\right)$$
has finite image, factors over $W_{F}/P_{F}\times W_{F}/P_{F}$, and
is a $2$-cocycle from $W_{F}/P_{F}$ into
$A:=\ker({T}_{\phi}({\mathcal W}_e(\bar k))\To{}{T}_{\phi}(\bar k))$ endowed with the
action $\alpha$. Having finite image, it is continuous for the
discrete topology on $A$. If this cocycle is
cohomologically trivial, that is, if there is some continuous map $t: \, W_{F}/P_{F} \to
A$  such that
$c_{2}(w,w')=t(w) ({^{\tilde{\varphi}(w)}t(w')})t(ww')^{-1}$, then the map $w\mapsto
\varphi(w):=t(w)^{-1}\tilde\varphi(w)$ is a continuous lift of $\bar\varphi$. Now, if
$\ell$ denotes the characteristic of $\bar k$, the group $A$ is certainly
$\ell'$-divisible (i.e. $m$-divisible for any $m$ prime to $\ell$), but not
$\ell$-divisible, so that $H^{2}(W_{F}/P_{F},A)$ is not a priori 
trivial. However, if $\mathcal{\bar O}$ denotes the ring of integers of an algebraic closure of
${\mathcal W}_e(\bar k)$, then the group
$A'=\ker({T}_{\phi}({\mathcal{\bar O}})\To{}{T}_{\phi}(\bar k))$ is divisible hence, by
Lemma \ref{cohomological_lemma}, 
$c_{2}$ is cohomologically trivial there, and we get a lift $\varphi$ of $\bar\varphi$
valued in $\LG(\mathcal{\bar O})$. 

We now modify this lift $\varphi$ so that it has finite image. To do so we introduce the
maximal subtorus $C_{\phi}$ of $T_{\phi}$ on which $W_{F}/P_{F}$ acts trivially. This is the split
torus over ${\mathcal W}_e(\bar k)$ whose group of characters is the torsion-free quotient of the $W_{F}/P_{F}$-coinvariants of
the group of characters of $T_{\phi}$. Now, pick an integer $m$ such that
$\bar\varphi({\rm Fr}^{m})=1$ and $\varphi(\Fr^{m})$ is central in $\varphi(W_{F})$
(this is possible since $\varphi(I_{F})$ is finite). The element $\varphi({\rm Fr}^{m})\in
A'$ then  belongs to
$T_{\phi}(\mathcal{\bar O})^{W_{F}/P_{F}}$. Since the group scheme
$T_{\phi}^{W_{F}/P_{F}}$ is an extension of a finite diagonalizable group scheme
by the torus $C_{\phi}$, some power of
$\varphi({\rm Fr}^{m})$, say $\varphi({\rm Fr}^{m'})$,  belongs to
    $C_{\phi}(\mathcal{\bar O})\cap \ker(T_{\phi}(\mathcal{\bar O})\to T_{\phi}(\bar k)) 
    = \ker( C_{\phi}(\mathcal{\bar O})\to  C_{\phi}(\bar k))$.  But the latter is a divisible group
  so we may pick there an element $c$ such that $c^{m'}=\varphi({\rm
    Fr}^{m'})$. Consider then $\varphi': w\mapsto c^{-\nu(w)}\varphi(w)$. This is still a
$\LG(\mathcal{\bar O})$-valued  lift of $\bar\varphi$ and it has finite image. 



\medskip
\emph{A section over a quasi-finite flat extension.}
Now, the existence of such a lift shows that the morphism of finite presentation
$\uSigma(W_{F}^{0},\mathcal{T}_{\phi})_{\phi,\alpha}\To{}\Spec(\mathcal{O}_{K_{e}}[\frac 1p])$
is dominant and, even better,
that there is a finite quotient $W$ of $W_{F}$ such that
$\uSigma(W,\mathcal{T}_{\phi})_{\phi,\alpha}\To{}\Spec(\mathcal{O}_{K_{e}}[\frac 1p])$
is dominant (with obvious notation).
Therefore, we can find a finite extension $K$ of $K_{e}$ and an integer $N$
such that
$\uSigma(W_{F}^{0},\mathcal{T}_{\phi})_{\phi,\alpha}$ has a section over $\mathcal{O}_{K}[\frac 1N]$
that corresponds to a morphism  $\varphi
: W_{F}\To{} \mathcal{T}_{\phi}(\mathcal{O}_{K}[\frac 1N])$ which
factors over a finite quotient of $W_{F}$.

\medskip
\emph{Sections over the missing points.}
Let us fix a prime $\lambda$ of $K$ that divides $N$ but not $p$, and denote by
$K_{\lambda}$ the completion of $K$ at $\lambda$ and by $\mathcal{O}_{\lambda}$ its ring of
integers. Using the inclusion
$\mathcal{O}_{K}[\frac 1N]\hookrightarrow K_{\lambda}$ we get a morphism 
 $\varphi : W_{F}\To{} \mathcal{T}_{\phi}(K_{\lambda})$. We would like to
 conjugate it, so that it factors though
 $\mathcal{T}_{\phi}(\mathcal{O}_{{\lambda}})$. We will show that this is possible after
 maybe passing to a ramified extension of $K_{\lambda}$. Indeed, the problem is to find
 some $t\in T_{\phi}(K_{\lambda})$ such that $ t\varphi(w)t^{-1}\in
 \mathcal{T}_{\phi}(\mathcal{O}_{\lambda})$ for all $w\in W_{F}$. Observe that
  $\mathcal{T}_{\phi}(K_{\lambda})=T_{\phi}(K_{\lambda})\mathcal{T}_{\phi}(\mathcal{O}_{\lambda})$,
  so that $T_{\phi}(\mathcal{O}_{\lambda})$ is a normal subgroup of $\mathcal{T}_{\phi}(K_{\lambda})$
  with quotient of the form
  $$\mathcal{T}_{\phi}(K_{\lambda})/T_{\phi}(\mathcal{O}_{\lambda}) =
  \left({T}_{\phi}(K_{\lambda})/T_{\phi}(\mathcal{O}_{\lambda})\right) \rtimes
  \tilde\pi_{0}(\phi).$$
  So we see that the existence of $t$ as above is equivalent to the existence of $\bar t\in
 {T}_{\phi}(K_{\lambda})/T_{\phi}(\mathcal{O}_{\lambda})$ such
that $\bar t \varphi \bar t^{-1}$ coincides with the trivial section
$W_{F}\To{\alpha}\tilde\pi_{0}(\phi)\To{}\mathcal{T}_{\phi}(K_{\lambda})/T_{\phi}(\mathcal{O}_{\lambda})$
(we have denoted again by $\varphi$ the composition of
$\varphi$ with the projection to the above quotient). Therefore, the existence of
$t$ as above  is equivalent to the vanishing of $\varphi_{T}$ in
$H^{1}(W',{T}_{\phi}(K_{\lambda})/T_{\phi}(\mathcal{O}_{\lambda}))$, where
$\varphi_{T}$ is defined by $\varphi(w)=\varphi_{T}(w)\rtimes\alpha(w)$ and
$W'$ is any finite quotient of $W_{F}$ through which $\varphi$ (hence also $\alpha$)
factors.

Now, let $v_{\lambda}$ be the normalized valuation on $K_{\lambda}$ and let
$X^{*}(T_{\phi})$ be the group of cocharacters of $T_{\phi}$. The pairing
$(t,\mu)\mapsto v_{\lambda}(\mu(t))$ for $t\in T_{\phi}(K_{\lambda})$ and $\mu\in
X^{*}(T_{\phi})$ induces an isomorphism of abelian groups
  $$ {T}_{\phi}(K_{\lambda})/T_{\phi}(\mathcal{O}_{\lambda})
\To{\sim} \Hom(X^{*}(T_{\phi}),\ZM)$$
which shows that ${T}_{\phi}(K_{\lambda})/T_{\phi}(\mathcal{O}_{\lambda})$ is a free
abelian group of rank  $\dim(T_{\phi})$ and that
$H^{1}(W',{T}_{\phi}(K_{\lambda})/T_{\phi}(\mathcal{O}_{\lambda}))$ has no reason to
vanish. However, let $\bar K_{\lambda}$ be an algebraic closure of $K_{\lambda}$ with ring
of integers $\mathcal{\bar O}_{\lambda}$ and denote by $v_{\lambda}$ the unique
extension of $v_{\lambda}$ to $\bar K_{\lambda}$. Then the same pairing as above induces
an isomorphism
 $$ {T}_{\phi}(\bar K_{\lambda})/T_{\phi}(\mathcal{\bar O}_{\lambda})
\To{\sim} \Hom(X^{*}(T_{\phi}),\mathbb{Q})$$
which shows that ${T}_{\phi}(\bar K_{\lambda})/T_{\phi}(\mathcal{\bar O}_{\lambda})$ is a
$\mathbb{Q}$-vector space, and therefore  that the group 
$H^{1}(W',{T}_{\phi}(\bar K_{\lambda})/T_{\phi}(\mathcal{\bar O}_{\lambda}))$ 
vanishes. It follows that there is some finite extension $K'_{\lambda}$ of $K_{\lambda}$
with ring of integers $\mathcal{O'}_{\lambda}$,
and some element $t'\in T_{\phi}(K'_{\lambda})$ such that $\varphi_{\lambda}
:=t'\cdot\varphi(w)\cdot {t'}^{-1}$ defines a section of
$\uSigma(W_{F}^{0},\mathcal{T}_{\phi})_{\phi,\alpha}$ over $\mathcal{O'}_{\lambda}$.

\medskip
\emph{Conclusion.}
With $\varphi$ and the $\varphi_{\lambda}$,  we have found a section
of $\uSigma(W_{F}^{0},\mathcal{T}_{\phi})_{\phi,\alpha}$  over the finite fpqc covering
 $\coprod_{\lambda|N,\lambda\nmid p} \Spec(\mathcal{O'}_{\lambda}) \cup \Spec(\mathcal{O}_{K}[\frac 1N])$
of $\Spec(\mathcal{O}_{K_{e}}[\frac 1p])$. Since
$\uSigma(W_{F}^{0},\mathcal{T}_{\phi})_{\phi,\alpha}$ is finitely
presented, there also exists  a section over a fppf covering.
Moreover, $\varphi$ and the $\varphi_{\lambda}$'s factor over a finite
quotient $W$ of $W_{F}$, so they provide a section of 
$\uSigma(W,\mathcal{T}_{\phi})_{\phi,\alpha}$ over an fpqc covering,
and we also  deduce that
$\uSigma(W,\mathcal{T}_{\phi})_{\phi,\alpha}$ has  a section over an
fppf covering of ${\rm Spec}({\mathcal O}_{E}[\frac 1p])$.
\end{proof}

In the above proof, we have used the following result in order to pass
from the non-emptyness of $\uZ^{1}(W_{F}^{0},\hat G)_{\phi,\alpha}$ to
that of $\uSigma(W_{F}^{0},\mathcal{T}_{\phi})_{\phi,\alpha}$.

\begin{prop}
  \label{Borel_preserving_acf}
Let $K$ be an algebraically closed field of characteristic different
from $p$, let $\varphi : W_{F}\To{} {^{L}G}(K)$ be a continuous
$L$-morphism, and let $\phi:=\varphi_{|P_{F}}$.
 Then there is another extension $\varphi'=\eta\cdot\varphi$ of $\phi$,
 with  $\eta \in Z^{1}_{{\rm Ad}_{\varphi}}(W_{F}/P_{F},C_{\hat
   G}(\phi)^{\circ}(K))$, and whose
  conjugation action $\Ad_{\varphi'}$ on $C_{\hat G}(\phi)$ preserves
  a Borel pair of $C_{\hat G}(\phi)^{\circ}$.
\end{prop}
\begin{proof}
Fix a Borel pair $\mathcal{B}_{\phi}$ of $C_{\hat G}(\phi)^{\circ}$.  
  Since $C_{\hat G}(\phi)^{\circ}$ acts
transitively on its Borel pairs, we may choose for all $\bar w\in W_{F}/P_{F}$ an element
$\alpha(\bar w)\in C_{\hat G}(\phi)^{\circ}(K)$ such that  $\Ad_{\alpha(\bar w)}\circ
\Ad_{\varphi(w)}$ stabilizes $\mathcal{B}_{\phi}$, where $w$ is any lift of $\bar w$ in
$W_{F}$ (note that the restriction of $\Ad_{\varphi(w)}$ to $C_{\hat G}(\phi)^{\circ}$ does
not depend on the choice of such a lift). Moreover, we may and will choose $\alpha(\bar w)$
so that it only depends on $\Ad_{\varphi(w)}$, ensuring in turn that the map $\bar w\mapsto
\alpha(\bar w)$ is continuous.  Since the stabilizer of
$\mathcal{B}_{\phi}$ in $C_{\hat G}(\phi)^{\circ}$ is $T_{\phi}$, we see that the
automorphism $\Ad_{\alpha(\bar w)}\circ \Ad_{\varphi(w)}$ of $T_{\phi}$ does not depend on
the choice of $\alpha(\bar w)$, and this defines an action of $W_{F}/P_{F}$ on $T_{\phi}$
by algebraic automorphisms. Note that this action is the same as the one given by the
image of  $\Ad_{\varphi(w)}$ in ${\rm Out}(C_{\hat G}(\phi)^{\circ})$ through the canonical
identification of $T_{\phi}$ with the ``abstract'' torus of the root datum of $C_{\hat
  G}(\phi)^{\circ}$. In particular,  this action is finite since it
factors through the quotient of the normalizer 
of $\phi(P_{F})$ in ${^{L} G}$ by $C_{\hat G}(\phi)^{\circ}$, which is
a finite group. Now we remark that the map 
$$(\bar w, \bar w')\mapsto \alpha(\bar w)\varphi(w)\alpha(\bar w')\varphi(w')(\alpha(\bar
w\bar w')\varphi(ww'))^{-1}= \alpha(\bar w) \Ad_{\varphi(w)}(\alpha(\bar w'))\alpha(\bar w
\bar w')^{-1}$$ defines
a continuous $2$-cocycle from $W_{F}/P_{F}$ to $T_{\phi}(K)$ with respect to the action described
above. If this cocycle is a coboundary, that is, if there is a continuous map $\beta
:W_{F}/P_{F}\To{} T_{\phi}(K)$ such that
$$\alpha(\bar w) \Ad_{\varphi(w)}(\alpha(\bar w'))\alpha(\bar w\bar w')^{-1}=
\beta(\bar w) (\Ad_{\alpha(\bar w)}\circ \Ad_{\varphi(w)}(\beta(\bar w'))) \beta(\bar w\bar w')^{-1},$$ then
the map $\eta : \bar w\mapsto \beta(\bar w)^{-1}\alpha(\bar w)$ is in
$Z^{1}_{\Ad_{\varphi}}(W_{F}/P_{F},C_{\hat G}(\phi)^{\circ}(K))$ and the parameter
$\varphi'=\eta\cdot\varphi$ normalizes the Borel pair $\mathcal{B}_{\phi}$ as desired.

Hence the obstruction to finding $\eta$ as desired lies in
$H^{2}(W_{F}/P_{F},T_{\phi}(K))$. However, since $T_{\phi}(K)$ is a divisible group, the
following lemma shows that this cohomology group
vanishes. 
\end{proof}

\def\sig{{s}}

\begin{lemma}\label{cohomological_lemma}
  Denote $W=W_F/P_F$ and $I=I_F/P_F$, and let $A$ be an abelian group with a
  finite action of $W$. We consider only continuous cohomology of $W$ with respect to the
  discrete topology on the coefficients. 
  \begin{enumerate}
  \item There is a short  exact sequence
    $$1\to H^1(W/I,H^{1}(I,A))\to H^2(W,A)\to H^2(I,A)^{W/I}\to 1.$$
  \item We have $H^{2}(I,A)={\rm colim}_{(n,p)=1} (A^{I}/N_{M}(A)^{n})$ where
    \begin{itemize}
    \item $M$ is the order of the action of a pro-generator $\sig$ of $I$ and
      $N_{M}(a)=a\sig(a)\cdots\sig^{M-1}(a)$
    \item $\{n\in\mathbb{N}, (n,p)=1\}$ is ordered by divisibility and the transition map $A^{I}/N_{M}(A)^{n}\rightarrow
      A^{I}/N_{M}(A)^{n'}$ for $n|n'$ is induced by the map $a\mapsto a^{n'/n}$.
    \end{itemize}
    In particular, $H^{2}(I,A)=\{1\}$ whenever $A$ contains a $p'$-divisible group of finite
    index.
  \item We  have $H^1(W/I,H^{1}(I,A))= H^{1}(I,A)_{\rm Fr}=[N_{M}^{-1}(A[p'])/A(\sig)]_{\rm Fr}$ where
    \begin{itemize}
    \item $A[p']$ is the prime-to-$p$ torsion  of $A$ and
      $A(\sig)=\{a\sig(a)^{-1}, a\in A\}$.
    \item $\rm Fr$ is a Frobenius lift in $W$. Moreover, if $m$ is the order
      of the action of $\rm Fr$ on $A$, then ${\rm Fr}^{-m}$ acts on $H^{1}(I,A)$ by
      raising to the power $q^{m}$. 
    \end{itemize}
In particular, $H^1(W/I,H^{1}(I,A))=\{1\}$ whenever $A$ is a $p'$-divisible group.
  \end{enumerate}
\end{lemma}
\begin{proof}
  (1) follows from the Hochschild-Serre spectral sequence with the facts that $W/I=\ZZ$
  and $H^n(\ZZ, M)=1$ for any $n\geq 2$ and any $\ZZ[\ZZ]$-module
  $M$. Note that the existence of the spectral sequence follows from
  Proposition 5 and the subsequent Remark (2) of \cite{Casselman-Wigner}, but the short
  exact sequence here can also be
  simply deduced by taking colimits of similar short exact sequences
  for discrete quotients $W/J$ with $J\subset
  I$ open and contained in the kernel of the action of $W$ on $A$.

(2) By identifying $I$ with the inverse limit of $\ZZ/nM\ZZ$ for $(n,p)=1$, we can write $H^2(I,A)$ as the
direct limit of $H^2(\ZZ/nM\ZZ,A)$, indexed by integers $n$ coprime to $p$  and ordered by divisibility.
The standard formula for the $H^{2}$ of a cyclic group tells us that 
$ H^{2}(\ZZ/nM\ZZ,A) = A^{\sig}/N_{M}(A)^{n}$ and that the transition map $A^{\sig}/N_{M}(A)^{n}\rightarrow
A^{\sig}/N_{M}(A)^{n'}$ for $n|n'$ is induced by the map $a\mapsto a^{n'/n}$. Now,
suppose $B$ is a  $p'$-divisible subgroup of $A$ such that $(A/B)^{N}=1$ for some integer
$N\geq 1$. Then
$N_{M}(B)$ is  a $p'$-divisible subgroup of $A^{\sig}$ hence it is contained in $N_{M}(A)^{n}$ for
all $n>0$ with $(n,p)=1$. Since $N_{M}(A)$ contains $(A^{\sig})^{M}$, we see that 
  each  $A^{\sig}/N_{M}(A)^{n}$  has exponent dividing $NM$. Moreover this exponent is also
prime to $p$ since both $M$ and $n$ are prime to $p$.
It follows that, denoting by $N'$ the prime-to-$p$ part of $N$, the
transition maps $A^{\sig}/N_{M}(A)^{n}\rightarrow
A^{\sig}/N_{M}(A)^{n'}$  vanish whenever $nN'M|n'$, showing that the colimit vanishes, whence $H^{2}(I,A)=1$.

(3) By the continuity constraint on cocycles, the map $Z^{1}(I,A)\to A$, $\eta\mapsto \eta(\sig)$ 
identifies $Z^{1}(I,A)$ with the subgroup $\{a\in A,\exists n\in\NN,(n,p)=1,N_{n}(a)=1\}$.
Since $N_{nM}(a)=N_{n}(N_{M}(a))=N_{M}(a)^{n}$, this is also the subgroup
$\{a\in A,\exists n\in\NN, (n,p)=1, N_{M}(a)^{n}=1\}$.
In other words, with the notation of the lemma we have $ Z^{1}(I,A) = N_{M}^{-1}(A[p']).$
As a consequence $H^{1}(I,A)$, being by definition the quotient of $ Z^{1}(I,A)$ by
$\sig$-conjugacy under $A$,
is also the quotient by the subgroup $A(\sig)$, and the formula of (3) follows.


Let us make the action of $\Fr$ on $H^{1}(I,A)$ more explicit.
Note first that the action of $\Fr^{-1}$ on $Z^{1}(I,A)$
is given by
 $\Fr^{-1}(\eta)(\sig)=
 \Fr^{-1}(\eta(\Fr\sig\Fr^{-1}))=\Fr^{-1}(\eta(\sig^{q}))=\Fr^{-1}(N_{q}(\eta(\sig)))$.
If $m\in\NM$ is such that  $\Fr^{m}$ acts trivially on $A$, then $\Fr^{-m}(\eta)(\sig)=N_{q^{m}}(\eta(\sig))$.
But since the image 
 $\overline{N_{q^{m}}(\eta(\sig))}$ of $N_{q^{m}}(\eta(\sig))$ in 
 $Z^{1}(I,A)/A(\sig)$ is $\overline{\eta(\sig)}^{q^{m}}$, we see that the action of
 $\Fr^{-m}$ on $H^{1}(I,A)$ is simply given by the $q^{m}$-th power map.
Therefore, 
the space of $\Fr^{m}$ co-invariants is the
quotient $$H^{1}(I,A)_{\Fr^{m}}=H^1(I,A)/\left(H^1(I,A)\right)^{q^{m}-1},$$ which is trivial whenever
$H^1(I,A)$ is $p'$-divisible. The latter holds if $Z^{1}(I,A)$ is $p'$-divisible, and this
holds in turn if $A$ is $p'$-divisible.
\end{proof}

\begin{rem} \label{rk_epinglage}
  This lemma is the main point in proving the existence of $L$-morphisms that
  preserve a Borel pair.  
  When the center $Z_{\phi}$ of $C_{\hat G}(\phi)^{\circ}$ is a torus, and more generally
  when it is \emph{smooth over} $\OO_{K_{e}}[\frac 1p]$, then $Z_{\phi}(K)$
  is a $p'$-divisible group for any algebraically
  closed field $K$ of characteristic not $p$, so that the
  same lemma implies that
  $H^{2}(W_{F}/P_{F},Z_{\phi}(K))$ vanishes.
  In this case, fix a pinning $\varepsilon_{\phi}$ of $C_{\hat G}(\phi)^{\circ}$ and
  consider its normalizer $\mathcal{Z}_{\phi}$ in
  $C_{\LG}(\phi)$, which is an extension of $\pi_{0}(\LG)$ by $Z_{\phi}$.
  Thanks to this vanishing result, the same argument as in
  Theorem \ref{thm_reduction} shows that
  $\uSigma(W_{F}^{\circ},\mathcal{Z}_{\phi})_{\phi,\alpha}$ is a fppf torsor under the
  diagonalisable group scheme $\uZ^{1}_{\alpha}(W_{F}^0/P_{F},Z_{\phi})$, and therefore
  that we can find $\varphi_{\alpha}$ as in Theorem \ref{Borel_preserving} \emph{ with the
  additional property that $\Ad_{\varphi_{\alpha}}$ preserves the pinning $\varepsilon_{\phi}$}.
   In this case, the group scheme $C_{\hat  G}(\phi)^{\circ}\cdot
   \varphi(W_{F})$ is isomorphic to a suitable quotient of the 
  Langlands dual group scheme over $\mathcal{O}_{K'_{e}}[\frac 1p]$
 of some tamely ramified reductive group  over $F$,
 namely ``the'' quasi-split reductive group $G_{\phi,\alpha}$ dual to $C_{\hat
   G}(\phi)^{\circ}$ over $\bar F$ and whose $F$-structure is induced by the outer
 action
$$ W_{F}\To{\alpha} \tilde\pi_{0}(\phi) \To{} {\rm Out}(C_{\hat G}(\phi)^{\circ}).$$
In particular, when $C_{\hat G}(\phi)$ is connected, $\Sigma(\phi)$ is trivial so we get
a single associated quasi-split reductive group $G_{\phi}$ over $F$ and, under the
hypothesis of this Remark, we have an isomorphism over $\mathcal{O}_{K'_{e}}[\frac 1p]$
$$ \LG_{\phi}=C_{\hat G}(\phi)\rtimes_{{\rm Ad}_{\varphi}} W_{e} \To\sim C_{\LG}(\phi).$$
\end{rem}

\begin{ex}[Classical groups] Let us assume that $p>2$ and consider the
  case where $\LG$ is a Langlands
  dual group of a quasi-split classical group $G$ over $F$, so that $\hat
  G$ is one of $\Sp_{2n}$, $\SO_{2n+1}$ or $\SO_{2n}$. Then the following
  holds :
  \begin{enumerate}
  \item \emph{$C_{\hat G}(\phi)$ is connected for all
    $\phi\in\Phi_{e}$. More precisely, it is isomorphic to a product $\hat G'\times
    \GL_{n_{1}}\times\cdots\times \GL_{n_{r}}$ with $\hat G'$  of the
    same type as $\hat G$.}

  This follows from the fact that the only self-dual
  irreducible representation of a $p$-group is the trivial representation. Indeed,
  decomposing the underlying symplectic or orthogonal space as a sum of
  $\phi(P_{F})$-isotypic components, this fact shows that each  pair of dual non-trivial
  irreducible representations contributes a factor $\GL$ to the 
  centralizer, while the trivial representation contributes a classical group of the same
  sign. We then deduce the following :


\item \emph{$G_{\phi}$ is a (possibly non split) classical group times a product of restrictions of scalars 
  of general linear groups and unitary groups.} 
\item \emph{If $G$ is symplectic, then we can find an extension $\varphi$ of $\phi$ such that
  ${\rm Ad}_{\varphi}$ preserves a pinning of $C_{\hat G}(\phi)$ \emph{(because $\HG$ is adjoint
  and thanks to the previous remark)}. In particular we
  get an isomophism $\LG_{\phi}\To\sim C_{\LG}(\phi)$ as above.} Recall that even though $G$
is split here, we take $\LG = \HG\times W_{e}$ where $P_{F}/P_{F}^{e}$ \emph{injects} into $W_{e}$.
  \end{enumerate}
  
\end{ex}


The next lemma provides many  examples  to which Remark
\ref{rk_epinglage} applies.

\begin{lemma} \label{lemma:center_smooth}
  Let $H$ be a reductive group scheme over $\bar\ZZ[\frac 1p]$ and let
  $P$ be a finite $p$-group of automorphisms of $H$. If  the
  center $Z(H)$ of $H$ is smooth over $\bar\ZZ[\frac 1p]$, then so is
  the center $Z(H^{P,\circ})$ of the connected centralizer
  $H^{P,\circ}$ of $P$. 
\end{lemma}
\begin{proof}
  Recall that the center $Z$ of a reductive group scheme is a group of
  multiplicative type, associated to an \'etale sheaf $X^{*}(Z)$ of
  finitely generated abelian groups. In particular, $Z$ is flat over the base, and it
  is smooth if and only if the order of the torsion subgroups of all
  stalks of $X^{*}(Z)$ are invertible on  the
  base. In our case, since $\Spec(\bar\ZZ[\frac 1p])$ is connected, it
  suffices to check the $\bar\QQ$-stalk. Hence we see that  $Z$ is smooth if and only if
  the torsion subgroup of $X^{*}(Z_{\bar\QQ})$  has $p$-power order, if
  and only if $\pi_{0}(Z_{\bar\QQ})$ has $p$-power order.

  As a consequence, we are reduced to prove a statement for reductive
  groups over $\bar\QQ$:
\emph{if $H$ is a reductive algebraic group over $\bar\QQ$ with an action
  of a $p$-group $P$ and such that $\pi_{0}(Z(H^{\circ}))$ has
  $p$-power order, then $\pi_{0}(Z(H^{P,\circ}))$ has also $p$-power order}.

Note that if  $P_{1}$ is a normal subgroup of $P$ with quotient $P_{2}:=P/P_{1}$, then
$H^{P_{1}}$ is a reductive algebraic group and 
$H^{P}=(H^{P_{1}})^{P_{2}}$. Therefore, if the above statement is true for the
action of $P_{1}$ on $H$ and that of $P_{2}$ on $H^{P_{1}}$, it is true for the action of
$P$ on $H$. By using a central series of $P$, we may thus argue by induction and we see
that it suffices to treat the case where $P$ is cyclic of order $p$. Moreover, we may also
assume that $H$ is connected since only $Z(H^{\circ})$ and
$H^{P,\circ}=(H^{\circ})^{P,\circ}$ appear in the above statement.
Now, the quotient morphism  $H\To{}H_{\ad}$ induces a surjective morphism
$H^{P,\circ}\twoheadrightarrow (H_{\ad})^{P,\circ}$ whose kernel is $K:=Z(H)^{P}\cap
H^{P,\circ}$. So $\pi_{0}(K)$ is dual to the torsion subgroup of $X^{*}(K)$,
which is a quotient of the torsion subgroup of the coinvariants $X^{*}(Z(H))_{P}$, which has
$p$-power order. Since $Z(H^{P,\circ})$ is an extension of 
$Z((H_{\ad})^{P,\circ})$ by $K$, we see that it suffices to prove
that $\pi_{0}(Z((H_{\ad})^{P,\circ}))$ has $p$-power order. Note that
$P$ permutes the simple factors of $H_{\rm ad}$, and it suffices to
treat the case where this permutation is transitive. If $H_{\rm ad}$ is not
simple, then this permutation is also simply transitive (since $P$ is
simple), and $(H_{\rm ad})^{P}$ is isomorphic to a simple factor of $H_{\rm ad}$
(diagonally embedded in $H_{\rm ad}$). So we
are left with the case where $H_{\rm ad}$ is simple.
Let $\theta$ be a  generator of $P$. Note that $\theta$ is a semi-simple element of
 $H_{\rm ad}\rtimes P$, hence it is in particular quasi-semisimple in
 the sense of Steinberg. Let $(B,T)$ be a Borel pair fixed by
 $\theta$, and write $\theta= \Ad_{t}\circ \sigma$ with $\sigma$
 quasi-central (see \cite[Def. 1.19]{DM18}) and $t\in
 T^{\theta,\circ}=T^{\sigma,\circ}$, as per \cite[Prop. 1.16 (1)]{DM18}. 
If $\theta$ is inner on $H_{\rm ad}$, then $\sigma=1$, hence $t$ has
order $p$.  Otherwise, by the classification of quasi-central elements
below Proposition 1.22 of \cite{DM94}, we
must have $p=2$ or $p=3$,  and $\sigma$ has always order $p$, so that $t$ 
also has order $p$. In all cases, Theorem 3.11 of \cite{DM18} implies that
the order of  $\pi_{0}(Z((H_{\ad})^{P,\circ}))$
divides $p^{2}$.
\end{proof}

Using Remark \ref{rk_epinglage}, we can now strengthen Theorem
\ref{Borel_preserving} for a certain class of groups, by replacing
``Borel pair'' by ``pinning''.

\begin{thm}
  \label{prop:pinning_preserving}
  Suppose that the center of $\HG$ is smooth over $\ZZ[\frac 1p]$.
  Then there is a finite extension $K'_{e}$ of $K_{e}$ such that for any
  admissible  $\phi\in\Phi_{e}$ and any admissible $\alpha\in\Sigma(\phi)$, there
  is some $\varphi_{\alpha} \in Z^{1}(W_{F}^{0},\hat G(\mathcal{O}_{K'_{e}}[\frac 1p]))_{\phi,\alpha}$ 
such that $\varphiL_{\alpha}(W_{F}^{0})$ is finite
and $\Ad_{\varphi_{\alpha}}$ preserves \emph{a pinning} of the
split reductive group scheme  $C_{\hat G}(\phi)^{\circ}$.
\end{thm}

\section{Moduli of Langlands parameters}
\label{section:modulioflanglandsparameters}

We maintain the setup and notation of the previous section. In particular, $\HG$ is a
split reductive group scheme over $\ZM[\frac 1p]$ endowed with a finite
action of $W_{F}$, and $\LG = \HG\rtimes W$ is an adjustable
associated ``$L$-group'' of finite type.

\subsection{The moduli space of cocycles}

Let us fix a ``depth'' $e\in \mathbb{N}$ such that the action of
$P_{F}^{e}$ on $\HG$ is trivial. The functor $R\mapsto
Z^{1}(W_{F}^{0}/P_{F}^{e},\HG(R))$ is representable by an affine
scheme of finite presentation over $\mathbb{Z}[\frac 1p]$ that we
denote by  $\uZ^{1}(W_{F}^{0}/P_{F}^{e},\HG)$, and whose affine ring we denote by
$R^{e}_{\LG}$. By construction, it comes with a universal $1$-cocycle
$$ \varphi^{e}_{\rm univ}: \, W_{F}^{0}/P_{F}^{e}\To{} \HG(R^{e}_{\LG}).$$
Restriction to $P_{F}$ provides us with a morphism of
$\mathbb{Z}[\frac 1p]$-schemes
\begin{equation}
\uZ^{1}(W_{F}^{0}/P_{F}^{e},\HG) \To{}
\uZ^{1}(P_{F}/P_{F}^{e},\HG)\label{eq:forget_map}
\end{equation}
with the notation of appendix A. Using the notation of Theorem
\ref{wild_type_recap}, we have a decomposition of the right hand side
over $\mathcal{O}_{K_{e}}[\frac 1p]$ as follows
$$ \uZ^{1}(P_{F}/P_{F}^{e},\HG)_{\mathcal{O}_{K_{e}}[\frac 1p]}
= \coprod_{\phi\in\Phi_{e}}
\HG\cdot \phi,$$
where $\HG\cdot\phi$ denotes the orbit of $\phi$, which in this context is a smooth
affine scheme that
represents the quotient sheaf $\HG/C_{\HG}(\phi)$ on the big \'etale site of
$\mathcal{O}_{K_{e}}[\frac 1p]$ (see Remark
\ref{rem_orbits}). 
This induces in turn a decomposition
\begin{equation}
\uZ^{1}(W_{F}^{0}/P_{F}^{e},\HG)_{\mathcal{O}_{K_{e}}[\frac 1p]} = 
\coprod_{\phi\in \Phi_{e}^{\rm adm}}  \HG \times^{C_{\HG}(\phi)}
\uZ^{1}(W_{F}^{0},\HG)_{\phi}\label{eq:dec}
\end{equation}
with the notation of the last section. Here the summand 
$\HG \times^{C_{\HG}(\phi)} \uZ^{1}(W_{F}^{0},\HG)_{\phi}$ is
an affine scheme that represents the quotient sheaf of
 $\HG \times \uZ^{1}(W_{F}^{0},\HG)_{\phi}$ by the 
 action of $C_{\HG}(\phi)$ by right translations on $\HG$ and by
 (twisted) conjugation on 
$\uZ^{1}(W_{F}^{0},\HG)_{\phi}$. Recall that $\phi$ is called
 ``admissible'' if this summand is non-empty and we have denoted by 
$\Phi_{e}^{\rm adm}$ the subset of admissible elements. In terms of rings, we have the
decomposition
\begin{equation}
 R^{e}_{\LG} \otimes_{\ZM[\frac 1p]} \mathcal{O}_{K_{e}}[\frac 1p]
= \prod_{\phi\in\Phi_{e}^{\rm adm}}  R_{\LG,[\phi]}
= \prod_{\phi\in\Phi_{e}^{\rm adm}}  \left(\mathcal{O}_{\HG}\otimes_{\ZM[\frac 1p]}
  R_{\LG,\phi}\right)^{C_{\HG}(\phi)}.\label{eq:dec_rings}
\end{equation}
The $[\phi]$-part of the universal $1$-cocycle
$$ \varphi^{[\phi]}_{\rm univ}: \, W_{F}^{0}/P_{F}^{e}\To{} \HG(R_{\LG,[\phi]})$$
is universal for $1$-cocycles $\varphi : W_{F}^{0}\To{}\HG(R)$ such that
$\varphi_{|P_{F}}$ is \'etale-locally (over $R$)  $\HG$-conjugate to $\phi$.
Over $R_{\LG,\phi}$ we have an extension of $\phi$
$$ \varphi^{\phi}_{\rm univ}: \, W_{F}^{0}/P_{F}^{e}\To{} \HG(R_{\LG,\phi})$$
which is universal for $1$-cocycles $\varphi :W_{F}^{0}\To{}\HG(R)$ such that $\varphi_{|P_{F}}=\phi$.
The $1$-cocycles $ \varphi^{[\phi]}_{\rm univ}$ and $ \varphi^{\phi}_{\rm univ}$
determine each other in the following ways.
\begin{itemize}
\item $\varphi^{\phi}_{\rm univ}$ is deduced from $ \varphi^{[\phi]}_{\rm univ}$ by
  pushing out along the morphism 
$$\left(\mathcal{O}_{\HG}\otimes_{\ZM[\frac 1p]}  R_{\LG,\phi}\right)^{C_{\HG}(\phi)}
\To{} \left(\mathcal{O}_{\HG}\otimes_{\ZM[\frac 1p]}  R_{\LG,\phi}\right)
\To{\varepsilon_{\HG}\otimes\id} R_{\LG,\phi}$$
\item $\varphi^{[\phi]}_{\rm univ}$ is deduced from $ \varphi^{\phi}_{\rm univ}$ by
the formula 
$$\varphi^{[\phi]}_{\rm univ}(w) :\, 
\mathcal{O}_{\HG} \To{{\rm Ad}_{w}^{*}} \mathcal{O}_{\HG} \otimes_{\ZM[\frac 1p]} \mathcal{O}_{\HG}
\To{\id\otimes \varphi^{\phi}_{\rm univ}(w)} 
\mathcal{O}_{\HG} \otimes_{\ZM[\frac 1p]} R_{\LG,\phi} $$
where ${\rm Ad}_{w}^{*}$ is induced by the $w$-twisted conjugation action of $\HG$ on itself, and the
composition lands into $\left(\mathcal{O}_{\HG}\otimes_{\ZM[\frac 1p]}  R_{\LG,\phi}\right)^{C_{\HG}(\phi)}$.
\end{itemize}
We now recall the decomposition of the previous section
\begin{equation}
 \uZ^{1}(W_{F}^{0},\HG)_{\phi} =\coprod_{\alpha\in \Sigma(\phi)^{\rm adm}}
\uZ^{1}(W_{F}^{0},\HG)_{\phi,\alpha}\label{eq:dec2}
\end{equation}
and we fix, for each $\alpha\in \Sigma(\phi)^{\rm adm}$, a
$1$-cocycle  $\varphi_{\alpha} :
W_{F}^{0}\To{}\HG(\mathcal{O}_{K'_{e}}[\frac 1p])$ as in Theorem
\ref{Borel_preserving}. Then we have an isomorphism $\rho\mapsto \rho\cdot \varphi_{\alpha}$
\begin{equation}
\uZ^{1}_{\Ad_{\varphi_{\alpha}}}((W_{F}/P_{F})^{0},C_{\HG}(\phi)^{\circ}) \To{\sim}
\uZ^{1}(W_{F}^{0},\HG)_{\phi,\alpha} \times_{\mathcal{O}_{K_{e}}
  [\frac 1p]}\mathcal{O}_{K'_{e}}{\textstyle [\frac 1p]}\label{eq:iso}
 \end{equation}
where the LHS is a space of tame parameters as  studied in Section 2.
Accordingly, we have a decomposition of $\mathcal{O}_{K_{e}}[\frac 1p]$-algebras
$R_{\LG,\phi}=\prod_{\alpha}
R_{\LG,\phi,\alpha}$ and, for each $\alpha$, the  $\alpha$-component 
 of $\varphi^{\phi}_{\rm univ}$ is given, over 
$R_{\LG,\phi,\alpha}\otimes_{\mathcal{O}_{K_{e}}[\frac
   1p]}\mathcal{O}_{K'_{e}}[\frac 1p]$ 
by 
\begin{equation}
 \varphi^{\phi,\alpha}_{\rm  univ} = \rho_{\LG_{\varphi_{\alpha}}}\cdot\varphi_{\alpha} : W_{F}^{0}\To{}
\HG\left(R_{\LG,\phi,\alpha}\otimes_{\mathcal{O}_{K_{e}}[\frac
   1p]}\mathcal{O}_{K'_{e}}{\textstyle[\frac 1p]}\right)\label{eq:univ-morphism}
\end{equation}
where $\rho_{\LG_{\varphi_{\alpha}}}$ is the universal $1$-cocycle
over $\uZ^{1}_{\Ad_{\varphi_{\alpha}}}((W_{F}/P_{F})^{0},C_{\HG}(\phi)^{\circ})$.
We are now in position to prove :

\begin{thm}\label{thm:geometry}
i)  The scheme $\uZ^{1}(W_{F}^{0}/P_{F}^{e},\HG)$  is syntomic (flat and local complete
  intersection) over $\ZM[\frac 1p]$ and generically smooth, of pure absolute dimension $\dim(\HG)$. 

ii) For any prime $\ell\neq p$, the ring $R_{\LG}^{e}$ is $\ell$-adically separated and
the pushforward of  $\varphiL^{e}_{\rm univ}$ to  $R_{\LG}^{e}\otimes \ZM_{\ell}$ extends
uniquely to a $\ell$-adically continuous $L$-morphism
$$ \varphiL^{e}_{\ell-\rm univ}:\, W_{F}/P_{F}^{e}\To{}\LG(R_{\LG}^{e}\otimes\ZM_{\ell})$$
which is universal for $\ell$-adically continuous $L$-morphisms as in Definition \ref{def:continuity}.
\end{thm}
\begin{proof}
  i) Since $\mathcal{O}_{K'_{e}}[\frac 1p]$ is a syntomic cover of $\ZM[\frac 1p]$, it suffices
  to prove i) after base change to this ring. 
In what follows, we  implicitly base-change $\HG$ and all schemes introduced above to
this ring, but we omit it in the notation to  keep it readable. 
So, it suffices to prove i) for each summand 
$\HG\times^{C_{\HG}(\phi)} \uZ^{1}(W_{F}^{0},\HG)_{\phi}$ of the decomposition (\ref{eq:dec}).
Consider the morphism
\begin{equation}
\HG\times^{C_{\HG}(\phi)} \uZ^{1}(W_{F}^{0},\HG)_{\phi}\To{}
\HG\cdot \phi\label{eq:forget2}
\end{equation}
obtained by restriction of (\ref{eq:forget_map}). Its base change along the orbit morphism
$\HG\To{} \HG\cdot\phi$ is the first projection
\begin{equation}
 \HG \times \uZ^{1}(W_{F}^{0},\HG)_{\phi}\To{}  \HG 
\label{eq:forget3}
\end{equation}
Thanks to the decomposition (\ref{eq:dec2}) and the isomorphisms (\ref{eq:iso}), we may
apply Corollary \ref{cor:lci} and Proposition \ref{prop:reduced} to deduce that 
$\uZ^{1}(W_{F}^{0},\HG)_{\phi}$ is syntomic over $\mathcal{O}_{K'_{e}}[\frac 1p]$ and
generically smooth,  of pure absolute dimension $\dim(C_{\HG}(\phi))$. It follows 
that the morphism (\ref{eq:forget3}) is syntomic of pure relative
dimension $\dim(C_{\HG}(\phi))-1$ and that the source space is
generically smooth since the target is smooth. 
Since the orbit morphism is surjective and smooth (because $C_{\HG}(\phi)$ is smooth), the same
property holds for the morphism (\ref{eq:forget2}) by descent. But the orbit $\HG.\phi$ itself is
smooth over $\mathcal{O}_{K'_{e}}[\frac 1p]$ (since it is  a summand of
$\underline\Hom(P_{F}/P_{F}^{e},\HG)$) and has relative dimension
$\dim(\HG)-\dim(C_{\HG}(\phi))$. So i) follows. 

ii) The $\ell$-adic separatedness of $R^{e}_{\HG}$ follows from Corollary
\ref{cor:separatedness} and (\ref{eq:dec_rings}).  Moreover, (\ref{eq:univ-morphism})
together with Theorem \ref{thm:universal family} show that for each $\phi\in\Phi_{e}$, the
universal $L$-morphism  $\varphiL^{\phi}_{\rm univ}$ extends uniquely and $\ell$-adically
continuously to an $L$-morphism 
$\varphiL^{\phi}_{\ell-\rm univ}:\, W_{F}\To{}\LG(R'_{\LG,\phi}\otimes\ZM_{\ell})$.
Here we have written 
$R'_{\LG,\phi}:=R_{\LG,\phi}\otimes_{\mathcal{O}_{K_{e}}[\frac
   1p]}\mathcal{O}_{K'_{e}}[\frac 1p]$, and we have used the fact that the $\varphi_{\alpha}$ occuring in
 (\ref{eq:univ-morphism}) has finite image, hence extends uniquely to $W_{F}$ by continuity.
Using the relation between
$\varphi^{[\phi]}_{\rm univ}$ and $\varphi^{\phi}_{\rm univ}$, we see ultimately 
that $\varphiL^{e}_{\rm univ}$ extends to an $\ell$-adically continuous $L$-morphism
$\varphiL^{e}_{\ell-\rm univ}:\, W_{F}\To{}\LG(R^{\prime e}_{\LG}\otimes\ZM_{\ell})$ where 
$R^{\prime e}_{\LG}=R^{e}_{\LG}\otimes \mathcal{O}_{K'_{e}}[\frac 1p]$. 
We now claim that
$\varphiL^{e}_{\ell-\rm univ}$ factors through
$\LG(R^{e}_{\LG}\otimes\ZM_{\ell})$. Indeed,  its pushforward to 
$\LG(R^{\prime e}_{\LG}\otimes\ZM/{\ell}^{n}\ZM)$ has the same image as the pushforward of 
$\varphiL^{e}_{\rm univ}$ by continuity, hence it factors through $\LG(R^{e}_{\LG}\otimes\ZM/{\ell}^{n})$
for all $n$. But since $R^{\prime e}_{\LG}$ is locally free of finite rank over $R^{e}_{\LG}$,
the claim follows.  The universal property is straightforward.
\end{proof}

Statement ii) clarifies a bit the dependence of our moduli space
$\uZ^{1}(W_{F}^{0}/P_{F}^{e},\HG)$ on our initial choices of a topological generator $s$ of
$I_{F}/P_{F}$ and of a lift of Frobenius $\Fr$ in $W_{F}/P_{F}$ when  defining the subgroup
$W_{F}^{0}$ of $W_{F}$. 

\begin{cor}\label{cor:indep_moduli}
  For any prime $\ell\neq p$, the base change
  $\uZ^{1}(W_{F}^{0}/P_{F}^{e},\HG)_{\ZZ_{\ell}}$ is independent of the
  choices made to define the subgroup $W_{F}^{0}$, up to canonical isomorphism. Namely, let $W_{F}^{0'}$ be another choice of subgroup, then there
is a unique isomorphism of $\ZZ_{\ell}$-schemes $\uZ^{1}(W_{F}^{0}/P_{F}^{e},\HG)_{\ZZ_{\ell}}\To\sim
  \uZ^{1}(W_{F}^{0'}/P_{F}^{e},\HG)_{\ZZ_{\ell}}$ compatible with the
 universal $\ell$-adically continuous $1$-cocycles on each side.
\end{cor}

Besides the above result, our main conjecture over $\overline\ZM[\frac 1p]$ states that the
decomposition (\ref{eq:dec2}) is the decomposition into connected components.

\begin{conj} \label{conj:connected}
  For any pair $(\phi,\alpha)$, the $\mathcal{O}_{K_{e}}[\frac 1p]$-scheme
  $\uZ^{1}(W_{F}^{0},\HG)_{\phi,\alpha}$ is connected and remains connected
  after any finite flat integral base change. 
\end{conj}

The isomorphisms in  (\ref{eq:iso}) reduce this conjecture to the following one :

\begin{conj} \label{conj:connectedbis}
  For any split $\HG$ over $\ZZ[\frac 1p]$ with a tamely ramified Galois action that preserves a Borel pair,
  the tame summand $\uZ^{1}(W_{F}^{0}/P_{F},\HG)_{\overline\ZZ[\frac 1p]}$ is connected.
\end{conj} 

In Subsection \ref{sec:conn-comp-over}, we prove the last statement under the additional
assumption that $\HG$ has smooth center (Corollary
\ref{cor:connected_smooth_center}), or the Galois action preserves a
pinning (Theorem \ref{thm:connected_pinning_preserved}).
Thanks to  
Theorem \ref{prop:pinning_preserving},
this is enough to get the following result towards Conjecture \ref{conj:connected} :


\begin{thm} \label{thm:connected}
  Conjecture \ref{conj:connected} holds if the center of $\HG$ is smooth over $\ZZ[\frac 1p]$. 
\end{thm}

\subsection{Decomposition after localization at a prime $\ell\neq p$} \label{sec:decomp-at-ell}
For each choice of a prime $\ell\neq p$, statement ii) of Theorem \ref{thm:geometry}
allows us to refine the decomposition (\ref{eq:dec_rings}) after tensoring by
$\ZM_{\ell}$. Indeed, denote by $I_{F}^{\ell}$ the maximal closed subgroup of $I_{F}$
with prime-to-$\ell$ pro-order. Then, since $\varphiL^{e}_{\ell-\rm univ}$ is $\ell$-adically continuous,
the kernel $I_{F}^{\ell,e}$ of $(\varphiL^{e}_{\ell-\rm univ})_{|I_{F}^{\ell}}$ is open
in $I_{F}^{\ell}$. It follows that restriction to $I_{F}^{\ell}$ provides a morphism of
$\ZM_{\ell}$-schemes
$$ \uZ^{1}(W_{F}^{0}/P_{F}^{e},\HG)_{\ZM_{\ell}} \To{}  \uZ^{1}(I_{F}^{\ell}/I_{F}^{\ell,e},\HG)_{\ZM_{\ell}}.$$
But since the finite group $I_{F}^{\ell}/I_{F}^{\ell,e}$ has order invertible in
$\ZM_{\ell}$, we can apply the results of appendix A. In particular, there is a finite
\'etale extension $\Lambda_e$ of $\ZM_{\ell}$ and a finite set $\Phi_{e}^{\ell}\subset
Z^{1}(I_{F}^{\ell}/I_{F}^{\ell,e}, \HG(\Lambda_e))$ such that
$$\uZ^{1}(I_{F}^{\ell}/I_{F}^{\ell,e},\HG)_{\Lambda_e} =
\coprod_{\phi^{\ell}\in \Phi_{e}^{\ell}} \HG\cdot \phi^{\ell},$$
from which we deduce a decomposition similar to (\ref{eq:dec})
$$\uZ^{1}(W_{F}^{0}/P_{F}^{e},\HG)_{\Lambda_{e}}
=\coprod_{\phi^{\ell}\in \Phi_{e}^{\ell}}
\HG \times^{C_{\HG}(\phi^{\ell})} \uZ^{1}(W_{F}^{0}/P_{F}^{e},\HG)_{\Lambda_{e},\phi^{\ell}}, $$
where $\uZ^{1}(W_{F}^{0}/P_{F}^{e},\HG)_{\Lambda_{e},\phi^{\ell}}$ denotes the closed
subscheme of $\uZ^{1}(W_{F}^{0}/P_{F}^{e},\HG)_{\Lambda_{e}}$ defined by
$(\varphi^{e}_{\ell-\rm univ})_{|I_{F}^{\ell}}=\phi^{\ell}$. Then we can play the same
game as in Subsection \ref{sec:some-defin-constr}. Namely, taking an $L$-group $\HG\rtimes W_{e}$ 
such that $I_{F}^{\ell}/I_{F}^{\ell,e}$ injects into $W_{e}$, we define
$C_{\LG}(\phi^{\ell})$, $\tilde\pi_{0}(\phi^{\ell})$ and $\Sigma(\phi^{\ell})$ exactly as
in that subsection. This allows us to decompose further 
$$\uZ^{1}(W_{F}^{0}/P_{F}^{e},\HG)_{\Lambda_{e},\phi^{\ell}}=
\coprod_{\alpha^{\ell}\in \Sigma(\phi^{\ell})}
\uZ^{1}(W_{F}^{0}/P_{F}^{e},\HG)_{\Lambda_{e},\phi^{\ell},\alpha^{\ell}}.$$
We will say again that $\phi^{\ell}$ and $\alpha^{\ell}$ are admissible if the
corresponding summand is non empty. Moreover, we have an analogue of Theorem \ref{Borel_preserving} with the
same proof (actually, the proof simplifies a bit since we work here over a DVR).
\begin{thm}\label{Borel_preserving_ell}
  There is an integral finite flat extension $\Lambda'_{e}$ of $\Lambda_{e}$ such that, for each
  admissible $\phi^{\ell},\alpha^{\ell}$, we can find a cocycle
  $\varphi_{\alpha^{\ell}}\in Z^{1}(W_{F}^{0}/P_{F}^{e},\HG)_{\Lambda'_{e},\phi^{\ell},\alpha^{\ell}}$
  with finite image and such that
  ${\rm Ad}_{\varphi_{\alpha^{\ell}}}$ normalizes a Borel pair of
  $C_{\HG}(\phi^{\ell})^{\circ}$.
\end{thm}
As in Remark \ref{rk_epinglage}, this can be improved in certain
circumstances. Namely, if the center $Z(C_{\HG}(\phi^{\ell})^{\circ})$
is smooth over $\Lambda_{e}$  -- equivalently, if $\ell$ does not divide the order of
$X^{*}(Z(C_{\HG}(\phi^{\ell})^{\circ}))_{\rm     tors}$ -- then
  one can find $\varphi_{\alpha^{\ell}}$ such that $\Ad_{\varphi_{\alpha^{\ell}}}$
  stabilizes a pinning of $C_{\HG}(\phi^{\ell})^{\circ}$.
Using a version of Lemma \ref{lemma:center_smooth} where $\overline\ZZ[\frac 1p]$ is replaced by $\overline\ZZ_{\ell}$ and
$P$ is replaced by any solvable group of order prime to $\ell$, one obtains the following
analogue of Theorem \ref{prop:pinning_preserving}.

\begin{thm} \label{thm:pinning_preserving_ell}
  Assume that the center of $\HG$ is smooth over $\ZZ_{(\ell)}$. Then there is an
  integral finite flat extension $\Lambda'_{e}$ of $\Lambda_{e}$ such that, for each 
  admissible $\phi^{\ell},\alpha^{\ell}$, we can find a cocycle
  $\varphi_{\alpha^{\ell}}\in Z^{1}(W_{F}^{0}/P_{F}^{e},\HG)_{\Lambda'_{e},\phi^{\ell},\alpha^{\ell}}$
  with finite image and such that
  ${\rm Ad}_{\varphi_{\alpha^{\ell}}}$ fixes a pinning of  $C_{\HG}(\phi^{\ell})^{\circ}$.
\end{thm}
In particular, this result applies to classical groups whenever $\ell\neq 2$.

Fix $\varphi_{\alpha^{\ell}}$ as in one of the above theorems. Having finite image, it extends to
$W_{F}$ and the conjugation action ${\rm Ad}_{\varphi_{\alpha^{\ell}}}$ factors over $W_{F}/I_{F}^{\ell}$. Then the
usual map $\rho\mapsto \rho\cdot\varphi_{\alpha^{\ell}}$ provides an isomorphism of
$\Lambda'_{e}$-schemes
$$ \uZ^{1}_{{\rm Ad}_{\varphi_{\alpha^{\ell}}}}
\left((W_{F}/P_{F})^{0},C_{\HG}(\phi^{\ell})^{\circ}\right)_{\Lambda'_{e},1^{\ell}}
\To\sim
\uZ^{1}(W_{F}^{0}/P_{F}^{e},\HG)_{\Lambda'_{e},\phi^{\ell},\alpha^{\ell}}$$
where the subscript $1^{\ell}$ on the left hand side denotes the closed and open subscheme
of
$
\uZ^{1}_{{\rm Ad}_{\varphi_{\alpha^{\ell}}}}
\left((W_{F}/P_{F})^{0},C_{\HG}(\phi^{\ell})^{\circ}\right)_{\Lambda'_{e}}$ where the
universal $\ell$-adically continuous tame parameter restricts trivially to
$I_{F}^{\ell}$.


\begin{thm} \label{thm:connectedZl}
  For each pair $(\phi^{\ell},\alpha^{\ell})$, the $\Lambda_{e}$-scheme
  $\uZ^{1}(W_{F}^{0}/P_{F}^{e},\HG)_{\Lambda_{e},\phi^{\ell},\alpha^{\ell}}$ has a
  geometrically connected special fiber. In particular, it is
  connected and its base change to any integral  finite flat extension of $\Lambda_{e}$
  remains connected.
\end{thm}

We will prove this result  after some preparation
on categorical quotients. Meanwhile, we note that the second part of the statement follows from the first one since
$\uZ^{1}(W_{F}^{0}/P_{F}^{e},\HG)_{\Lambda_{e},\phi^{\ell},\alpha^{\ell}}$
is the spectrum of an $\ell$-adically separated $\ZM_{\ell}$-algebra. The collection of
these results for all $\ell\neq p$ will be the main ingredient in the proof of Theorem \ref{thm:connected}.


\subsection{Quotients, moduli spaces of parameters}\label{sec:quotients}
The group scheme $\HG$ acts by (twisted) conjugation on
$\uZ^{1}(W_{F}^{0}/P_{F}^{e},\HG)$. There 
are several type of quotients which can be considered here : the stacky quotient, the
quotient as fppf sheaves, or the quotient in the category of affine schemes, which is
simply $\Spec((R^{e}_{\LG})^{\HG})$.  Whatever type of quotient is considered, let us
denote it by $\uH^{1}(W_{F}^{0}/P_{F}^{e},\HG)$. Then, (\ref{eq:dec}) induces a decomposition
$$ \uH^{1}(W_{F}^{0}/P_{F}^{e},\HG)_{\mathcal{O}_{K_{e}}[\frac 1p]} 
= \coprod_{\phi\in \Phi_{e}^{\rm adm}} \uZ^{1}(W_{F}^{0},\HG)_{\phi}/C_{\HG}(\phi),$$
where the quotients on the right hand side are of the same type.
Next, (\ref{eq:dec2}) gives for each $\phi$ a decomposition
$$\uZ^{1}(W_{F}^{0},\HG)_{\phi}/C_{\HG}(\phi)^{\circ} =
\coprod_{\alpha\in\Sigma(\phi)^{\rm adm}}
\uZ^{1}(W_{F}^{0},\HG)_{\phi,\alpha}/C_{\HG}(\phi)^{\circ}$$
(beware the $\circ$) while (\ref{eq:iso}) provides for each $\alpha$ an isomorphism
$$ \uH^{1}_{\rm Ad_{\varphi_{\alpha}}}(W_{F}^{0}/P_{F},C_{\HG}(\phi)^{\circ})
\To\sim \left(\uZ^{1}(W_{F}^{0},\HG)_{\phi,\alpha}/C_{\HG}(\phi)^{\circ}\right)_{\mathcal{O}_{K'_{e}}[\frac 1p]}.$$
Now, let us denote by $\Sigma(\phi)_{0}$ a set of representatives of
$\pi_{0}(\phi)$-orbits in $\Sigma(\phi)$ and by $\pi_{0}(\phi)_{\alpha}$ the stabilizer of $\alpha$ in
$\pi_{0}(\phi)$. Let $C_{\HG}(\phi)_{\alpha}$ be the closed
subgroup scheme of
$C_{\HG}(\phi)$ inverse image of $\pi_{0}(\phi)_{\alpha}$. It
 stabilizes the summand $\uZ^{1}(W_{F}^{0},\HG)_{\phi,\alpha}$
of $\uZ^{1}(W_{F}^{0},\HG)_{\phi}$, whence an action of $\pi_{0}(\phi)_{\alpha}$ on 
$\uH^{1}_{\rm Ad_{\varphi_{\alpha}}}(W_{F^{0}}/P_{F},C_{\HG}(\phi)^{\circ})$ through
the last isomorphism. We thus have obtained an isomorphism

$$\uH^{1}(W_{F}^{0}/P_{F}^{e},\HG)_{\mathcal{O}_{K'_{e}}[\frac 1p]} 
= \coprod_{\phi\in \Phi_{e}^{\rm adm}} \coprod_{\alpha\in\Sigma(\phi)^{\rm adm}_{0}}
\uH^{1}_{\rm Ad_{\varphi_{\alpha}}}(W_{F}^{0}/P_{F},C_{\HG}(\phi)^{\circ})_{/\pi_{0}(\phi)_{\alpha}}.$$
In the case of the affine categorical quotient, we will use the familiar notation
$$\uZ^{1}(W_{F}^{0}/P_{F}^{e},\HG)\sslash\HG := \Spec((R^{e}_{\LG})^{\HG}).$$
From the above discussion we deduce :
\begin{prop}
The affine categorical quotient $\uZ^{1}(W_{F}^{0}/P_{F}^{e},\HG)\sslash\HG$ is a flat, reduced, $\ell$-adically separated
affine scheme of finite presentation over $\ZM[\frac 1p]$ and its ring
of functions decomposes as 
$$ (R^{e}_{\LG})^{\HG}\otimes\mathcal{O}_{K'_{e}}{\textstyle [\frac 1p]} = \prod_{\phi\in \Phi_{e}^{\rm
    adm}} \prod_{\alpha\in\Sigma(\phi)^{\rm adm}_{0}} 
\left(\left(R_{\LG_{\varphi_{\alpha}}}\right)^{C_{\HG}(\phi)^{\circ}}\right)^{\pi_{0}(\phi)_{\alpha}}.$$
With similar notation, we also have local decompositions for each prime $\ell\neq p$
$$ (R^{e}_{\LG})^{\HG}\otimes\Lambda'_{e} = \prod_{\phi^{\ell}\in \Phi_{e}^{\ell,\rm
    adm}} \prod_{\alpha^{\ell}\in\Sigma(\phi^{\ell})^{\rm adm}_{0}} 
\left(\left(R_{\LG_{\varphi_{\alpha^{\ell}}},1^{\ell}}\right)^{C_{\HG}(\phi^{\ell})^{\circ}}\right)^{\pi_{0}(\phi^{\ell})_{\alpha^{\ell}}}.$$
\end{prop}
\begin{proof}
The first decomposition has been explained above and the second one is similar, based on
section \ref{sec:decomp-at-ell}. The claimed properties of 
$(R^{e}_{\LG})^{\HG}$ follow from Theorem \ref{thm:geometry} except for its finite
generation as a $\ZM[\frac 1p]$-algebra, which is a
difficult result of Thomason \cite[Thm 3.8]{thomason}. 
\end{proof}


\subsection{Closed orbits over an algebraically closed field}
Let $L$ be an
algebraically closed field of characteristic $\ell$ different from $p$
($\ell=0$ is allowed here).
Let us consider the affine categorical quotient 
$$\uZ^{1}(W_{F}^{0}/P_{F}^{e},\HG)_{L}\sslash\HG_{L}= \Spec((R^{e}_{\LG}\otimes L)^{\HG_{L}}).$$
Its relation with the affine quotient over $\ZM[\frac 1p]$ can be extracted from
Alper's paper \cite{Alper}, which builds on the work of Seshadri \cite{seshadri} and Thomason \cite{thomason} on
Geometric Invariant Theory over arbitrary bases.
\begin{prop}
  The canonical map
$ (R^{e}_{\LG})^{\HG}\otimes L\To{} (R^{e}_{\LG}\otimes L)^{\HG_{L}}$
is injective. It is surjective if $\ell=0$ and, when $\ell>0$, there is an integer $r$ such that its image contains $\{f^{\ell^{r}},
f\in (R^{e}_{\LG}\otimes L)^{\HG_{L}}\}$.
In particular the canonical  morphism of $L$-schemes 
$$\uZ^{1}(W_{F}^{0}/P_{F}^{e},\HG)_{L}\sslash\HG_{L} \To{}
\left(\uZ^{1}(W_{F}^{0}/P_{F}^{e},\HG)\sslash\HG\right)_{L}$$
is a universal homeomorphism, and even an isomorphism when $\ell=0$.
\end{prop}
\begin{proof} The case $\ell=0$ is easy, so we assume that $\ell$ is prime.
 Consider first the map $(R^{e}_{\LG})^{\HG}/\ell (R^{e}_{\LG})^{\HG} \To{}
 (R^{e}_{\LG}/\ell R^{e}_{\LG})^{\HG}$. It is injective because $R^{e}_{\LG}$ is
 $\ell$-torsion free. Moreover, since $\HG$ is geometrically reductive over $\ZM[\frac 1p]$ in
the sense of \cite[Def. 9.1.1]{Alper} (by Theorem 9.7.5 of \cite{Alper}), it follows from \cite[Rk 5.2.2]{Alper} that
this map  is an ``adequate'' homeomorphism, in the sense of \cite[Def 3.3.1]{Alper}. In
particular it is ``universally adequate'', hence the map of the proposition is adequate
too, and \cite[Lemma 3.2.3]{Alper} insures the existence of $r$ as claimed in the proposition.
\end{proof}

\begin{rem}
  In Theorem \ref{thm:structureGITfield}, we will get an explicit
  bound on the set of primes $\ell$ for which the canonical morphism
  of this proposition is an isomorphism.
\end{rem}

By classical Geometric Invariant Theory, we know that the $L$-points of the
affine quotient $\uZ^{1}(W_{F}^{0}/P_{F}^{e},\HG)_{L}\sslash\HG_{L}$
correspond bijectively to Zariski closed $\HG(L)$-orbits in 
$Z^{1}(W_{F}^{0}/P_{F}^{e},\HG(L))$.
On the other hand, a theorem of
Richardson provides a criterion to decide when the
$\HG(L)$-orbit of $\varphi\in Z^{1}(W_{F}^{0}/P_{F}^{e},\HG(L))$ is closed. 

\begin{defn} 
We say that $\varphi\in Z^{1}(W_{F}^{0}/P_{F}^{e},\HG(L))$ is
\emph{$\LG$-semisimple} 
if the Zariski closure
$\overline{\varphiL(W_{F}^{0})}$ of the image of $\varphiL$ in $\LG(L)$ is
a \emph{completely reducible} subgroup of $\LG(L)$ in the sense of \cite{bmr05}. 
\end{defn}
Let us recall the definition from \cite{bmr05} :
a closed subgroup $\Gamma$ of $\LG(L)$ is called 
\emph{completely reducible}
if for all R-parabolic subgroups $P(L)$ of $\LG(L)$ containing $\Gamma$, there exists a R-Levi
subgroup of $P(L)$ containing $\Gamma$. 
Here, and as in \cite[\S 6]{bmr05},  we use Richardson's definition of parabolic and Levi subgroups
via cocharacters, which makes perfect sense for non-connected reductive
groups. 
Actually,  the definition applies verbatim to $\Gamma$ an arbitrary
subgroup, see  \cite[\S 2.6]{bmr05}, and we have that $\Gamma$ is completely reductible if and
only if its closure is completely reductible, so that, in the above
definition, we may only require that the image $\varphiL(W_{F}^{0})$
be completely reducible.

It wouldn't be difficult to check directly that
for a continuous $1$-cocycle $\varphi :W_{F}^{0}\to \HG(L)$,
the property of being $\LG$-semisimple
 neither depends on the choice of an integer $e$ such that $\varphiL$ factors
through $W_{F}^{0}/P_{F}^{e}$, nor on the  particular choice of $L$-group we
make. Anyway, this fact is also a
consequence of Richardson's theorem that we now state.
\begin{thm}[Richardson]
  The $\HG(L)$-orbit of $\varphi\in Z^{1}(W_{F}^{0}/P_{F}^{e},\HG(L))$
  is closed if and only if  $\varphi$ is $\LG$-semisimple.
\end{thm}
\begin{proof}
  Recall that the map $\varphi\mapsto \varphiL$ identifies the set
  $Z^{1}:=Z^{1}(W_{F}^{0}/P_{F}^{e},\HG(L))$ with the
  set of $L$-homomorphisms
  $W_{F}^{0}/P_{F}^{e}\To{}\LG(L)$, which is contained in the set 
$H$ of all homomorphisms $W_{F}^{0}/P_{F}^{e}\To{}\LG(L)$. Both $Z^{1}$ and $H$
have a  natural reduced $L$-scheme structure, and $Z^{1}$ is open and
closed in $H$. In particular, the $\HG(L)$-orbit of $\varphi\in Z^{1}$
is closed in $Z^{1}$ if and only the $\HG(L)$-orbit of $\varphiL$ is closed in $H$. Now, on $H$
the action of $\HG(L)$ extends to $\LG(L)$ and, since $\HG$ has finite index
  in $\LG$, we see that the $\HG(L)$-orbit of $\varphiL$ is closed 
if and only if its $\LG(L)$-orbit is closed. 

Now, let $w_{1},\cdots, w_{n}$ be a finite set of generators of the
group $W_{F}^{0}/P_{F}^{e}$. Then the map $\psi\mapsto
(\psi(w_{1}),\cdots,\psi(w_{n}))$ is an $\LG(L)$-equivariant
closed embedding of $H$ into $\LG_{L}^{n}$. So we see that the $\LG(L)$-orbit of $\varphiL$ in
$H$ is closed if and only if the $\LG(L)$-orbit of
$(\varphiL(w_{1}),\cdots,\varphiL(w_{n}))\in \LG(L)^{n}$ is closed in $\LG(L)^{n}$.
Now, Richardson's theorem (see Cor 3.7 and \S6.3 of \cite{bmr05})
tells us that the latter orbit is closed if and only if the closure of
the subgroup of $\LG(L)$ generated by
$(\varphiL(w_{1}),\cdots,\varphiL(w_{n}))$ is completely reducible
in the sense recalled above.
\end{proof}

In view of this result, we may drop the $\LG$ and  simply say that ``$\varphi$ is semisimple''.
The following  result will be crucial in our study of the affine quotient.

\begin{prop} \label{prop:estimate_ss}
  Any semisimple $1$-cocycle $\varphi :W_{F}^{0}\to \HG(L)$
  extends  continuously and uniquely to $W_{F}$. Moreover,   
the prime-to-$p$ part $|^{L}\varphi(I_{F})|_{p'}$ of the cardinality of
$\varphiL(I_{F})$ is bounded independently
of $\varphi$ and of the field $L$.
More precisely, $|\varphiL(I_{F})|_{p'}$ divides $e. \chi_{\HG,\tilde \Fr}(q)^{2}$ where
   \begin{itemize}
  \item 
     $e$  is the tame ramification of the finite
    extension $F'$ of $F$ given by the kernel of the map $W_{F}\To{}\pi_{0}(\LG)$,
  \item $\tilde \Fr$ is any lift of Frobenius in $W_{F}$.
  \item  $\chi_{\HG,\tilde \Fr}\in\ZZ[T]$ is introduced in the appendix \ref{eq:defchi}.
  \end{itemize}
\end{prop}


  \begin{rem}
    If we restrict attention to fields of characteristic $\ell>0$, then the statement that
    $\varphi$ extends continuously to $W_{F}$ is true for all $1$-cocycles, by (ii) of
    Theorem \ref{thm:geometry}. However, there is obviously no uniform bound on
    $|\varphiL(I_{F})|_{p'}$ without the semisimplicity hypothesis, when
    we vary the field $L$.
  \end{rem}

\begin{proof}[Proof of the proposition]
Recall from \cite[Thm. 2 (iii)]{Iwasawa} that there exist lifts
  $\tilde s$ and $\tilde\Fr$ of $s$ and $\Fr$ in $W_{F}$ such 
  that  $\tilde\Fr.\tilde s.\tilde\Fr^{-1}= \tilde s^{q}$ and
  that $W_{F}$ decomposes as a semi-direct product
  $W_{F}=P_{F}\rtimes \overline{\langle\tilde s,\tilde\Fr\rangle}$. Accordingly,
  $W_{F}^{0}$ decomposes as
  $W_{F}^{0}=P_{F}\rtimes {\langle\tilde s,\tilde\Fr\rangle}$.
  Then we see that a continuous $1$-cocycle $\varphi$ from $W_{F}^{0}$ extends continuously to
  $W_{F}$ if and only if $\varphiL(\tilde s)$ has finite order, in which case this
  order is prime-to-$p$, the
  extension is unique, and it satisfies $\varphi(W_{F})=\varphi(W_{F}^{0})$.

  Let us now assume that $\varphi$ is $\LG$-semisimple and show that
  $\varphiL(\tilde s)$ then has finite order. Let $F'$ and $e$ be as
 in the statement of the proposition. Note that
 $\tilde s^{e}\in W_{F'}$ and thus $\varphiL(\tilde s)^{e}\in \HG(L)$.   Since $\varphiL(P_{F'})$ is finite, there
  certainly is an integer $m$ such that $\varphiL(\tilde s)^{em}$ commutes with
  $\varphiL(P_{F'})$. This means that $\langle\varphiL(\tilde s)^{em}\rangle$ is a normal
  subgroup of $\varphiL(I_{F'}^{0})$, which is a normal subgroup of
  $\varphiL(W_{F}^{0})$ (here we have set $I_{F'}^{0}=I_{F'}\cap  W_{F'}^{0}$).
  Taking Zariski closures, we get that $\overline{\langle\varphiL(\tilde s)^{em}\rangle}$ is a normal
  subgroup of $\overline{\varphiL(I_{F'}^{0})}$, which is a normal subgroup of
  $\overline{\varphiL(W_{F}^{0})}$.  Now
  recall from \cite[Thm 3.10]{bmr05} that any normal closed subgroup of a
  completely reducible closed subgroup of $\LG(L)$ is completely reducible. So we infer that 
 $\overline{\langle\varphiL(\tilde s)^{em}\rangle}$, hence also $\langle\varphiL(\tilde s)^{em}\rangle$, is a completely reducible subgroup of $\LG(L)$,
hence also a completely reducible
subgroup of $\HG(L)$. This means that $\varphiL(\tilde s)^{em}$ is a semi-simple element of
$\HG(L)$. Since it is conjugate to its $q$-power under
$\varphiL(\tilde\Fr)\in\HG(L)$,  Proposition \ref{prop:char_pol} (2) shows
that  $\varphiL(\tilde s)^{em}$ has finite  order, and this order 
divides $\chi_{\HG,\Ad_{\varphi(\tilde\Fr)}}(q)=\chi_{\HG,\tilde\Fr}(q)$. 

It now remains to estimate $m$ and  prove that $m=\chi_{\HG,\tilde\Fr}(q)$ works.
For this, we may assume that $\varphi$ belongs to some
$Z^{1}(W_{F}^{0},\HG)_{\phi,\alpha}$ and write $\varphi = \rho\cdot\varphi_{\alpha}$ as in (\ref{eq:iso}).
By construction $\varphiL_{\alpha}$ has finite image in $\LG(\overline\ZZ[\frac 1p])$. Let $m$ be
the order of the  element $\varphiL_{\alpha}(\tilde s)^{e}$, which lies in
$\HG(\overline\ZZ[\frac 1p])$.
Then we have
 $$\varphiL(\tilde s)^{em}= \rho(\tilde s).\Ad_{\varphi_{\alpha}(\tilde
s)}(\rho(\tilde s)) \cdots \Ad_{\varphi_{\alpha}(\tilde s)^{em-1}}(\rho(\tilde s)) 
 \in C_{\HG}(\phi)^{\circ}(L),$$
hence $\varphiL(\tilde s)^{em}$ commutes with $\varphi(P_{F'})$ as desired. But observe
now that $\varphiL_{\alpha}(\tilde s)^{e}$ 
is also a semisimple element of $\HG(\overline\QQ)$ that is conjugate to
its $q^{th}$-power under $\varphiL_{\alpha}(\tilde\Fr)$. Hence, as above, its order $m$ divides
$\chi_{\HG,\tilde\Fr}(q)$. 
\end{proof}

Let us denote by $N_{\HG}$ the l.c.m of all $|\varphiL(I_{F})|_{p'}$ for $\varphi$ and $L$
varying as in the proposition,  and where $\LG$ is the minimal
$L$-group, i.e. $\LG=\HG\rtimes\Gamma$ with $\Gamma$ the image of
$W_{F}\To{}{\rm Aut}(\HG)$.

\begin{cor} \label{cor:existIe} For each ``depth'' $e$, there
  is an open normal subgroup $I_{F}^{e}$ of $I_{F}$ with index dividing $N_{\HG}.[P_{F}:P_{F}^{e}]$ such
  that any semisimple cocycle
  $\varphi : W_{F}^{0}/P_{F}^{e}\To{}\HG(L)$ is trivial on
  $I_{F}^{e}\cap W_{F}^{0}$,
  and therefore extends canonically to $W_{F}/I_{F}^{e}$.
\end{cor}
\begin{proof}
  Define $I_{F}^{e}$ to be the intersection of the kernels of the $L$-homomorphisms
  $\varphiL : W_{F}^{0}/P_{F}^{e}\To{}\LG(L)$ associated to  all
  semisimple cocycles  $\varphi : W_{F}^{0}/P_{F}^{e}\To{}\HG(L)$ with $L$ an
  algebraically closed field of characteristic $\ell\neq p$.
  Here $\LG$ is the minimal $L$-group, as above.
  Then $P_{F}\cap I_{F}^{e}$ contains $P_{F}^{e}$, hence it is open in $P_{F}$ of index dividing
  that of $P_{F}^{e}$. Moreover, by the above proposition, the cyclic group
  $I_{F}/(P_{F}.I_{F}^{e})$ is killed by $N_{\HG}$, hence $I_{F}^{e}$ has open image in
  $I_{F}/P_{F}$. It follows that $I_{F}^{e}$ is open in $I_{F}$ and that its index divides
  $N_{\HG}.[P_{F}:P_{F}^{e}]$.
\end{proof}

With $I_{F}^{e}$ as in this corollary, we may consider the $\HG$-stable closed subscheme $\uZ^{1}(W_{F}/I_{F}^{e},\HG)$ of 
$\uZ^{1}(W_{F}^{0}/P_{F}^{e},\HG)$ consisting of $1$-cocycles that are trivial on $I_{F}^{e}$.
We will denote by $S^{e}_{\LG}$ its affine ring, which is thus a quotient of
$R^{e}_{\LG}$. 

\begin{prop}\label{prop:GITiso}
  i) The homomorphism of rings
\begin{equation}
  (R^{e}_{\LG})^{\HG} \To{} (S^{e}_{\LG})^{\HG}
  \label{eq:mapinvring}
\end{equation}
  is injective and its image
contains $\{f^{N}, f\in (S^{e}_{\LG})^{\HG}\}$ for some integer $N>0$.

ii) The corresponding morphism of schemes
 \begin{equation}
\uZ^{1}(W_{F}/I_{F}^{e},\HG)\sslash\HG 
\To{} \uZ^{1}(W_{F}^{0}/P_{F}^{e},\HG)\sslash\HG 
\label{eq:mapaffquo}
\end{equation}
is a finite universal homeomorphism and becomes an
isomorphism after extending scalars to $\mathbb{Q}$.
\end{prop}
\begin{proof}
  i) \emph{Injectivity of (\ref{eq:mapinvring})}. For any algebraically closed field $L$ in which $p$ is
  invertible, the last corollary tells us that all closed
orbits of $\uZ^{1}(W_{F}^{0}/P_{F}^{e},\HG)(L)$ are contained in
$\uZ^{1}(W_{F}/I_{F}^{e},\HG)(L)$, hence the morphism (\ref{eq:mapaffquo}) is
bijective on $L$-points. It follows that the kernel of (\ref{eq:mapinvring}) is contained in the
nilradical of $(R^{e}_{\LG})^{\HG}$, which is trivial since $R^{e}_{\LG}$ is reduced,
being syntomic over $\ZM[\frac 1p]$ and generically smooth by Theorem    \ref{thm:geometry}. 

\emph{Image of (\ref{eq:mapinvring})}. By \cite[Thm 9.7.5]{Alper},
$\HG$ is geometrically reductive in the sense
of \cite[Def. 9.1.1]{Alper}. By the characterization of this property given in \cite[Lem
9.2.5 (2)']{Alper}, it follows that the map (\ref{eq:mapinvring}) is ``universally
adequate''. Then, Proposition 3.3.5 of \cite{Alper} provides the desired $N$ (note that
the map $(R^{e}_{\LG})^{\HG}\To{}(S^{e}_{\LG})^{\HG}$ is of finite type since
$(S^{e}_{\LG})^{\HG}$ is finitely generated over $\ZM[\frac 1p]$ by \cite[Thm 3.8]{thomason}).

ii) 
now follows from Lemmas 3.1.4 and 3.1.5 of \cite{Alper}.
\end{proof}

Note that the ring $S^{e}_{\LG}$ may not share the nice properties of
$R^{e}_{\LG}$; it may not be reduced nor be flat over $\ZZ[\frac 1p]$. However, the last
proposition implies that the nilradical of $(S^{e}_{\LG})^{\HG}$ coincides with its $\ZZ[\frac 1p]$-torsion ideal.
Moreover, the fact that (\ref{eq:mapaffquo}) is an isomorphism after
tensoring by $\QQ$ shows that
$(\uZ^{1}(W_{F}^{0}/P_{F}^{e},\HG)\sslash\HG)_{\QQ}$ is canonically
independent of our initial choice of subgroup $W_{F}^{0}$ in $W_{F}$. 
Actually, we can do better with a little more work :
\begin{thm} \label{thm:indep_quotient}
  The affine quotient $\uZ^{1}(W_{F}^{0}/P_{F}^{e},\HG)\sslash\HG$ is canonically independent
  of the choice of a topological generator $s$ of $I_{F}/P_{F}$ and a lift of Frobenius
  $\Fr$ in $W_{F}/P_{F}$  in the definition of the subgroup $W_{F}^{0}$.
\end{thm}
\begin{proof}
  Let  $(\Fr',s')$ be another choice, leading to a subgroup $W_{F}^{0'}$ of
  $W_{F}$. Denote by $R^{e'}_{\LG}$ the affine ring of
  $\uZ^{1}(W_{F}^{0'}/P_{F}^{e},\HG)$. Denote by $\iota$ the embedding
  (\ref{eq:mapinvring}) and by $\iota'$ the analogous embedding 
  $(R^{e'}_{\LG})^{\HG} \hookrightarrow(S^{e}_{\LG})^{\HG}$.  For each prime $\ell\neq p$,
  we have a canonical isomorphism  $(R^{e'}_{\LG})^{\HG} \otimes\ZZ_{\ell}\simeq
  (R^{e}_{\LG})^{\HG} \otimes\ZZ_{\ell}$ from Corollary \ref{cor:indep_moduli}. By
  construction it commutes with the base changes of $\iota$ and $\iota'$ to $\ZZ_{\ell}$,
  which means that
  $\iota((R^{e}_{\LG})^{\HG})\otimes\ZZ_{\ell}=\iota'((R^{e'}_{\LG})^{\HG})\otimes\ZZ_{\ell}$
  inside $(S^{e}_{\LG})^{\HG} \otimes\ZZ_{\ell}$.  This implies that $\ell$ is not in the
  support of the quotient $\ZZ[\frac 1p]$-module 
$(\iota((R^{e}_{\LG})^{\HG})+\iota'((R^{e'}_{\LG})^{\HG}))/\iota((R^{e}_{\LG})^{\HG})$.
Since this is true for all $\ell\neq p$,
  it follows that this quotient is $0$, hence
  $\iota((R^{e}_{\LG})^{\HG})=\iota'((R^{e'}_{\LG})^{\HG})$ inside $(S^{e}_{\LG})^{\HG}$. 
\end{proof}

We may wonder whether such an independence result still holds for the stacky quotient. We
believe that, at least, the categories of quasi-coherent sheaves on such stacks might be
equivalent.

\subsection{Geometric connected components in positive characteristic} We maintain our setup
of an algebraically closed field $L$ of characteristic $\ell\neq p$, and we assume that
$\ell>0$. In order to parametrize the connected components of
$\uZ^{1}(W_{F}^{0}/P_{F}^{e},\HG)_{L}$, we first observe that, since $\HG$ is connected, the canonical morphism
$$\uZ^{1}(W_{F}^{0}/P_{F}^{e},\HG)_{L}\To{} \uZ^{1}(W_{F}^{0}/P_{F}^{e},\HG)_{L}\sslash\HG_{L}$$
induces a bijection on the respective sets of connected components. Hence, Proposition
\ref{prop:GITiso} invites us to study the connected components of 
$\uZ^{1}(W_{F}/I_{F}^{e},\HG)_{L}\sslash\HG_{L}$.

\emph{Note : in order to lighten the notation a bit we will sometimes denote by $\uH^{1}$ the
  categorical quotient of cocycles modulo the relevant group action.}

Using restriction to $I_{F}^{\ell}$, we have already obtained a decomposition 
$$
 \uH^{1}(W_{F}/I_{F}^{e},\HG_{L}) =
\coprod_{\phi^{\ell}\in \Phi_{e}^{\ell,\rm adm}}
\coprod_{\alpha^{\ell}\in\Sigma(\phi^{\ell})^{\rm adm}}
\left(\uH^{1}_{{\rm Ad}_{\varphi_{\alpha^{\ell}}}}\left(W_{F}/I_{F}^{e}I_{F}^{\ell},C_{\HG}(\phi^{\ell})^{\circ}_{L}\right)\right)_{\sslash\pi_{0}(\phi^{\ell})_{\alpha^{\ell}}}.
$$

The following result shows that each summand is connected and will provide a topological
description of these summands.

\begin{prop} \label{prop:connectltame}
 Assume that the action of $W_{F}$ on $\HG$ is trivial on $I_{F}^{\ell}$ and stabilizes a Borel
 pair $(\hat B,\hat T)$. Then the following holds.
 \begin{enumerate}
\item The reduced fixed-points subgroup $(\HG_{L})^{I_{F}}$  is a connected reductive
  subgroup of $\HG_{L}$ and the reduced fixed-points subgroup $(\hat T_{L})^{I_{F}}$ is a
  maximal torus of $(\HG_{L})^{I_{F}}$ whose Weyl group is the $I_{F}$-fixed subgroup
  $\Omega^{I_{F}}$ of the Weyl group $\Omega$ of $\hat T$ in $\HG$.
\item The closed immersion $\uZ^{1}(W_{F}/I_{F},\HG_{L}^{I_{F}})\hookrightarrow
  \uZ^{1}(W_{F}/I_{F}^{e}I_{F}^{\ell},\HG_{L})$ induces an homeomorphism
  $$\uZ^{1}(W_{F}/I_{F},\HG_{L}^{I_{F}})\sslash \HG_{L}^{I_{F}} \To{}
  \uZ^{1}(W_{F}/I_{F}^{e}I_{F}^{\ell},\HG_{L})\sslash \HG_{L}.$$
\item The map $t\mapsto (\varphi :\, \Fr\mapsto t\rtimes \Fr)$ induces an isomorphism
  $$ (\hat T_{L}^{I_{F}})_{\Fr}\sslash \Omega^{W_{F}} \To\sim \uZ^{1}(W_{F}/I_{F},\HG_{L}^{I_{F}})\sslash \HG_{L}^{I_{F}}.$$
 \end{enumerate}
\end{prop}

\begin{proof}
  (1) The group $I_{F}$ acts on $\HG$ through a cyclic $\ell$-power quotient, any generator
  of which is a quasi-semisimple automorphism of $\HG$ (in the sense of
  Steinberg). Therefore, the first assertion follows from Thm 1.8 i) (reductivity) and
  Cor. 1.33 (connectedness) of \cite{DM94}.

  (2) We first make the following observation.
 If $\phi:\, I_{F}\To{} \HG(L)$ is a semisimple cocycle trivial on $I_{F}^{\ell}$, then it is $\HG(L)$-conjugate to the
trivial cocycle (note that the latter is indeed semisimple by
the characterization given in \cite[Cor 3.5 (v)]{bmr05}).
To prove this, let us use the ``minimal'' $L$-group $\HG\rtimes \Gamma$, where $\Gamma$ is
the image of $W_{F}$ in ${\rm Aut}(\HG)$. Then  the image $C$ of $I_{F}$ in $\Gamma$ is a
cyclic $\ell$-group. Let $\bar s$ be the image of the pro-generator $s$ of $I_{F}/P_{F}$  in $C$.
By \cite[7.2]{steinberg_endo}, the element $\phiL(s):=(\sigma,\bar s)$ normalizes a Borel subroup of $\HG$. After conjugating by some element of $\HG(L)$ we
may assume that it normalizes $\hat B$, thus $\phiL$ factors trough  the minimal R-parabolic subgroup
$\hat B\rtimes C$ of $\hat G\rtimes C$. Since $\phi$ is assumed to be semisimple, $\phiL$
should factor through some  R-Levi subgroup of $\hat
B\rtimes C$. But these Levi subgroups are $\hat B$-conjugated  to $\hat T\rtimes
C$. Therefore we may conjugate again $\phiL$ so that it factors through $\hat T\rtimes
C$, which means that $\phi\in Z^{1}(C,\hat T(L))$. But since $\hat T(L)$ is a $\ell$-torsion free divisible group, we
have $H^{1}(C,\hat T(L))=\{1\}$, which means that $\phi$ is conjugate to the trivial cocycle.

We deduce that the subset
$Z^{1}(I_{F}/I_{F}^{e}I_{F}^{\ell},\HG(L))^{\rm ss}$ of
$Z^{1}(I_{F}/I_{F}^{e}I_{F}^{\ell},\HG(L))$ that  consists of
semisimple cocycles is closed, since it is a single orbit and this orbit is closed by definition of
semisimple. Moreover this closed subset identifies with $\HG(L)/\hat G(L)^{I_{F}}$.
By pull-back, we deduce that the subset $Z^{1}(W_{F}/I_{F}^{e}I_{F}^{\ell},\HG(L))^{I_{F}-\rm ss}$ of
  $Z^{1}(W_{F}/I_{F}^{e}I_{F}^{\ell},\HG(L))$ that consists of all cocycles $\varphi: W_{F}\To{}\HG(L)$ such that
  $\varphi_{|I_{F}}$ is semisimple, is closed and identifies with
  $\HG(L)\times^{\HG(L)^{I_{F}}} Z^{1}(W_{F}/I_{F},\HG(L)^{I_{F}})$.

 Now by \cite[Thm 3.10]{bmr05} we know that any semisimple $\varphi:W_{F}\To{}\HG(L)$ has
semisimple restriction to $I_{F}$, so that the above closed subset contains
all closed orbits of $Z^{1}(W_{F}/I_{F}^{e}I_{F}^{\ell},\HG(L))$.
 So denote  by $\uZ^{1}(W_{F}/I_{F}^{e}I_{F}^{\ell},\HG_{L})^{I_{F}-\rm ss}$ the (reduced) closed subscheme
of $Z^{1}(W_{F}/I_{F}^{e}I_{F}^{\ell},\HG_{L})$ associated to this closed subset. Then the same argument as
 in Proposition \ref{prop:GITiso} shows that the canonical morphism 
$$\uZ^{1}(W_{F}/I_{F}^{e}I_{F}^{\ell},\HG_{L})^{I_{F}-\rm ss}\sslash\HG_{L}
\To{} \uZ^{1}(W_{F}/I_{F}^{e}I_{F}^{\ell},\HG_{L})\sslash\HG_{L}$$
is a finite universal homeomorphism. Using that
$$\uZ^{1}(W_{F}/I_{F}^{e}I_{F}^{\ell},\HG)_{L}^{I_{F}-\rm ss}=
\HG_{L}\times^{\HG_{L}^{I_{F}}} \uZ^{1}(W_{F}/I_{F},\HG_{L}^{I_{F}})$$
we infer statement (2).

(3) This is \cite[Prop. 7.1]{DM15} applied with $\mathbf{G}^{1}=\HG^{I_{F}}\rtimes\Fr$
and $\mathbf{T}^{1}=\HT^{I_{F}}\rtimes\Fr$, and $\sigma=t\rtimes\Fr$ for any element
$t\in\HT^{I_{F}}$ such that $t\rtimes\Fr$ is quasi-central (note that
$(\HT^{I_{F}})_{\sigma}=(\HT^{I_{F}})_{\Fr}$ and $(\Omega^{I_{F}})^{\sigma}=(\Omega^{I_{F}})^{\Fr}$).
\end{proof}

We now use the results and notation of Subsection \ref{sec:decomp-at-ell} to spread out
this result.

\begin{cor}
Let $\phi^{\ell}\in \Phi_{e}^{\ell}$ and $\alpha^{\ell}\in\Sigma(\phi^{\ell})_{0}$ be
admissible, and fix $\varphi:=\varphi_{\alpha^{\ell}}\in Z^{1}(W_{F}^{0}/P_{F}^{e},\hat
G(\Lambda'_{e}))_{\phi^{\ell},\alpha^{\ell}}$ with finite image and such that
${\rm Ad}_{\varphi_{\alpha^{\ell}}}$ normalizes a Borel pair $( B_{\phi^{\ell}},
T_{\phi^{ \ell}})$ of $C_{\HG}(\phi^{\ell})^{\circ}$. We denote by $\Omega^{\circ}_{\phi^{\ell}}$
the Weyl group of $ T_{\phi^{ \ell}}$ in $C_{\HG}(\phi^{\ell})^{\circ}$ and by
$\Omega_{\phi^{\ell}}=\Omega_{\phi^{\ell}}^{\circ}\rtimes \pi_{0}(\phi^{\ell})$
its ``Weyl group'' in $C_{\HG}(\phi^{\ell})$.
 \begin{enumerate}
\item Let $C_{\HG_{L}}(\varphi_{|I_{F}})= (\HG_{L})^{\varphi(I_{F})}$ be the reduced  centralizer of $^{L}\varphi(I_{F})$ in
  $\HG_{L}$.
  \begin{enumerate}
  \item $C_{\HG_{L}}(\varphi_{|I_{F}})^{\circ}$  is reductive with  maximal torus
    $( T_{\phi^{\ell},L})^{\varphi(I_{F})}$ and Weyl group
    $(\Omega_{\phi^{\ell}}^{\circ})^{\varphi(I_{F})}$. 
  \item $\pi_{0}(C_{\HG_{L}}(\varphi_{|I_{F}}))=\pi_{0}(\phi^{\ell})^{\varphi(I_{F})}$ and
    the ``Weyl group'' of $( T_{\phi^{\ell},L})^{\varphi(I_{F})}$     in  $C_{\HG_{L}}(\varphi_{|I_{F}})$
 is    $(\Omega_{\phi^{\ell}})^{\varphi(I_{F})}\simeq
    (\Omega_{\phi^{\ell}}^{\circ})^{\varphi(I_{F})}\rtimes \pi_{0}(\phi^{\ell})^{\varphi(I_{F})}$.

  \end{enumerate}

\item The natural closed immersion
  $$\uZ^{1}_{\rm Ad_{\varphi}}(W_{F}/I_{F}, C_{\HG_{L}}(\varphi(I_{F}))^{\circ})\hookrightarrow
  \uZ^{1}_{\rm
  Ad_{\varphi}}(W_{F}/I_{F}^{e}I_{F}^{\ell},C_{\HG_{L}}(\phi^{\ell})^{\circ})$$
  induces an homeomorphism
  $$\uH^{1}_{\rm Ad_{\varphi}}(W_{F}/I_{F}, C_{\HG_{L}}(\varphi(I_{F}))^{\circ})
  \To{}
  \uH^{1}_{\rm Ad_{\varphi}}(W_{F}/I_{F}^{e}I_{F}^{\ell},C_{\HG_{L}}(\phi^{\ell})^{\circ})$$
  which is equivariant for the natural actions of $\pi_{0}(\phi^{\ell})_{\alpha^{\ell}}=\pi_{0}(\phi^{\ell})^{\varphi(W_{F})}$.
\item The map $t\mapsto (\varphi_{t} :\, \Fr\mapsto t\rtimes \Fr)$ induces an isomorphism
  $$ ( T_{\phi^{\ell},L}^{\varphi(I_{F})})_{\varphi(\Fr)}\sslash
  (\Omega_{\phi^{\ell}}^{\circ})^{W_{F}}
  \To\sim \uH^{1}_{\rm Ad_{\varphi}}(W_{F}/I_{F}, C_{\HG_{L}}(\varphi_{|I_{F}})^{\circ})$$
  and subsequently an isomorphism
   $$ ( T_{\phi^{\ell},L}^{\varphi(I_{F})})_{\varphi(\Fr)}\sslash
  (\Omega_{\phi^{\ell}})^{W_{F}}
  \To\sim \left(\uH^{1}_{\rm Ad_{\varphi}}(W_{F}/I_{F}, C_{\HG_{L}}(\varphi_{|I_{F}})^{\circ})\right)_{\sslash
  \pi_{0}(\phi^{\ell})_{\alpha^{\ell}}}.$$
   \end{enumerate}
\end{cor}
\begin{proof}
  (1)(a) we have
  $C_{\HG}(\varphi_{|I_{F}})^{\circ}=((C_{\HG}(\phi^{\ell})^{^{L}\varphi(I_{F})})^{\circ}
  =((C_{\HG}(\phi^{\ell})^{\circ})^{^{L}\varphi(I_{F})})^{\circ}$
  and (1) of the previous proposition applied to
  $C_{\HG}(\phi^{\ell})^{\circ}$ implies
  $((C_{\HG}(\phi^{\ell})^{\circ})^{^{L}\varphi(I_{F})})^{\circ}=(C_{\HG}(\phi^{\ell})^{\circ})^{^{L}\varphi(I_{F})}$.  For (1)(b), observe first
  that 
  the fact that $\Omega_{\phi^{\ell}}$ is a split extension
  $\Omega_{\phi^{\ell}}^{\circ}\rtimes\pi_{0}(\phi^{\ell})$ of
$\pi_{0}(\phi^{\ell})$ by $\Omega_{\phi^{\ell}}^{\circ}$ comes from the fact it contains
the subgroup
 $N_{C_{\HG}(\phi^{\ell})}( T_{\phi^{\ell}}, B_{\phi^{\ell}})/ T_{\phi^{\ell}}\simeq\pi_{0}(\phi^{\ell})$.
Since the action of $W_{F}$ through $\rm Ad_{\varphi}$ on $C_{\HG}(\phi^{\ell})$ stabilizes 
$ T_{\phi^{\ell}}$ and $ B_{\phi^{\ell}}$, the induced action on $\Omega_{\phi^{\ell}}$  preserves the
semi-direct product decomposition, hence in particular the $\varphi(I_{F})$-invariants are
given by $(\Omega_{\phi})^{\varphi(I_{F})}= (\Omega_{\phi}^{\circ})^{\varphi(I_{F})}\rtimes
    \pi_{0}(\phi^{\ell})^{\varphi(I_{F})}$. 
Moreover we have $H^{1}(I_{F}, T_{\phi^{\ell}}(L))=1$ since $I_{F}$ acts on 
the  uniquely $\ell$-divisible abelian group $ T_{\phi^{\ell}}(L)$ through a cyclic $\ell$-group, therefore the map
$N_{C_{\HG}(\phi^{\ell})}( T_{\phi^{\ell}}, B_{\phi^{\ell}})^{\varphi(I_{F})}\To{}
\pi_{0}(\phi^{\ell})^{\varphi(I_{F})}$ is surjective, which shows that
$\pi_{0}(C_{\HG_{L}}(\varphi(I_{F})))=\pi_{0}(\phi^{\ell})^{\varphi(I_{F})}$. 

(2) follows from (2) of the last proposition except for the equivariance under the group
$\pi_{0}(\phi^{\ell})_{\alpha^{\ell}}=\pi_{0}(\phi^{\ell})^{\varphi(W_{F})}$ which is straightforward.

The first statement of (3) follows directly from (3) of the last proposition, and we infer the
second statement from the equality $(\Omega_{\phi})^{\varphi(W_{F})}= (\Omega_{\phi}^{\circ})^{\varphi(W_{F})}\rtimes
    \pi_{0}(\phi^{\ell})^{\varphi(W_{F})}$, which we already explained above. 
\end{proof}

Applying this corollary to $L=\oFl$ we see that the special fiber of the $\Lambda_{e}$-scheme
$\uZ^{1}(W_{F}^{0}/P_{F}^{e},\HG)_{\Lambda_{e},\phi^{\ell},\alpha^{\ell}}$ of subsection
\ref{sec:decomp-at-ell} is
geometrically connected. This  finishes the proof of Theorem \ref{thm:connectedZl}.

\begin{cor}
  There are natural bijections between the following sets :
  \begin{enumerate}
  \item The set of connected components of
    $\uZ^{1}(W_{F}^{0}/P_{F}^{e},\HG)_{\Lambda_{e}}$
  \item The set of connected components of
    $\uZ^{1}(W_{F}^{0}/P_{F}^{e},\HG)_{\oFl}=\uZ^{1}(W_{F}/P_{F}^{e},\HG)_{\oFl}$
  \item The set of pairs $(\phi^{\ell},[\alpha^{\ell}])$ with
    $\phi^{\ell}\in\Phi^{\ell,\rm adm}_{e}$ and $[\alpha^{\ell}]$ a
    $\pi_{0}(\phi^{\ell})$-conjugacy class in $\Sigma(\phi^{\ell})^{\rm adm}$. 
  \item The set of $\HG(\oFl)$-conjugacy classes of admissible pairs $(\phi^{\ell},\alpha^{\ell})$ where
    $\phi^{\ell}\in Z^{1}(I_{F}^{\ell}/I_{F}^{\ell,e},\HG(\oFl))$ and  $\alpha^{\ell}\in
    \Sigma(\phi^{\ell})$.
  \item The set of $\HG(\oFl)$-conjugacy classes of admissible pairs $(\phi,\alpha)$ where
    $\phi\in Z^{1}(I_{F}/I_{F}^{e},\HG(\oFl))^{\rm ss}$ is $\LG$-semisimple and  $\alpha\in
    \Sigma(\phi)$.
  \item  The set of $\HG(\oFl)$-conjugacy classes of pairs $(\phi,\beta)$ where
    $\phi\in Z^{1}(I_{F}/I_{F}^{e},\HG(\oFl))^{\rm ss}$ is $\LG$-semisimple and
    $\beta\in \tilde\pi_{0}(\phi)$ is the image of some element in
    $C_{\LG}(\phi)\cap (\HG(\oFl)\rtimes \Fr)=\{\tilde\beta\in \HG(\oFl)\rtimes \Fr,
    \tilde\beta{^{L}\phi}(i)\tilde\beta^{-1}={^{L}\phi}(\Fr.i.\Fr^{-1})\}$.
  \item The set of equivalence classes in $Z^{1}(W_{F}/P_{F}^{e},\HG(\oFl))$ for the
    relation defined by $\varphi \sim\varphi'$ if and only if there is $\hat g\in\hat
    G(\oFl)$ such that  $\varphi_{|I_{F}^{\ell}}={^{\hat g}\varphi'_{|I_{F}^{\ell}}}$ and
    $\pi\circ\varphi= \pi\circ {^{\hat g}\varphi'}$ with $\pi$ the map 
$C_{\LG}(\varphi_{|I_{F}^{\ell}}) \twoheadrightarrow \pi_{0}(C_{\LG}(\varphi_{|I_{F}^{\ell}})).$
  \end{enumerate}
Moreover, one can replace $\oFl$ by any algebraically closed field $L$
of
characteristic $\ell$.
\end{cor}
\begin{proof}
The bijections between (1), (2) and (3) follow from Theorem \ref{thm:connectedZl} which we
have just proved. The bijection between (3) and (4) follows from the definitions, and so
does the bijection between (4) and (7).
We now describe bijections between (4), (5) and (6) in a circular way.

(4)$\rightarrow$(5). Start with an admissible pair, $(\phi^{\ell},\alpha^{\ell})$. Choose
an extension  $\varphi$ of $\phi^{\ell}$ that preserves some chosen Borel pair of $C_{\HG}(\phi^{\ell})^{\circ}$. Then
$\phi:=\varphi_{|I_{F}}$ is certainly $\LG$-semisimple and $\alpha :=\pi\circ \varphi$ is
an element of $\Sigma(\phi)$ (here $\pi$ is the projection
$C_{\LG}(\phi)\To{}\tilde\pi_{0}(\phi)$ as usual). We need to check that any other choice $\varphi'$
leads to a conjugate of $(\phi,\alpha)$. Since all Borel pairs are conjugate, we may
assume that $\varphi'$ and $\varphi$  fix the same Borel pair,
and denote it by $( B_{\phi^{\ell}}, T_{\phi^{\ell}})$ . Then 
$\varphi'=\eta\cdot\varphi$ for some  $\eta\in 
Z^{1}_{\rm Ad_{\varphi}}(W_{F}/I_{F}^{\ell}, T_{\phi^{\ell}})$.  Since $H^{1}_{\rm Ad_{\varphi}}(I_{F}/I_{F}^{\ell},
T_{\phi^{\ell}})=0$ (because $ T_{\phi^{\ell}}$ is uniquely $\ell$-divisible), we have
$H^{1}_{\rm Ad_{\varphi}}(W_{F}/I_{F}^{\ell}, T_{\phi^{\ell}})=H^{1}_{\rm Ad_{\varphi}}(W_{F}/I_{F},(
T_{\phi^{\ell}})^{I_{F}})$, which means that 
we can ``$\Ad_{\varphi}$-conjugate''
$\eta$ by an element $t\in  T_{\phi^{\ell}}$ so that it
factors through a cocycle in
$Z^{1}_{\rm Ad_{\varphi}}(W_{F}/I_{F}, T_{\phi^{\ell}}^{I_{F}})$. So, after
conjugating $\varphi'$ by $t$, it has the form $\eta\cdot\varphi$ with
$\eta\in Z^{1}_{\rm Ad_{\varphi}}(W_{F}/I_{F}, T_{\phi^{\ell}}^{I_{F}})$.
We now certainly have $(\eta\cdot\varphi)_{|I_{F}}=\varphi_{|I_{F}}$ and, since
$ T_{\phi^{\ell}}^{I_{F}}$ is connected, we also have $\pi\circ (\eta\cdot\varphi)=\pi\circ\varphi$.

(5)$\rightarrow$(6). To a pair $(\phi,\alpha)$ we associate $(\phi,\beta)$ with $\beta:=\alpha(\Fr)$.

(6)$\rightarrow$(4).
Start with a pair $(\phi,\beta)$ and put $\phi^{\ell}:=\phi_{|I_{F}^{\ell}}$. Choose a
lift $\tilde\beta$ of $\beta$ in $C_{\LG}(\phi)\cap (\HG(\oFl)\rtimes \Fr)$. Then there is a
unique extension $\varphi$ of $\phi$ such that $\varphi(\Fr)= \tilde\beta$. This extension
certainly factors through $C_{\LG}(\phi^{\ell})$ and we put
$\alpha^{\ell}:=\pi^{\ell}\circ \varphi$ with $\pi^{\ell}:C_{\LG}(\phi^{\ell})\To{}\tilde\pi_{0}(\phi^{\ell})$.
Note that any other choice of lift of $\beta$ is of the form $c\tilde\beta$ with $c\in
C_{\HG}(\phi)^{\circ}$. Since $C_{\HG}(\phi)^{\circ}$ is contained in
$C_{\HG}(\phi^{\ell})^{\circ}$, such a choice defines the same $\alpha^{\ell}$.

The composition of these three applications, starting from any set  (4), (5) and (6) is
easily seen to be the identity.
\end{proof}

We finish this paragraph with another view on the  topological description of the affine
categorical quotient $\uZ^{1}(W_{F}^{0}/P_{F}^{e},\HG_{L})\sslash\HG_{L}$ that we have obtained above,
which makes it strikingly similar to what we will get over fields of characteristic $0$.


For a pair  $(\phi,\beta)$ as in (6) of the last corollary and a Borel pair $( B_{\phi},
T_{\phi})$ of the reductive algebraic group $C_{\HG_{L}}(\phi)^{\circ}$ we have an action
of $\beta$ on the torus
$ T_{\phi}$ and on its Weyl group $\Omega_{\phi}=\Omega^{\circ}_{\phi}\rtimes
\pi_{0}(\phi)$ in $C_{\HG}(\phi)$ (namely, the conjugation action of
 any lift $\tilde\beta$ of $\beta$ in $C_{\LG}(\phi)$ that preserves $( B_{\phi},
 T_{\phi})$).
Putting together the last corollary and the previous
proposition, we get :

\begin{cor}\label{cor:up_to_homeo}
  Let $\Psi_{e}(L)$ be a set of representatives of $\HG_{L}$-conjugacy classes of pairs
  $(\phi,\beta)$ as in (6) of the last corollary. For each such pair,
  choose a lift $\tilde\beta$ of $\beta$ in $C_{\LG}(\phi)$ that
  normalizes 
  a Borel pair $( B_{\phi}, T_{\phi})$ of $C_{\HG}(\phi)^{\circ}$, and
  denote by $\varphiL_{\tilde\beta}:\,W_{F}\To{}\LG(L)$ the corresponding
  extension of $\phiL$.
  Then the collection of morphisms 
  $\uZ^{1}_{\rm Ad_{\beta}}(W_{F}/I_{F}, T_{\phi})\To{}\uZ^{1}(W_{F}^{0}/P_{F}^{e},\HG_{L})$,
  $\eta\mapsto \eta\cdot \varphi_{\tilde\beta}$ induce an homeomorphism
  $$ \coprod_{(\phi,\beta)\in\Psi_{e}(L)} ( T_{\phi})_{\beta}\sslash
  (\Omega_{\phi})^{\beta} \To{\approx} \uZ^{1}(W_{F}^{0}/P_{F}^{e},\HG_{L})\sslash \HG_{L},$$
where $(T_{\phi})_{\beta}$ denotes the $\beta$-coinvariants of $T_{\phi}$ (i.e. the
cokernel of the morphism $T_{\phi}\To{}T_{\phi}, t\mapsto t^{-1}\beta(t)$).
\end{cor}

In Theorem \ref{thm:structureGITfield} we will see that these homeomorphisms are actually isomorphisms
when $\ell$ is ``$\LG$-banal''.

\subsection{Connected components over $\ZM[\frac 1p]$.}
\label{sec:conn-comp-over}
\def\HB{{\hat B}}
\def\HU{{\hat U}}
\def\HT{{\hat T}}
\def\HH{{\hat H}}
\def\LH{{\tensor*[^L]H{}}}

In this subsection, we assume that the
action of $W_{F}$ on $\HG$ is trivial on $P_{F}$ and stabilizes a Borel pair 
$(\HB,\HT)$, and we study 
the connectedness of 
the depth $0$ scheme $\uZ^{1}(W_{F}^{0}/P_{F},\HG)$ considered in Section 2, and
of all its base changes to finite flat integral extensions of $\ZM[\frac 1p]$.
Our general strategy relies on what we already know about the connected components of the  base change to
$\overline\ZZ_{\ell}$ for all $\ell\neq p$.
The following result implements this strategy under some additional hypothesis, that are
fulfilled for example if the action of $W_{F}$ on $\HG$ is unramified.
After proving it, we will show that this additional hypothesis is more generally satisfied when the action of
$I_{F}$ stabilizes a pinning.

\begin{prop}\label{prop:connected}
  Assume that there is a prime $\ell_{0}\neq p$ such that, for each subgroup $I$ of finite
  index of $I_{F}$, the $\ZZ[\frac 1p]$-group scheme $\HG^{I}$ 
   has connected geometric fibers, and is smooth over $\ZZ[\frac  1{\ell_{0}p}]$. Then the $\ZZ[\frac 1p]$-scheme 
  $\uZ^{1}(W_{F}^{0}/P_{F},\HG)$ is connected, and so are all its base changes to finite
  flat integral extensions of $\ZZ[\frac 1p]$.
\end{prop}
\begin{proof}
  Let $C$ be a connected component of $\uZ^{1}(W_{F}^{0}/P_{F},\HG)$.
  Since $C$ is flat and of finite type over $\ZZ[\frac 1p]$, we certainly have $C(\overline\QQ)\neq\emptyset$.
  Let us
     consider the set $\mathcal{I}$ of open subgroups $I$ of $I_{F}$ that contain $P_{F}$
     and such that 
  $$C(\overline\QQ)\cap Z^{1}(W_{F}/I,\HG^{I})(\overline \QQ)\neq \emptyset.$$
  By Proposition \ref{prop:GITiso} (ii), the set $\mathcal{I}$ is not empty, 
  so we may pick a maximal  $I\in\mathcal{I}$.

  We claim that $I$ has $\ell_{0}$-power index in $I_{F}$. Indeed, suppose the contrary,
  let  $\ell\neq \ell_{0}$ be a prime that divides $[I_{F}:I]$ and let $I'\supset I$ be the
  unique subgroup of $I_{F}$ such that $I'/I$ is the $\ell$-primary part of $I_{F}/I$.
  Since $I\in\mathcal{I}$,   there is a connected
  component $C_{I}$ of $\uZ^{1}(W_{F}/I,\HG^{I})$ contained in $C$ and such that
  $C_{I}(\overline\QQ)\neq \emptyset$. Our hypothesis on $\HG$ and
  \cite[Thm 1.8]{DM94} imply that $\HG^{I}$ is
  reductive over $\ZZ[\frac 1{\ell_{0}p}]$, hence Lemma
  \ref{lemma:lifting_reduction} (2) applies and ensures that $C_{I}(\overline\FF_{\ell})$ is not empty.
  Moreover, by looking fibrewise and using again \cite[Thm 1.8]{DM94}, we see that the pair $(\HB^{I},\HT^{I})$ is a
  Borel pair of $\HG^{I}$ over $\ZZ[\frac 1{\ell_{0}p}]$. 
  Now, since $I'/I$ has $\ell$-power order,  we can repeat the argument of the proof of (2) of Proposition
  \ref{prop:connectltame} and deduce
  that the injective map
  $\uZ^{1}(W_{F}/I',\HG^{I'}(\oFl))^{I_{F}-\rm ss}\hookrightarrow
  \uZ^{1}(W_{F}/I,\HG^{I}(\oFl))^{I_{F}-\rm ss}$ induces a bijection between the respective
  sets of conjugacy classes, which in turn implies that the morphism
  $\uZ^{1}(W_{F}/I',\HG^{I'})_{\oFl}\sslash \HG^{I'}_{\oFl }\To{}
  \uZ^{1}(W_{F}/I,\HG^{I})_{\oFl}\sslash \HG^{I}_{\oFl }$ is a homeomorphism, and
  consequently that the morphism $\uZ^{1}(W_{F}/I',\HG^{I'})_{\oFl}\hookrightarrow
  \uZ^{1}(W_{F}/I,\HG^{I})_{\oFl}$ induces a bijection on $\pi_{0}$. Therefore, there is a
  component $C_{I'}$ of $\uZ^{1}(W_{F}/I',\HG^{I'})$ that maps into $C_{I}$, hence also into
  $C$, and such that $C_{I'}(\overline\FF_{\ell})\neq\emptyset$.  But since the index of $I'$
  in $I_{F}$ is prime to $\ell$, Lemma \ref{lemma:lifting_reduction} (3) ensures that $C_{I'}(\overline\QQ)\neq\emptyset$,
  which  contradicts the maximality of $I$ unless $I'=I$.

  Now that we know that $I$ has $\ell_{0}$-power index in $I_{F}$, we shrink it so that it
  still has $\ell_{0}$-power index in 
  $I_{F}$ and its image in $\Aut(\HG)$ has prime-to-$\ell_{0}$ order. For this new $I$,
  Lemma \ref{hom_smooth} ensures that the group scheme
  $\HG^{I}$ is also smooth at $\ell_{0}$, hence, by Lemma \ref{lemma:lifting_reduction} (2) again, we have
  $C_{I}(\overline\FF_{\ell_{0}})\neq\emptyset$.
  But  the map
  $\uZ^{1}(W_{F}/I_{F},\HG^{I_{F}})_{\overline\FF_{\ell_{0}}}\hookrightarrow
  \uZ^{1}(W_{F}/I,\HG^{I})_{\overline\FF_{\ell_{0}}}$ induces a bijection on $\pi_{0}$, by the same
  argument as above,
  hence $C_{I}$ contains a component
  $C_{I_{F}}$ of  $\uZ^{1}(W_{F}/I_{F},\HG^{I_{F}})$, and so does $C$.
  So we have shown that the closed immersion $\uZ^{1}(W_{F}/I_{F},\HG^{I_{F}})\hookrightarrow
  \uZ^{1}(W_{F}^{0}/P_{F},\HG)$ is surjective on $\pi_{0}$, 
  and our statement follows from the fact that
  $\uZ^{1}(W_{F}/I_{F},\HG^{I_{F}})\simeq \HG^{I_{F}}$ is connected, under our assumption.
  Moreover, the same argument works similarly after base change to any integral finite
  flat extension of $R$ of $\ZZ[\frac 1p]$ by reducing modulo prime ideals of $R$ rather than prime numbers.
\end{proof}


\begin{lemma} \label{lemma:lifting_reduction}
  Let $\ell\neq p$ be a prime, and let $I\subset I_{F}$ be a
 subgroup  of finite index that contains $P_{F}$ and such that the group scheme $\HG^{I}$ is reductive over
 $\ZM_{(\ell)}$. Then, for any connected component $C$ of $\uZ^{1}(W_{F}/I,\HG^{I})$, we
 have :
 \begin{enumerate}
 \item If $L$ is an algebraically closed field and a $\ZZ_{(\ell)}$-algebra   with $C(L)\neq
   \emptyset$, then $C(L)$ contains a semisimple cocycle valued in $N_{\HG^{I}}(\HT^{I})(L)$.
 \item $C(\overline \QQ)\neq \emptyset\Rightarrow C(\overline\FF_{\ell})\neq\emptyset$.
 \item If $\ell$ does not divide the index $[I_{F}:I]$, then
   $C(\overline\FF_{\ell})\neq\emptyset\Rightarrow C(\overline \QQ)\neq \emptyset$.
 \end{enumerate}
\end{lemma}
\begin{proof}  We lighten the notation a bit by putting $\HH:=\HG^{I}$, $B_{\HH}:=\HB^{I}$
  and $T_{\HH}:=\HT^{I}$. The action of $W_{F}/I$ on $\HH$ stabilizes the Borel pair
  $(B_{\HH},T_{\HH})$ and factors over some finite quotient $W$. We also put
  $\LH:=\HH\rtimes W$ and we still denote by $s$ the image of $s$ in $W$.

  (1) If $C(L)$ is not empty, then $C(L)$ certainly contains a semisimple $1$-cocycle
  $\varphi\in Z^{1}(W_{F}/I,\HH(L))$. As usual, we denote by $\varphiL$ the associated $L$-homomorphism
  $W_{F}/I\To{}\LH(L)$. 
   By \cite[7.2]{steinberg_endo} the element $\varphiL(s)$ of $\LH(L)$ normalizes a Borel subgroup of
 $\HH$. Since $C$ is stable under conjugation by $\HH$,  we may conjugate $\varphi$ so that
 $\varphiL(s)$  normalizes $B_{\HH}$. Then
 $\varphiL(s)$ belongs to the R-Borel 
 subgroup $B_{\HH}(L)\rtimes W$ of $\LH(L)$.  Since $\varphi$ is semisimple, $\varphiL(s)$ generates
 a completely reducible subgroup of $\LH(L)$, hence it belongs to a R-Levi subgroup of
 $B_{\HH}\rtimes W$. Since all these R-Levi subgroups are $B_{\HH}$-conjugate to $T_{\HH}\rtimes W$, we may
 conjugate further $\varphi$ 
 so that $\varphiL(s)\in T_{\HH}(L)\rtimes s$. In this situation,
 $(T_{\HH}(L)^{\varphiL(s)})^{\circ}$ is a  maximal torus of $(\HH(L)^{\varphiL(s)})^{\circ}$
 whose centralizer in $\HH(L)$ is $T_{\HH}(L)$, \cite[Thm 1.8 iv)]{DM94}. 
 Now, $\varphiL(\Fr)$ normalizes
 $(\HH(L)^{\varphiL(s)})^{\circ}=(\HH(L)^{\varphiL(s)^{q}})^{\circ}$, hence it conjugates
 $(T_{\HH}(L)^{\varphiL(s)})^{\circ}$ to another maximal torus therein. Pick  $c\in
 (\HH(L)^{\varphiL(s)})^{\circ}$ that conjugates back this torus
 to $(T_{\HH}(L)^{\varphiL(s)})^{\circ}$. So $c.\varphiL(\Fr)$ normalizes
 $(T_{\HH}(L)^{\varphiL(s)})^{\circ}$, hence also its centralizer
 $T_{\HH}(L)$ in $\HH(L)$.
Hence the unique  $1$-cocycle $\varphi^{c}:W_{F}/I\To{}\HH(L)$ such
that 
 $$
\left\{ \begin{array}{l}
   {^{L}(\varphi^{c})}(s):={^{L}\varphi}(s)\in
          T_{\HH}(L)\rtimes s  \\
   {^{L}(\varphi^{c})}(\Fr):= c.\varphiL (\Fr) \in N_{\HH}(T_{\HH})(L)\rtimes \Fr.       
 \end{array}\right.
$$
is valued in $N_{\HH}(T_{\HH})(L)$ as desired, and it remains to prove that $\varphi^{c}$ belongs to $C(L)$.
But the cocycle $\varphi^{c}$ makes sense for any $c\in
(\HH(L)^{\varphiL(s)})^{\circ}$, so it is an element in the image of an algebraic morphism
$(\HH(L)^{\varphiL(s)})^{\circ}\To{}\uZ^{1}(W_{F}/I,\LH(L))$. Since the source of this
morphism is connected, its image is contained in $C(L)$.

(2) Let us first prove that $C(\overline\QQ)$ contains a $1$-cocycle $\varphi$ such that  $\varphiL$ has finite image, which is here equivalent to $\varphiL(\Fr)$ having
finite order. By (1), we may start with $\varphi$ valued in  $N_{\HH}(T_{\HH})(\overline\QQ)$. Then, a convenient power $\varphiL(\Fr)^{r}$ of $\varphiL(\Fr)$
belongs to $(T_{\HH}(\overline\QQ)^{\varphiL(s)})^{\circ}$. But the latter is a divisible group, so
it contains an element $t$ such that $t^{-r}=\varphiL(\Fr)^{r}$. Then, the cocycle
$\varphi^{t}$ defined as above has finite image.
Now, we argue as in Proposition \ref{prop:integral point} with the building of $\HH(\overline\QQ_{\ell})$ to see that
$\varphi^{t}$ can be $\HH(\overline\QQ_{\ell})$-conjugated so that it becomes
$\HH(\overline\ZZ_{\ell})$-valued. Then its image in $Z^{1}(W_{F}/I,\HH(\overline\FF_{\ell}))$
belongs to $C(\overline\FF_{\ell})$. 

(3)  By (1) we may start with $\varphi\in C(\overline\FF_{\ell})$ taking values in
$N_{\HH}(T_{\HH})(\overline\FF_{\ell})$, and we will show that it can be lifted to a $1$-cocycle
$W_{F}/I\To{} N_{\HH}(T_{\HH})(\overline\ZZ_{\ell})$. As in the proof of Theorem \ref{thm_reduction},  the obstruction to lifting
$\varphi$ belongs to $H^{2}(W_{F}/I,K)$ where $K$ is the kernel of the reduction map
$N_{\HH}(T_{\HH})(\overline\ZZ_{\ell})\To{} N_{\HH}(T_{\HH})(\overline\FF_{\ell})$, which is also the kernel
of the reduction map $T_{\HH}(\overline\ZZ_{\ell})\To{} T_{\HH}(\overline\FF_{\ell})$. In particular, $K$ is
a uniquely $\ell'$-divisible abelian group. Since $I_{F}/I$ has prime to $\ell$ order, it
follows that $H^{1}(I_{F}/I,K)=H^{2}(I_{F}/I,K)=\{0\}$. Since we also have
$H^{2}(W_{F}/I_{F}, K^{I_{F}})=\{0\}$, we see that $H^{2}(W_{F}/I,K)=0$ and there is no
obstruction to lift $\varphi$.
\end{proof}

Proposition \ref{prop:connected} shows in particular that
$\uZ^{1}(W_{F}^{0}/P_{F},\HG)$ is connected in the case where $I_{F}$ acts trivially on
$\HG$. We will now show this connectedness property in the more general case where $I_{F}$
preserves a pinning of $\HG$.
The next lemma starts with a particular subcase.

\begin{lemma} \label{lemma:sssc}
  Assume that $\HG$ is semi-simple and simply connected, and that $I_{F}$
  stabilizes a pinning of $\HG$. Then, for any
  subgroup $I$ of finite index of $I_{F}$,
  the closed subgroup scheme $\HG^{I}$ has connected geometric fibers and is smooth over $\Spec(\ZZ[\frac 1{2p}])$.
\end{lemma}  
\begin{proof}
The connectedness of geometric fibers is Steinberg's  theorem in \cite[Thm
8.2]{steinberg_endo}, so we focus on the  smoothness of $\HG^{I}$.
 Consider the action of $I$ on the set of simple factors of $\HG$. By treating distinct
 $I$-orbits seperately, we may assume that
$I$ has a single orbit, so that $\HG= {\rm ind}_{I'}^{I} \HG'$ where $\HG'$ is simple and
$I'$ has finite index in $I$.  Then $\HG^{I}=(\HG')^{I'}$, so we are reduced to the case where $\HG$ is
simple. In this case, the image of $I$  in $\Aut(\HG)$, which is cyclic since $P_{F}$
acts trivially, has order either $2$ or $3$.  
If it
has order $2$, the smoothness of $\HG^{I}$ over $\ZM[\frac 1{2p}]$ follows from Lemma
\ref{hom_smooth}.
If it has order $3$ then $\HG={\rm Spin}_{8}$. 
Over fields of characteristic $0$, it is
known that the subgroup of ${\rm Spin}_{8}$ fixed by the triality automorphism is
$G_{2}$. It may be true that $\HG^{I}=G_{2}$ in our context too, but we find it easier to argue as follows.
The big cell $C=\hat U^{-}\HT\hat U$ associated to the Borel pair $(\HT,\HB)$ is
stable under $I$, and it suffices to prove smoothness of $C^{I}= (\hat U^{-})^{I} \HT^{I}
\hat U^{I}$. Since $\HG$ is simply connected, $\HT^{I}$ is a torus by Steinberg's theorem,
hence it is smooth. By
symmetry, it remains to prove smoothness of $\hat U^{I}$. Choose an ordering
of the set $\Phi^{+}/I$ of $I$-orbits of positive roots and, for each orbit $\bar\alpha\in
\Phi^{+}/I$, choose an ordering of this orbit.
To these choices is associated a decomposition $\HU=
\prod_{\bar \alpha\in \Phi^{+}/I} \HU_{\bar \alpha}$ with $\HU_{\bar \alpha}=\prod_{\alpha\in
  \bar\alpha} \HU_{\alpha}$.
Now, the point here is that there is
no pair of $I$-conjugate positive roots whose sum is again a root.
This implies that $\HU_{\alpha}$ and $\HU_{\alpha'}$ commute with each other if
$\alpha,\alpha'\in\bar\alpha$, and it follows that $\HU^{I}= \prod_{\bar \alpha}
\HU_{\bar\alpha}^{I}$. This also implies that the $I$-invariant pinning
$(X_{\alpha})_{\alpha\in\Delta}$ (with $X_{\alpha}$ a basis of $\Lie(\HU_{\alpha})$)
can be extended  to an $I$-invariant pinning $(X_{\alpha})_{\alpha\in \Phi^{+}}$ for all
positive roots. Then, to each $X_{\alpha}$ corresponds an isomorphism
$\GG_{a}\To\sim\HU_{\alpha}$ and the product of these isomorphisms induces 
 $(\GG_{a})_{\rm diag} \To\sim\HU_{\bar\alpha}^{I}$. Whence the smoothness of $\HU_{\bar\alpha}^{I}$.
\end{proof}

\begin{rem} (1) The same lemma holds with ``adjoint'' instead of ``simply connected''. The
  reference to Steinberg's result has to be replaced by a reference to
  \cite[Remarque 1.30]{DM94} for example (note that a pinning-preserving automorphism is
  quasi-central in the sense of \cite{DM94}).
  
(2)  If $\HG=\SL_{2n+1}$ with a topological generator of $I_{F}$ acting by the non-trivial automorphism
  that preserves the standard pinning, then $(\SL_{2n+1})^{I_{F}}$ is
  \emph{not} smooth over $\ZZ_{(2)}$. For example, with standard
  coordinates $x=x_{12},y=x_{23},z=x_{13}$ for the upper unipotent
  subgroup $\HU$ of $\SL_{3}$, the invariants $\HU^{I}$ are given by
  equations $x=y$ and $xy=2z$. However, it is likely that in any
  simple simply connected case not of type $A_{2n}$, the
  $I$-invariants are smooth over $\ZZ$. During the reviewing process of this paper, this
  expectation was indeed proved (and the above lemma reproved and generalized) in  \cite[Thm 1.1 (3)]{ALRR}.
  \end{rem}

This lemma, together with Proposition \ref{prop:connected}, shows that
$\uZ^{1}(W_{F}^{0}/P_{F},\HG)_{\overline\ZZ[\frac 1p]}$ is connected for $\HG$ as in the lemma. In order to
spread a bit this result, we will use the next two lemmas.

\begin{lemma}
  Assume given
 another split reductive group $\HG'$ over $\ZM[\frac 1p]$ equipped with an action of
  $W_{F}/P_{F}$  and with an equivariant surjective  morphism $\HG'\To{f}\HG$ whose  kernel is a
  torus.
Then the morphism  $\uZ^{1}(W_{F}^{0}/P_{F},\HG')\To{f_{*}}\uZ^{1}(W_{F}^{0}/P_{F},\HG)$
induces a surjection on $\pi_{0}$.
\end{lemma}
\begin{proof}
Put $\hat S:=\ker f$. For any prime $\ell\neq p$ and
  $\varphi\in Z^{1}(W_{F}/P_{F},\HG(\oFl))$, the obstruction to lifting $\varphi$ to an
  element of $Z^{1}(W_{F}/P_{F},\HG'(\oFl))$ lies in the group
  $H^{2}(W_{F}/P_{F},\hat S(\oFl))$, which vanishes by Lemma \ref{cohomological_lemma}
  since $\hat S(\oFl)$ is a divisible group. Therefore the map
$$f_{*}: \, Z^{1}(W_{F}/P_{F},\HG'(\oFl))\To{}
Z^{1}(W_{F}/P_{F},\HG(\oFl))$$ is surjective. Since any connected component of
$\uZ^{1}(W_{F}^{0}/P_{F},\HG)$ has $\oFl$-points for some $\ell$ (and even for all $\ell$),
 this implies the lemma.
\end{proof}

\begin{lemma}
  There exists a split reductive group $\HG'$ over $\ZM[\frac 1p]$ equipped with an action of
  $W_{F}/P_{F}$ and an equivariant surjective morphism $\HG'\To{}\HG$ whose
  kernel is a torus, and such that, for all open subgroups $I$ of $I_{F}$, we have :
  \begin{enumerate}
  \item $(\HG')^{I}$ has geometrically connected fibers, and
  \item $(\HG')^{I}$ is smooth over $\ZZ[\frac 1{2p}]$ if $I_{F}$ preserves a
    pinning of $\HG$.
  \end{enumerate}
\end{lemma}
\begin{proof}
  By Theorem 5.3.1 of \cite{conrad_luminy} (or \cite[Exp XXII, \S 6.2]{MR0274460}), 
there is a unique closed semi-simple subgroup scheme
$\HG_{\rm der}$ of $\HG$ over $\ZM[\frac 1p]$ that represents the fppf
sheafification of the set-theoretical derived subgroup presheaf. 
Further, by Exercise 6.5.2 of \cite{conrad_luminy}, there is a canonical central isogeny
 $\HG_{\rm sc}\To{} \HG_{\rm der}$ over $\ZM[\frac 1p]$, such that all the geometric fibers of $\HG_{\rm sc}$ are
 simply connected semi-simple groups. Being canonical, the action of $W_{F}/P_{F}$ on $\HG_{\rm
   der}$ lifts uniquely to $\HG_{\rm sc}$ and still preserves a Borel pair or a
 pinning, depending on the case. 
Now denote by $R(\HG)$ the radical of $\HG$, which is a split torus.
Then the natural morphism $R(\HG)\times \HG_{\rm sc}\To{} \HG$ is 
a $W_{F}$-equivariant central isogeny. We already know that $(\HG_{\rm sc})^{I}$ satisfies
properties (1) and (2) for finite index any subgroup $I\subset I_{F}$, by Lemma \ref{lemma:sssc}.
On the other hand, $R(\HG)^{I}$ is the diagonalisable group 
associated to the abelian group $X^{*}(R(\HG))_{I}$ of $I$-coinvariants in
$X^{*}(R(\HG))$, which may have torsion. Let $W$ be a finite quotient of $W_{F}/P_{F}$
through which the action of $W_{F}$ on $\HG$ factors. Choosing a dual basis of the lattice 
$X^{*}(R(\HG))$, we get a $W_{F}$-equivariant embedding $X^{*}(R(\HG))\hookrightarrow \ZM[W]^{\dim
  R(\HG)}$ where the target has torsion-free $I$-coinvariants for all $I$. Dually we get a surjective
$W_{F}$-equivariant morphism of tori $\hat S\twoheadrightarrow R(G)$ such that $(\hat
S)^{I}$ is a torus,  hence is smooth with connected geometric fibers. Thus we have a
$W_{F}$-equivariant surjective morphism  $\HG'':=\hat S\times \HG_{\rm sc}\twoheadrightarrow
\HG$  whose source satisfies both properties (1) and (2), but
 whose kernel $D$, a diagonisable subgroup, is not necessarily a torus.

Now let us choose a surjective $W_{F}$-equivariant morphism $X\To{} X^{*}(D)$ such that
$X$ is a permutation module (i.e. $W_{F}$ permutes a $\ZM$-basis of $X$) and such that for
every  finite index subgroup $I\subset I_{F}$, the map on $I$-invariants
$X^{I}\To{}X^{*}(D)^{I}$ is surjective. 
For example, one can take $X=\bigoplus_{I} \ZM[W/I]\otimes Y_{I}$ where $I$ runs
over subgroups of the image of $I_{F}$ inside $W$ and $Y_{I}$ is any free
abelian group mapping surjectively to  $X^{*}(D)^{I}$. Dually, we have
a $W_{F}$-equivariant  embedding
$D\hookrightarrow \hat S'$ of $D$ into the split torus $\hat S'$ over
$\ZM[\frac 1p]$ with character group $X^{*}(\hat S')=X$. Since its
character group is a permutation module, the torus $\hat S'$ satisfies
both properties (1) and (2). Namely, $(\hat S')^{I}$ is a torus for any finite index
subgroup $I\subset I_{F}$.
Now, by \cite[Exp. XXII \S 4.3]{MR0274460} the quotient
$\HG' := \hat S' \times^{D} \HG''$ is representable by a split reductive group
scheme, which by construction  is a $W_{F}$-equivariant extension of $\HG$ by the torus $\hat S'$.
Let us prove that $\HG'$ has property (1).
Fix a prime $\ell\neq p$ and look at the exact sequence
$$ (\hat S'(\oFl) \times \HG''(\oFl))^{I}\To{}  \HG'(\oFl)^{I} \To{} H^{1}(I,D(\oFl))
\To{} H^{1}(I,\hat S'(\oFl) \times \HG''(\oFl)).$$ 
Since $I$ is procyclic, say with topological generator $t$, we have $H^{1}(I,D(\oFl))\simeq
D(\oFl)_{I}= (D/(\id-t)D)(\oFl)=\Hom(X^{*}(D)^{I},\oFl^{\times})$.
Therefore, it follows that the map $ H^{1}(I,D(\oFl)) \To{} H^{1}(I,\hat S'(\oFl))$ identifies with the map
$$ \Hom(X^{*}(D)^{I},\oFl^{\times})\To{} \Hom(X^{*}(\hat S')^{I},\oFl^{\times}),$$
hence it is injective
since the map $X^{*}(S')^{I}\To{}X^{*}(D)^{I}$ is surjective. It follows that the map
$$(\hat S'(\oFl) \times \HG''(\oFl))^{I}\To{}  \HG'(\oFl)^{I}$$
is surjective. But the source of this map is a connected variety since $X^{*}(S')$ is a
permutation module, so the target is also connected as desired. This
proves the connectedness of the closed geometric fibers of
$(\HG')^{I}$, and that of the generic geometric fiber follows since
$(\HG')^{I}$ is smooth with reductive neutral component, hence \'etale
component group, after restriction to a suitable open subset of ${\rm
  Spec} \ZZ[\frac 1p]$.
Let us now prove that
$\HG'$ has property (2). So we assume that $I_{F}$ preserves a pinning of $\HG$, which
implies that it also preserves a pinning of $\HG''$ and $\HG'$, and we shall prove
smoothness of $(\HG')^{I}$ over $\ZZ[\frac 1{2p}]$.  Since the big cell
$(\HU'^{-}\HT'\HU')^{I} = (\HU'^{-})^{I}(\HT')^{I}(\HU')^{I}$ is non-empty and
$(\HG')^{I}$ is connected, it suffices to
prove its smoothness. We already know from Lemma \ref{lemma:sssc} that
$(\HU'^{-})^{I}$ and $(\HU')^{I}$ are smooth over $\ZZ[\frac 1{2p}]$ so we may concentrate on the diagonalisable
subgroup $(\HT')^I$. But, by the same argument as for property (1) above, this
diagonalisable group has geometrically connected fibers, so this is a torus (because the
base is connected and has fibers of at least two distinct residual characteristics), and in
particular it is smooth.
\end{proof}

Putting the last two lemmas and Proposition \ref{prop:connected} together, we get the following result.
\begin{thm} \label{thm:connected_pinning_preserved}
  If $I_{F}$ preserves a pinning of $\HG$, the scheme
  $\uZ^{1}(W_{F}^{0}/P_{F},\HG)_{\overline\ZZ[\frac 1p]}$ is connected.
\end{thm}

\begin{cor}  \label{cor:connected_smooth_center}
  If the center of $\HG$ is smooth, then
  $\uZ^{1}(W_{F}^{0}/P_{F},\HG)_{\overline\ZZ[\frac 1p]}$ is connected.
\end{cor}
\begin{proof}
  Indeed, in this case there is
  $\varphi\in \uZ^{1}(W_{F}^{0}/P_{F},\HG)(\overline\ZZ[\frac 1p])$ such that $\Ad_{\varphi}$
  preserves a pinning (see Remark \ref{rk_epinglage}), and right  multiplication by $\varphi$ provides an
  isomorphism $\uZ^{1}_{\Ad_{\varphi}}(W_{F}^{0}/P_{F},\HG)_{\overline\ZZ[\frac
    1p]}\To\sim\uZ^{1}(W_{F}^{0}/P_{F},\HG)_{\overline\ZZ[\frac 1p]}$. 
\end{proof}

\section{Unobstructed points}


In this section, we fix an algebraically closed field $L$ of
characteristic  $\ell\neq p$. From \ref{sec:exist-unobstr-points} on, we will further
assume that $\ell$ is finite.

\subsection{Deformation Theory}
Here, for an $L$-point $x$ of
$\uZ^1(W_F^0/P_F^e,{\hat G})$, we are interested in  the tangent space
$T_x \uZ^1(W_F^0/P_F^e,{\hat G}_L)$ and, in particular, we wish to
compute its dimension.
We will need the $L$-linear  continuous representation
$\Ad\varphi_{x}$ of $W_{F}^{0}$ on the Lie algebra $\Lie(\HG_{L})$
obtained by composing $\varphiL_{x}$ with the adjoint representation
of $\LG$. 

Recall that an element of $T_x \uZ^1(W_F^0/P_F^e,{\hat G}_L)$ is given
by a map ${\tilde x}: \Spec L[\epsilon]/\epsilon^2 \rightarrow
\uZ^1(W_F^0/P_F^e,{\hat G})$ 
whose composition with the natural map $\Spec L \rightarrow \Spec
L[\epsilon]/\epsilon^2$ is equal to $x$.   In particular, the zero element ${\tilde x}_0$
of $T_x \uZ^1(W_F^0/P_F^e,{\hat G}_L)$ is given by the composition of
$x$ with the natural map $\Spec L[\epsilon]/\epsilon^2 \rightarrow
\Spec L$.
Given such a $\tilde x$ we form a cocycle for $\Ad\varphi_{x}$  as follows:
for each $w \in W_F^{0}$, the element
$\varphiL_{\tilde x}(w) \varphiL_{{\tilde x}_0}(w)^{-1}$ is a tangent
vector to $\hat G$ at the identity element of ${\hat G}(L)$; that is, 
an element of $\Lie({\hat G}_{L})$.  In this way one obtains a
continuous $1$-cocycle $W_{F}^{0}\To{}\Lie({\hat G}_{L})$ that lives in $Z^1(W_F^{0}, \Ad \varphi_x)$, and this sets
up an isomorphism 
\begin{equation}
T_x \uZ^1(W_F^0/P_F^e,{\hat G}_L) \simeq  Z^1(W_F^{0}, \Ad \varphi_x).\label{eq:tangent}
\end{equation}
To compute the dimension of this tangent space, we use the following
familiar-looking cohomological lemma.
\begin{lemma} 
  For any finite dimensional $L$-vector space $V$ with a
  continuous linear action of  $W_{F}^{0}$,  we have  :
  \begin{enumerate}
  \item $H^{i}(W_{F}^{0},V)=0$ for $i>2$,
  \item $\dim H^2(W_F^{0},V) - \dim H^1(W_F^{0},V) + \dim H^0(W_F^{0},V) = 0$
  \item $H^{2}(W_{F}^{0},V)^{*} \simeq H^{0}(W_{F}^{0},V^{*}\otimes\omega)$ where
$\omega$ is the cyclotomic character of $W_F$ and $^{*}$ denotes the
$L$-linear dual.
  \end{enumerate}
\end{lemma}
\begin{proof}
  The open compact subgroup $P_{F}$ of $W^{0}_{F}$ is a pro-$p$ group,
  hence the functor of $P_{F}$-invariants on continuous
  $L$-representations is exact and commutes with taking $L$-linear duals. Hence it suffices to prove (1), (2) and
  (3) for $L$-representations of the discrete group
  $W^{0}:=W_{F}^{0}/P_{F}=\langle s,\Fr\rangle$. To this aim, observe that the equality
  $$ (1-s^{q})(1-\Fr)=(t_{q}-\Fr)(1-s) \hbox{ with } t_q=1+s+\cdots+s^{q-1}$$
in $L[W^{0}]$, enables us to define the following complex :
 $$0\To{} L[W^{0}]\To{\delta} L[W^{0}]\oplus
L[W^{0}]\To{\gamma}L[W^{0}]\To{\varepsilon}L\To{} 0, $$
$$ \hbox{ where}\left\{\begin{array}[c]{l}
    \varepsilon \hbox{   is the augmentation map,} \\
   \gamma(f,g)=f(1-\Fr)-g(1-s) \\
   \delta(h)=\left(h(1-s^{q}),h(t_q-\Fr)\right) 
         \end{array}\right.
       $$
       We claim that this complex is exact.
       Admitting this for now, this gives us a projective resolution of the trivial
  representation, and shows that $H^{*}(W^{0},V)$ is the cohomology of a complex of the form
  $V\To{}V^{\oplus 2}\To{}V$. This implies (1) and (2).  Moreover, this shows that
  $H^{2}(W^{0},V)= V/((1-s^{q})V+(t_q-\Fr)V)$. Observe that the inclusion $(1-s^{q})V\subset
  (1-s)V$ has to be an equality for dimension reasons, since the action of $\Fr$ induces
  an isomorphism $(1-s)V\To\sim (1-s^{q})V$. Similarly, we have
  $(1-s^{q})V=(1-s^{q^{-r}})V$ for all $r\in\NM$.  Denoting $I^{0}:=s^{\ZM[\frac 1q]}$, this
  means that the canonical map $V/(1-s^{q})V\To{} V_{I^{0}}={\rm colim}_{r}V/(1-s^{q^{-r}})V$
  is an isomorphism. Since $t_q$ acts as multiplication by $q$ on $V_{I^{0}}$ this induces in turn an isomorphism
  $$ H^{2}(W^{0},V)\To\sim V_{I^{0}}/(q-\Fr)V_{I^{0}}= (V\otimes\omega^{-1})_{W^{0}},$$
from which we deduce (3).

  Let us now prove the exactness of the above complex. Note first that
  $\delta$ is injective since multiplication by $1-s^{q}$ is injective, and $\varepsilon$ is clearly
  surjective. To see that $\ker\varepsilon=\im\gamma$, it suffices to see that $1-w\in
  L[W^{0}](1-\Fr)+L[W^{0}](1-s)$ for all $w\in W^{0}$, which follows from the fact that $s$ and $\Fr$ generate $W^{0}$.
  It remains to check  that $\ker\gamma=\im\delta$. So let $(f,g)\in \ker\gamma$. If we
  can prove that $f$ has the form $f=h(1-s^{q})$, then $g(1-s)=f(1-\Fr)=h(t_{q}-\Fr)(1-s)$,
  hence $g=h(t_{q}-\Fr)$ since $1-s$ is not a zero divisor, and $(f,g)\in\im\delta$.   Writing
 $f=\sum_{i,j}a_{i,j}\Fr^{i}s^{j}$ with $i\in \ZM$ and $j\in\ZM[\frac 1q]$, it thus
 suffices to prove that $\sum_{k\in\ZM}a_{i,j+qk}=0$ 
 for all $(i,j)$. In the expansion $f(1-\Fr)=\sum_{i,j}b_{i,j}\Fr^{i}s^{j}$, we have
 $b_{i,j}=a_{i,j}-a_{i-1,qj}$. The fact that $f(1-\Fr)\in L[W^{0}](1-s)$ translates into
 $\sum_{k\in\ZM}b_{i,j+k}=0$ for all $i,j$, that is
 $\sum_{k\in\ZM}a_{i,j+k}=\sum_{k\in\ZM}a_{i-1,qj+qk}$ for all $i,j$,
 which we can rewrite as
 $$\forall i,j, \,\, \sum_{k\in\ZM}a_{i,j+qk}=\sum_{k\in\ZM}a_{i+1,j/q+k}.$$
 Writing $k=r+qk'$ in the right hand sum, we get
 $$\sum_{k\in\ZM}a_{i+1,j/q+k}=\sum_{r=0}^{q-1}\sum_{k'\in\ZM}a_{i+1,j/q+r+qk'}=\sum_{r=0}^{q-1}
 \sum_{k'\in\ZM}a_{i+2,j/q^{2}+r/q+k'} =\sum_{k\in\ZM}a_{i+2,j/q^{2}+k/q}.$$
Proceeding by induction, we get for any $s\in\NM$ :
 $$ \forall i,j, \,\, \sum_{k\in\ZM}a_{i,j+qk}=\sum_{k\in\ZM}a_{i+s,(j+qk)/q^{s}} .$$
  But for $s>>0$, the right hand side vanishes, hence so does the left hand side. This
  implies that $\ker\gamma=\im\delta$ and completes the proof that the complex above is exact.
\end{proof}

From this lemma and (\ref{eq:tangent}), we get : 
\begin{prop} \label{prop:tangent}
For an $L$-valued point $x: \Spec L \rightarrow \uZ^1(W_F^0/P_F^e,{\hat G})$, the
dimension of $T_x \uZ^1(W_F^0/P_F^e,{\hat G}_L)$ over 
$L$ is equal to $\dim {\hat G}_{L} + \dim H^0(W_{F}^{0}, (\Ad \varphi_x)^{*} \otimes \omega)$, where
$\omega$ is the cyclotomic character of $W_F$.  
\end{prop}
\begin{proof}
We have seen that the dimension of $T_x \uZ^1(W_F^0/P_F^e,{\hat G}_L)$ is equal to that of
$Z^1(W_{F}^{0}, \Ad \varphi_{x})$.   The  latter is equal to
the dimension of $H^1(W_{F}^{0}, \Ad \varphi_x)$ plus the dimension of the space of coboundaries (principal crossed homomorphisms).  These are
all of the form $w \mapsto wy - y$, where $y$ is an element of $\Ad \varphi_x$.  The dimension of $\Ad \varphi_x$ is equal to $\dim \hat G_{L}$,
and those $y$ that give the zero element of $Z^1(W_{F}^{0},\Ad \varphi_x)$ are precisely those fixed by $W_{F}^{0}$.  Thus we have
$$\dim Z^1(W_{F}^{0}, \Ad \varphi_x) = \dim H^1(W_{F}^{0}, \Ad \varphi_x) + \dim {\hat G}_{L} - \dim H^0(W_{F}^{0}, \Ad \varphi_x).$$
Hence the proposition follows from the last lemma applied with $V=\Ad\varphi_{x}$.
\end{proof}

\begin{cor}
  The point $x$ is a smooth point of
$\uZ^1(W_F^0/P_F^e,{\hat G}_L)$ if and only if $H^0(W_{F}^{0}, (\Ad \varphi_x)^{*} \otimes \omega)=0$.
\end{cor}
\begin{proof}
  We know that the algebraic $L$-scheme $\uZ^1(W_F^0/P_F^e,{\hat G}_L)$ has pure dimension
  $\dim\HG_{L}$. Therefore the local ring at the closed point $x$ has dimension $\dim\HG_{L}$, while
  its tangent space has dimension $\dim\HG_{L}+\dim H^0(W_{F}^{0}, (\Ad \varphi_x)^{*} \otimes \omega)$
  by the last proposition.
\end{proof}

\begin{rem}
  It is interesting to note that the obstruction theory naturally suggested by the moduli
  problem is ``optimal'', in that it faithfully detects smoothness of points.
Namely, let $A$ be a finite length local $L$-algebra with residue field $L$, $\tilde x: \Spec A \rightarrow \uZ^1(W_F^0/P_F^e,{\hat G}_L)$
a map whose composition with the map $\Spec L \rightarrow \Spec A$ is equal to $x$, and let $A'$ be a small extension of $A$; that is, a finite length local $L$-algebra
with residue field $L$, and a principal ideal $I \subseteq A'$ such that $I$ is annihilated by the maximal ideal of $A'$, and an isomorphism $A'/I \cong A$. The
problem of lifting $\tilde x$ to a map $\tilde x': \Spec A' \rightarrow
\uZ^1(W_F^0/P_F^e,{\hat G}_L)$ is equivalent to lifting  the  $1$-cocycle
$\varphi_{\tilde x}:W_{F}^{0}/P_{F}^{e}\To{} \HG(A)$ to a $1$-cocycle
$\varphi_{\tilde x'}:W_{F}^{0}/P_{F}^{e}\To{} \HG(A')$.
This problem is standard: let $\varphi'$ be any lift of $\varphi_{\tilde x}$ to a
continuous {\em function} (not necessarily 
a cocycle): $W_{F}^{0} \rightarrow \HG(A')$.  Then the map taking $w_1,w_2 \in W_{F}^{0}$ to $\varphiL'(w_1w_2)^{-1}\varphiL'(w_1)\varphiL'(w_2)$ is a $2$-cocycle with values in
$(\Ad \varphi_x) \otimes I$, and we can adjust our choice of $\varphi'$ to yield a
$1$-cocycle $\varphi_{\tilde x'}$ lifting $\varphi_{\tilde x}$ if, and only if, this
$2$-cocycle is 
a coboundary.  We thus obtain an obstruction theory for $\uZ^1(W_F^0/P_F^e,{\hat G}_{L})$ in a
formal neighborhood of $x$ with values in $H^2(W_{F}^{0}, \Ad
\varphi_x)=H^{0}(W_{F}^{0},(\Ad\varphi_{x})^{*}\otimes\omega)$.  
Now, since an unobstructed point is smooth, the last corollary says that the 
obstruction to lifting vanishes if and only if the space which it
naturally belongs to vanishes. \emph{For this reason, we will indifferently
use the words ``unobstructed'' or ``smooth''  to denote
these points in the rest of this section.}
\end{rem}


\subsection{Existence of unobstructed points} \label{sec:exist-unobstr-points}


The primary goal of this section is to show that if the characteristic $\ell$ of $L$ does not lie in
an explicit finite set (depending only on $\HG$ and its $W_{F}$-action), the fiber $\uZ^1(W_F^0/P_F^e,{\hat G}_L)$ is generically smooth.  We have already established this smoothness in characteristic zero,
\emph{so we assume henceforth that $L$ has finite characteristic $\ell$}. In this case,
the restriction map $\uZ^1(W_F/P_F^e,{\hat G}_L)\To\sim\uZ^1(W_F^0/P_F^e,{\hat G}_L)$ is
an isomorphism, so there is no need to distinguish between $W_{F}^{0}$ and $W_{F}$.

\begin{notn}
  Let $\varphi: W_F \rightarrow \HG(L)$ be a continuous $1$-cocycle.
 For any $g \in C_{\HG}(\varphi_{|I_{F}})(L)$ there is a unique continuous  $1$-cocycle
 $\varphi^g$, whose restriction to $I_F$ is 
  $\varphi_{|I_{F}}$, and such that $\varphi^g(\Fr) = g\varphi(\Fr)$.
\end{notn}
The rough version of our main result here is :

\begin{thm} \label{thm:unobstructed}
 There is a finite set of primes $S$, depending only on $\HG$ and the image of $W_{F}$ in
 ${\rm Out}(\HG)$, such that, if
$\ell\notin S$, then for any 
continuous $1$-cocycle  $\varphi: W_F \rightarrow \HG(L)$, there exists a $g \in
C_{\HG}(\varphi_{|I_{F}})^\circ(L)$ such that $\varphi^g$ is unobstructed.
\end{thm}
\def\XX{{\mathbb X}}
 \def\MM{{\mathbb M}}
 In order to state a more precise version, we need notations
  (\ref{eq:defchi}) and (\ref{eq:defchistar}) of
 the appendix.
 In particular, $h_{\HG,1}$ is the Coxeter number of the root system
 of $\HG$.

 \begin{thm} \label{thm:unobstructed_explicit}
    Let $\Fr$ be a lift of Frobenius in $W_{F}$, and denote by $e$ the tame
   ramification index of the finite extension of $F$ whose Weil group is 
 the kernel of $W_{F}\To{}{\rm Out}(\HG)$. Then the set $S$ in Theorem
   \ref{thm:unobstructed} can be taken as 

   (1) $S=\{\hbox{primes } \ell$ dividing $e. \chi_{\HG,\Fr}^{*}(q)\}$, whatever $\HG$ is.

   (2) $S=\{\hbox{primes } \ell$ dividing $e. \chi_{\HG,\Fr}(q) .(h_{\HG,1})!\}$
   if $\HG$ has no exceptional factor.
 \end{thm}

 Here, ``exceptional'' includes triality forms of $D_{4}$. Note also that $\ell$ not
 dividing $\chi_{\HG,\Fr}^{*}(q)$ is equivalent to
  $q$ having order greater than $h_{\HG,\Fr}$ in $\FF_{\ell}^{\times}$, which implies $\ell>h_{\HG,\Fr}$ hence also
  $\ell> h_{\HG,1}$.  We will also prove that, in the case where $\HG$ has no exceptional
  factor and the action of $W_{F}$ is unramified  and $\ell>h_{\HG,1}$,  
the condition $\chi_{\HG,\Fr}(q)\neq 0$ in $\FF_{\ell}$ is also necessary to have generic smoothness.

\medskip

We now start the proofs of Theorems \ref{thm:unobstructed} and
\ref{thm:unobstructed_explicit}.  Fix a $\varphi$  as in Theorem \ref{thm:unobstructed}; our
first step will be to reduce to a setting in which the action of $W_{F}$ on $\HG$ is
unramified and stabilizes a pinning, and the image of $\varphi_{|I_{F}}$ is unipotent.

Denote by $\phi^{\ell}$ the restriction of $\varphi$ to $I_{F}^{\ell}$ and by
$\alpha^{\ell}$ the composition
$W_{F}\To{\varphi}C_{\LG}(\phi^{\ell})\To{}\tilde\pi_{0}(\phi^{\ell})$, so that $\varphi$
lies in the closed subscheme
$\uZ^{1}(W_{F}/P_{F}^{e},\HG_{L})_{\phi^{\ell},\alpha^{\ell}}$, as defined in subsection \ref{sec:decomp-at-ell}.
Then, according to Theorem \ref{thm:connectedZl}, the connected
component of $\uZ^{1}(W_{F}/P_{F}^{e},\HG_{L})$ that contains $\varphi$ has the form
$$ \HG\times^{C_{\HG}(\phi^{\ell})^{\circ}} \uZ^{1}(W_{F}^{0}/P_{F}^{e},\HG_{L})_{\phi^{\ell},\alpha^{\ell}}.$$
Thus we see that $\varphi$ is a smooth point of $\uZ^{1}(W_{F}^{0}/P_{F}^{e},\HG_{L})$ if and only if
it is a smooth point of
$\uZ^{1}(W_{F}/P_{F}^{e},\HG_{L})_{\phi^{\ell},\alpha^{\ell}}$.

By Theorem \ref{Borel_preserving_ell}, there exists
$\varphi'\in Z^{1}(W_{F}/P_{F}^{e},\HG(L))_{\phi^{\ell},\alpha^{\ell}}$ such that the
action of $W_{F}$ on $C_{\HG}(\phi^{\ell})^{\circ}$ via 
$\Ad_{\varphi'}$   preserves a Borel pair. Actually, we a have better result in this setting :
\begin{prop} \label{prop:pinning_preserving_acf}
  We can choose $\varphi'\in Z^{1}(W_{F}/P_{F}^{e},\HG(L))_{\phi^{\ell},\alpha^{\ell}}$ so
  that the action of $W_{F}$ on $C_{\HG}(\phi^{\ell})^{\circ}$ via  
  $\Ad_{\varphi'}$   preserves a pinning.
\end{prop}
\begin{proof}
 We take up the  proof of Proposition \ref{Borel_preserving_acf}, replacing $P_{F}$ by
 $I_{F}^{\ell}$ and ``Borel pair'' by ``pinning''. Since the fixator of a pinning of 
$C_{\HG}(\phi^{\ell})^{\circ}$ under conjugation is the center 
 $Z$ of $C_{\HG}(\phi^{\ell})^{\circ}$, the argument of that proof shows that the
 obstruction to the existence of a cocycle $\varphi'$ as in this proposition lies in the
 group $H^{2}(W_{F}/I_{F}^{\ell},Z(L))$. On the other hand, repeating the proof of Lemma
 \ref{cohomological_lemma} for  $W_{F}/I_{F}^{\ell}$ instead of $W_{F}/P_{F}$ shows that
 this cohomology group vanishes if $Z(L)$ can be proved to be  $\ell$-divisible.

To prove this, recall that the group scheme $Z$ is diagonalisable and let $M$ be its
character group. Then we have a
(non canonical) decomposition $M\simeq M_{\ell-\rm tors}\times M_{\ell'-\rm tors}\times
M_{\rm free}$ which induces a decomposition $Z\simeq Z_{\ell} \times Z^{\ell} \times T_{Z}$, where
$T_{Z}$ is a torus, $Z^{\ell}$ is finite smooth and $Z_{\ell}$ is finite infinitesimal.
Correspondingly we get $Z(L)\simeq  Z^{\ell}(L)\times T_{Z}(L)$.
Now, $T_{Z}(L)$ is clearly $\ell$-divisible since $L$ is
algebraically closed and $Z^{\ell}(L)$ has prime-to-$\ell$ order hence is also $\ell$-divisible.
\end{proof}

Choose $\varphi'$ as in this proposition and recall that the action $\Ad_{\varphi'}$ on $C_{\HG}(\phi^{\ell})$
factors over the quotient $W_{F}/I_{F}^{\ell}$. Then we have an isomorphism $\eta\mapsto \eta\cdot\varphi'$
$$\uZ^{1}_{\Ad_{\varphi'}}(W_{F}/I_{F}^{\ell},C_{\HG}(\phi^{\ell})^{\circ}) \To\sim
\uZ^{1}(W_{F}/P_{F}^{e},\HG_{L})_{\phi^{\ell},\alpha^{\ell}},
$$
The isomorphism above shows that $\varphi$ is an unobstructed point of
$\uZ^{1}(W_{F}/P_{F}^{e},\HG_{L})$ if, and only if, $\varphi\cdot{\varphi'}^{-1}$ is
an unobstructed point of
$\uZ^{1}_{\Ad_{\varphi'}}(W_{F}/I_{F}^{\ell},C_{\HG}(\phi^{\ell})^{\circ})$.
So we
are reduced to study unobstructedness in a much simpler case, but in order to make this reduction step effective, we
need some control on  $\Ad_{\varphi'}$.

\begin{lemma}\label{lemma:control_ad}
  Fix $\varphi'$ as in Proposition \ref{prop:pinning_preserving_acf}, let $w\in W_{F}$
  and denote by  $o_{w}$ its order in ${\rm Out}(\HG)$. Then
 $\Ad_{\varphi'}(w)$ has order dividing $o_{w}|\Omega_{\HG}|$ in ${\rm
   Aut}(C_{\HG}(\phi^{\ell})^{\circ})$, where $\Omega_{\HG}$ denotes the Weyl group of $\HG$. 
\end{lemma}
\begin{proof}
  Put $\hat H:=C_{\HG}(\phi^{\ell})^{\circ}$ and let $T_{\hat H}$ be a maximal torus of
  $\hat H$ that is
part of a pinning stable under $\Ad_{\varphi'}$. Pick a maximal torus $\HT$ of $\HG$ that
contains $T_{\hat H}$. Then there is an element $m$ of the centralizer of $T_{\hat H}$ in $\HG$ such
that $g_{w}:=m  {^{L}\varphi'}(w)$ normalizes $\HT$.
The action of  $(g_{w})^{o_{w}}$ on $X^{*}(\HT)$ is the action of an element of $\HG$ that normalizes
$\HT$, so its order divides $|\Omega_{\HG}|$. Hence,
$\Ad_{\varphi'}(w)^{o_{w}|\Omega_{\HG}|}$ acts trivially on $T_{\hat H}$ and therefore also on
$\hat H$, since it stabilizes a pinning and fixes the maximal torus of this pinning.
\end{proof}

As in Theorem \ref{thm:unobstructed_explicit}, denote by $e$ the tame ramification index
of the finite extension of $F$ whose Weil group is 
the kernel of $W_{F}\To{}{\rm Out}(\HG)$. Applying this lemma to a suitable lift of our
generator $s$ of tame inertia, we see that,
\emph{if we assume that $\ell$ is prime to $e |\Omega_{\HG}|$} (which is satisfied if
$\ell$ is prime to $e$ and $\ell>h_{\HG,1}$), then
the action $\Ad_{\varphi'}$ is \emph{unramified}.
For this reason, we will now focus on the following particular setting :
\begin{equation}
  \label{eq:setting}
  \left\{
    \begin{array}[c]{l}
      \hbox{-- the action of $W_{F}$ on $\HG$ is unramified and stabilizes a pinning,}\\
      \hbox{-- the restriction of $\varphi$ to ${I_{F}^{\ell}}$ is trivial.}
    \end{array}
  \right.  
\end{equation}
In this setting, $\LG$ is the Langlands dual group of a uniquely determined quasi-split
unramified reductive group $G$ over $F$, and the restriction  $\varphi_{|I_{F}}$ is
determined by $u:=\varphi(s)$ which is a unipotent element of $\HG(L)$. By definition, the
cocycle $\varphi$ corresponds to an  unobstructed point if, and only if, $q^{-1}$ is not
an eigenvalue of $(\Ad_{\varphi})^{*}(\Fr)$ on $(\Lie\HG_{L})^{*,\Ad^{*} u}$.

\begin{lemma} \label{lemma:isogenies}
  Let $\HG$ and $\varphi$ be as in (\ref{eq:setting}), put $u:=\varphi(s)$, and assume
  that $\ell$ is prime to $|\pi_{1}(\HG_{\rm ad})|$. Endow the
  reductive group $\HG':=\HG_{\rm ad}\times\HG_{\rm ab}$ with the unique action of $W_{F}$
  that makes the isogeny $\pi:\HG\To{}\HG'$ equivariant, and put  $\varphi':=\pi\circ\varphi$.
  Then there is $g\in C_{\HG}(u)^{\circ}(L)$ such that $\varphi^{g}$ is
  unobstructed if, and only if, there is $g'\in C_{\HG'}(\pi(u))^{\circ}(L)$ such that ${\varphi'}^{g'}$
  is unobstructed.
\end{lemma}
\begin{proof}
As with any isogeny, $\pi$ induces an isomorphism from the unipotent subvariety of $\HG$ to
that of $\HG'$. In particular, the map $C_{\HG}(u)(L)\To{\pi}C_{\HG'}(\pi(u))(L)$ is surjective for
all unipotent $u\in\HG(L)$, and so is the map $C_{\HG}(u)^{\circ}(L)\To{\pi}C_{\HG'}(\pi(u))^{\circ}(L)$.
Therefore, it suffices to prove that $\varphi$ is unobstructed if and only if $\varphi'$
is unobstructed. Note that $\ker\pi=\ker(\HG_{\rm der}\To{}\HG_{\rm ad})$  is a finite
diagonalisable group scheme
whose order divides the order
of $\pi_{1}(\HG_{\rm ad})$, hence is prime to $\ell$ by our assumption.
So $\pi$ is a separable isogeny and $d\pi_{L}$ induces a
$(\Ad{\varphi},\Ad\varphi')$-equivariant isomorphism 
$\Lie(\HG_{L})\To\sim \Lie(\HG'_{L})$, and the desired property follows.
\end{proof}

 \begin{rem}\label{rem:lift}
   Let $\pi : \HG\To{}\HG'$ be as in the lemma or, more generally, any surjective morphism
   with central $\Fr$-stable kernel. Then any $\varphi'\in
   Z^{1}(W_{F}/I_{F}^{\ell},\HG')$ lifts through $\pi$, in the sense that there is some
   $\varphi \in   Z^{1}(W_{F}/I_{F}^{\ell},\HG)$ such that $\varphi'=\pi\circ\varphi$.
   Indeed, let $\varphi(s)\in \HG(L)$ be the unique unipotent element above $\varphi'(s)$,
   and let $\varphi(\Fr)\in \HG(L)$ be any  lift of  $\varphi'(\Fr)$. Then
   $(\varphi(\Fr)\rtimes\Fr)\varphi(s)(\varphi(\Fr)\rtimes\Fr)^{-1}$ is unipotent and above
     $\varphi'(s)^{q}$, so it is equal to $\varphi(s)^{q}$.
 \end{rem}


This lemma allows us to further reduce the 
setting (\ref{eq:setting}) to the cases where $\HG$ is a torus or an adjoint group.
 Dealing with tori is quite easy :
 \begin{lemma} \label{lemma:torus}
   If $\HG$ as in (\ref{eq:setting}) is a torus, then the following are equivalent :
   \begin{enumerate}
   \item there is an unobstructed $\varphi$ in $\uZ^{1}(W_{F}/I_{F}^{\ell},\HG_{L})$,
   \item any $\varphi$ in $\uZ^{1}(W_{F}/I_{F}^{\ell},\HG_{L})$ is unobstructed,
   \item $\chi_{\HG,\Fr}(q)\neq 0$ in $L$.
   \end{enumerate}

 \end{lemma}
 \begin{proof}
   For any $\varphi$ in $Z^{1}(W_{F}/I_{F}^{\ell},\HG(L))$, we have $\varphi(s)=1$ and
   $^{L}\varphi(\Fr)\in \HG(L)\rtimes\Fr$, so the condition for $\varphi$ to be unobstructed is
   that $H^{0}(\langle\Fr\rangle,\omega\otimes\Lie(\HG_{L})^{*})=0$, which is independent of $\varphi$, and 
   equivalent to $q^{-1}$ not being an eigenvalue of $(\Ad_{\Fr})^{*}$ on $\Lie(\HG_{L})^{*}$. Since
   $\Lie(\HG_{L})^{*}=X^{*}(\HG)\otimes L$, we have
   $$\det\left(q(\Ad_{\Fr})^{*}-\id | \Lie(\HG_{L})^{*}\right)=
   \det\left(q\Fr-\id | X^{*}(\HG)\right)_{L}= \chi_{\HG,\Fr}(q^{-1})_{L}$$
   where the subscripts $L$ denote the image of an integer in $L$.
   Hence we see that $q^{-1}$ is an eigenvalue of $\Ad_{\Fr}$ if and only if
   $\chi_{\HG,\Fr}(q^{-1})=0$ in $L$, which is equivalent to $\chi_{\HG,\Fr}(q)=0$ in $L$
   since $\chi_{\HG,\Fr}$ is a product of cyclotomic polynomials.
\end{proof}

Let us now deal with the adjoint part. We have a $\Fr$-equivariant decomposition as a
product of simple adjoint groups 
\begin{equation}
\HG_{\rm ad} = \underbrace{\HG_{11}\times\cdots \times  \HG_{1f_{1}}}_{\HG_{1}}\times
\cdots \times \underbrace{\HG_{r1}\times\cdots \times \HG_{rf_{r}}}_{\HG_{r}}\label{eq:decompsimple}
\end{equation}
where $\Fr$ permutes cyclically  $\HG_{i1} \to \HG_{i2}\to\cdots\to \HG_{if_{i}}$ and
$\Fr^{f_{i}}$ restricts to an outer automorphism of $\HG_{i1}$. Accordingly, $\varphi$
decomposes as a product $\varphi_{1}\times\cdots\times\varphi_{r}$ with $\varphi_{i}\in
Z^{1}(W_{F},\HG_{i}(L))$
and we see that $\varphi$ is unobstructed if and only if each $\varphi_{i}$ is unobstructed.
Denote by $F_{f_{i}}$  the unramified extension of degree $f_{i}$ of $F$.  Then we have a
Shapiro morphism $\uZ^{1}(W_{F},\HG_{i})\To{} \uZ^{1}(W_{F_{f_{i}}},\HG_{i1})$   
 given by $\varphi_{i}\mapsto \varphi'_{i}:=\pi_{i1}\circ (\varphi_{i})_{|W_{F_{f_{i}}}}$, where
 $\pi_{i1}$ is the projection onto $\HG_{i1}$.
 \begin{lemma}
  The Shapiro morphism $\uZ^{1}(W_{F},\HG_{i})\To{} \uZ^{1}(W_{F_{f_{i}}},\HG_{i1})$ is smooth.
\end{lemma}
\begin{proof}
  Denote  by $\varphi_{ij}:=\pi_{ij}\circ
  \varphi_{i}$ the $j$-th component of $\varphi_{i}$ and define $\uZ_{i1}$ to be the
  affine scheme whose $R$-points are given  by
  $$ \uZ_{1i}(R):=\left\{f: W_{F}\To{} \HG_{i1}(R),\forall w'\in
    W_{F_{f_{i}}}, f(w'w)=f(w')\cdot{^{w'}f}(w)\right\}$$
  for any $\ZM[\frac 1p]$-algebra $R$.  We claim that the map
  $\varphi_{i}\mapsto \varphi_{i1}$ induces an isomorphism
  $ \uZ^{1}(W_{F},\HG_{i})\To\sim \uZ_{i1}$.
Indeed, denoting by $\tilde\Fr$ a lift of $\Fr$ in $W_{F}$, the cocycle condition on $\varphi_{i}$
  implies that for any integer $j$ we have
  $$ \varphi_{i}(w)= {^{\Fr^{j}}}
  \left(\varphi_{i}(\tilde\Fr^{-j})^{-1}\varphi_{i}(\tilde\Fr^{-j}w)\right).$$
  Taking the $j$-th component, this shows that  $\varphi_{ij}$ is determined by
  $\varphi_{i1}$ and this gives a formula for the putative inverse to $\varphi_{i}\mapsto \varphi_{i1}$.
  Namely, given $f\in \uZ_{i1}(R)$, define $\varphi_{i} :
  W\To{}\HG_{i}(R)$ component-wise by $\varphi_{ij}(w):={^{\Fr^{j}}}
  \left(f(\tilde\Fr^{-j})^{-1}f(\tilde\Fr^{-j}w)\right)$. Then a computation shows that
  $\varphi_{i}\in Z^{1}(W_{F},\HG_{i}(R))$ and that this defines the desired inverse
  isomorphism.
  But now, the map $\varphi_{i1} \mapsto ( (\varphi_{i1})_{|W_{F_{f_{i}}}},
  \varphi_{i1}(\tilde\Fr),\cdots, \varphi_{i1}(\tilde\Fr^{f_{i}-1}))$
  defines an isomorphism
  $\uZ_{i1}\To{}\uZ^{1}(W_{F_{f_{i}}},\HG_{i1})\times (\HG_{i1})^{f_{i}-1}$
  and the Shapiro morphism becomes the projection on the first factor.
\end{proof}

  Resuming the discussion above the lemma,  we see that $\varphi$ is unobstructed if,
and only if, for each $i=1,\cdots,r$, the cocycle
$\varphi'_{i}:\, W_{F_{f_{i}}}\To{}\HG_{i1}(L)$ is unobstructed.
On the other hand, it follows from the definitions that
$\chi_{\HG_{i},\Fr}(T)=\chi_{\HG_{i1},\Fr}(T^{f_{i}})$ so that, coming
back to a general $\HG$, we have the following equality :
\begin{equation}
  \label{eq:chi_decomp}
  \chi_{\HG,\Fr}(T)= 
  \chi_{\HG_{\rm ab},\Fr}(T) \chi_{\HG_{11},\Fr^{f_{1}}}(T^{f_{1}})\cdots \chi_{\HG_{r1},\Fr^{f_{r}}}(T^{f_{r}}). 
\end{equation}
In this way,  we are reduced to study the case where $\HG$ is \emph{simple} and \emph{adjoint}.

\subsection{The simple adjoint case}
In light of the above discussion, \emph{we will now focus on the case
  where $\HG$ is simple adjoint in the  setting  (\ref{eq:setting}).}

In this case, it will come in handy to express the unobstructedness condition on the adjoint
representation, as opposed to the coadjoint one. Recall
that a prime $\ell$ is \emph{good} for $\HG$ if it does not divide the
coefficient of any root of $\HG$ when expressed as a linear combination of simple
roots. Moreover, $\ell$ is called \emph{very good} if it is good and does
not divide the order of the fundamental group of the root system of $\HG$. 

\begin{thm}[Springer-Steinberg]
  Suppose $\HG$ is simple adjoint and $\ell$ is very good for $\HG$. Then a suitable rational multiple of the  Killing
  form on $\Lie\HG_{\rm sc}$ induces a non-degenerate bilinear form on $\Lie \HG_{L}$.
\end{thm}
\begin{proof}
According to \cite[p.180]{Spr_St} (see also \cite[\S 5]{Gross_Nebe}), the discriminant of
the Killing form on $\Lie\HG_{\rm sc}$ divided by $2$ times the dual Coxeter number of $\HG$ is
prime to $\ell$. Moreover, since $\ell$ does not divide the degree of the isogeny
$\HG_{\rm sc}\To{}\HG$, this isogeny induces an isomorphism  $\Lie(\HG_{\rm sc})_{L}\To\sim
\Lie\HG_{L}$.
\end{proof}
Since the Killing form is invariant under the automorphism group of $\HG$, we see that
a cocycle $\varphi$ corresponds to an  unobstructed point if, and only if, $q^{-1}$ is not
an eigenvalue of $\Ad_{\varphi}(\Fr)$ on $(\Lie\HG_{L})^{\Ad u}$.

Note that
$(\Lie\HG_{L})^{\Ad u}$ is the Lie algebra of the scheme-theoretic centralizer $C_{\HG_{L}}(u)$ of $u$ in
$\HG_{L}$, which may not be reduced. Following standard notation, we denote by $\HG_{u}$
the \emph{reduced} centralizer of $u$, which is a closed smooth algebraic subgroup of
$\HG_{L}$.  The following result of Slodowy will be useful in our discussion below. 
\begin{thm}\cite[p.38]{Slodowy} \label{thm:slodowy}
 If $\ell$ is very good for $\HG$, then $C_{\HG_{L}}(u)$ is smooth for all unipotent
 elements $u\in\HG(L)$, so that $C_{\HG_{L}}(u)=\HG_{u}$ and
 $(\Lie\HG_{L})^{\Ad u}=\Lie \HG_{u}$.
\end{thm}

Our arguments below will extensively make use of
the following tool to construct points in $\uZ^{1}(W_{F}/I_{F}^{\ell},\HG)$. 
Assume that we are given 
\begin{itemize}
\item a homomorphism $\lambda: \SL_2 \rightarrow {\hat G}_{L}$ and
\item an element $\CF\in (\HG(L)\rtimes\Fr)_{\lambda}$, i.e. an element of $\LG(L)$ that centralizes $\lambda$
  and projects to $\Fr$.
\end{itemize}
Then there is a unique $1$-cocycle 
$\varphi:\,W_{F}/I_{F}^{\ell}\To{}\HG(L)$ such that
\begin{equation}
 \varphi(s)=\lambda(U) \hbox{ and } \varphiL(\Fr)=\lambda(S)\CF,\label{eq:nice_form}
 \end{equation}
where $S$ and $U$ denote the matrices $\left(\begin{smallmatrix}
  q^{\frac{1}{2}}&0\\0&q^{-\frac{1}{2}}\end{smallmatrix}\right)$ and 
$\left(\begin{smallmatrix} 1&1\\0&1\end{smallmatrix}\right)$ in
$\SL_2(L)$, respectively, and where $q^{\frac 12}$ is a choice of a square root of $q$ in
$L$.

However, we will need a condition to ensure  exhaustivity of this construction.
Recall that over characteristic zero fields, for any unipotent element $u$ in $\HG(L)$ there is a
homomorphism $\lambda: \SL_2 \rightarrow {\hat G}_{L}$ such that $\lambda(U) = u$
and, moreover, $\lambda$ is
unique up to ${\hat G}_{L}$-conjugacy.
In finite characteristic $\ell$, the situation  is more subtle. An obvious necessary
condition for the existence of $\lambda$ is that $u$ have order $\ell$. When $\ell$ is good
for $\HG$, this was proven to
be sufficient by Testerman in \cite{testerman}.
In order to study uniqueness, Seitz~\cite{seitz} has introduced the following notion :
 a morphism $\lambda: \SL_2 \rightarrow {\hat G}_L$ over $L$ is a ``good $\SL_{2}$''
if the weights of the conjugation action of the maximal torus
$T_2 \subset \SL_{2}$ on $\Lie({\hat G})$ are bounded above by $2\ell - 2$ (here we identify
$T_{2}$ to $\mathbb{G}_{m}$ via the map $\left(\begin{smallmatrix} z
    &0\\ 0&z^{-1 }\end{smallmatrix}\right)\mapsto z$
). 

\begin{thm}[\cite{seitz}, Theorems 1.1 and 1.2]
Suppose $\HG$ is simple adjoint and $\ell$ is a good prime for $\hat
G$, and let $u$ be a unipotent element of ${\hat G}(L)$ of order
$\ell$.  Then there is a ``good $\SL_2$'' $\lambda: \SL_2 \rightarrow
{\hat G}_{L}$ such that $\lambda(U) = u$.  Moreover, any two such $\lambda$ are conjugate  by an $L$-point of the unipotent radical $R_u({\hat G}_u)$.  Finally, the centralizer ${\hat G}_{\lambda}$
of $\lambda$ in ${\hat G}$ is reductive, and $\hat G_u = {\hat G}_{\lambda} R_u({\hat G}_u)$.
\end{thm}

In order to ensure that all non-trivial unipotent elements of $\HG(L)$ have order $\ell$, we will
henceforth assume that
\begin{center}
  \emph{$\ell> h$, where $h=h_{\HG,1}$ is the Coxeter number of $\HG$.}
\end{center}
Indeed, since $h$ is one plus the height of the highest positive root of $\HG$, it follows  from 
Proposition 3.5 of~\cite{seitz} and the Bala-Carter classification, that any nontrivial
unipotent element of ${\hat G}(L)$ has order $\ell$ under this hypothesis. Moreover, such
an $\ell$ is also automatically good for $\HG$, so that Seitz'
theorem applies to any $u$ under this hypothesis, and even very
good for $\HG$, so that Slodowy's theorem \ref{thm:slodowy} also holds.

\begin{cor} \label{cor:reduc-to-nice-form}
  Let $\HG$ and $\varphi$ be as in (\ref{eq:setting}) with $\HG$
  simple adjoint, and
  suppose that $\ell>h_{\HG,1}$. Then there is $g\in (\HG_{u})^{\circ}(L)$ such that
  $\varphi^{g}$ is of the form (\ref{eq:nice_form}) associated to a pair
  $(\lambda,\CF)$ such that $\CF$ normalizes a Borel pair (or even a pinning) of $(\HG_{\lambda})^{\circ}$. 
\end{cor}
\begin{proof}
      Let us choose a ``good $\SL_2$'' $\lambda: \SL_2 \rightarrow {\hat G}_{L}$ with $\lambda(U) = u:=\varphi(s)$.
 Set $\CF_1: = \lambda(S)^{-1} . {^{L}\varphi}(\Fr) $. 
Then $\CF_1\in \HG\rtimes\Fr$ centralizes $u$, so $^{\CF_1}\lambda$ is a second ``good $\SL_2$'' that takes
$U$ to $u$.  
Since any two such are conjugate by an element centralizing $u$, 
we have a unipotent element $u' \in R_u({\hat G}_{u})$ such that
$^{u'}\lambda={^{\CF_1}\lambda}$; then $\CF_2 = {u'}^{-1} \CF_1$ centralizes $\lambda$ and, in
particular, normalizes $(\HG_{\lambda})^{\circ}$. 
Choose a pinning $\varepsilon$
in $({\hat G}_{\lambda})^\circ$; then there
exists $h \in ({\hat G}_{\lambda})^\circ(L)$ such that $^{h}\varepsilon={^{\CF_{2}}\varepsilon}$.
Then $\CF := h^{-1} \CF_2$ 
still centralizes $\lambda$, and preserves $\varepsilon$.
Now,
$$ {^{L}\varphi}(\Fr)=  \lambda(S) u' h\CF = (\lambda(S) u'h \lambda(S)^{-1})(\lambda(S)\CF)$$
with $u'$ in the unipotent radical of ${\hat G}_{u}$ and $h$ in $({\hat G}_{\lambda})^\circ(L)$.
Thus $hu'$ lies in $(\HG_{u})^\circ(L)$, and since $\lambda(S)$ normalizes $({\hat G}_u)^{\circ}$, it
follows that $ \lambda(S)\CF \in ({\hat G}_{u})^\circ(L). {^{L}\varphi}(\Fr)$.
\end{proof}

We now consider a particular case, which shows that the condition $\chi_{\HG,\Fr}(q)\neq 0$
in $L$ is necessary for the existence of unobstructed translates.

\begin{prop}\label{prop:regular-unip}
  Let $\HG$ be simple adjoint,
  and assume that $\ell>h_{\HG,1}$. Then there exists $\varphi$ as in
  (\ref{eq:setting}) such that $\varphi(s)$ is regular unipotent. Moreover, the
  following properties are equivalent :
  \begin{enumerate}
  \item There is an unobstructed $\varphi$ such that $\varphi(s)$ is regular unipotent.
  \item Any $\varphi$ with $\varphi(s)$ regular unipotent is unobstructed.
  \item $\chi_{\HG,\Fr}(q)\neq 0$ in $L$.
  \end{enumerate}
\end{prop}
\begin{proof} Fix a pinning $(\hat T, \hat B, (X_{\alpha})_{\alpha\in\Delta})$ stable
  under $\Fr$. The sum  $E=\sum_{\alpha\in\Delta }X_{\alpha}$ is a regular nilpotent
  element of $\Lie(\HG)$, which is fixed by $\Fr$. Moreover,
  the sum $H=\sum_{\beta\in\Phi^{+}}
  \beta^{\vee}\in \Lie(\hat T_{L})$ is also fixed by $\Fr$ (here $\Phi^{+}$
  denotes the set of positive roots and we denote by $\check\beta$ the
  image of the associated coroot in $\Lie(\HT_{L})\simeq
  X_{*}(\HT)\otimes L$). Then the pair $(H,E)$ is part of a unique
  principal $\mathfrak{s}\mathfrak{l}_{2}$-triple, which is also fixed under $\Fr$.
  Now, pick a regular  unipotent $u\in\HG(L)$ and a good $\SL_{2}$, say $\lambda:\,\SL_{2}\To{}\HG_{L}$,
  such that  $\lambda(U)=u$. Then, evaluating $d\lambda$ on the
  standard basis $\mathfrak{s}\mathfrak{l}_{2}$ yields another
  principal $\mathfrak{s}\mathfrak{l}_{2}$-triple. The latter has to
  be  conjugate  to $(F,H,E)$ by some element $g\in\HG(L)$, which  means that, after
  conjugating by $g$, we may assume that $\lambda$ (and therefore $u$) is
  fixed by $\Fr$. Then we can construct $\varphi$ as desired by putting
  $\varphi(s):={\lambda(U)}$ and $\varphiL(\Fr):={\lambda(S)}\rtimes \Fr$.

  If $\varphi'$ is another cocycle with $\varphi'(s)$ regular unipotent, then we
  may conjugate it so that $\varphi'(s)=\varphi(s)=u$, and this does not affect the property
  of being unobstructed. Then $^{L}\varphi'(\Fr)={^{L}\varphi}(\Fr)g$ for some $g\in\HG_{u}(L)$, and
  $\varphi'$ is unobstructed if and only if $q^{-1}$ is not an eigenvalue of
  $(\Ad\varphi')(\Fr)$ on $\Lie(\HG)^{\Ad_{u}}$. But under our
running assumption $\ell>h_{\HG}$, which implies that $\ell$ is very
good for $\HG$, Theorem \ref{thm:slodowy} implies
that $\Lie(\HG)^{\Ad u}=\Lie(\HG_{u})$. Moreover $\HG_{u}$ is known to be commutative,  hence
  $(\Ad\varphi')(\Fr)=(\Ad\varphi)(\Fr)$ on $\Lie(\HG)^{\Ad_{u}}$ and we have the equivalence of (1)
  and (2).

  It remains to study when $q^{-1}$ is an eigenvalue of
  $(\Ad\varphi)(\Fr)$. Observe that
  $\Lie(\HG)^{\Ad_{u}}$ coincides with the centralizer $\Lie(\HG)_{E}$
  of $E$ in $\Lie(\HG)$. Moreover, our hypothesis $\ell> h_{\HG,1}$
  implies that $\ell$ does not divide the order of the Weyl group $\Omega_{\HG}$. Therefore, we
  can use Kostant's section theorem as in subsection \ref{sec:kost-sect-theor}. In particular,
  Proposition \ref{prop:kostant} tells us that  
$$\det\left(q\Ad_{\lambda(S)\Fr}-\id\,|\Lie(\HG)^{\Ad_{u}} \right) =
  \pm \chi_{\HG,\Fr}(q),$$
  which shows that $\varphi$ is unobstructed if and only if
$\chi_{\HG,\Fr}(q)\neq 0$ in $L$. 
\end{proof}

\begin{rem}
  Let $\CF$ be any automorphism of $\HG$ and suppose $\lambda$ is a $\CF$-invariant good
  $\SL_{2}$ such that $u=\lambda(U)$ is regular in $\HG$. Then the same proof shows that
  $\det\left(q\Ad_{\lambda(S)\CF}-\id\,|\Lie(\HG)^{\Ad_{u}} \right) = \pm\chi_{\HG,\CF}(q).$
\end{rem}

In order to study more general unipotent classes, the following lemma will allow us to
use inductive arguments.

\begin{lemma}\label{lemma:red_to_disting}
  Let $\HG$ and $\varphi$ be as in (\ref{eq:setting}), and let $\hat S$ be a torus in the
  centralizer $C_{\HG}(\varphi)$ of $\varphi$. Then :
  \begin{enumerate}
  \item  $\exists h\in \HG(L)$ such that $\hat M:=hC_{\HG}(\hat S)h^{-1}$ is a $\Fr$-stable
    Levi subgroup of $\HG$.
  \item If the $h$-conjugate $^{h}\varphi$ is unobstructed in $\uZ^{1}(W_{F}/I_{F}^{\ell},\hat M)$,
    then there is $g\in (\HG_{u})^{\circ}(L)$ such that $\varphi^{g}$ is unobstructed. 
  \end{enumerate}
\end{lemma}
\begin{proof}
(1)  The centralizer $C_{\LG}(\hat S)$ contains $\varphiL(\Fr)$, hence it surjects onto
$\pi_{0}(\LG)$. By \cite[Lemma 3.5]{borel_corvallis}, it is a ``Levi subgroup'' of $\LG$
in Borel's sense. It is thus conjugate by some $h\in\HG(L)$ to the standard Levi subgroup of
a standard parabolic subgroup of $\LG$. Such standard Levi subgroups are of the form
${^{L}M}=\hat M\rtimes\langle\Fr\rangle$.

(2) Since unobstructedness is invariant by conjugacy, we may and will assume that $h=1$.
Then observe that ${\varphiL}$ factors indeed through $^{L}M$, and also that $\Lie(\hat
M)$ is the weight $0$ subspace of $\Lie(\HG)$ in the decomposition
$\Lie(\HG)=\bigoplus_{\kappa\in X^{*}(\hat S)}\Lie(\HG)_{\kappa}$ of $\Lie(\HG)$ as a  sum
of weight spaces for the adjoint action of $\hat S$.
So, for any element $s\in \hat S(L)$, unobstructedness of $\varphi^{s}$ is equivalent to $q^{-1}$
not being an eigenvalue of $\varphiL(\Fr)s$ on each $\Lie(\HG_{u})_{\kappa}$. For
$\kappa=0$, this property is fulfilled by our hypothesis, since $s$ acts trivially on
$\Lie(\HG_{u})_{0}$. For any other $\kappa$, this property is fulfilled for $s$ outside a
proper Zariski closed subset of $\hat S$, because $s$ commutes with $\varphiL(\Fr)$. Therefore we can find $s$ that works for all
 $\kappa$.
\end{proof}

Following a standard terminology, we  will say that a $1$-cocycle $\varphi$ is \emph{discrete} if
$C_{\HG}(\varphi)$ contains no non-central torus of $\HG$. 
In the case where $\varphi$ is given by a pair $(\lambda,\CF)$ as in Lemma
\ref{cor:reduc-to-nice-form},  this
is equivalent to $C_{\HG}(\lambda)^{\CF}$ not containing any
non-central torus of $\HG$, since $C_{\HG}(\lambda)=\HG_{\lambda}$ is a Levi factor
of $C_{\HG}(\varphi(s))=\HG_{u}$. In this case we will also say that the pair
$(\lambda,\CF)$ is discrete.
If $\varphi(s)=\lambda(U)$ is a \emph{distinguished unipotent} element
(meaning that its centralizer does not contain any non-central torus), then $\varphi$ is certainly
discrete. The converse is not always true, but we note that
if $C_{\HG}(\lambda)$ has \emph{positive} semisimple rank, then $\varphi$ is \emph{not} discrete. 


In the next proposition, we include triality forms of $D_{4}$ (i.e. any group of type
$D_{4}$ with action of $\Fr$ of order $3$) in the ``exceptional types''.

\begin{prop} \label{prop:classical}
  Let $\HG$ be as in (\ref{eq:setting}) with no simple factor of exceptional type. Assume
  that $\ell> h_{\HG,1}$, i.e. $\ell$ is greater than the Coxeter numbers of the simple factors of $\HG$. Then the
  following are equivalent :
  \begin{enumerate}
  \item For all $\varphi$ as in (\ref{eq:setting}), there exists $g\in
    (\HG_{\varphi(s)})^{\circ}$ such that $\varphi^{g}$ is unobstructed.
  \item $\chi_{\HG,\Fr}(q)\neq 0$ in $L$.
  \end{enumerate}
\end{prop}
\begin{proof}
 By Lemma \ref{lemma:isogenies}, Lemma \ref{lemma:torus},
  decomposition (\ref{eq:decompsimple}) and equality (\ref{eq:chi_decomp}), we may assume
  that $\HG$ is simple and adjoint. In this case,
  the implication (1)$\Rightarrow$(2) follows from Proposition \ref{prop:regular-unip}.
  So we  now focus on the  other implication.   Using again Lemma \ref{lemma:isogenies}
  together with Remark \ref{rem:lift}  and the fact that $\chi_{\HG,\Fr}(q)$ is insensitive 
  to isogenies, we may assume that $\HG$ is either   ${\rm PGL}_{n}$, $\Sp_{2n}$ or $\SO_{N}$.
  Actually, the ${\rm PGL}_{n}$ case can be treated on $\GL_{n}$, since Remark
  \ref{rem:lift} also applies to the central morphism $\GL_{n}\To{\pi}{\rm PGL}_{n}$, while
  $\chi_{\GL_{n},\Fr}(T)=\chi_{{\rm PGL}_{n},\Fr}(T)\chi_{{Z(\GL_{n})},\Fr}(T)$ and
  $\chi_{{Z(\GL_{n})},\Fr}(T)$ divides $\chi_{{\rm PGL_{n}},\Fr}(T)$ for $n>1$.

  So let $\HG$ be either $\GL_{n}$, $\Sp_{2n}$ or $\SO_{N}$. Then $\Fr$ acts on $\HG$ by
  an automorphism of order at most $2$, and we will let $\LG$ denote the minimal form of
  the $L$-group. Now let $\varphi$ be as in (\ref{eq:setting}).
  By Corollary \ref{cor:reduc-to-nice-form}, we may assume that $\varphi$ is given by a pair $(\lambda,\CF)$.
  Moreover, Lemma \ref{lemma:red_to_disting},
  Proposition \ref{prop:char_pol} (1) and an inductive argument allow us to restrict attention to
  \emph{discrete} pairs $(\lambda,\CF)$.

  \medskip
  \emph{Case $\HG=\GL_{N}$ with $\Fr=\id$.}  Let $V$ be an $L$-vector space of dimension
  $N$, and $\lambda:\SL_{2}\To{}\GL(V)$ a morphism.
  Since $\ell>N$,  the  $\SL_{2}$-module $V$ is semi-simple and, for any $d\leq N$,
  the $d$-dimensional representation $S_{d}=\Sym^{d-1}(L^{2})$ of $\SL_{2}(L)$ is irreducible. Let $V_{d}$ be the
  $S_{d}$-isotypic part of $V$.  We then have decompositions $V=\bigoplus_{d\geq 0}V_{d}$
  and $S_{d}\otimes W_{d}\To\sim V_{d}$, where $W_{d}:=\Hom_{\SL_{2}}(S_{d},V_{d})$.
  In particular, we get that  $\GL(V)_{\lambda}=\prod_{d}\GL(W_{d})$.  Since this is a
  connected group, we may assume that $\CF=1$. Then we see that $(\lambda,1)$ is discrete if
  and only $\lambda$ is principal, i.e.  $u=\lambda(U)$ is regular. In this case we conclude
  thanks to Proposition \ref{prop:regular-unip} (note that this proposition applies
  directly to  ${\rm PGL}_{n}$ and extends to $\GL_{n}$ thanks to Lemma \ref{lemma:torus}
  and Remark \ref{rem:lift}).


  \medskip
  We now assume that $V$ is endowed
  with a non-degenerate bilinear form of sign $\varepsilon$ and we denote by $I(V)$ the
  isometry group, so that $I(V)\simeq \Sp_{N}$ if $\varepsilon=-1$ and $I(V)\simeq {\rm O}_{N}$
  if $\varepsilon=1$. We take up the above notations, assuming that $\lambda$ factors through
  $I(V)$. Then each $V_{d}$ is a non-degenerate subspace of $V$ and the decomposition
  $V=\bigoplus_{d}V_{d}$ is orthogonal. Further, each $S_{d}$ carries a natural non-degenerate
  bilinear form of sign $(-1)^{d-1}$ such that $\SL_{2}$ acts through $I(S_{d})$. Then
  $W_{d}$ inherits a non-degenerate form of sign $(-1)^{d-1}\varepsilon$ such that the isomorphism
  $S_{d}\otimes W_{d}\To\sim V_{d}$ is compatible with the tensor product form. It follows
  in particular that 
  $ I(V)_{\lambda} = \prod_{d} I(W_{d})$.  Writing $r_{d}:=\dim(W_{d})$, we have
  \begin{enumerate}
  \item $I(W_{d})\simeq {\rm O}_{r_{d}}\simeq \SO_{r_{d}}\rtimes\ZZ/2\ZZ$ if $(-1)^{d-1}\varepsilon=1$.
  \item $I(W_d)\simeq \Sp_{r_{d}}$ if $(-1)^{d-1}\varepsilon=-1$.
  \end{enumerate}
  In particular, $\pi_{0}(I(V)_{\lambda})$ admits a section into  $I(V)_{\lambda}$, and we may
  take $\CF$ in the image of such a section, so that $\CF$ has order at most $2$.
  Moreover, we see that $(I(V)_{\lambda})^{\CF}$ contains a
non-trivial torus whenever there is a symplectic factor (associated to some $d$ such that
$(-1)^{d-1}\varepsilon=-1$ and $W_{d}\neq 0$).
Since we may restrict attention to discrete $(\lambda,\CF)$, we will assume that
$I(V)_{\lambda}$ has  no symplectic factor.
In particular, this fixes the parity of the $d$'s such that $V_{d}\neq 0$.

We now need to investigate the eigenvalues of $q\Ad_{\CF}\Ad_{\lambda(S)}$ on
$\Lie(\HG)^{\Ad_{u}}$, where $u=\lambda(U)$. 
We have  decompositions
$$ \End_{u}(V) =\prod_{d,d'}\Hom_{u}(V_{d},V_{d'}) \simeq \prod_{d,d'}
    \Hom_{U}(S_{d},S_{d'})\otimes\Hom_{L}(W_{d},W_{d'}).$$
 The weights of $\lambda(T_{2})$ on $V_{d}$ are the weights of $S_{d}$, i.e. $d-1,d-3,\cdots, 1-d$,
  hence the weights of $\Ad_{\lambda(T_{2})}$ on $\Hom_{u}(V_{d},V_{d'})$ are the same as those on
  $\Hom_{U}(S_{d},S_{d'})$, i.e. $d+d'-2i$ for $1\leq i\leq{\rm min}(d,d')$, each one occurring with
  multiplicity $r_{d}r_{d'}$. In particular, these weights
are bounded above by $N-2$ if  $d\neq d'$ (because then $d+d'\leq N$) or if 
 $d=d'\leq \frac N2$. On the other hand, there is at most one $d> \frac
 N2$ with $W_{d}\neq 0$ and in this case $r_{d}=1$. So any weight $k> N-2$ of
 $\lambda(T_{2})$ on $\End_{u}(V)$ is even and occurs with
 multiplicity $1$. Actually, it is easy to exhibit a weight vector. Namely, put
 $e:=d(\lambda_{|\GG_{a}})(1)$, which is a nilpotent endomorphism of $V$. We also have
 $e=\log(u)$ since the logarithm is well defined under our hypothesis $\ell>h_{\HG,1}$.
 Then $e$ is a
 weight $2$ element of $\End_{u}(V)=\End_{e}(V)$ and for any $k=2k'> N-2$, the element
 $e^{k'}$ generates the subspace of weight $k$ whenever it is non-zero. In other words, we have a decomposition
 $$ \End_{u}(V) = \left\langle e^{k'}\right\rangle_{k'\geq \left\lfloor\frac N2\right\rfloor} \oplus \End_{u}(V)_{\leq N-2}$$
 where the last term is the sum of weight spaces of weight $\leq N-2$.

Now, let $\tau$ denote the involution $\psi\mapsto -\psi^{*}$  of $\End(V)$ associated with the
bilinear form on $V$. We have $\Lie(I(V))^{\Ad_{u}} = \End_{u}(V)^{\tau}$ and $\tau(\Hom(V_{d},V_{d'}))=\Hom(V_{d'},V_{d})$. 
Using the fact that $\tau(e^{k'})=(-1)^{k'+1}e^{k'}$, we get :
$$ \Lie(I(V))^{\Ad_{u}}=\End_{u}(V)^{\tau} = \left\langle e^{k'}\right\rangle_{k'\geq
  \left\lfloor\frac N2\right\rfloor,\,\rm odd} \oplus \End_{u}(V)^{\tau}_{\leq N-2}.$$

 \medskip
 \emph{Case $\HG$ is symplectic or odd orthogonal.} In this case, $\Fr$ acts trivially and
 we have
 $\chi_{\HG,\Fr}(T)=\prod_{d=1}^{\lfloor \frac N2\rfloor}(T^{2d}-1)$. The eigenvalues of $q\Ad_{\varphi(\Fr)}=q\Ad_{\CF}\Ad_{\lambda(S)}$ on
 $\End_{u}(V)_{\leq N-2}^{\tau}$ are of the form $\pm q^{k}$ for $k$ such that $0< 2k
\leq N$. For such an eigenvalue to be equal to $1$, we need that $q$ be a root of
$T^{2k}-1$, which is a factor of $\chi_{\HG,\Fr}(T)$. On the other 
hand each non-zero $e^{k}$ is an eigenvector of $q\Ad_{\varphi(\Fr)}$ with eigenvalue
$q^{k+1}$. Of course $e^{N}=0$, so $k\leq N-1$ and we have seen that $k$ must be odd. So
$k+1$ is even, between $2$ and $N$. Therefore, an eigenvalue $q^{k+1}$ is $1$
in $L$ only if $q$ is a root of $\chi_{\HG,\Fr}(T)$, as desired.

\medskip
\emph{Case $\HG$ is even orthogonal.} Here we set $\HG=\SO(V)$,
endowed with an outer action of $\Fr$ of order $f=1$ or $2$.
In these cases, setting $N=2n$, we have
$$\chi_{\HG,\Fr}(T)=(T^{n}+(-1)^{f})\prod_{d=1}^{n-1}(T^{2d}-1).$$ 
We will take advantage 
of the fact that, when $f=2$, we have  ${\rm O}(V)\simeq {^{L}\SO(V)}$. 
A pair $(\lambda,\CF)$ thus defines a $L$-homomorphism $^{L}\varphi$ for $\SO(V)$ endowed with
a trivial, resp. quadratic, action of $\Fr$ if  $\det\CF=1$, resp. if $\det\CF=-1$.

As above, each $e^{k}$ is an
eigenvector of $q\Ad_{\varphi(\Fr)}$ with eigenvalue $q^{k+1}$ with $k+1$ even. Moreover
we have $e^{N-1}=0$ (no Jordan matrix of rank $N$ is orthogonal), so $k+1\leq N-2$ and we
see that $q^{k+1}=1$ only if $q$ is a root of $\chi_{\HG,\Fr}(T)$.

Next, the weight spaces
with weight $<N-2$ are treated exactly as in the previous case, but the weight space
$\End_{u}(V)^{\tau}_{N-2}$ of weight $N-2$
needs more attention. Indeed, we already know that the eigenvalues of $q\Ad_{\varphi(\Fr)}$ on this weight
space are of the form $\pm q^{n}$, but we need more precise information since, for example,
$T^{n}+1$ does not divide $\chi_{\HG,\Fr}(T)$ when
$f=1$,  and $T^{n}-1$ may not divide $\chi_{\HG,\Fr}(T)$ when $f=2$. 
Since $\Hom_{u}(V_{d},V_{d'})_{N-2}$ is zero unless
$d+d'=N$, we have to consider two cases.

(1)  $V=V_{d}\oplus V_{d'}\simeq S_{d}\oplus S_{d'}$, with (necessarily) $d$ and $d'$ odd
and, say $d>d'$. In this setting, $\CF$ belongs
to the center $\{\pm 1\}\times\{\pm 1\}$ of ${\rm O}(V_{d})\times{\rm O}(V_{d'})$. Writing
$\CF=(\varepsilon_{d},\varepsilon_{d'})$, we see that $\CF$ acts on
$\Hom_{u}(V_{d},V_{d'})$ and $\Hom_{u}(V_{d'},V_{d})$ by multiplication by
$\varepsilon_{d}\varepsilon_{d'}$, and since $d$ and $d'$ are odd, we have $\varepsilon_{d}\varepsilon_{d'}=\det\CF$.
Now we have
$$ \End_{u}(V)^{\tau}_{N-2} = \langle e^{n-1}\rangle^{\tau} \oplus
\left(\Hom_{u}(V_{d},V_{d'}) \oplus \Hom_{u}(V_{d'},V_{d})\right)^{\tau}_{N-2}.$$
So if $\det\CF=1$, the action of $\CF$ on $\End_{u}(V)^{\tau}_{N-2}$ is trivial, hence the
eigenvalue of $q\Ad_{\varphi(\Fr)}$ is $q^{n}$ and we are done, since $T^{n}-1$ divides
$\chi_{\HG,\Fr}(T)$ when $f=1$.  Suppose now $\det\CF=-1$. Then $\CF$ acts on the second summand
of $\End_{u}(V)^{\tau}_{N-2}$ by $-1$, so the
eigenvalue of $q\Ad_{\varphi(\Fr)}$ is $-q^{n}$, which is fine since $T^{n}+1$ divides
$\chi_{\HG,\Fr}(T)$ when $f=2$. On the other hand, $\CF$ acts trivially on the first summand,
but the latter is non-zero only if $n$ is even, in which case $T^{n}-1$ also divides
$\chi_{\HG,\Fr}(T)$.

(2) $V=V_{n}$. Then we may decompose $V$ as an orthogonal sum of two $\lambda(\SL_{2})$-stable non-degenerate subspaces
$V=V_{n}^{1}\oplus V_{n}^{2}$. Moreover, since $n$ has to be odd, $(e_{|V_{n}^{i}})^{n-1}$ is not in
$\End_{u}(V)^{\tau}$ so we have
$$ \End_{u}(V)^{\tau}_{N-2} = 
\left(\Hom_{u}(V_{n}^{1},V_{n}^{2}) \oplus \Hom_{u}(V_{n}^{2},V_{n}^{1})\right)^{\tau}_{N-2}.$$
On the other hand, ${\rm O}(V)_{\lambda}\simeq {\rm O}_{2}$ acts on this space through its
component group $\{\pm 1\}$ with the non trivial element acting as $\psi\mapsto
\psi^{*}$. So, in particular, $\CF$ acts by multiplication by $\det\CF$. The eigenvalue
of $q\Ad_{\varphi(\Fr)}$ is thus $\det\CF. q^{n}$ and it equals $1$ only if
$q^{n}-\det\CF=0$, hence also only if $\chi_{\HG,\Fr}(q)=0$. 


\medskip
\emph{Case $\HG=\GL_{N}$ and $\Fr\neq \id$.} Here we have $\chi_{\HG,\Fr}(T)=\prod_{d=1}^{N}(T^{d}-(-1)^{d})$.
We continue with the same notations $V,\lambda,u$ etc, and we assume that there is $\CF\in
(\HG\rtimes\Fr)_{\lambda}$ that fixes a Borel pair of
$\HG_{\lambda}$.
Using the explicit description
$\HG_{\lambda}=\prod_{d}\GL(W_{d})$, we see that $(\lambda,\CF)$ is discrete if and only if
$r_{d}=1$ for all $d$  (so that
$\HG_{\lambda}=\GG_{m}\times\cdots\times\GG_{m}$ is the center of $\prod_{d}\GL(V_{d})$) and 
 $(\HG_{\lambda})^{\CF} =\{\pm 1\}\times\cdots\times\{\pm 1\}$.  This implies that $\CF$
 normalizes each $\GL(V_{d})$ and induces the non-trivial element $\alpha_{d}$ of
 ${\rm Out}(\GL(V_{d}))$.  Since $u_{|V_{d}}$ is regular, it follows from Proposition
 \ref{prop:regular-unip} and the subsequent remark that no  eigenvalue of
 $q\Ad_{\CF\lambda(S)}$ on $\End_{u}(V_{d})$ equals $1$ unless  $\chi_{\GL(V_{d}),\alpha_{d}}(q)=0$ in $L$, in which case
 we also have $\chi_{\HG,\Fr}(q)=0$ by Proposition \ref{prop:char_pol} (1).
Let us now focus on the eigenvalues of $q\Ad_{\CF\lambda(S)}$ on each
$\Hom_{u}(V_{d},V_{d'})$ for $d\neq d'$. As we have already seen, the eigenvalues of $q\Ad_{\lambda(S)}$ are
of the form $q^{\frac 12(d+d')-i}$ with $0\leq i <{\rm min}(d,d')$. So it remains to
understand how $\CF$ acts. Note that $\CF^{2}\in(\HG_{\lambda})^{\CF}$ so at least we know
that $\CF^{4}=1$.
We distinguish two cases.

(1) Suppose that all the $d$'s occuring  have the same parity. Then
  there is a non-degenerate bilinear form on $V$ (symplectic if the $d$'s are even,
 orthogonal if they are odd) such that $u\in I(V)$, see \cite[Cor. 3.6 (2)]{Liebeck} for
 example.
 We may then conjugate $\lambda$ so that it
 factors through $I(V)$. But $I(V)$ is the fixed-point subgroup of an involution given by
 conjugation by an element of the  form $g\rtimes \Fr$. So we may set $\CF$ to this element and we have
 achieved $(\Ad_{\CF})^{2}=1$.  It follows that the eigenvalues of $q\Ad_{\CF\lambda(S)}$ on each
 $\Hom_{u}(V_{d},V_{d'})$ are of the form $\pm q^{k}$ for some integer
 $k\leq \frac N2$. Should such an eigenvalue be equal to $1$, we would have $q^{2k}-1=0$, hence a
 fortiori $\chi_{\HG,\Fr}(q)=0$.

(2) Suppose there are both even and odd $d$'s. Write $(\CF^{2})_{d}$ for the component of
$\CF^{2}$ in $\GL(V_{d})$. This is a central element of $\GL(V_{d})$ equal to $\pm 1$. We
then decompose $V=V_{+}\oplus V_{-}$ where $V_{\pm}=\bigoplus_{d,\,(\CF^{2})_{d}=\pm 1}
V_{d}$. We have $\GL(V)^{\CF^{2}}=\GL(V_{+})\times\GL(V_{-})$, hence also
$\GL(V)^{\CF}=\GL(V_{+})^{\CF}\times\GL(V_{-})^{\CF}$. But $\CF$ acts on both $\GL(V_{-})$
and $\GL(V_{+})$ as an involution that induces the non trivial outer automorphism. So 
each $\GL(V_{\pm})^{\CF}$ is an orthogonal or symplectic group. This implies that all
$d$'s occurring in the decomposition of $V_{+}$, resp. $V_{-}$, have the same parity
(because all multiplicities  $r_{d}$ are $1$). As a
consequence, we see that $\CF^{2}$ acts on $\Hom_{u}(V_{d},V_{d'})$ by multiplication by
$(-1)^{d+d'}$. So we now have two subcases :
\begin{itemize}
\item if $d,d'$ have the same parity, the eigenvalues of $q\Ad_{\CF\lambda(S)}$ on 
 $\Hom_{u}(V_{d},V_{d'})$ are of the form $\pm q^{\frac 12 k}$ for some \emph{even} integer
 $k\leq  d+d' \leq N$. As before, should such an eigenvalue be equal to $1$, we
 would have $q^{k}-1=0$, hence a  fortiori $\chi_{\HG,\Fr}(q)=0$. 
\item if $d,d'$ have different parities, the eigenvalues of $q\Ad_{\CF\lambda(S)}$ on 
 $\Hom_{u}(V_{d},V_{d'})$ are of the form $\zeta q^{\frac 12 k}$ for some \emph{odd} integer $k\leq
 d+d'\leq N$ and a primitive $4^{th}$-root of unity $\zeta$ in $L$. This time, should such
 an eigenvalue be equal to $1$, we 
 would have $q^{k}+1=0$ hence, again, $\chi_{\HG,\Fr}(q)=0$.
\end{itemize}
\end{proof}

It may be tempting to believe that the nice equivalence of Proposition \ref{prop:classical}
holds in general. However, it fails in the case of triality, i.e. a group of type $D_{4}$
with Frobenius acting with order $3$. In this case,  the irreducible factors of
$\chi_{\HG,\Fr}(T)=(T^{2}-1)(T^{6}-1)(T^{8}+T^{4}+1)$ are $\Phi_{n}(T)$ for
$n=1,2,3,6,12$. But to get an equivalence, we need also $\Phi_{4}(T)$ :

\begin{lemma} \label{lemma:triality}
  Assume that  $\HG={\rm PSO}_{8}$ with  $\Fr$ of order $3$ (triality), and
  $\ell>h_{\HG}=6$. Then the following are equivalent :
  \begin{enumerate}
   \item For all $\varphi$ as in (\ref{eq:setting}), there exists $g\in
    (\HG_{\varphi(s)})^{\circ}$ such that $\varphi^{g}$ is unobstructed.
  \item $\chi_{\HG,\Fr}'(q)\neq 0$ in $L$, where $\chi_{\HG,\Fr}'(T)=T^{12}-1$.
  \end{enumerate}
\end{lemma}

\begin{proof}
As in the proof of Proposition \ref{prop:classical}, we may focus on discrete pairs
$(\lambda,\CF)$ with $\lambda:\SL_{2}\To{}\HG$ and $\CF\in (\HG\rtimes\Fr)_{\lambda}$
(where $\LG$ is the minimal L-group, so that $\pi_{0}(\LG)=\ZM/3\ZM$). We still denote by
$\lambda$ the unique lift $\SL_{2}\To{}\SO_{8}$ and see $\SO_{8}$ as $\SL(V)\cap I(V)$ for
an $8$-dimensional vector space with a non-degenerate symmetric bilinear form. With the
notation of  
the proof of Proposition \ref{prop:classical},
there are only three possible types of decomposition of $V$ associated to such a
$\lambda$. Either $V=V_{7}\oplus V_{1}$, or
$V=V_{5}\oplus V_{3}$ or, $V=V_{3}\oplus V_{1}$ with $V_{3}=S_{3}^{2}$ and $V_{1}=S_{1}^{2}$.
  
(1) Type $(7,1)$. This is the regular orbit, so it is covered by Proposition
\ref{prop:regular-unip}.

(2) Type $(5,3)$. This is the only distinguished non-regular orbit, so it is stable under
$\Fr$. Therefore, for $\lambda$ of type $(5,3)$, there exists
$\CF\in (\HG\rtimes\Fr)_{\lambda}$. Since $\HG_{\lambda}=Z(\HG)=\{1\}$, we have
$(\Ad_{\CF})^{3}=\id$. On the other hand, the $\lambda(T_{2})$-weights on
$\Lie(\HG)^{\Ad_{u}}=\End_{u}(V)^{\tau}$ are $2$, $4$ and $6$,
so the eigenvalues of
$q\Ad_{\CF\lambda(S)}$ are respectively of the form $\zeta q^{2}$, $q^{3}$ or $\zeta q^{4}$
for some $3^{rd}$-root of unity $\zeta$.
If any of these numbers equals $1$ in $L$, 
then $q^{12}=1£$, hence $\chi_{\HG,\Fr}'(q)=0$.
However, it is actually possible to prove that $q^{4}$ is not an eigenvalue, so that the
polynomial $\chi_{\HG,\Fr}$ is still good for this orbit.

(3) Type $(3,3,1,1)$. This orbit intersects $G_{2}=\HG^{\Fr}$ along its non-regular
distinguished orbit. So we may pick a relevant $\lambda$ that is centralized by $\Fr$. Then
$\pi_{0}(\LG_{\lambda})$ is isomorphic to $\ZM/6\ZM$ and contains two elements such that
$(\lambda,\CF)$ is discrete : $\Fr$ of order $3$, and $c\Fr$ of order $6$, where $c$ is
the image in $\HG_{\lambda}$ of a
reflection that generates $\pi_{0}(I(V)_{\lambda}\cap \SL(V))$.
The weights are $0,2$ and $4$. Hence the eigenvalues of $q\Ad_{\CF\lambda(S)}$ on weight $0$ and $2$ spaces are of the
form $\zeta q^{2}$ or $\zeta q$ for a sixth root of unity $\zeta$, so they are different
from $1$ unless $q^{12}=1$. On the other hand, the weight $4$ space has dimension
$1$ and comes from $G_{2}$. So $\Fr$ acts trivially on it, and $c\Fr$ acts by $\pm 1$.
Hence the corresponding eigenvalue is $\pm q^{3}$ and is also different from $1$ unless $q^{12}=1$.
Now, the computation in $G_{2}$ of Remark \ref{rem:G2} shows that $-1$ is an  eigenvalue of $\CF=c\Fr$ 
on the weight $2$ space, so $q\Ad_{\CF\lambda(S)}$ has eigenvalue $-q^{2}$, and it is
different from $1$ if and only if $\Phi_{4}(q)\neq 0$.
\end{proof}

We now turn to the exceptional groups. Recall the polynomials $\chi_{\HG,\Fr}^{*}$ from (\ref{eq:defchistar}).

\begin{prop} \label{prop:exceptional}
  Suppose that  $\HG$ is simple of exceptional type. If
  $\chi_{\HG,\Fr}^{*}(q)\neq 0$ in $L$, then for all $\varphi$ as in (\ref{eq:setting}), there exists $g\in
    (\HG_{\varphi(s)})^{\circ}$ such that $\varphi^{g}$ is unobstructed.
\end{prop}
Recall that $\chi_{\HG,\Fr}^{*}(q)\neq 0$ is equivalent to ``$q$ has
order greater than $h_{\HG,\Fr}$ in $L^{\times}$'', which implies  $\ell>h_{\HG,\Fr}$,
hence also $\ell>h_{\HG,1}$.

\begin{proof}
  Thanks to Lemma \ref{lemma:isogenies} and equality (\ref{eq:chi_decomp}), we may assume
  $\HG$ is adjoint.
  We will use the tables in Chapter 11 of \cite{lawther-testerman}. These tables cover all the
  nilpotent classes of exceptional groups, including a description of the reductive
  quotient $C$ of the centralizers (both the neutral component, denoted there by $C^{\circ}$
  and the $\pi_{0}$, denoted there by $C/C^{\circ}$), and the weights of an associated
  cocharacter $\tau$ on the Lie algebra centralizer (denoted by $m$ there).  Actually,
  they even describe the weights on each subquotient of the central series of the
  nilpotent part of the Lie algebra centralizer (with integer $n$ denoting the $n^{th}$
  step of the central series).

  Using a Springer isomorphism $e\leftrightarrow u$ between  the nilpotent cone and the unipotent variety, we
  get a table of unipotent classes, and we may identify the centralizers
  $\HG_{u}=\HG_{e}$. Then for any good $\lambda$ associated to $u$, we have  identifications
  $\HG_{\lambda}\simeq C$, hence the table provides us with descriptions of $(\HG_{\lambda})^{\circ}=C^{\circ}$,
  $\pi_{0}(\HG_{\lambda})=C/C^{\circ}$ and the weights $m$ of $\lambda(T_{2})$ on
  $\Lie(R_{u}(\HG_{u}))$.

  Thanks to Corollary \ref{cor:reduc-to-nice-form} we may focus on $\varphi$ associated to a
  pair $(\lambda,\CF)$. Then, using Lemma \ref{lemma:red_to_disting} (together with
  Proposition \ref{prop:classical} an inductive argument for
  the $E$ series), we may restrict attention to \emph{discrete} $(\lambda,\CF)$. In the
  case where $\Fr$ acts trivially on $\HG$, this means
  that, in the tables of \emph{loc. cit.}, we may restrict to classes such that
  $C^{\circ}$ is a torus and consider all $\CF\in C$  such that $(C^{\circ})^{\CF}$ is finite.
  In this setting, the exponent of $(C^{\circ})^{\CF}$ divides the
  order $f$ of
  the image $\bar\CF$ of $\CF$ in the component group $C/C^{\circ}$
  (since the endomorphism $t\mapsto t({^{\bar\CF}t})\cdots({^{\bar\CF^{f-1}}t})$ of $C$
    vanishes), hence  $\CF$ has  finite order dividing $f^{2}$,
   and this order does
  only depend on the connected component of $C$ that contains $\CF$. So the
  eigenvalues of $q\Ad_{\CF\lambda(S)}=q\Ad_{\varphi(\Fr)}$ on $\Lie(\HG)^{\Ad_{u}}$ are of
  the form $\zeta q^{\frac m2+1}$ for some root of unity $\zeta$ whose order $t$ divides the
  order of $\CF$. Therefore, what we have to check is that, in all cases, we have  $t(\frac
  m2+1)\leq h_{\HG,\Fr}$ when $t(\frac  m2+1)$ is an integer, or $t(m+2)\leq h_{\HG,\Fr}$
  else.
  In the only twisted case $^{2}E_{6}$ where $\Fr$ has order $2$, we will apply the same
  strategy except that here  $\CF\in\HG\rtimes\Fr$.


 Below we list all ``discrete'' orbits except  
 the regular ones,
 which are treated in Proposition \ref{prop:regular-unip}. The
 numbering is that of \cite[\S 11]{lawther-testerman}.

  \medskip
  $G_{2}$, orbit $3$.  Here $h=6$, $m=2$ or $4$, and $C=S_{3}$, so
  that $t=1,2$ or $3$. Hence the desired inequalities  $t(\frac
  m2+1)\leq h=6$ always hold
  except  if $t=3$ and $m=4$. But this case doesn't happen since the weight $4$ subspace
  is $1$ dimensional, so $C=S_{3}$ acts on it via a character, hence via an
  element of order $2$.  

  \medskip
  $F_{4}$, orbit $10$. Here $h=12$, $C=S_{4}$ hence $t\leq
  4$, and the desired inequality holds trivially
  for all weights except possibly for weight 6. But the weight 6 subspace has dimension
  $2$, so the action of $C=S_{4}$ factors over a quotient isomorphic to $S_{3}$, hence
  $t\leq 3$ on this subspace.

    \medskip
  $F_{4}$, orbit $13$. Here $h=12$, $C=S_{2}$ hence $t\leq
  2$ and weights $m$  are even and $\leq 10$, hence the desired inequality holds.
  
    \medskip
  $F_{4}$, orbit $14$. Here $h=12$ and $C=S_{2}$, so only the weight $m= 14$ space might contradict the
  desired inequality, but the columns $\mathcal{Z}^{\natural}$ and $\mathcal{Z}$ of the
  table show that $C$ acts trivially on this 
  space, so $t=1$ and $t(\frac m2+1)=8\leq 12$.

  \medskip
  $E_{6}$, orbits 17 and 19. Here again, $h=12$ and $C=S_{2}$ or $\{1\}$. In each case,
  the desired inequalities follow directly from the list of weights.

  \medskip
  $E_{6}$, orbit 11. Our source here is  section 9.3.4 of \emph{loc.cit}.
  This orbit comes from the distinguished non regular orbit of a Levi subgroup $H$ of root system $D_{4}$ and
  $C^{\circ}$ is the two-dimensional connected center of $H$, while
  $C$ normalizes a Borel pair of  $H$
  and has
  component group $C/C^{\circ}=S_{3}$. The action of $S_{3}$ on $X^{*}(C^{\circ})$ is
  the  standard representation, as can be seen by embedding $E_{6}$ as
  a Levi subgroup of $E_{7}$ and using the description of the
  reductive centralizer $C_{7}$ of this orbit from \emph{loc.cit}.
  Therefore,
  an element of order $1$ or $2$ of $C/C^{\circ}$ fixes
  a subtorus, and we see that if $(\lambda,\CF)$ is to be discrete, then $\bar \CF$ should
  have order $3$ in $C/C^{\circ}$. Then $\CF$ itself has order $3$ or $9$ in $C$. To prove
  it has order $3$, we  embed $E_{6}$ as a Levi subgroup of $E_{8}$ and consider
  the reductive centralizer $C_{8}$ there.  Then
  section 9.3.4 of \emph{loc.cit} exhibits two elements $c_{1}$ and $c_{2}$ of order $2$ in $C$, whose
  images generate $C/C^{\circ}=C_{8}/C_{8}^{\circ}$,  and that  act on $C_{8}^{\circ}$ by fixing a
  pinning. But  $C_{8}^{\circ}$  is a simple group of type $D_{4}$, so its center has
  exponent dividing $2$. Hence the element  $(c_{1}c_{2})^{3}$, which belongs to $C^{\circ}$ and
  fixes a pinning of $C_{8}^{\circ}$ is central in $C_{8}^{\circ}$, hence has order
  dividing $2$. Since we have seen that $c_{1}c_{2}$  has order $3$ or $9$, we conclude it has order
  $3$. Hence $\CF$ has order $3$, and all desired inequalities follow from the list of weights.



    \medskip
    $^{2}E_{6}$, orbit 11. Here, $h_{\HG,\Fr}=18$. Again, we refer to section 9.3.4 of
    \emph{loc.cit.}, except that  we find it easier to argue with the nilpotent
    representative $e':=\Ad_{g}(e)$ in their notation (and the same
    cocharacter $\tau$ described in table 3 of their chapter
    6). Indeed, for the $\lambda$ corresponding to $(e',\tau)$, we easily
    see that $h_{2}(-1)\rtimes \Fr \in \LG_{\lambda}$.
    Then,  $\LG_{\lambda}/(\LG_{\lambda})^{\circ}$ is an extension of
  $\ZZ/2\ZZ$ by $\HG_{\lambda}/(\HG_{\lambda})^{\circ}=C/C^{\circ}=S_{3}$. Such an extension has to be
  split and contains a central element of order $2$. In the present
  case,   using notations $c'_{2}:= g c_{2} g^{-1}$, the element
  $\CF_{0}:=c'_{2}h_{2}(-1)\rtimes\Fr$ belongs to
  $(\HG\rtimes \Fr)_{\lambda}$ and its image in $\LG_{\lambda}/(\LG_{\lambda})^{\circ}$ is the central element of order
  $2$. A computation shows that  $\CF_{0}$ acts
  on $(\LG_{\lambda})^{\circ}=C^{\circ}$ by inversion.
  It follows that the pair $(\lambda,\CF_{0})$ is discrete  and that,
  more generally, a pair   $(\lambda,\CF)$ is discrete if, and only
  if, writing $\CF=c\CF_{0}$ for some $c\in C$, the image of $c$ in $C/C^{\circ}$ has
  order $1$ or $3$. 
  Putting $\CF_{1}:=c'_{1}c'_{2}\CF_{0}$, this means that any
  $\CF$ such that the pair $(\lambda,\CF)$ is discrete is a
  $C^{\circ}$-translate of $\CF_{0}$ or $\CF_{1}^{\pm 1}$. Let us
  compute their orders. Since $\CF_{0}^{2}\in C^{\circ}$ is  fixed by inversion, we have
  $\CF_{0}^{4}=1$. On the other hand, since $\CF_{0}$ is central in
  $\pi_{0}(\LG_{\lambda})$, there is some $c\in C^{0}$ such that
  $(c'_{1}c'_{2})\CF_{0}(c'_{1}c'_{2})^{-1}=\CF_{0}c$, which yields
  $(c'_{1}c'_{2})\CF_{0}^{2}(c'_{1}c'_{2})^{-1}=\CF_{0}c\CF_{0}c=\CF_{0}^{2}c^{-1}c=\CF_{0}^{2}$. Hence
  $\CF_{0}^{2}$ belongs to the $c'_{1}c'_{2}$-fixed subgroup of $C^{\circ}$, which has order
  $3$. This implies $\CF_{0}^{2}=1$.  On the other hand, there is some $c'\in C^{\circ}$
  such that $\CF_{1}^{3}=\CF_{0}c'$, which, as above, implies
  $\CF_{1}^{6}=(\CF_{0}c')^{2}=\CF_{0}^{2}=1$.
  Having computed the orders $2$ and $6$, we now see that the desired inequalities  follow from the list of
   weights, except maybe for weight $6$ when $\CF=\CF_{1}^{\pm 1}$. 
  However,an easy computation shows that the action of $\CF_{0}$ on the weight $6$ space is trivial, so that
  the action of $\CF_{1}$ on it actually has order $3$, and the desired inequalities
  hold too.


    \medskip
  $^{2}E_{6}$, orbits 14, 16, 18. Here $\HG_{\lambda}=C=\GG_{m}$, and $\CF$ should
  map to the non-trivial element of $\pi_{0}(\LG_{\lambda})=\pi_{0}(\LG)$, so $\CF$ has
  order dividing $4$. 
  The desired
  inequalities are then straightforward for weights $\leq 7$ since $h=18$. For the weight
  8, 10 or 14 spaces, we use the fact that $C^{\circ}$ acts trivially on them, so
  $\Ad_{\CF}$ has order $t\leq 2$ there, whence the wanted inequality.

  \medskip
  $^{2}E_{6}$, orbits 17. Here $\HG_{\lambda}=C=\{\pm 1\}$, hence $\CF^{2}=\pm 1$, and $\CF$
  has a priori order $2$ or $4$. But the representative $e$ of the table is visibly invariant
  under $\Fr$, which means that we can pick $u$ and $\lambda$ invariant under $\Fr$, and set
  $\CF=\Fr$ or $\CF= (- 1).\Fr$. In each case, $\CF$  has order $2$ and the desired
  inequalities follow from the list of weights. 
  
  \medskip
  $^{2}E_{6}$, orbit 19. 
  Here $\HG_{\lambda}=C=\{1\}$, so $\CF^{2}=1$, and the desired inequalities follow 
  from the list of weights (recall $h=18$).

    \medskip
  $E_{7}$, orbit 24. Here $h=18$, $C^{\circ}$ is a torus and $\pi_{0}(C)=S_{2}$. So 
  $\CF$ has order dividing $4$. The desired inequality $4(\frac m2+1)\leq h=18$ holds for
  all weights, except weight $8$, but $C^{\circ}$ acts trivially on this weight space, so
    $\Ad_{\CF}$ has order $t=2$ there and the inequality holds too.

    \medskip
  $E_{7}$, orbits 33, 37,41,42,43. In these cases $C=S_{3}$ or $S_{2}$ or $\{1\}$, and the
  inequalities are straightforward, except for the weight 18 space in  orbit 41, where we
  need to use the column $\mathcal{Z}$ to ensure $C$ acts
  trivially on this weight space.

    \medskip
    $E_{7}$, orbit 39. Here $C^{\circ}=\GG_{m}$ and  $C/C^{\circ}=S_{2}$ acts non
    trivially on $C^{\circ}$, but $C$ is not a semi-direct product of
    $C^{\circ}$ by $S_{2}$, so  $\CF$ has order $4$ with
    $\CF^{2}=-1\in (C^{\circ})^{\CF}$. However, the
    explicit form of $C^{\circ}$ given in the table shows that it acts with even weights on all root
    subgroups (trivially on the simple roots of the Levi subsystem $E_{6}$ and with weight
    2 on the remaining simple root).
    Hence  $\Ad_{\CF}$ has order $2$, and the desired inequalities follow
    since all weights are even and less than $16$.

    \medskip
  $E_{8}$, orbit 41. Here $h=30$ and $C=S_{5}$, so $t\leq 6$. Hence the
  desired inequalities are at least satisfied for all weights $\leq 8$. This leaves us
  with the $4$-dimensional weight $10$ space $Z_{10}$, which is stable under $C=S_{5}$. We claim
  that $Z_{10}$ is isomorphic to the standard representation of
  $S_{5}$. This implies that the eigenvalues of the elements of order $6$ of 
  $S_{5}$ have order $1,2$ or $3$, and not $6$. So $t\leq 5$ on this space, and the desired
  inequalities still hold.  To justify the claim, we use the notation of 9.3.17 of
  \cite{lawther-testerman}. There, the authors exhibit three elements $c_{1}$, $c_{2}$ and
  $c_{3}$ that generate $C=S_{5}$, as well as  a basis $z_{10}^{1},\cdots, z_{10}^{4}$ of
$Z_{10}$. The element $c_{1}$ is a $5$-cycle, and all $z_{10}^{i}$ are
eigenvectors of $\Ad(c_{1})$, with respective eigenvalues
$\zeta,\zeta^{3},\zeta^{4},\zeta^{2}$ where $\zeta$ is a primitive $5^{th}$-root of unity. This implies 
that $Z_{10}$ is either the standard representation or its twist by the sign character. To
show it is the untwisted standard representation, it suffices to show that the trace of a
transposition is $2$.  According to \emph{loc.cit.} the element $c_{2}c_{3}$ is a
transposition. The action of $\Ad(c_{2})$ and $\Ad(c_{3})$ on $Z_{10}$ is not made
explicit in \emph{loc.cit.} but according to the authors (private communication), they are
given by matrices
$$\Ad(c_2) = \left(\begin{array}{cccc}
  0 & 0 & 0 & 1 \\
  1 & 0 & 0 & 0 \\
  0 & 1 & 0 & 0 \\
  0 & 0 & 1 & 0 \\
\end{array}
\right),
\Ad(c_3) = {\textstyle\frac{1 + 2\phi}{5}}
\left(\begin{array}{cccc}
  -1 & \phi & 1 & 1 + \phi \\
  \phi & 1 & 1 + \phi & -1 \\
  1 & 1 + \phi & -1 & \phi \\
  1 + \phi & -1 & \phi & 1 \\
\end{array}
\right),
$$
where $\phi=\zeta^{2}+\zeta^{3}$. Thus the trace of $\Ad(c_2 c_3)$ is $\frac{(1 + 2\phi)(2 + 4\phi)}{5} = 2$, as desired. 

  \medskip
     $E_{8}$, orbits 47, 50, 52. Here $h=30$ and $C=\GG_{m}$ and $C/C^{\circ}=S_{2}$. So $\CF$
     has order dividing $4$, which makes directly all desired inequalities hold except for
     weights $14$ and $16$ subspaces, but the latter are fixed by $C^{\circ}$ according to
     column $\mathcal{Z}^{\sharp}$, so   $\Ad_{\CF}$ has order $t\leq 2$ there, and the inequalities hold too. 

  \medskip
  $E_{8}$, orbits 54, 58, 60, 62, 63, 65, 66, 67, 68. Here $C=S_{3}$, $S_{2}$ or $\{1\}$,  and all
  inequalities are straightforward, except for the one dimensional weight 22 space in
  orbit 60 and weight 34 space in orbit 66. But the latter are fixed by $C$ in each case
  according to column $\mathcal{Z}$,
  so $t=1$ there and the inequalities still hold.

  \medskip
  $E_{8}$, orbit 55. Here $C^{\circ}=\GG_{m}$ and $C/C^{\circ}=S_{2}$. The table features
  a lift $c$ in $C$ of the non-trivial element of $C/C^{\circ}$. One can  compute that $c^{2}=1$
  (e.g. by using the list of positive roots of $E_{8}$ in Bourbaki). So we can take
  $\CF=c$ and the desired inequalities follow from the list of weights.

\end{proof}

\begin{rem} \label{rem:G2}
  Consider the orbit 3 of $G_{2}$, on page 73 of \cite{lawther-testerman}. The reflection
  $c_{2}$ exchanges the two root vectors  $e_{11}$ and $e_{21}$, which have both weight
  $2$. So $-1$ is an eigenvalue of $\CF:=c_{2}$ on the space generated by these vectors,
  hence $-q^{2}$ is an eigenvalue of  $q\Ad_{\CF\lambda(S)}$. But $\Phi_{4}(T)$ does not
  divide $\chi_{G_{2},1}(T)=(T^{2}-1)(T^{6}-1)$,
  so we see that in this case the equivalence of Proposition \ref{prop:classical} with the
  polynomial $\chi_{\HG,\Fr}$ really fails, just as for $^{3}D_{4}$.
 It fails also for $F_{4}$ due to the weight $8$ space of orbit 13, which requires
  $\Phi_{5}$  and  $\Phi_{10}$ (depending on $\CF$) although none of these polynomials divides $\chi_{F_{4},1}$. The same
  orbit and the same weight space viewed in $E_{6}$ and $^{2}E_{6}$
  through the identification of $F_{4}$ with the fixed points of the
  outer involution  (orbit 17 in \emph{loc.cit.}) again
  requires $\Phi_{5}$ and $\Phi_{10}$, although $\Phi_{10}$  does not divide
  $\chi_{E_{6},1}$ and $\Phi_{5}$ does not divide $\chi_{^{2}E_{6},\Fr}$. In orbit 33 of
  $E_{7}$, taking $\CF=c_{1}$, the weight $8$ space requires $\Phi_{15}$, which does not
  divide $\chi_{E_{7},1}$. Finally, in orbit 66 of $E_{8}$, taking $\CF=c$, the weight 26
  space requires $\Phi_{28}$, which does not divide $\chi_{E_{8},1}$.
  Note it is certainly possible in each case 
  to compute explicitly a
  polynomial $\chi'$ dividing $\chi^{*}$ for which equivalence between $\chi'(q)\neq 0$
  and generic smoothness holds. 

  For convenience of the reader, we include a table showing
  the prime factors of $\chi_{\HG,\beta}$ in $\ZM[T]$ for the exceptional types.

\begin{table}[h]
  \resizebox{\columnwidth}{!}{
$
\begin{array}{c|c|c|c|c|c|c|c|}
 \hbox{type of }\HG,\beta & ^{3}D_{4}  & G_{2} & F_{4} & E_{6} & ^{2}E_{6} & E_{7} & E_{8} \\
 \hline  \left\{n, \Phi_{n}|\chi_{\HG,\beta}\right\}
                          &
\left\{\begin{array}[c]{l}1,2,3\\6,12\end{array}\right\} 
                                       &
\left\{\begin{array}[c]{l}1,2\\3,6\end{array}\right\} 
&
\left\{\begin{array}[c]{l} 1,\ldots, 4, \\ 6,8,12 \end{array}\right\}
                          &
\left\{\begin{array}[c]{l} 1,\ldots, 6,\\ 8,9,12\end{array} \right\} 
&
\left\{\begin{array}[c]{l} 1,\ldots, 4\\6,8,10\\12,18\end{array} \right\} 
&
\left\{\begin{array}[c]{l} 1,\ldots, 10,\\ 12,14,18\end{array}\right\}
&
\left\{\begin{array}[c]{l} 1,\ldots, 10,12\\ 14,15,18\\20,24,30\end{array}\right\}

\end{array}
$
}
\end{table}

\end{rem}

\begin{cor} 
  Let $G$ be as in (\ref{eq:setting}). If $\chi_{\HG,\Fr}^{*}(q)\neq 0$ in $L$, then for
  any $\varphi$ as in (\ref{eq:setting}), there is $g\in \HG_{\varphi(s)}^{\circ}$ such
  that $\varphi^{g}$ is unobstructed.
\end{cor}
\begin{proof}
  This follows from  Lemma \ref{lemma:isogenies}, Lemma \ref{lemma:torus},
  decomposition (\ref{eq:decompsimple}), Proposition \ref{prop:classical}, Lemma \ref{lemma:triality} and Proposition
  \ref{prop:exceptional}. Note again that $\chi_{\HG,\Fr}^{*}(q)\neq 0$ is equivalent to
  $q$ having order greater than $h_{\HG,\Fr}$, which implies $\ell>h_{\HG,\Fr}$ hence also
  $\ell> h_{\HG}$. It also implies that $\chi_{\HG,\Fr}(q)\neq 0$.
\end{proof}

\begin{proof}[Proof of Theorem \ref{thm:unobstructed_explicit}]

(1) We assume that $\ell$ does not divide $e\chi_{\HG,\Fr}^{*}(q)$.  Fix $\varphi\in
Z^{1}(W_{F},\HG(L))$ and choose $\varphi'$ as in Proposition
\ref{prop:pinning_preserving_acf}. By Lemma \ref{lemma:control_ad}, the action
$\Ad_{\varphi'}$ of $W_{F}$ on $\HH:=C_{\HG}(\varphi(I_{F}^{\ell}))^{\circ}$ is unramified and
$\eta:=\varphi\cdot(\varphi')^{-1}\in Z^{1}(W_{F}/I_{F}^{\ell},\HH(L))$.  By Proposition
\ref{prop:char_pol}, we have $\chi^{*}_{\HH,\Ad_{\varphi(\Fr)}}(q)\neq 0$ in $L$, so the last
Corollary gives us an element $h\in(\HH_{\eta(s)})^{\circ}$
such that $\eta^{h}$ is unobstructed in $Z^{1}(W_{F}/I_{F}^{\ell},\HH(L))$. We have
explained after Proposition \ref{prop:pinning_preserving_acf} that $\eta^{h}\cdot\varphi'$ is then unobstructed  in
$Z^{1}(W_{F},\HG(L))$,
but we have $(\HH_{\eta(s)})^{\circ}= (\HG_{\tau})^{\circ}$ and $\eta^{h}\cdot\varphi'=(\eta\cdot\varphi')^{h}=\varphi^{h}$.

(2) We assume here that $\HG$ has no exceptional factor, that $\ell>h_{\HG,1}$, and that
$\ell$ does not divide $\chi_{\HG,\Fr}(q)$. Then we repeat the above argument, observing
that $\HH=C_{\HG}(\varphi(I_{F}^{\ell}))$ is again a group with no exceptional
component. Indeed, it suffices to check this in a classical group where it is fairly
standard.
However, the action of $\varphi'(\Fr)$ on $\HH$ may feature instances of
triality. Fortunately, this is harmless because the modified polynomial $\chi'_{\HH,\Fr}$
still divides $\chi_{\HG,\Fr}$. Indeed, if $\Phi_{12}(T)$ divides $\chi_{\HG,\Fr}$ for
$\HG$ a classical group, then so does $\Phi_{4}(T)$.

\end{proof}

\subsection{$\LG$-banal primes}\label{subsection:LGbanal} We keep the general setup of this section.

\begin{prop} \label{prop:reduced fibers}
  Let $\ell\neq p$ be a prime.  Then the  following are equivalent :

  (1) For every algebraically closed field of characteristic $\ell$, and every continuous
  $L$-homomorphism $\varphi: W_F \rightarrow \LG(L),$ there is $g\in
  C_{\HG}(\varphi(I_{F}))^{\circ}$ such that $\varphi^{g}$ is unobstructed.
  
  (2) For any $e\in\NN$ and any finite place $v$ of $\OO_{K_{e}}[\frac 1p]$ such that the residue field $k_v$ has characteristic $\ell$, the fiber $\uZ^1(W_F^0/P_F^e,{\hat G}_{k_v})$ of $\uZ^1(W_F^0/P_F^e,{\hat G})$
is reduced. 
\end{prop}
\begin{proof}
  Assume (1). It suffices to prove reducedness  for $k_v$ replaced by
its algebraic closure $L$.  Let $x$ be an $L$-point of $\uZ^1(W_F^0/P_F^e,{\hat G}_L)$
contained on exactly one 
irreducible component of $\uZ^1(W_F^0/P_F^e,{\hat G}_L)$.  Let $\tau$ be the restriction of $\varphi_x$ to $I_F$; there then exists a $g \in {\hat G}_{\tau}^\circ(L)$
such that $\varphi_x^g$ is unobstructed.  The corresponding $L$-point
$y$ of $\uZ^1(W_F^0/P_F^e,{\hat G}_L)$ is smooth and
lies in the same irreducible component of $\uZ^1(W_F^0/P_F^e,{\hat G}_L)$
as $x$, so that irreducible component is generically reduced.  Since $x$ was arbitrary, we deduce that $\uZ^1(W_F^0/P_F^e,{\hat G}_L)$ is generically reduced; since it is
also a local complete intersection, it must be reduced.

Now assume (2). Following the same reduction process as above Proposition
\ref{prop:pinning_preserving_acf}, we may assume that $\LG$ and $\varphi$ are as in
(\ref{eq:setting}). Recall from the discussion above Lemma \ref{lemma:jordan} the map
$$ \HG_{L}\times \HG_{\varphi(s)}^{\circ} \To{} \uZ^{1}(W_{F}/I_{F}^{\ell},\HG)_{L},\,\,
(h,g)\mapsto {^{h}(\varphi^{g})}.$$
We have shown there that $\uZ^{1}(W_{F}/I_{F}^{\ell},\HG)_{L}$ is covered by the images of
finitely many of these maps, and that these images all have the same dimension. This
implies that the closure of these images are the irreducible components of 
$\uZ^{1}(W_{F}/I_{F}^{\ell},\HG)_{L}$. In particular,  the image of
the above map is dense in one of the components that contain $\varphi$.
Therefore, since reducedness implies generic smoothness of all components, we get (1).
\end{proof}

\begin{defn} A prime $\ell \neq p$ is $\LG$-banal if the properties of the last
  proposition hold for $\ell$.
\end{defn}



One way to view this reducedness of fibers from a philosophical standpoint is to say that there are no nontrivial ``congruences'' between Langlands parameters
modulo an $\LG$-banal prime $\ell$: the closures 
of distinct irreducible components in
characteristic zero remain distinct modulo $\ell$.  One expects that this should 
correspond, on the other side of the local Langlands correpondence, to a lack of nontrivial congruences between admissible smooth representations of the reductive group
$G$ over $F$ whose $L$-group is $\LG$. So this should be related to the representation
theoretic notion of ``banal''.
Recall indeed that, for a reductive group $G$ over $F$, a prime $\ell\neq p$ is called \emph{banal} if
it does not divide the order of a torsion element of $G(F)$. For the sake of precision, we
will reterm this as ``$G$-banal''.

\begin{lemma}
  Suppose that $G$ is a reductive group over $\mathcal{O}_{F}$.
  \begin{enumerate}
  \item A prime $\ell\neq p$ is $G$-banal if and only if it does not divide the order of $G(k_{F})$.
  \item The set of $G$-banal primes only depends on the isogeny class of $G$.
  \item We have $|G(k_{F})|=q^{N}\cdot\chi_{G,\Fr}(q)$ where $N$ is
    the dimension of a maximal unipotent subgroup of $G$.
  \end{enumerate}
\end{lemma}
\begin{proof}
(1) Let $g\in G(F)$ have finite order prime to $p$. Then it stabilizes a facet of the Bruhat-Tits
building of $G(F)$, and fixes its barycenter. This barycenter becomes a
hyperspecial
point in the building of $G(F')$ for some  totally ramified extension of $F$. So the order
of $g$ divides $|G(k_{F'})|$, but $k_{F'}=k_{F}$.
Conversely, let $\ell$ be a prime that divides $|G(k_{F})|$ and pick
an element $\bar g\in G(k_{F})$ with order $\ell$. Choose a
lift  $g\in G({\mathcal O}_{F})$ of $\bar g$ and consider
the topological Jordan decomposition  $g=g_{as}g_{tu}$ of $g$ as in \cite[Thm
2.38]{Spice}. Then $g_{as}\in G({\mathcal O}_{F})$ and it has order $\ell$.

(2) This follows from (1), see the proof of Theorem \ref{thm:chevallet-steinberg}.

(3) This is the Chevalley-Steinberg formula, see Theorem \ref{thm:chevallet-steinberg}.

\end{proof}

\begin{cor}\label{cor:banalvsbanal}
  Suppose $G$ is an unramified group over $F$ with no exceptional factor, denote by
  $\LG=\HG\rtimes\langle\Fr\rangle$ its Langlands dual group, and let $\ell$ be a prime greater than the
  Coxeter number of $G$. 
  Then $\ell$ is $\LG$-banal if and only if it is $G$-banal. 
\end{cor}
\begin{proof}
  This follows from the above lemma together with Proposition
  \ref{prop:classical}, and the equality $\chi_{\HG,\Fr}(T)=\chi_{G,\Fr}(T)$.
\end{proof}

It is a bit surprising that our results in Lemma \ref{lemma:triality} and Remark
\ref{rem:G2} show that this equivalence does not hold for exceptional groups.



\section{The GIT quotient in the banal case}

Our aim in this section  is to get a complete description of the affine quotient
$\uZ^{1}(W_{F}^{0}/P_{F}^{e},\HG)\sslash\HG$ after base change to $\overline\ZM[\frac 1N]$ for
some sufficiently well controlled integer $N$. Our strategy rests on the universal
homeomorphism  (\ref{eq:mapaffquo}) 
$$\uZ^{1}(W_{F}/I_{F}^{e},\HG)\sslash\HG 
\To{} \uZ^{1}(W_{F}^{0}/P_{F}^{e},\HG)\sslash\HG. $$
We have already singled out the so-called $\LG$-banal primes, which are
particularly well behaved for the RHS. On the other hand, the integer
$N_{\HG}$ defined above Corollary \ref{cor:existIe} plays a particular role regarding the LHS:

\begin{lemma}
  The structural morphism $\uZ^{1}(W_{F}/I_{F}^{e},\HG)\To{}\Spec(\ZZ[\frac 1p])$
  is smooth over $\Spec(\ZZ[\frac 1{pN_{\HG}}])$.
\end{lemma}
\begin{proof}  Since the finite group $I_{F}/I_{F}^{e}$
  has invertible order in $\ZZ[\frac 1{pN_{\HG}}]$, Lemma \ref{hom_smooth} tells us that
  $\uZ^{1}(I_{F}/I_{F}^{e},\HG)$ is smooth over $\Spec(\ZZ[\frac  1{pN_{\HG}}])$.
  Let $\phi_{\rm univ}$ denote the universal $1$-cocycle
  $I_{F}/I_{F}^{e}\To{}\HG({\mathcal O}_{\uZ^{1}(I_{F}/I_{F}^{e},\HG)})$.
  Then  the map $\varphi \mapsto \varphi(\Fr)$ identifies $\uZ^{1}(W_{F}/I_{F}^{e},\HG)$
  with the $\HG$-transporter from $^{\Fr}\phi_{\rm univ}$ to $\phi_{\rm univ}$, as a
  scheme over $\uZ^{1}(I_{F}/I_{F}^{e},\HG)$.
  Hence, by Lemma \ref{hom_smooth} again, the restriction
  map $\uZ^{1}(W_{F}/I_{F}^{e},\HG)\To{}\uZ^{1}(I_{F}/I_{F}^{e},\HG)$ is also smooth, and the lemma follows. 
\end{proof}

Recall that the universal homeomorphism  (\ref{eq:mapaffquo}) becomes an isomorphism after tensoring by $\QQ$.
The next result gives a bound on the set of integers that actually need to be inverted.

 \begin{prop} \label{prop:homeisombanal}
  The morphism  $\uZ^{1}(W_{F}/I_{F}^{e},\HG)\sslash\HG 
  \To{} \uZ^{1}(W_{F}^{0}/P_{F}^{e},\HG)\sslash\HG$ of (\ref{eq:mapaffquo}) is an isomorphism after inverting
$N_{\HG}$ and the  non $\LG$-banal primes.
\end{prop}
\begin{proof}
  Consider the dual map  (\ref{eq:mapinvring}) on rings of functions $  (R^{e}_{\LG})^{\HG}
  \To{} (S^{e}_{\LG})^{\HG}$. We already know it is injective and its cokernel is a torsion
  abelian group. Let $\ell\neq p$ be an associated prime of this cokernel. If $\ell$ does not
  divide $N_{\HG}$, there is no $\ell$-torsion in $S^{e}_{\LG}$ (by the last lemma), hence the reduced map 
$  (R^{e}_{\LG})^{\HG}\otimes\FM_{\ell} \To{} (S^{e}_{\LG})^{\HG}\otimes \FM_{\ell}$ is
not injective. But this map induces a bijection on $\overline\FM_{\ell}$-points, so its kernel
lies in the Jacobson radical, and we deduce that $(R^{e}_{\LG})^{\HG}\otimes\FM_{\ell}$ is
not reduced. On the other hand, $(R^{e}_{\LG})^{\HG}$ is an $\ell$-adically saturated submodule of
$R^{e}_{\LG}$, so that the map $(R^{e}_{\LG})^{\HG}\otimes\FM_{\ell}\To{}
(R^{e}_{\LG}\otimes\FM_{\ell})^{\HG}$ is actually injective. So we infer that
$R^{e}_{\LG}\otimes\FM_{\ell}$ is not reduced, hence $\ell$ is not $\LG$-banal.
\end{proof}

\begin{rem}
  When $\LG$ is the Langlands dual group of an unramified group, Proposition \ref{prop:regular-unip} and the
  estimate of Proposition \ref{prop:estimate_ss} show that the prime divisors of $N_{\HG}$ are non $\LG$-banal. We
  believe this is true in general. 
\end{rem}

In view of the last proposition, we focus in the next subsection
on the explicit description of $\uZ^{1}(W_{F}/I_{F}^{e},\HG)\sslash\HG$,
  over  $\overline\ZM[\frac 1{pN_{\HG}}]$.
The description that we obtain in Theorem \ref{thm:structureGIT} bears a striking analogy
with the usual description of the  Bernstein center.
 Actually, in Subsection \ref{subsec:comp-with-hain},  we extend scalars to $\CM$
   and we show that our description gives back Haines' definition of a structure of algebraic
   variety on the set of semisimple complex Langlands parameters.

\subsection{Description of $\uZ^{1}(W_{F}/I_{F}^{e},\HG)\sslash \HG$ over
  $\overline\ZM[\frac 1{pN_{\HG}}]$} \label{subsec:description}
Since the order of  $I_{F}/I_{F}^{e}$ is invertible in
$\ZZ[\frac 1{pN_{\HG}}]$, we
can obtain decompositions of
$\uZ^{1}(W_{F}/I_{F}^{e},\HG)_{\overline\ZZ[\frac 1{pN_{\HG}}]}$ similar to (\ref{eq:dec}) and
(\ref{eq:dec2}) by restricting cocycles to $I_{F}$ instead of restricting to $P_{F}$.
Indeed, we first infer the following results from Theorems \ref{h1_finite_etale},
\ref{representatives}, \ref{pi0fini} and  \ref{split_red_gp} in the appendix.
\begin{prop}
 There is a finite extension $\tilde{K}_{e}$ of $K_{e}$ and a set
 $$\tilde\Phi_{e} \subset Z^{1}\left(I_{F}/I_{F}^{e},\HG\left(\mathcal{O}_{\tilde{K}_{e}}
     {\textstyle[\frac1{pN_{\HG}}]}\right)\right), \hbox{ such that}$$
  \begin{itemize}
  \item for each $\phit\in\tilde\Phi_e$, the group scheme
    $C_{\HG}(\phit)^{\circ}$ is split reductive and
    $\pi_{0}(\phit):=\pi_{0}(C_{\HG}(\phit))$ is constant over
    $\mathcal{O}_{\tilde{K}_{e}}[\frac 1{pN_{\HG}}]$ and
  \item we have an orbit decomposition
$$\uZ^{1}(I_{F}/I_{F}^{e},\HG)_{\mathcal{O}_{\tilde{K}_{e}}[\frac
  1{pN_{\HG}}]}=\coprod_{\phit\in\tilde\Phi_e} \HG\cdot\phit \simeq
\coprod_{\phit\in\tilde\Phi_e} \HG/C_{\HG}(\phit)$$
where each summand represents the corresponding \'etale sheaf quotient.
\end{itemize}
\end{prop}
The above  decomposition induces in turn  the following ones :
$$
\uZ^{1}(W_{F}/I_{F}^{e},\HG)_{\mathcal{O}_{\tilde{K}_{e}}[\frac 1{pN_{\HG}}]} = 
\coprod_{\phit\in \tilde\Phi_e^{\rm adm}}  \HG \times^{C_{\HG}(\phit)}
\uZ^{1}(W_{F},\HG)_{\phit}.
$$

\begin{equation}
(\uZ^{1}(W_{F}/I_{F}^{e},\HG)\sslash\HG)_{\mathcal{O}_{\tilde{K}_{e}}[\frac 1{pN_{\HG}}]} = 
\coprod_{\phit\in \tilde\Phi_e^{\rm adm}}  
\uZ^{1}(W_{F},\HG)_{\phit}\sslash C_{\HG}(\phit).\label{eq:decompquot}
\end{equation}
Here, $\uZ^{1}(W_{F},\HG)_{\phit}$ is the affine scheme over
$\mathcal{O}_{\tilde{K}_{e}}[\frac 1{pN_{\HG}}]$ that classifies all $1$-cocycles
$\varphi:\, W_{F}\To{}\HG$ such that $\varphi_{|I_{F}}=\phit$ and, as usual,
we say that $\phit$ is admissible if this scheme is not empty.

Define the $\Fr$-twist of $\phit$ by
${^{\Fr}\phit}(i):= \Fr(\phit(\Fr^{-1}i\Fr))$. Then we have an
isomorphism $\varphi\mapsto \varphi(\Fr)$
$$\uZ^{1}(W_{F},\HG)_{\phit} \To\sim  T_{\HG}({^{\Fr}\phit},\phit)$$
where the RHS denotes the transporter in $\HG$ from 
${^{\Fr}\phit}$ to $\phit$ for the natural action of $\HG$
on $\uZ^{1}(I_{F},\HG)$. This isomorphism is $C_{\HG}(\phit)$-equivariant if we let
$C_{\HG}(\phit)$ act on the transporter by $\Fr$-twisted conjugation
$c\cdot t:= c t \Fr(c)^{-1}$. 
On the other hand, $T_{\HG}({^{\Fr}\phit},\phit)$ is also a left pseudo-torsor over $C_{\HG}(\phit)$
under composition $(c,t)\mapsto c  t$.
When $\phi$ is admissible, $T_{\HG}({^{\Fr}\phit},\phit)$ is actually a 
$C_{\HG}(\phit)$-torsor for the \'etale topology, and the \'etale sheaf quotient
$\pi_{0}({^{\Fr}\phit},\phit):=T_{\HG}({^{\Fr}\phit},\phit)/C_{\HG}(\phit)^{\circ}$
is a 
$\pi_{0}(\phit)$-torsor. Therefore $\pi_{0}({^{\Fr}\phit},\phit)$ is
representable by a finite \'etale $\mathcal{O}_{\tilde{K}_{e}}[\frac 1{pN_{\HG}}]$-scheme and,
after maybe enlarging $\tilde K_{e}$, we may and will assume that
it is  \emph{constant}. Then we get a further decomposition
$$T_{\HG}({^{\Fr}\phit},\phit) =\coprod_{\beta\in \pi_{0}({^{\Fr}\phit},\phit)}
T_{\HG}({^{\Fr}\phit},\phit)_{\beta}$$
which is nothing but the decomposition into connected components, and where each component
is a left $C_{\HG}(\phit)^{\circ}$-torsor. Moreover, the  $\Fr$-twisted
conjugation action of $C_{\HG}(\phit)$ on $T_{\HG}({^{\Fr}\phit},\phit)$
induces an action of
$\pi_{0}(\phit)$ on $\pi_{0}({^{\Fr}\phit},\phit)$. Denote by
$\pi_{0}(\phit)_{\beta}$ the stabilizer of $\beta$ for this action, and by
$\pi_{0}({^{\Fr}\phit},\phit)_{0}$ a set of representatives of orbits.
Then we get
\begin{equation}
T_{\HG}({^{\Fr}\phit},\phit)\sslash C_{\HG}(\phit) =\coprod_{\beta\in \pi_{0}({^{\Fr}\phit},\phit)_{0}}
 \left(T_{\HG}({^{\Fr}\phit},\phit)_{\beta}\sslash
   C_{\HG}(\phit)^{\circ}\right)_{/\pi_{0}(\phit)_{\beta}}.\label{eq:decompquot2}
\end{equation}
 Our next result will allow us to compute each term of this decomposition.
Before we can state it, note that if $R$ is an
$\mathcal{O}_{\tilde{K}_{e}}[\frac 1{pN_{\HG}}]$-algebra and $\tilde\beta\in
T_{\HG(R)}({^{\Fr}\phit},\phit)$, then conjugation by $\tilde\beta\rtimes \Fr$ in
$\HG(R)\rtimes W_{F}$ normalizes $C_{\HG(R)}(\phi)$. We denote the automorphism thus
induced by ${\rm Ad}_{\tilde\beta}$.

\begin{thm} \label{thm:fix-epinglage}
  Fix a pinning $\varepsilon_{\phit}=(B_{\phit}, 
  T_{\phit}, (X_{\alpha})_{\alpha})$ of
  $C_{\HG}(\phit)^{\circ}$ over $\mathcal{O}_{\tilde K_{e}}[\frac 1 {pN_{\HG}}]$. Then,  after
maybe enlarging the finite extension  $\tilde{K}_{e}$,  we can find  for each
$\beta\in\pi_{0}({^{\Fr}\phit},\phit)$,  a lift $\tilde\beta\in
T_{\HG(\mathcal{O}_{\tilde K_{e}}[\frac 1 {pN_{\HG}}])}({^{\Fr}\phit},\phit)$
of $\beta$ such that ${\rm Ad}_{\tilde\beta}$ normalizes $\varepsilon_{\phit}$.
\end{thm}

\begin{proof} The proof goes along the same argument as for Theorem
  \ref{Borel_preserving}. Let us first do the translation to the notation of Subsection
  \ref{sec:some-defin-constr}. 
  To this aim,  choose an
$L$-group such that $I_{F}/I_{F}^{e}$ embeds into $\pi_{0}(\LG)$. Then we can
form the subgroup scheme $C_{\LG}(\phit)$ as in Subsection
\ref{sec:some-defin-constr} and, 
denoting  by $\overline\Fr$ the image of $\Fr$ in $\pi_{0}(\LG)$, we see
that the map $\varphi\mapsto \varphiL(\Fr)$ defines an isomorphism  
$$\uZ^{1}(W_{F},\HG)_{\phit} \To\sim  C_{\LG}(\phit)\cap (\HG\rtimes
\overline\Fr),$$ 
and that the map $\tilde\beta\mapsto \tilde\beta\rtimes \overline\Fr$ defines a second
isomorphism
$$ T_{\HG}({^{\Fr}\phit},\phit) \To\sim C_{\LG}(\phit)\cap (\HG\rtimes
\overline\Fr)$$
whose composition with the previous one is the isomorphism introduced just above.
Now, define $\tilde\pi_{0}(\phit)$ and $\Sigma(\phit)$ as 
above Definition \ref{sigma_admissible}, so that we have for each admissible $\phit$ a further decomposition
$$\uZ^{1}(W_{F},\HG)_{\phit} =\coprod_{\alphat\in \Sigma(\phit)}
\uZ^{1}(W_{F},\HG)_{\phit,\alphat}.$$
Then we have a bijection $\alphat\mapsto \alphat(\Fr)$ between
$\Sigma(\phit)$ and the fiber of  the map $\tilde\pi_{0}(\phit)\to\pi_{0}(\LG)$
over $\overline\Fr$. On the other hand, we have a natural injection
$\pi_{0}({^{\Fr}\phit},\phit) \hookrightarrow \tilde\pi_{0}(\phit)$ whose
image is precisely the said fiber. So we get a bijection $\beta\leftrightarrow
\alphat$ between $ \pi_{0}({^{\Fr}\phit},\phit)$ and $\Sigma(\phit)$,
and it is easily checked that the map $\varphi\mapsto\varphi(\Fr)$ identifies 
$\uZ^{1}(W_{F},\HG)_{\phit,\alphat}$ with $T_{\HG}({^{\Fr}\phit},\phit)_{\beta}$.

Now the same proof as that of Theorem  \ref{Borel_preserving} applies, and actually 
 the stronger variant of  Remark \ref{rk_epinglage} applies too, because what is needed from
Lemma \ref{cohomological_lemma} in the proof of this variant is now
trivial : since $W_{F}/I_{F}\simeq \ZM$, we have $H^{2}(W_{F}/I_{F},A)=\{0\}$ for any
abelian group $A$ with action of $W_{F}/I_{F}$.

So we get the existence of a finite extension  $\tilde{K}_{e}$ and, for each $\alphat$,  a cocycle
\begin{equation}
\varphi_{\alphat}:W_{F}\To{}\HG\left(\mathcal{O}_{\tilde{K}_{e}}{\textstyle[\frac 1{pN_{\HG}}]}\right)\label{eq:extending_cocycle}
\end{equation}
that restricts to $\phit$, induces $\alphat$, normalizes
$\varepsilon_{\phit}$  and has finite image.
Writing $\varphi_{\alphat}(\Fr)=\tilde\beta\rtimes\overline\Fr$ provides us with the
desired element $\tilde\beta$.
\end{proof}

With the notation of this theorem we now have an identification $ c\mapsto c\tilde\beta$
$$ C_{\HG}(\phi)^{\circ} \To\sim T_{\HG}({^{\Fr}\phit},\phit)_{\beta} $$
  and the $\Fr$-twisted conjugation action of  $C_{\HG}(\phit)^{\circ}$ on
  $ T_{\HG}({^{\Fr}\phit},\phit)_{\beta} $ corresponds to  the
    ${\rm  Ad}_{\tilde\beta}$-twisted conjugation action of $C_{\HG}(\phit)^{\circ}$ on itself.
    We thus get
$$ T_{\HG}({^{\Fr}\phit},\phit)_{\beta}\sslash C_{\HG}(\phit)^{\circ} 
 = (C_{\HG}(\phit)^{\circ}\rtimes {\rm Ad}_{\tilde\beta})\sslash C_{\HG}(\phit)^{\circ}$$
where the notation on the right hand side is meant to emphasize that
$C_{\HG}(\phi)^{\circ}$ acts via $\Ad_{\tilde\beta}$-twisted conjugation.
 
Now, denote by
$\Omega_{\phit}^{\circ}$ the Weyl group of the maximal torus
$T_{\phit}$ of $C_{\HG}(\phit)^{\circ}$, and denote by  
$\Omega_{\phit}:=N_{C_{\HG}(\phit)}(T_{\phit})/T_{\phit}$ its ``Weyl
group'' in  $C_{\HG}(\phit)$. The natural map
$N_{C_{\HG}(\phit)}(T_{\phit},B_{\phit})\To{}\pi_{0}(\phit)$ induces an isomorphism 
$N_{C_{\HG}(\phit)}(T_{\phit},B_{\phit})/T_{\phit}\simeq\pi_{0}(\phit)$, 
hence  $\Omega_{\phit}=\Omega_{\phit}^{\circ}\rtimes\pi_{0}(\phit)$ is a split extension of
$\pi_{0}(\phit)$ by $\Omega_{\phit}^{\circ}$. Since the automorphism
${\rm Ad}_{\tilde\beta}$ of $C_{\HG}(\phit)$ stabilizes 
$T_{\phit}$ and $B_{\phit}$, it acts on $\Omega_{\phit}$ and preserves the
semi-direct product decomposition. Note that the actions of ${\rm Ad}_{\tilde\beta}$ on
$T_{\phit}$  and $\Omega_{\phit}$ only depend on $\beta$ and not
on the choice of $\tilde\beta$ as in the theorem. We will thus denote these actions
 simply by ${\rm Ad}_{\beta}$.
 Observe that the invariant subgroup
$(\Omega_{\phit})^{{\rm Ad}_{\beta}}$ of $\Omega_{\phit}$ decomposes as
$$(\Omega_{\phit})^{{\rm Ad}_{\beta}}
= (\Omega_{\phit}^{\circ})^{{\rm Ad}_{\beta}}\rtimes \pi_{0}(\phit)_{\beta}$$
and acts naturally on the coinvariant torus 
$(T_{\phit})_{{\rm Ad}_{\beta}}$.

\begin{prop}
  The inclusion
 $T_{\phit}\hookrightarrow C_{\HG}(\phit)^{\circ}$  induces
 an isomorphism 
$$ (T_{\phit})_{{\rm Ad}_{\beta}}\sslash (\Omega_{\phit}^{\circ})^{{\rm Ad}_{\beta}} \To\sim 
\left(C_{\HG}(\phit)^{\circ}\rtimes 
{\rm Ad}_{\tilde\beta}\right){\sslash C_{\HG}(\phit)^{\circ}}$$
\end{prop}
\begin{proof}
  Consider the inclusion
  $T_{\phit}\rtimes {\rm Ad}_{\tilde\beta} \hookrightarrow C_{\HG}(\phit)^{\circ}\rtimes
  {\rm Ad}_{\tilde\beta}$. Under the conjugation action of $C_{\HG}(\phi)^{\circ}$ on the RHS, the
  LHS is stable by  the subgroup scheme $N_{C_{\HG}(\phit)^{\circ}}(T_{\phit})_{\beta}$ of
$N_{C_{\HG}(\phit)^{\circ}}(T_{\phit})$ given as the inverse image of $(\Omega_{\phit}^{\circ})^{{\rm
    Ad}_{\beta}}$. Whence a morphism
$$\left(T_{\phit}\rtimes 
{\rm Ad}_{\tilde\beta}\right)\sslash N_{C_{\HG}(\phit)^{\circ}}(T_{\phit})_{\beta} \To{}
\left(C_{\HG}(\phit)^{\circ}\rtimes 
{\rm Ad}_{\tilde\beta}\right){\sslash C_{\HG}(\phit)^{\circ}}.$$
Now observe that $(T_{\phit})_{{\rm Ad}_{\beta}} = \left(T_{\phit}\rtimes 
{\rm Ad}_{\tilde\beta}\right)\sslash T_{\phit}$, so that
the above morphism induces in turn
a  morphism
$$  (T_{\phit})_{{\rm Ad}_{\beta}}\sslash (\Omega_{\phit}^{\circ})^{{\rm Ad}_{\beta}} \To{} 
(C_{\HG}(\phit)^{\circ}\rtimes 
{\rm Ad}_{\tilde\beta}){\sslash C_{\HG}(\phit)^{\circ}}.$$
Now, for any algebraically closed field $L$ over
$\mathcal{O}_{\tilde{K}_{e}}[\frac 1{pN_{\HG}}]$, Lemma 6.5 of  \cite{borel_corvallis} tells us that 
this morphism induces a bijection on $L$-points. In particular the
corresponding map on rings of functions is injective since the source is reduced.
Its surjectivity can be proved as in \cite[Prop. 6.7]{borel_corvallis}, which deals with complex
coefficients.  Namely, put $R:=\mathcal{O}_{\tilde{K}_{e}}[\frac 1{pN_{\HG}}]$ and let
$X$ denote the character group of $T_{\phi}$. Then the ring of functions of 
$(T_{\phit})_{{\rm Ad}_{\beta}}\sslash (\Omega_{\phit}^{\circ})^{{\rm Ad}_{\beta}} $
is $R[X^{\Ad_{\beta}}]^{(\Omega_{\phi}^{\circ})^{\Ad_{\beta}}}$, hence has a natural
$R$-basis given by $(\Omega_{\phi}^{\circ})^{\Ad_{\beta}}$-orbits in
$X^{\Ad_{\beta}}$. Any such orbit has a unique representative in the antidominant cone of
$X$ with respect to $B_{\phi}$. So let $\lambda \in X^{\Ad_{\beta}}$ be antidominant in
$X$ and let $\mathcal{L}_{\lambda}$ be the corresponding invertible sheaf on the flag
variety $C_{\HG}(\phi)^{\circ}/B_{\phi}$. Then
$M_{\lambda}:=H^{0}(C_{\HG}(\phi)^{\circ}/B_{\phi},\mathcal{L}_{\lambda})$ is a free
$R$-module of finite rank with an
algebraic action of $C_{\HG}(\phi)^{\circ}$. Actually, since $\lambda$ is
$\Ad_{\beta}$-invariant, it defines a character of the group scheme $T_{\phi}\rtimes
\langle \Ad_{\beta}\rangle$, and since
$C_{\HG}(\phi)^{\circ}/B_{\phi}=(C_{\HG}(\phi)^{\circ}\rtimes\langle \Ad_{\beta}\rangle)/(B_{\phi}\rtimes\langle \Ad_{\beta}\rangle)$,
we see that $M_{\lambda}$ is actually a
$C_{\HG}(\phi)^{\circ}\rtimes\langle \Ad_{\beta}\rangle$-module. 
In particular, the map $g\mapsto \tr(g\rtimes\Ad_{\beta}|M_{\lambda})$ is in the ring of
functions of $(C_{\HG}(\phit)^{\circ}\rtimes  {\rm Ad}_{\tilde\beta}){\sslash C_{\HG}(\phit)^{\circ}}$.
Its restriction to $T_{\phi}$ factors over $(T_{\phi})_{\Ad_{\beta}}$ and is of the
form
$$ c\left(\sum_{\lambda'\in (\Omega_{\phi}^{\circ})^{\Ad_{\beta}}.\lambda} \lambda'\right)  +
\sum_{\mu>\lambda} a_{\mu} \mu,\,\,\, a_{\mu}\in \NN,$$
where $c$ denotes the eigenvalue of $\Ad_{\beta}$ on the $\lambda$-eigenspace of $T_{\phi}$ in
$M_{\lambda}$ (which is a free direct factor of rank $1$).
So we deduce inductively the desired surjectivity.
 \end{proof}


Eventually, after choosing a pinning $\varepsilon_{\phi}$ for each $\phi\in \tilde\Phi_{e}$
and inserting the result of the above proposition inside decompositions (\ref{eq:decompquot}) and 
 (\ref{eq:decompquot2}), 
we get our desired description of the affine quotient over $\mathcal{O}_{\tilde{K}_{e}}[\frac 1{pN_{\HG}}]$.
\begin{thm} \label{thm:structureGIT}
The collection of embeddings $T_{\phi}\hookrightarrow C_{\HG}(\phi)$ induce an isomorphism of
$\mathcal{O}_{\tilde{K}_{e}}[\frac 1{pN_{\HG}}]$-schemes 
$$\coprod_{\phit\in \tilde\Phi_e^{\rm adm}}  
\coprod_{\beta\in \pi_{0}({^{\Fr}\phit},\phit)_{0}}
(T_{\phit})_{{\rm Ad}_{\beta}}\sslash (\Omega_{\phit})^{{\rm Ad}_{\beta}}  \To\sim
(\uZ^{1}(W_{F}/I_{F}^{e},\HG)\sslash\HG)_{\mathcal{O}_{\tilde{K}_{e}}[\frac 1{pN_{\HG}}]}  
.
$$
\end{thm}

We note that the LHS does not depend on the choices of elements $\tilde\beta$ as in
Theorem \ref{thm:fix-epinglage}. But the maps from the LHS to the RHS a priori depend on
these choices.

 \subsection{The GIT quotient over a banal algebraically closed field}
With Theorem \ref{thm:structureGIT} and Proposition \ref{prop:homeisombanal}, we now have a
description of the affine quotient $\uZ^{1}(W_{F}^{0}/P_{F}^{e},\HG)\sslash\HG$ after inverting
$N_{\HG}$ and the non $\LG$-banal primes. Let us now consider affine quotients over algebraically
closed fields.


\begin{thm} \label{thm:structureGITfield}
 Let $L$ be an algebraically closed field over $\mathcal{O}_{\tilde{K}_{e}}[\frac
 1{pN_{\HG}}]$ and of $\LG$-banal characteristic. Then the natural maps induce
 isomorphisms
 \begin{eqnarray*}
\coprod_{\phit\in \tilde\Phi_e^{\rm adm}}  
\coprod_{\beta\in \pi_{0}({^{\Fr}\phit},\phit)_{0}}
   (T_{\phit,L})_{{\rm Ad}_{\beta}}\sslash (\Omega_{\phit})^{{\rm Ad}_{\beta}}
   & \To\sim&
              \uZ^{1}(W_{F}/I_{F}^{e},\HG_{L})\sslash\HG_{L}\\
   \uZ^{1}(W_{F}/I_{F}^{e},\HG_{L})\sslash\HG_{L}
   & \To\sim&
   \uZ^{1}(W_{F}^{0}/P_{F}^{e},\HG_{L})\sslash\HG_{L}\\
   \uZ^{1}(W_{F}^{0}/P_{F}^{e},\HG_{L})\sslash\HG_{L} &\To\sim& (\uZ^{1}(W_{F}^{0}/P_{F}^{e},\HG)\sslash\HG)_{L}.
 \end{eqnarray*}
 \end{thm}
 \begin{proof}
 The first isomorphism holds  without the $\LG$-banal hypothesis, and is proved exactly
 as the isomorphism of Theorem \ref{thm:structureGIT}. We then have a commutative diagram
 $$\xymatrix{
{\coprod}_{(\phit,\beta)}
   (T_{\phit,L})_{{\rm Ad}_{\beta}}\sslash (\Omega_{\phit})^{{\rm Ad}_{\beta}}
\ar[d]  \ar[r]^-{\sim} & 
              \uZ^{1}(W_{F}/I_{F}^{e},\HG_{L})\sslash\HG_{L} \ar[d] \\
{\coprod}_{(\phit,\beta)}
   \left((T_{\phit})_{{\rm Ad}_{\beta}}\sslash (\Omega_{\phit})^{{\rm Ad}_{\beta}}\right)_{L}
   \ar[r]^-{\sim} &   \left(\uZ^{1}(W_{F}/I_{F}^{e},\HG)\sslash\HG\right)_{L}
   }$$
 If, in addition, the order of each
 $(\Omega_{\phi})^{\rm Ad_{\beta}}$ is invertible in $L$, which is certainly the case if
 ${\rm char}(L)$ is $\LG$-banal, then
 the left vertical map is an isomorphism, and it follows that the right vertical map is
 also an isomorphism.
So we now have a commutative square involving the analogous map for $W_{F}^{0}/P_{F}^{e}$,
  which, in terms of rings, reads
$$\xymatrix{
  (R_{\LG}^{e})^{\HG}\otimes L  \ar[r] \ar[d]^{\sim} &
  (R_{\LG}^{e}\otimes L)^{\HG} \ar[d] \\
  (S_{\LG}^{e})^{\HG}\otimes L \ar[r]^{\sim} &
  (S_{\LG}^{e}\otimes L)^{\HG} 
}$$
The right vertical map is surjective since the bottom map is surjective, and it is also
injective since it induces a bijection on $L$-points and  $(R_{\LG}^{e}\otimes L)^{\HG}$
is reduced.  Therefore it is an isomorphism, and so is the upper map.
 \end{proof}

\begin{rem}\label{rem:structureGIT}

 Let $L$ be an algebraically closed field as in the theorem.

i) The index set in the first isomorphism can be replaced by any set
$\Psi_{e}(L)$ of representatives of 
$\HG(L)$-conjugacy classes of pairs $(\phit,\beta)$ consisting of a
cocycle $\phit : I_{F}/I_{F}^{e}\To{}\LG(L)$ and an element $\beta\in\pi_{0}({^{\Fr}\phit},\phit)$.
Since any $\phit$ as above is automatically semisimple, this is the same set as in Corollary \ref{cor:up_to_homeo}.

ii) As usual, the set of $L$-points of $\uZ^{1}(W_{F}/I_{F}^{e},\HG_{L})\sslash\HG_{L}$
is the set of closed $\HG_{L}$-orbits in $Z^{1}(W_{F}/I_{F}^{e},\HG(L))$. For a cocycle
$\varphi : W_{F}/I_{F}^{e}\To{} \HG(L)$ in some
$\uZ^{1}(W_{F}/I_{F}^{e},\HG_{L})_{\phi}$, we claim that the following statements are
equivalent, \emph{provided $L$ has characteristic $0$} :
\begin{enumerate}
\item its $\HG_{L}$-orbit is closed in $\uZ^{1}(W_{F}/I_{F}^{e},\HG_{L})$, 
\item its  $C_{\HG}(\phi)_{L}$-orbit is closed in
  $\uZ^{1}(W_{F}/I_{F}^{e},\HG_{L})_{\phi}$,
\item $\varphiL(\Fr)$  is a semisimple element of $\LG(L)$.
\item $\varphiL(W_{F})$ consists of semisimple elements.
\end{enumerate}
Indeed, $(1)\Rightarrow (2)$ since the small orbit is the intersection of the big one with
the closed subset $\uZ^{1}(W_{F}/I_{F}^{e},\HG_{L})_{\phi}$. Moreover,
$(2)$ is equivalent to the orbit of $\varphiL(\Fr)$ being closed in $C_{\LG}(\phi)(L)$, which
in turn is equivalent to $\varphiL(\Fr)$ being a semisimple element of $C_{\LG}(\phi)(L)$,
hence also of $\LG(L)$. Further, $(3)$, being equivalent to $(2)$, applies to any lift of Frobenius, so implies
$(4)$. Eventually, $(4)$ implies that the $\HG(L)$-orbits of a finite set of generators of
$\varphiL(W_{F})$ are closed, which implies $(1)$.


The cocycles that satisfy property (3) are often called ``Frobenius semi-simple'' in the literature.
When $L$ has positive characteristic,  a cocycle with closed orbit may not be Frobenius
semi-simple.
For example suppose $q=q_{F}$ has prime order $\ell\neq p$ in some
$(\ZM/n\ZM)^{\times}$ with $n$ prime to both $p$ and $\ell$,
and consider the character $\theta :\,I_{F}\twoheadrightarrow \mu_{n}\hookrightarrow\overline\FM_{\ell}^{\times}.$
Extend this character to $I_{F}\cdot \Fr^{\ell\ZM}$ by setting $\theta(\Fr^{\ell})=1$ and
induce to $W_{F}$. We obtain an irreducible representation
$\varphi:\, W_{F}\To{}\GL_{\ell}(\overline\FM_{\ell})$ such that $\varphi(\Fr)$ has order $\ell$. 
\end{rem}

\subsection{Comparison with the Haines variety}\label{subsec:comp-with-hain}
In this subsection, we assume that the action of $W_{F}$ stabilizes a pinning
of $\HG$, so that $\LG$ is an $L$-group associated to
some reductive group $G$ over $F$. 
As noted  in point ii) of the last remark, the set of $\mathbb{C}$-points of the affine
categorical quotient
$$\uZ^{1}(W_{F}/I_{F}^{e},\HG)_{\mathbb{C}}\sslash \HG_{\mathbb{C}}$$ is the set of
$\HG(\mathbb{C})$-conjugacy classes of Frobenius semisimple $L$-homomorphisms
$W_{F}/I_{F}^{e}\To{} \LG(\mathbb{C})$.  
In \cite{haines}, Haines  endows this set with the
structure of a complex affine variety that mimics Bernstein's description of the center of the
category of complex representations of $G(F)$.  We will denote by $\Omega_{e}(\HG)$ the Haines
variety and we wish to compare his construction to ours.
Note that, in the notation of \cite{haines}, $\Omega_{e}(\HG)$ is a summand of
$\mathfrak{Y}$ and is the union of all components $\mathfrak Y_{\mathfrak t}$
corresponding to inertial classes of parameters that are trivial on $P_{F}^{e}$.
In a rather
abstract form, the main result of this section is the following.
\begin{thm}\label{thm:comparison}
The set-theoretic  identification between
    $(\uZ^{1}(W_{F}/I_{F}^{e},\HG)\sslash \HG)(\mathbb{C})$  
and $\Omega_{e}(\HG)$ is induced by an  isomorphism of varieties
$$  \Omega_{e}(\HG) \simeq \uZ^{1}(W_{F}/I_{F}^{e},\HG)_{\mathbb{C}}\sslash \HG_{\mathbb{C}} .$$
\end{thm}



We need to recall some features of Haines' construction in Section 5
of \cite{haines}.
Let $\varphiL : W_{F}/I_{F}^{e}\to\LG(\mathbb{C})$ be a Frobenius-semisimple $L$-morphism
(called an ``infinitesimal character'' by Haines and Vogan), and choose a Levi subgroup
$\mathcal{M}$ of $\LG$ that contains $\varphiL(W_{F})$ and is minimal for this
property. Here we consider Levi  subgroups in the sense of Borel 
\cite[\S 3]{borel_corvallis}. In particular, $\mathcal{M}^{\circ}=\mathcal{M}\cap \HG$ 
is a Levi subgroup of $\HG$ and $\pi_{0}(\mathcal{M})\To\sim \pi_{0}(\LG)$ is a
quotient of $W_{F}$. As a
consequence, the action of $\mathcal{M}$ by conjugation on the center
$Z(\mathcal{M}^{\circ})$ of $\mathcal{M}^{\circ}$ factors through an action of
$\pi_{0}(\LG)$, and provides thus a canonical action of $W_{F}$ on
$Z(\mathcal{M}^{\circ})$. 
We may then consider the torus
$(Z(\mathcal{M}^{\circ})^{I_{F}})^{\circ}$ given by  the neutral component of the
$I_{F}$-invariants, and which still carries an action of $W_{F}/I_{F}=\langle\Fr\rangle$.
To any $z\in (Z(\mathcal{M}^{\circ})^{I_{F}})^{\circ}$, Haines associates a new parameter
$z\cdot\varphiL$ defined by
$$ (z\cdot \varphiL)(w):=z^{\nu(w)}\varphiL(w), 
\,\,\hbox{ with }
\nu: W_{F}\twoheadrightarrow \ZZ\hbox{ defined by }w\Fr^{-\nu(w)}\in I_{F}.$$

The conjugacy class $(z\cdot \varphiL)_{\HG}$ only depends on the image of $z$ in the 
 $\Fr$-coinvariants $(Z(\mathcal{M}^{\circ})^{I_{F}})^{\circ}_{\Fr}$, hence we get a map
 \begin{equation}
   (Z(\mathcal{M}^{\circ})^{I_{F}})^{\circ}_{\Fr} \To{} \Omega_{e}(\HG)=
   (\uZ^{1}(W_{F}/I_{F}^{e},\LG)\sslash \HG)(\mathbb{C})
   \label{eq:morphism_haines}
 \end{equation}
By Haines' definition of the variety structure on $\Omega_{e}(\HG)$, this map is a
morphism of algebraic varieties $ (Z(\mathcal{M}^{\circ})^{I_{F}})^{\circ}_{\Fr} \To{}
\Omega_{e}(\HG)$.
Even better, there is a finite group
$W_{\mathcal{M},\varphi}$ of algebraic automorphisms of
$(Z(\mathcal{M}^{\circ})^{I_{F}})^{\circ}_{\Fr}$, whose precise definition is not needed
here, such that the map (\ref{eq:morphism_haines}) factors over an injective map
$(Z(\mathcal{M}^{\circ})^{I_{F}})^{\circ}_{\Fr}/W_{\mathcal{M},\varphi}\hookrightarrow
\Omega_{e}(\HG)$. Then, the corresponding morphism of varieties
$ (Z(\mathcal{M}^{\circ})^{I_{F}})^{\circ}_{\Fr}\sslash
W_{\mathcal{M},\varphi} \To{}\Omega_{e}(\HG)$ is an isomorphism
onto a connected component of $\Omega_{e}(\HG)$, by Haines' construction. Moreover, all
connected components are obtained in this way. 

At this point, we have recalled enough to prove one direction.
\begin{lemma}
The set-theoretic  identification between 
  $\Omega_{e}(\HG)$ and
  $(\uZ^{1}(W_{F}/I_{F}^{e},\HG)\sslash\HG)(\mathbb{C})$  is induced
  by a  morphism of varieties 
$$  \Omega_{e}(\HG) \To{} \uZ^{1}(W_{F}/I_{F}^{e},\HG)_{\mathbb{C}}\sslash
\HG_{\mathbb{C}} .$$ 
\end{lemma}
\begin{proof}
By the foregoing discussion, it now suffices to prove that 
 each  map (\ref{eq:morphism_haines}) is induced by a
  morphism of schemes $$ (Z(\mathcal{M}^{\circ})^{I_{F}})^{\circ}_{\Fr}
  \To{} \uZ^{1}(W_{F}/I_{F}^{e},\HG)_{\mathbb{C}}\sslash \HG_{\mathbb{C}}.$$

 By construction, the map (\ref{eq:morphism_haines}) is part of a
  commutative diagram
$$\xymatrix{
  (Z(\mathcal{M}^{\circ})^{I_{F}})^{\circ} \ar[d] \ar[r] &
  \uZ^{1}(W_{F}/I_{F}^{e},\HG)(\mathbb{C}) \ar[d] \\
  (Z(\mathcal{M}^{\circ})^{I_{F}})^{\circ}_{\Fr} \ar[r] &
  (\uZ^{1}(W_{F}/I_{F}^{e},\HG)\sslash \HG)(\mathbb{C}) }$$
where the top map is given by $z\mapsto z\cdot \varphiL$. Denote by
$\zeta \in \LG(\mathbb{C}[(Z(\mathcal{M}^{\circ})^{I_{F}})^{\circ}])$
the element corresponding to the closed immersion
$(Z(\mathcal{M}^{\circ})^{I_{F}})^{\circ}\hookrightarrow
\LG$. Then $\zeta\cdot\varphiL$ is an element of
$Z^{1}(W_{F}/I_{F}^{e},\HG(\mathbb{C}[(Z(\mathcal{M}^{\circ})^{I_{F}})^{\circ}]))$, hence
corresponds to a morphism
$(Z(\mathcal{M}^{\circ})^{I_{F}})^{\circ}\To{}
\uZ^{1}(W_{F}/I_{F}^{e},\HG)$. By definition, this morphism induces
the top map of the above diagram on the respective sets of
$\mathbb{C}$-points.  Moreover, the composition of this morphism with
the morphism underlying the right vertical map of the diagram is
$\Fr$-equivariant for the trivial action of $\Fr$ on the target, so it
has to factor over a morphism which induces the bottom map of the
diagram, as desired.
\end{proof}




We now go in the other direction.

\begin{lemma}
The set-theoretic  identification between 
  $(\uZ^{1}(W_{F}/I_{F}^{e},\HG)\sslash\HG)(\mathbb{C})$ and  $\Omega_{e}(\HG)$ 
 is induced
  by a  morphism of varieties 
$$  \uZ^{1}(W_{F}/I_{F}^{e},\HG)_{\mathbb{C}}\sslash \HG_{\mathbb{C}} \To{}  \Omega_{e}(\HG) .$$ 
\end{lemma}
 \begin{proof}
The description of
Theorem \ref{thm:structureGIT} shows that it suffices  to prove that
for any pair $(\phit,\beta)$ as in Theorem \ref{thm:structureGIT}
and any choice of $\tilde\beta\in T_{\HG}({^{\Fr}\phit},\phit)_{\beta}$ as in Theorem \ref{thm:fix-epinglage},
   the map
   \begin{equation}\label{eq:morphism_git}
     (T_{\phit})_{\rm Ad_{\beta}}(\mathbb{C}) \To{}    
     (\uZ^{1}(W_{F}/I_{F}^{e},\HG)\sslash \HG)(\mathbb{C}) =\Omega_{e}(\HG)
   \end{equation}
   is induced by a morphism of algebraic varieties $ (T_{\phit})_{\rm Ad_{\beta}} \To{} \Omega_{e}(\HG)$.

   To prove this, we will identify the maps (\ref{eq:morphism_git}) to
   instances of maps (\ref{eq:morphism_haines}).
   Let us thus fix a pair $(\phit,\beta)$ and $\tilde\beta$ as in the statement, so that
   ${\rm Ad}_{\tilde\beta}$ fixes a pinning    $\varepsilon_{\phit}$ of $C_{\HG}(\phit)^{\circ}$ with maximal
   torus $T_{\phit}$.
   We denote by    $\varphi_{\tilde\beta}$ the unique extension of $\phi$ such that
   $\varphi_{\tilde\beta}(\Fr)=\tilde\beta$.
  Consider the torus $(T_{\phit})^{\rm
     Ad_{\tilde\beta},\circ}$.  Its centralizer
   $\mathcal{M}$ in $\LG$  contains $\varphiL_{\tilde\beta}(W_{F})$, hence maps
   onto $\pi_{0}(\LG)$ and is thus a Levi subgroup in the sense of
   Borel. Moreover, the canonical action of $W_{F}$ on
   $Z(\mathcal{M}^{\circ})$ is induced by ${\rm
     Ad}_{\varphi_{\tilde\beta}}$.  We claim that $\mathcal{M}$ is
   minimal among Levi subgroups of $\LG$ that contain
   $\varphiL_{\tilde\beta}(W_{F})$. Indeed, if
   $\varphiL_{\tilde\beta}(W_{F})\subset \mathcal{M}'\subset
   \mathcal{M}$, then $(T_{\phit})^{\rm
     Ad_{\tilde\beta},\circ} \subset
   Z(\mathcal{M})^{\circ} \subset Z(\mathcal{M'})^{\circ} \subset
   C_{\HG}(\varphi_{\tilde\beta})^{\circ}.$ But
   $(T_{\phit})^{\rm Ad_{\tilde\beta},\circ}$ is  a maximal torus of 
$C_{\HG}(\varphi_{\tilde\beta})^{\circ}= C_{\HG}(\phit)^{\rm Ad_{\tilde\beta},\circ}$ 
by \cite[Thm 1.8
   iii)]{DM94}, so all inclusions above have to be equalities and in
   particular $\mathcal{M}'= \mathcal{M}$ since a Levi subgroup of
   $\LG$ is the centralizer in $\LG$ of its connected center by
   \cite[Lem. 3.5]{borel_corvallis}.  Now, observe that
$$(Z(\mathcal{M}^{\circ})^{I_{F}})^{\circ} = Z(\mathcal{M}^{\circ})^{{\rm
    Ad}_{\phit(I_{F})},\circ}
\subset  \mathcal{M}^{\circ}\cap C_{\HG}(\phit)^{\circ}=
T_{\phit}.$$ Indeed, the last equality comes from \cite[Thm 1.8
iv)]{DM94} which implies that the centralizer of
$(T_{\phit})^{\rm Ad_{\tilde\beta},\circ}$ in
$C_{\HG}(\phit)^{\circ}$ is $T_{\phit}$. Taking $^{L}\varphi_{\tilde\beta}(\Fr)$-invariants,
we get 
$$ Z(\mathcal{M})^{\circ}=(Z(\mathcal{M}^{\circ})^{I_{F}})^{\rm
  Ad_{\tilde\beta},\circ} \subset (T_{\phit})^{\rm Ad_{\tilde\beta},\circ},$$
from which we deduce that $(Z(\mathcal{M}^{\circ})^{I_{F}})^{\rm
  Ad_{\tilde\beta},\circ} = (T_{\phit})^{\rm Ad_{\tilde\beta},\circ}$ since
$ (T_{\phit})^{\rm Ad_{\tilde\beta},\circ}\subset Z(\mathcal{M})$. 
Since the order of $\rm Ad_{\tilde\beta}$ on $T_{\phi}$ is finite, it follows that the inclusion of tori
$(Z(\mathcal{M}^{\circ})^{I_{F}})^{\circ} \subset T_{\phit}$ induces an isogeny 
\begin{equation}
 (Z(\mathcal{M}^{\circ})^{I_{F}})^{\circ}_{\Fr} \twoheadrightarrow  (T_{\phit})_{\rm
   Ad_{\tilde\beta}}.\label{eq:isogeny}
 \end{equation}
Here, note that the action of $\varphi_{\tilde\beta}(\Fr)$ on
$Z(\mathcal{M}^{\circ})$ is just the action of $\Fr$ since $\tilde\beta\in\mathcal{M}^{\circ}$.
Now, by construction, the map (\ref{eq:morphism_git}) is part of a
commutative diagram
$$\xymatrix{
  T_{\phit}(\mathbb{C}) \ar[d] \ar[r] &
  \uZ^{1}(W_{F}/I_{F}^{e},\HG)(\mathbb{C}) \ar[d] \\
  (T_{\phit})_{\rm Ad_{\tilde\beta}}(\mathbb{C}) \ar[r] &
  (\uZ^{1}(W_{F}/I_{F}^{e},\HG)\sslash \HG)(\mathbb{C}) }$$
where the top map is 
given by
$$
t\mapsto \left( w\mapsto   t^{\nu(w)}\varphi_{\tilde\beta}(w)\right),
$$
where $\nu$ is the projection $W_{F}\To{}W_{F}/I_{F}=\Fr^{\ZM}$.
This shows that the composition of the bottom map of the diagram (i.e. the map of the
lemma) with (\ref{eq:isogeny}) 
is an instance of (\ref{eq:morphism_haines}), which is a
morphism of algebraic varieties
$(Z(\mathcal{M}^{\circ})^{I_{F}})^{\circ}_{\Fr} \To{}\Omega_{e}(\HG)$
according to Haines' construction. This morphism is constant along the fibers of
(\ref{eq:isogeny}), so it has to factor over  (\ref{eq:isogeny}) since
the latter  is a quotient morphism. Hence the bottom map
of the diagram is a morphism of varieties $(T_{\phit})_{\rm Ad_{\tilde\beta}}\To{}\Omega_{e}(\HG)$.
\end{proof}

\begin{proof}[Proof of Theorem \ref{thm:comparison}] It follows from
  the two above lemmas.
\end{proof}

\begin{rem}
  The isomorphism of Theorem \ref{thm:comparison} induces of course a
  bijection between the sets of connected components on both
  sides. This bijection is easily described as follows :
  \begin{itemize}
  \item $\pi_{0}(\Omega_{e}(\HG))$ is the set of ``inertial classes''
    of Frobenius-semisimple cocycles $\varphi$,  as defined in
    \cite[\S 5.3, Def. 4.15]{haines}, that are trivial on $I_{F}^{e}$.
  \item
    $\pi_{0}(\uZ^{1}(W_{F}/I_{F}^{e},\HG)_{\mathbb{C}}\sslash\HG_{\mathbb{C}})$
    is the set of conjugacy classes of pairs
    $(\phit,\beta)$ as in ii) of Remark \ref{rem:structureGIT}.
  \item The bijection takes $\varphi$ to
    $(\varphi_{|I_{F}},p(\varphi(\Fr)))$ with 
    $p$ the projection $T_{\HG}({^{\Fr}\phit},\phit)\To{}\pi_{0}({^{\Fr}\phit},\phit)$.
\end{itemize}
\end{rem}

Now, we wish to compare more explicitly Haines' construction with our description. This can
be done component-wise, so let us fix data $(\phit,\beta,\varepsilon_{\phit},\tilde\beta)$ 
and put $\varphi=\varphi_{\tilde\beta}$ as in the last proof. Recall also the Levi subgroup
$\mathcal{M}=C_{\LG}((T_{\phit})^{\Ad_{\tilde\beta},\circ})$ of $\LG$  that
appeared in the last proof. So we have
an inclusion $Z(\mathcal{M}^{\circ})^{I_{F},\circ} \subset T_{\phit}$
that induces an isogeny 
$\pi:\, (Z(\mathcal{M}^{\circ})^{I_{F},\circ})_{\Fr}\twoheadrightarrow
(T_{\phit})_{\Ad_{\tilde\beta}}$ as in  (\ref{eq:isogeny}). The associated
connected component in Theorem \ref{thm:structureGIT} is
$(T_{\phit})_{\Ad_{\tilde\beta}}\sslash (\Omega_{\phit})^{{\rm Ad}_{\tilde\beta}}$, which
we may also write as a two-steps quotient :
$$ \left((Z(\mathcal{M}^{\circ})^{I_{F},\circ})_{\Fr}/\ker(\pi)\right)/(\Omega_{\phit})^{{\rm Ad}_{\tilde\beta}},$$
while the same component  is described as a two-steps quotient 
$$ \left((Z(\mathcal{M}^{\circ})^{I_{F},\circ})_{\Fr}/{\rm Stab}(\varphi)\right)/W_{\varphi,\mathcal{M}^{\circ}}$$
in Lemmas 4.19 and 4.20  of Section 5.3 in Haines' paper
\cite{haines}. The relation between  these two presentations can be summarized as follows :

\begin{lemma} Using the notation right above,
  \begin{enumerate}
    \item we have $\ker(\pi)\subset {\rm Stab}(\varphi)$ as subgroups
      of $(Z(\mathcal{M}^{\circ})^{I_{F},\circ})_{\Fr}$.
    \item there is a normal subgroup
      $K\subset (\Omega_{\phi})^{\Ad_{\beta}}$ whose action on
      $(T_{\phi})_{\Ad_{\beta}}$ factors over that of
      ${\rm Stab}(\varphi)/\ker(\pi)$  through 
      a surjective map
      $K\twoheadrightarrow {\rm Stab}(\varphi)/\ker(\pi)$.
    \item there is a natural isomorphism
      $(\Omega_{\phi})^{\Ad_{\beta}}/K\To\sim
      W_{\varphi,\mathcal{M}^{\circ}}$ compatible with the respective
      actions on
      $(Z(\mathcal{M}^{\circ})^{I_{F},\circ})_{\Fr}/{\rm
        Stab}(\varphi)$.
    \end{enumerate}
    Moreover,  when $C_{\HG}(\phi)$ is
    connected,  we have $\ker(\pi)= {\rm Stab}(\varphi)$,
    while $(\Omega_{\phit})^{{\rm Ad}_{\tilde\beta}}$ identifies with
    $W_{\varphi,\mathcal{M}^{\circ}}$ compatibly with the action on
    $(T_{\phi})_{\Ad_{\beta}}=(Z(\mathcal{M}^{\circ})^{I_{F},\circ})_{\Fr}/{\rm
      Stab}(\varphi)$ (i.e., the group $K$ above is trivial).  
\end{lemma}



\begin{proof}
 Let us simplify the notation by putting
  $Z:=(Z(\mathcal{M}^{\circ})^{I_{F},\circ})_{\Fr}$ (denoted
  $Y(\mathcal{M}^{\circ})$ in Haines' paper) and
  $\tilde Z := Z(\mathcal{M}^{\circ})^{I_{F},\circ}$.  Then Haines'
  definition of ${\rm Stab}(\varphi)$ (denoted ${\rm stab}_{\lambda}$
  there) is
$$ {\rm Stab}(\varphi)=\left\{z\in Z,\, \exists \tilde z \mapsto z, \exists
  m\in \mathcal{M}^{\circ}, \tilde z\cdot\varphi= {\rm
    Ad}_{m}(\varphi)\right\}.$$ Here, we use our notation
$c\cdot \varphi$ for the unique $1$-cocycle that restricts to $\phi$
on inertia and takes value $c\tilde\beta$ on $\Fr$ (this makes sense
for any $c\in C_{\HG}(\phi)$).
Note that if $\tilde z\cdot\varphi= {\rm Ad}_{m}(\varphi)$, then in
particular $\phi={\rm Ad}_{m}(\phi)$, i.e. $m\in C_{\HG}(\phit)$,
hence we also have
${\rm Ad}_{m}(\varphi)=(m{\rm Ad}_{\tilde\beta}(m)^{-1})\cdot
\varphi$. Moreover, since $m$ centralizes
$(T_{\phit})^{\rm Ad_{\tilde\beta},\circ}$, it also normalizes the
centralizer of this torus in $C_{\HG}(\phi)^{\circ}$, which is
$T_{\phi}$. So we see that
$$ {\rm Stab}(\varphi)=\left\{z\in Z,\, \exists \tilde z\mapsto z, \exists
  m\in \mathcal{M}^{\circ}\cap N_{C_{\HG}(\phit)}(T_{\phit}), \tilde z
  = m{\rm Ad}_{\tilde\beta}(m)^{-1}\right\}.$$ This certainly contains
 $$ \ker(\pi)=\left\{z\in Z,\, \exists \tilde
   z\mapsto z, \exists t\in T_{\phit}, \tilde z = t{\rm
     Ad}_{\tilde\beta}(t)^{-1}\right\}.$$ We have seen in the last
 proof that $\mathcal{M}^{\circ}\cap C_{\HG}(\phi)^{\circ}=T_{\phi}$,
 therefore, when $C_{\HG}(\phi)$ is connected, we have
 $\ker(\pi)={\rm Stab}(\varphi)$.  In general, we have
 \begin{eqnarray*} {\rm Stab}(\varphi)/\ker(\pi)&\simeq& \left\{
                                                         \begin{array}[c]{ll}
                                                           t\in (T_{\phi})_{\Ad_{\beta}}, & \exists \tilde
                                                                                            z \in \tilde Z \mapsto t, \exists m\in \mathcal{M}^{\circ}\cap N_{C_{\HG}(\phit)}(T_{\phit})
                                                           \\ & \tilde z=   m{\rm  Ad}_{\tilde\beta}(m)^{-1}
                                                         \end{array}
                                                                \right\}
   \\    &=&
             \left\{
             \begin{array}[c]{ll}
               t\in (T_{\phi})_{\Ad_{\beta}}, & \exists \tilde
                                                t \in T_{\phi} \mapsto t, \exists m\in \mathcal{M}^{\circ}\cap N_{C_{\HG}(\phit)}(T_{\phit})
               \\ & \tilde t=   m{\rm  Ad}_{\tilde\beta}(m)^{-1}
             \end{array}
                    \right\}.
 \end{eqnarray*}
 To see the last equality, start with $t$ in the last set and pick
 $(\tilde t,m)$ with $\tilde t\mapsto t$ and
 $\tilde t= m{\rm Ad}_{\tilde\beta}(m)^{-1}$. By surjectivity of
 (\ref{eq:isogeny}), there is $s\in T_{\phi}$ such that
 $s\tilde t \Ad_{\beta}(s)^{-1}=:\tilde z\in\tilde Z$. Then we have
 $\tilde z\mapsto t$ and $\tilde z = (sm)\Ad_{\beta}(sm)^{-1}$ with
 $sm\in \mathcal{M}^{\circ}\cap N_{C_{\HG}(\phit)}(T_{\phit})$.

 Now recall that $\Omega_{\phi}^{\Ad_{\beta}}=N_{\beta}/T_{\phi}$
 where
$$ N_{\beta}=\left\{n\in N_{C_{\HG}(\phi)}(T_{\phi}), n\Ad_{\tilde\beta}(n)^{-1}\in T_{\phi}\right\}.$$
From the description above, we see that
${\rm Stab}(\varphi)/\ker(\pi)$ is the image of the map
$N_{\beta}\cap \mathcal{M}^{\circ}\To{}(T_{\phi})_{\Ad_{\beta}}$ given
by $n\mapsto n\Ad_{\tilde\beta}(n)^{-1}$, and that this map factors
over the subgroup $K:=(N_{\beta}\cap\mathcal{M}^{\circ})/T_{\phi}$ of
$(\Omega_{\phi})^{\Ad_{\beta}}$.  Moreover, the action of
$n\in N_{\beta}\cap\mathcal{M}^{\circ}$ on $(T_{\phi})_{\Ad_{\beta}}$
is given by
$t\mapsto nt\Ad_{\beta}(n)^{-1}=ntn^{-1}(n\Ad_{\beta}(n)^{-1}) =
t(n\Ad_{\beta}(n)^{-1})$ because $n$ centralizes
$(T_{\phi})^{\Ad_{\beta},\circ}$. So this action factors through the
above morphism.

On the other hand, Haines' definition of
$W_{\varphi,\mathcal{M}^{\circ}}$ (denoted
$W^{\HG}_{[\lambda]_{\mathcal{M}^{\circ}}}$ there) is of the form
$W_{\varphi,\mathcal{M}^{\circ}}= N/\mathcal{M}^{\circ}$ with
$$ N
=\left\{n\in N_{\HG}(\mathcal{M}),\exists m\in \mathcal{M}^{\circ},
  \exists \tilde z\in \tilde Z, {\rm Ad}_{n}(\varphi)={\Ad_{m}}(\tilde
  z\cdot\varphi)\right\}.$$
We claim that $N_{\beta}\subset N$.  Indeed, note first that the
conjugation action of an element $n\in N_{\beta}$ on $T_{\phi}$
commutes with $\Ad_{\beta}$, so $n$ normalizes
$T_{\phi}^{\Ad_{\beta},\circ}$ hence it normalizes also
$\mathcal{M}$. Moreover, writing
$\tilde t:=n\Ad_{\tilde\beta}(n)^{-1}\in T_{\phi}$, we have
${\rm Ad}_{n}(\varphi)=\tilde t\cdot\varphi$. Finally, since
$\tilde Z$ surjects onto $(T_{\phi})_{\Ad_{\beta}}$, there are
$\tilde z\in \tilde Z$ and $m\in T_{\phi}$ such that
$\tilde t= m \Ad_{\beta}(m)^{-1}\tilde z$, hence also
$\tilde t\cdot\varphi= \Ad_{m}(\tilde z\cdot\varphi)$.  So we have
$N_{\beta}\subset N$, and since $T_{\phi}\subset\mathcal{M}^{\circ}$,
we get a map
\begin{equation}
  (\Omega_{\phi})^{\Ad_{\beta}} \To{} W_{\varphi,\mathcal{M}^{\circ}}.\label{eq:omegaW}
\end{equation}
We now claim that this map is surjective. Indeed, let $n\in N$ and
pick $\tilde z$ and $m$ such that
${\rm Ad}_{n}(\varphi)={\Ad_{m}}(\tilde z\cdot\varphi).$ The element
$n':=m^{-1}n$ has the same image as $n$ in
$W_{\varphi,\mathcal{M}^{\circ}}$, and we have
${\rm Ad}_{n'}(\varphi)=\tilde z\cdot\varphi$, hence also
${\rm Ad}_{n'}(\phi)=\phi$, i.e. $n'\in C_{\HG}(\phi)$. Moreover, $n'$
normalizes $Z(\mathcal{M})^{\circ}=(T_{\phi})^{\Ad_{\beta},\circ}$
(see the proof of the last lemma above), hence it normalizes the
connected centralizer of $(T_{\phi})^{\Ad_{\beta},\circ}$ in $C_{\HG}(\phi)$,
which is $T_{\phi}$. Hence we see that $n'\in N_{\beta}$ and we get
the surjectivity of (\ref{eq:omegaW}). We also see that
 $$ \ker ((\Omega_{\phi})^{\Ad_{\beta}} \To{} W_{\varphi,\mathcal{M}^{\circ}})
 =\im(N_{\beta}\cap \mathcal{M}^{\circ}\To{}
 (\Omega_{\phi})^{\Ad_{\beta}})=K.$$ In particular, when
 $C_{\HG}(\phi)$ is connected, we have
 $N_{\beta}\cap \mathcal{M}^{\circ}\subset C_{\HG}(\phi)^{\circ}\cap
 \mathcal{M}^{\circ}=T_{\phi}$, so the map (\ref{eq:omegaW}) is
 bijective in this case.
\end{proof}

\appendix
\section{Moduli of cocycles}

\subsection{Schemes of  cocycles}
  Let $H$ be an affine  group scheme over a noetherian ring $R$ and let $\Gamma$ be a finite
  group. Consider the functor $\underline\Hom(\Gamma,H)$, which to any $R$-algebra
$R'$ associates the set of homomorphisms $\Hom(\Gamma,H(R'))$.  It is represented  by a
closed and finitely presented $R$-subscheme of the affine $R$-scheme $H^{(\Gamma)}$, since
it is the inverse image of the closed subscheme
  $\{1_{H}\}^{(\Gamma\times\Gamma)}$ of $H^{(\Gamma\times\Gamma)}$ by the $R$-morphism
  $H^{(\Gamma)}\To{} H^{(\Gamma\times\Gamma)}$ defined by $(h_{\gamma})_{\gamma\in
    \Gamma}\mapsto
  (h_{\gamma}h_{\gamma'}h_{\gamma\gamma'}^{-1})_{(\gamma,\gamma')\in\Gamma\times\Gamma}$.

The group scheme $H$ acts by conjugation on $\underline\Hom(\Gamma,H)$.
 Given an $R$-algebra $R'$ and an homomorphism
$\phi\in \Hom(\Gamma,H(R'))$, the orbit maps $g\mapsto {\rm Ad}_{g}\circ \phi$,
$H(R'')\To{}\Hom(\Gamma,H(R''))$  define an $R'$-morphism
$H_{R'}\To{}\underline\Hom(\Gamma,H)_{R'}$ of finite presentation,
that we call an orbit morphism (here $R''$ runs over $R'$-algebras). The fiber over any other homomorphism $\phi'\in\Hom(\Gamma,H(R'))$ of this morphism
is the transporter $T_{H}(\phi,\phi')$ of $\phi$, which to any $R''$
over $R'$
associates the set-theoretic transporter from $\phi$ to $\phi'$ in $H(R'')$.

\begin{lemma}\label{hom_smooth}
  Assume that $H$ is smooth and that  $\Gamma$ has order invertible in $R$. 
Then $\underline\Hom(\Gamma,H)$ is smooth over $R$, all the orbit
morphisms are smooth and all transporters are smooth.
\end{lemma}
\begin{proof}
  By finite presentation, to
  prove smoothness it suffices to prove formal smoothness. Let $R'$ be
  an $R$-algebra and let $I$ be an ideal of $R'$
    of square $0$. We need to show that the map
    $\Hom(\Gamma,H(R'))\To{}\Hom(\Gamma,H(R'/I))$ is surjective. So let
    $\phi_{0}:\Gamma\To{}H(R'/I)$ be a group homomorphism. By smoothness of $H$ we may lift
    $\phi_{0}$ to a map $h:\Gamma\To{}H(R')$. 
Consider the map     $\Gamma\times\Gamma\To{} \ker(H(R')\To{}H(R'/I))$ that takes
$(\gamma,\gamma')\in \Gamma\times\Gamma$ to 
    $h(\gamma)h(\gamma')h(\gamma\gamma')^{-1}$. Note that conjugation by $h(\gamma)$
    endows the abelian group $\ker(H(R')\To{}H(R'/I))$ with an action of $\Gamma$ that actually
    only depends on $\phi_{0}$. In fact, if we identify $\ker(H(R')\To{}H(R'/I))$
    with the $R'/I$-module ${\rm Lie}(H)\otimes_{R}I$ then this action is induced by the
    $R'/I$-linear action of $\Gamma$ on ${\rm Lie}(H)\otimes_{R}R'/I$
 given by the adjoint representation composed with the homomorphism $\phi_{0}$. Now the map
 defined above is a $2$-cocycle, hence since
    $|\Gamma|$ is invertible in $R$, it has to be  cohomologically trivial, so there
    is a map $k :\, \Gamma\To{} \ker(H(R')\To{}H(R'/I))$ such that
    $h(\gamma)h(\gamma')h(\gamma\gamma')^{-1}=k(\gamma)({^{h(\gamma)}k}(\gamma'))k(\gamma\gamma')^{-1}$. Then
    the map $\gamma\mapsto \phi(\gamma):=k(\gamma)^{-1}h(\gamma)$ is a group homomorphism $\phi:\,\Gamma\To{}
    H(R')$ that lifts $\phi_{0}$, and the smoothness of $\underline\Hom(\Gamma,H)$ follows.

Now fix a homomorphism $\phi:\Gamma\To{}H(R)$
and let us show that the corresponding orbit morphism is smooth, by
using the infinitesimal criterion. Let again  $R'$ be an $R$-algebra together with an
ideal $I$ of square $0$, and let $\phi'$ be another homomorphism $\Gamma\To{} H(R')$ 
whose image $\phi'_{0}$ in $\Hom(\Gamma,H(R'/I))$ is conjugate to $\phi_{0}$ by an element 
$h_{0}$ in $H(R'/I)$. We must find an element $h\in H(R')$ that conjugates $\phi'$ to $\phi$. By
smoothness of $H$ we can pick an element $h'\in H(R)$ that maps to $h_{0}$.
Then the map $\gamma\mapsto  \phi(\gamma) (h'\phi'(\gamma)^{-1} {h'}^{-1})$
defines a $1$-cocycle of $\Gamma$ in  $\ker(H(R')\To{}H(R'/I))$ endowed with the action associated with
$\phi_{0}$ as above. By the same argument as above, this cocycle is a coboundary, so there is
some $k\in \ker(H(R')\To{}H(R'/I))$ such that $\phi(\gamma) h'\phi'(\gamma)^{-1} {h'}^{-1}=  
(\phi(\gamma)k \phi(\gamma)^{-1})k^{-1}$, from which we get an element $h=k^{-1}h'$ as desired.
Hence the orbit morphism is smooth. By base change,
the
centralizers and the transporters are therefore smooth too.
\end{proof}

Suppose now that we are given an action of $\Gamma$ on $H$ by automorphisms of group
schemes over $R$. Identifying $1$-cocycles $\Gamma\To{}H(R')$ with cross-section
homomorphisms $\Gamma\To{}H\rtimes\Gamma$ (i.e. homomorphisms whose
composition with the projection to $\Gamma$ is the identity), we see that the functor $R'\mapsto
Z^{1}(\Gamma,H(R'))$ is represented by an $R$-scheme that is a direct summand of
$\underline\Hom(\Gamma,H\rtimes \Gamma)$. We denote this scheme by $\uZ^{1}(\Gamma,H)$.
It is stable under the conjugation action of $H\rtimes \Gamma$ \emph{restricted} to $H$. 
 
When $H$ is smooth, so is $H\rtimes \Gamma$, hence the above lemma implies :
\begin{cor}
  Assume that $H$ is smooth and that  $\Gamma$ has order invertible in $R$. 
Then $\uZ^{1}(\Gamma,H)$ is smooth over $R$, all the $H$-orbit morphisms are smooth and all transporters are smooth.
\end{cor}

\subsection{The sheafy quotient} 
We henceforth assume that \emph{$H$ is smooth and $\Gamma$ has order
invertible in $R$}, and we are now  interested
in the quotient object $\underline{H}^{1}(\Gamma,H)$ of
$\uZ^{1}(\Gamma,H)$ by the conjugation action of
$H$. As for now, we define it as the quotient sheaf, say for the \'etale
topology, that is, the sheaf associated to $R\mapsto H^{1}(\Gamma,H(R))$.

\begin{cor}\label{cor_h1}
  Assume 
  that $R$ is a local 
  Henselian ring, and denote by $k$ its residue field. Then the map
  $H^{1}(\Gamma,H(R))\To{}H^{1}(\Gamma,H(k))$ is a bijection.
\end{cor}
\begin{proof}
By smoothness of $\uZ^{1}(\Gamma,H)$ and \cite[Thm 18.5.17]{EGAIV4}, any $k$-point of
$\uZ^{1}(\Gamma,H)$ extends to a section over $R$, that is, the
map $Z^{1}(\Gamma,H(R))\To{}Z^{1}(\Gamma,H(k))$ is surjective. Hence the
map of the lemma is surjective too. To prove injectivity, let $\phi$,
$\phi':\,\Gamma\To{} H(R)$ be two $1$-cocycles whose images
$\phi_{0}$, $\phi'_{0}$ are $H$-conjugate in $\Hom(\Gamma,H(k)\rtimes\Gamma)$ by some
$h_{0}\in H(k)$. By the
previous lemma, the transporter scheme $T_{H}(\phi,\phi')$ is smooth
over $R$. Hence, by \cite[Thm 18.5.17]{EGAIV4} again, its $k$-point $h_{0}$  extends to an
$R$-section $h$ that conjugates $\phi$ to $\phi'$.
\end{proof}

\begin{lemma}\label{lemma_h1}
  Assume that $H$ is reductive, that $R$ is  a  strictly
  Henselian local ring, and denote by $R'$ any non-zero $R$-algebra.

(i) the map   $H^{1}(\Gamma,H(R))\To{} H^{1}(\Gamma,H(R'))$ is
injective.

(ii) it is surjective if $R$ is a d.v.r. or a field and $R'$ is local  strictly Henselian.
\end{lemma}
\begin{proof} We adapt the proof of Thm 4.8 of \cite{BHKT}. 

(i) 
We need to prove that if two cocycles $\phi$, $\phi'$ in $Z^{1}(\Gamma,H(R))$ get
$H(R')$-conjugate in $\Hom(\Gamma,H(R')\rtimes\Gamma)$, then they are $H(R)$-conjugate. By the last
corollary, it suffices to prove that their images $\phi_{0}$,
$\phi'_{0}\in Z^{1}(\Gamma,H(k))$ are $H(k)$-conjugate.
We will need V. Lafforgue's theory of pseudocharacters for the group $H\rtimes \Gamma$.
This notion is introduced without name nor
formal definition  in the preamble of Proposition 11.7 of \cite{Lafforgue}. A formal
definition is given in \cite[Def 4.1]{BHKT} where the name
``pseudocharacter'' is also introduced. Unfortunately, unlike Lafforgue, these authors restrict attention to connected (split reductive)
groups. However,
one has merely to replace $\ZM[\hat G^{n}]^{\hat G}$ by
$\ZM[(H\rtimes\Gamma)^{n}]^{H}$ in \cite[Def 4.1]{BHKT} to get the correct definition for
the non-connected group $H\rtimes\Gamma$ (note that $H$ is a split
reductive group over $R$, since $R$ is strictly Henselian). Then, as in
\cite[Lemma 4.3]{BHKT}, it follows from the definition
that any homomorphism $\phi:\Gamma\To{}H(R)\rtimes\Gamma$ defines a
 ``$H\rtimes\Gamma$-pseudocharacter of $\Gamma$ over $R$'' denoted by
 $\Theta_{\phi}$. Moreover, if $\phi$, $\phi'$ in $Z^{1}(\Gamma,H(R))$ become $H(R')$-conjugate in
$\Hom(\Gamma,H(R')\rtimes\Gamma)$, then $\Theta_{\phi}\equiv\Theta_{\phi'}
[{\rm mod } I]$ where $I=\ker(R\To{} R')$ (as in lemmas 4.3 and 4.4.i of \cite{BHKT}).
Therefore we get  $\Theta_{\phi_{0}}=\Theta_{\phi'_{0}}$. 
Then, the main result on pseudocharacters asserts 
that \emph{the semi-simplifications} of $\phi_{0}$ and
$\phi'_{0}$ are conjugate under $H(k)$. Here, the notion of semi-simplicity is the
notion of $H\rtimes\Gamma$-complete reducibility of \cite[\S 6]{bmr05}. We note actually that this result is
proven in \cite[Prop. 11.7]{Lafforgue} when $k$ has characteristic $0$ and in \cite[Thm
4.5]{BHKT} in any characteristic, but in the \emph{connected} case. We leave it to the reader to
convince themselves that their argument can be adapted to the non-connected case in any characteristic.

It suffices now to show that
$\phi_{0}$ and $\phi'_{0}$ are actually $H\rtimes\Gamma$-completely reducible.
Choose an $R$-parabolic subgroup $P$ of $H\rtimes \Gamma$ containing
$\phi_{0}(\Gamma)$ and minimal for this property. Let $P\To{\pi}L_{P}$ be its
Levi quotient and let $L_{P}\To{\iota} P$ be a Levi section of $P$. Then
$\phi_{0}^{ss}:=\iota\circ\pi\circ\phi_{0}$ is by definition a
semisimplification of $\phi_{0}$. If we denote by $U_{P}$ the
unipotent radical of $P$, the map $\Gamma\To{} U_{P}$, $\gamma\mapsto
\phi^{ss}_{0}(\gamma)\phi_{0}(\gamma)^{-1}$ is a $1$-cocycle for the
action of $\Gamma$ by conjugation on $U_{P}$ through $\phi_{0}$. The
descending central series of $U_{P}$ is a $\Gamma$-stable descending
filtration of $U_{P}$ by smooth unipotent subgroup schemes whose
successive quotients are $k$-vector space schemes. Since  
 $|\Gamma|$ is invertible in
$k$, we know that  $H^{1}(\Gamma,V)$ is trivial for any
$k\Gamma$-module $V$. Therefore the above $1$-cocycle is a coboundary, and we can find some
$u\in U_{P}$ such that $\phi^{ss}_{0}(\gamma)\phi_{0}(\gamma)^{-1}=
u^{-1} \phi_{0}(\gamma) u \phi_{0}(\gamma)^{-1}$. So $u^{-1}$
conjugates $\phi_{0}$ to $\phi_{0}^{ss}$ and $\phi_{0}$ is semisimple (ie
$H\rtimes\Gamma$-completely reducible)
as claimed.

(ii) Denote by $k'$ the residue field of $R'$ and by $\bar K$ an
algebraic closure of the fraction
field of $R$. In the case where $R$ is a d.v.r, either the composition
$R\To{}k'$ factors as $R\To{} k\To{} k'$ or as $R\To{} \bar K \To{}
k'$. Applying the last corollary to both $R$ and $R'$, we are thus reduced to
showing the special case (a) of statement (ii) where $R'=\bar K$, and its variant
(b) where $R$ and $R'$ are algebraically closed fields.

(a) The case where $R'=\bar K$ is an algebraic closure of the fraction field
$K$ of $R$.
This case will follow from the following facts of Bruhat-Tits theory : 
\begin{enumerate}
\item[BT1.]  any vertex of the semi-simple building $B(H,K)$ and, more
  generally, the barycenter of any facet of $B(H,K)$ becomes a
  hyperspecial point in $B(H,K')$ for a suitable finite extension $K'$
  of $K$
\item[BT2.]  two hyperspecial points in $B(H,K')$ become
  $H(K'')$-conjugate in $B(H,K'')$ for some further finite extension
  $K''$.
\end{enumerate}
Note that, here,
$H$ is split over $K$ (since $R$ is strictly henselian), so these
facts are quite elementary, even in our setting where the discretely
valued field $K$ is Henselian but not necessarily
complete. For example, BT2 follows from
Corollary 7.11.5 of the forthcoming book \cite{kaletha-prasad}.
As for BT1, here is a sketch of the argument.  Choose a splitting
$(B,T,X)$ of $H$ over $R$ 
and denote by $A$ the appartment of $B(H,K)$ associated to $T$. It
contains the hyperspecial point $o$ corresponding to the integral model $H$ over
$R$. The pinning defines a ``Chevalley valuation'' of the root system
of $T$ in $H$, and then an affine root system on $A$. By \cite[Prop
6.4.1]{kaletha-prasad}, we know that,  taking $o$ as an origin of the $\mathbb R$-affine space
$A$, the affine roots on $A$ are translates of ordinary roots by
integers. It then follows that a point $y\in A$ is (hyper)special if
$\alpha(y)\in v(K^{\times})$ for all roots $\alpha$, and where $v$ is
the valuation of $K$. 
Now, let $x$ be the barycenter of some facet of $B(H,K)$.
After translating $x$  by  an
element of $G(K)$ we may assume that $x$ lies in $A$. Being the
barycenter of a facet, there is an integer $N$ such that $\alpha(x)\in
\frac 1N v(K^{\times})$ for all roots $\alpha$. So, if $K'$ is any
extension whose ramification index is a multiple of $N$, then $x$
becomes (hyper)special in $B(H,K')$.

Observe also that $\Gamma$ acts on the building $B(H,K)$ and fixes 
the hyperspecial point $o$. 
Now let  $\phi\in Z^{1}(\Gamma,H(\bar K))$. Then $\phi$ belongs to $Z^{1}(\Gamma,H(K_{1}))$
for some finite extension $K_{1}$ of $K$.
Pick a point $x$ of $B(H,K_{1})$ fixed by $^{L}\phi(\Gamma)\subset
H(K_{1})\rtimes \Gamma$. Up to replacing $x$ by the barycenter of the
facet that contains $x$, we may assume $x$ is the barycenter of this facet. 
So it becomes hyperspecial over some finite extension $K_{2}$ of $K_{1}$ and we may even assume that
there is some $h\in H(K_{2})$  such that $hx=o$. Then ${^{h}({^{L}\phi})}(\Gamma)$ fixes $o$ so,
writing $^{h}({^{L}\phi})(\gamma)=(^{h}\phi(\gamma),\gamma)\in H(K_{2})\rtimes \Gamma$, we see that
$^{h}\phi(\gamma)$ fixes $o$ hence belongs to $H(K_{2})_{o}=H(R_{2})Z_{H}(K_{2})$ for all $\gamma\in \Gamma$,
\emph{i.e.} $^{h}\phi\in Z^{1}(\Gamma,H(R_{2})Z_{H}(K_{2}))$, where $R_{2}$ is the normalization of $R$ in
$K_{2}$. Now, note that $H^{1}(\Gamma, Z_{H}(K_{2})/Z_{H}(R_{2}))$ may not be trivial, but 
maps trivially in $H^{1}(\Gamma, Z_{H}(K_{3})/Z_{H}(R_{3}))$ for any further finite extension
$K_{3}$ such that $|\Gamma|$ divides the exponent of $Z_{H}(K_{3})/Z_{H}(R_{3})Z_{H}(K_{2})$. This means that
there is $z\in Z_{H}(K_{3})$ such that  $^{zh}\phi\in Z^{1}(\Gamma,H(R_{3}))$.
But $R_{3}$ is an Henselian local $R$-algebra with the same residue field as $R$, so by the
previous corollary there is $h'\in H(R_{3})$ such that  $^{h'zh}\phi \in
Z^{1}(\Gamma,H(R))$. So the class $[\phi]$ in $H^{1}(\Gamma,H(\bar K))$ is the image of
$[^{h'zh}\phi]\in H^{1}(\Gamma,H(R))$, as desired.
   
(b) The case where $R$ and $R'$ are algebraically closed fields.
This case can certainly be handled
via pseudocharacters. Namely, using \cite[Thm 4.5]{BHKT} and the fact
that all morphisms $\Gamma\To{}H(R)\rtimes\Gamma$ are $H\rtimes\Gamma$-semisimple (as proved above), we see that
it suffices to prove that any $H\rtimes\Gamma$-pseudocharacter of $\Gamma$ over $R'$ is actually $R$-valued.
However, the result is true under the much more general  assumption that $H$ is smooth
over $R$.
 Indeed, since the orbit morphisms are smooth, the $H(R)$-orbits in
$Z^{1}(\Gamma,H(R))$ are open for the Zariski topology. Since two orbits are either equal
or disjoint, there are only finitely many of them. Let $\phi_{1},\cdots, \phi_{n}$ be
representatives. The orbit morphisms yield  a smooth surjective morphism
$(\sqcup_{i=1}^{n} H) \To{}\uZ^{1}(\Gamma,H)$ which induces in turn a surjection on $R'$-points 
$(\sqcup_{i=1}^{n} H(R')) \To{}\uZ^{1}(\Gamma,H)(R')$ since $R'$ is algebraically closed.
So we see that each $H(R')$-orbit in $Z^{1}(\Gamma,H(R'))$ comes from an $H(R)$-orbit in  $Z^{1}(\Gamma,H(R))$.
\end{proof}


Recall now the \'etale sheafification $\underline H^{1}(\Gamma,H)$ of the
functor $R'\mapsto H^{1}(\Gamma,H(R'))$ on $R$-algebras. Here we consider the ``big''
site of affine schemes of finite presentation over $R$ with the \'etale topology.
The maps $H^{1}(\Gamma,H(R))\To{}H^{1}(\Gamma,H(R'))$ define a morphism from
the constant presheaf associated to the set $H^{1}(\Gamma,H(R))$ to
the presheaf $R'\mapsto H^{1}(\Gamma,H(R'))$. It induces in turn a
morphism of sheaves
$$ \underline{H^{1}(\Gamma,H(R))} \To{} \underline H^{1}(\Gamma,H)$$
where the left hand side is a ``constant'' sheaf.

\begin{prop}
  Suppose that $H$ is reductive over a strictly Henselian discrete
  valuation ring $R$ in which the order of $\Gamma$ is invertible. 
  Then the above morphism of sheaves is an isomorphism. In particular,
  $\underline H^{1}(\Gamma,H)$ is representable by a product  of
  finitely many copies of $R$.
\end{prop}
\begin{proof}
  We first note that the functor $R'\mapsto H^{1}(\Gamma,H(R'))$ defined over
  all $R$-algebras commutes with filtered colimits. Indeed, this property is certainly
  true for the functors $R'\mapsto Z^{1}(\Gamma,H(R'))$ and $R'\mapsto H(R')$ since both
  these functors are represented by finitely presented $R$-algebras. Elementary
  formal nonsense shows that this property holds in turn for the quotient functor $R'\mapsto
  H^{1}(\Gamma,H(R'))$.


  Therefore, if $A$ is any $R$-algebra and $x$ is a geometric point of $\Spec(A)$ then,
  writing $A_{x}^{sh}$ for the strict henselization of $A$ at $x$, the set
  $H^{1}(\Gamma,H(A_{x}^{sh}))$ is the stalk of the sheaf $\underline{H}^{1}(\Gamma,H)$ at
  $x$. 
  So by the last lemma, the map
  $H^{1}(\Gamma,H(R))\To{}H^{1}(\Gamma,H(A_{x}^{sh}))$ is bijective. This means that the
  morphism of sheaves under consideration is an isomorphism on
  stalks. Thus it is an isomorphism.

  It remains to justify the finiteness of the set
  $H^{1}(\Gamma,H(R))$. But it follows from Corollary \ref{cor_h1} and
  the last paragraph of the proof of Lemma \ref{lemma_h1}.
\end{proof}

\begin{rem}
  Here is a concrete paraphrase of the proposition. First note that
  the map $Z^{1}(\Gamma,H(R))\To{} \underline H^{1}(\Gamma,H)(R)$ is
  surjective since $R$ is strictly Henselian, so that we can pick a
  finite subset $\Phi_{0}\subset  Z^{1}(\Gamma,H(R))$ mapping bijectively
  to $\uH^{1}(\Gamma,H)(R)$. Now, suppose that $A$ is an integral finitely
  presented $R$-algebra and let $\phi$ be a $1$-cocycle $\Gamma\To{} H(A)$. Then there is
  a unique cocycle $\phi_{0}\in \Phi_{0}$ and a faithfully \'etale map $A\To{} A'$
  such that $\phi$ is $H(A')$-conjugate to the ``constant'' cocycle $\phi_{0}$.
\end{rem}

We now globalize a bit the previous proposition.

\begin{thm} \label{h1_finite_etale}
Suppose that $H$ is reductive over a Dedekind $G$-ring
   $R$ in which the order of $\Gamma$ is invertible. 
  Then 
  $\underline H^{1}(\Gamma,H)$ is representable by a finite \'etale
  $R$-algebra.  
\end{thm}
\begin{proof}
  Let $\bar K$ be an algebraic closure of the fraction field $K$ of
  $R$. For a closed point $s$ of $\Spec(R)$, denote by $R_{s}^{sh}$  a strict henselization
  of $R$ at $s$ (depending on a choice of geometric point over $s$) and by $\bar K_{s}$ an algebraic closure of its fraction field. Let us 
  choose a set of representatives $\Phi_{s}\subset
  Z^{1}(\Gamma,H(R_{s}^{sh}))$ of $H^{1}(\Gamma,H(R_{s}^{sh}))$. Since
  $\uZ^{1}(\Gamma,H)$ is finitely presented, these representatives are defined over
  some \'etale $R$-domain $R'$, so that $\Phi_{s}$ comes from a subset $\Phi_{R'}\subset Z^{1}(\Gamma,H(R'))$.
  Now if $s'$ is another closed point of $\Spec(R)$ in the
  image of $\Spec(R')$ and if we choose an $R$-morphism $R'\To{} R_{s'}^{sh}$, then we
  claim that the natural map $\Phi_{R'}\to H^{1}(\Gamma,H(R_{s'}^{sh}))$ is also a bijection.
  Indeed, this follows from the  following commutative diagram
  $$\xymatrix{
    & H^{1}(\Gamma,H(R_{s}^{sh})) \ar[r]^{\sim} & H^{1}(\Gamma,H(\bar K_{s})) &  \\
    \Phi_{R'} \ar[ru] \ar[rd] \ar[rrr] & & & H^{1}(\Gamma,H(\bar K)) \ar[lu]_{\sim} \ar[ld]_{\sim} \\
    & H^{1}(\Gamma,H(R_{s'}^{sh})) \ar[r]^{\sim} & H^{1}(\Gamma,H(\bar K_{s'})) & 
  }$$
  where we have chosen an $R$-embedding $R'\hookrightarrow \bar K$ and two $R'$-embeddings $\bar
  K \hookrightarrow \bar K_{s}$ and $\bar K \hookrightarrow \bar K_{s'}$, and where the
  $\sim$ denote bijections granted by Lemma \ref{lemma_h1}.
  As a consequence, denoting by $\underline\Phi_{R'}$ the constant sheaf on $R'$-algebras
  associated to the set $\Phi_{R'}$, we see as in the last proof that the natural morphism of
  sheaves $\underline\Phi_{R'}\To{}\underline H^{1}(\Gamma,H)_{R'}$ is an isomorphism.

  Now, varying the point $s$ and using the quasicompacity of $\Spec(R)$ we get a
  faithfully \'etale morphism $R\hookrightarrow R''=R'_{1}\times\cdots\times R'_{n}$ and a
  set $\Phi_{R''}=\Phi_{R'_{1}}\times\cdots\times \Phi_{R'_{n}}$ such that the natural
  morphism of sheaves on $R''$-algebras
  $\underline\Phi_{R''}\To{}\underline H^{1}(\Gamma,H)_{R''}$ is an isomorphism. In
  particular, the sheaf $\underline H^{1}(\Gamma,H)$ is representable after base change to
  $R''$ by a sum of copies of $R''$. Since the map of $(\Spec R)_{\et}$-sheaves
  $\underline H^{1}(\Gamma,H)\times_{\Spec R}\Spec R'' \To{} \underline H^{1}(\Gamma,H)$ is
  visibly representable, \'etale and surjective, it follows that $\underline
  H^{1}(\Gamma,H)$ is an algebraic space over $(\Spec R)_{\et}$. This algebraic space has
  to be  finite \'etale (and in particular separated) over $R$ since it is so after base
  change to $R''$. Hence by Corollary II.6.17 of \cite{MR0302647}, this algebraic space is actually a
  scheme, and it is finite \'etale over $R$.
\end{proof}

\subsection{Relation with  the affine GIT quotient}
\def\OC{{\mathcal O}}
Let us investigate the relationship between $\underline
H^{1}(\Gamma,H)$ and another natural quotient
of $\uZ^{1}(\Gamma,H)$ by $H$. Namely, denote by $\OC$ the
$R$-algebra such that $\uZ^{1}(\Gamma,H)=\Spec(\OC)$. The action of $H$ on
$\uZ^{1}(\Gamma,H)$ translates into a comodule structure
$\OC\To{\rho}\OC\otimes_{R}R[H]$ on $\OC$ under the Hopf $R$-algebra $R[H]$ corresponding to
$H$. As usual, put
$$\OC^{H}:= \ker( \rho-\id\otimes\varepsilon)$$
where $\varepsilon$ is the unit of $R[H]$. Then the morphism
$\Spec(\OC)\To{}\Spec(\OC^H)$ is a categorical quotient of $\uZ^1(\Gamma,H)$ by
$H$ in the category of \emph{affine} $R$-schemes.

Note that $\underline H^1 (\Gamma,H)$ is a categorical quotient in the
much larger category of sheaves on 
the big \'etale site of $\Spec(R)$.
However, under suitable assumptions, Theorem \ref{h1_finite_etale}
shows that it is actually represented by an affine 
$R$-scheme. So, by  uniqueness of categorical quotients, we conclude that up to a unique
isomorphism, we have
$$ \underline H^1(\Gamma,H) = \Spec(\OC^H),$$
which we summarize in the following corollary.

\begin{cor} \label{h1_and_git}
  Suppose that $H$ is reductive over a Dedekind $G$-ring $R$ in which the order of $\Gamma$ is
  invertible. Then $\OC^{H}$ is a finite \'etale $R$-algebra and represents the sheaf
  $\underline H^{1}(\Gamma,H)$. In particular, its formation commutes
  with any change of rings $R\To{} R'$. 
\end{cor}

\subsection{Representatives} 


Suppose that $H$ is reductive over a Dedekind $G$-ring
   $R$ in which the order of $\Gamma$ is invertible. 
Theorem \ref{h1_finite_etale} ensures that after replacing $R$ by a finite \'etale
extension, $\underline H^{1}(\Gamma,H)$ is a constant sheaf (associated
to the set $\underline H^{1}(\Gamma,H)(R)$). 
The map $Z^{1}(\Gamma,H(R))\To{} \underline H^{1}(\Gamma,H)(R)$ need not be
surjective, but  if
$R_{0}$ is any $R$-algebra such that $\underline
H^{1}(\Gamma,H)(R)$ is in the image of the map
$Z^{1}(\Gamma,H(R_{0}))\To{}\underline H^{1}(\Gamma,H)(R_{0})$, then
 for any
finite set  $\Phi_{0}\subset Z^{1}(\Gamma,H(R_{0}))$ mapping
bijectively to  $\underline H^{1}(\Gamma,H)(R)$, the constant sheaf
property ensures that :
\emph{for any  connected $R_{0}$-algebra $A$ and 
  any $\phi\in Z^{1}(\Gamma,H(A))$, there is a unique $\phi_{0}\in \Phi_{0}$ such that $\phi$
 and $\phi_{0}$ become $H(A')$-conjugate in $Z^{1}(\Gamma,H(A'))$ for some faithfully \'etale
 $A$-algebra $A'$.}

By definition of $\underline H^{1}(\Gamma,H)$, we certainly can find a
$R_{0}$ as above that is faithfully \'etale over $R$.
However in general, it is not clear whether we can find
$R_{0}$ \emph{finite} \'etale over $R$. The following result uses the strong
approximation property to prove that, if $R$ is a localization of a
ring of integers in a number field, then we can at least find $R_{0}$ finite
(not necessarily \'etale) over $R$.

\begin{thm} \label{representatives}
Assume that $H$ is reductive over
a normal subring $R$ 
 of some number field $K$, and that $\Gamma$ has
invertible order in $R$. Then there is a finite extension $K_{0}$ of $K$
and a finite set $\Phi_{0}\subset Z^{1}(\Gamma,H(R_{0}))$ (with $R_{0}$
the normalization of $R$ in $K_{0}$)  such that for
any  connected $R_{0}$-algebra $A$ and  
  any $\phi\in Z^{1}(\Gamma,H(A))$, there is a unique $\phi_{0}\in \Phi_{0}$ such that $\phi$
 and $\phi_{0}$ becomes $H(A')$-conjugate in $Z^{1}(\Gamma,H(A'))$ for some faithfully \'etale
 $A$-algebra $A'$.
\end{thm}

\begin{proof} As we have just argued, we may assume that $\underline
  H^{1}(\Gamma,H)$ is a constant sheaf, and the problem boils down to
  finding $K_{0}$ such that the map 
$$ Z^{1}(\Gamma, H(R_{0}))\To{} \underline H^{1}(\Gamma,H)(R_{0})=\underline H^{1}(\Gamma,H)(R)$$ is
surjective. We certainly can find a faithfully \'etale $R'$ over $R$
such that any $[\phi]\in\uH^{1}(\Gamma,H)(R)$ has a representative
$\phi\in Z^{1}(\Gamma,H(R'))$. Let us choose such data, and assume
further that $H$ is split over $R'$.
Let $R'=\prod_{i=1}^{n}R'_{i}$ be the
decomposition of $R'$ in connected components and let
$\phi=(\phi_{i})_{i=1,\cdots, n}$ be the corresponding decomposition
of $\phi$. Replacing $R$ and all $R'_{i}$ by their respective normalizations in the residue field at
some generic point of $R'_{1}\otimes_{R}\cdots\otimes_{R}R'_{n}$, we may
assume that each $R'_{i}$ is a localization of $R$
(i.e. $\Spec(R')\To{}\Spec(R)$ is a Zariski cover). Then we simplify
the notation and write $R_{i}:=R'_{i}$. Since  all 
$\phi_{i}$ map to the same element $[\phi]\in H^{1}(\Gamma,H(K))$, 
they become pairwise $H$-conjugate over some  finite extension of $K$. Replacing $R$ by its
normalization in this finite extension, we may thus assume that they are
$H(K)$-conjugate in $Z^{1}(\Gamma,H(K))$. Actually we may, and we will,
even assume that they are pairwise $Z(H)^{\circ}(K)\times H_{\rm sc}(K)$-conjugate through the canonical isogeny
$Z(H)^{\circ}\times H_{\rm sc}\To{} H$, where 
$H_{\rm sc}$ denotes the simply connected covering group of the adjoint
group $H_{\rm ad}$.
We now try to construct a $\phi\in Z^{1}(\Gamma,H(K))$ that is
$H(K)$-conjugate to each $\phi_{i}$, and such that $\phi(\Gamma)\subset H(R)$.

\def\pG{{\mathfrak p}}
If $n=1$, we are obviously done. Otherwise, start with $\phi_{1}$ and pick elements
$(z_{i},h_{i})\in Z(H)^{\circ}(K)\times H_{\rm sc}(K)$ such that $^{z_{i} h_{i}}\phi_{1}=\phi_{i}$ in
$Z^{1}(\Gamma,H(K))$, for all $i=2,\cdots, n$. For any prime
$\pG\in S:=\Spec(R)\setminus \Spec(R_{1})$ there is some $i\geq 2$ such
that $\pG\in\Spec(R_{i})$. Pick such an $i$ and put $(z_{\pG},h_{\pG}):=(z_{i},h_{i})$.
Since $H_{\rm sc}$ is a split simply connected semisimple group over $K$, the
strong approximation theorem with respect to the finite set of archimedean places
ensures the existence of an element $h\in
H_{\rm sc}(R_{1})$ such that $h\in  H_{\rm sc}(R_{\pG}) h_{\pG}$ for
all $\pG\in S$. Then we have
$(^{h}\phi_{1})(\Gamma) \subset H(R_{1})$ and  
$(^{z_{\pG}h}\phi_{1})(\Gamma) \subset H(R_{\pG})$ for all $\pG\in S$. Now, since
$Z(H)^{\circ}$ is a split torus, say of dimension $d$, the
obstruction to finding $z\in Z(H)^{\circ}(R_{1})\cap\bigcap_{\pG}Z(H)^{\circ}(R_{\pG})z_{\pG}$
lies in the $d^{th}$ self-product of the ideal class group $\mathcal{C}\ell(K)^{d}$. Hence it vanishes over the Hilbert
class field $K^{h}$ of $K$ and we can at least  find
$z\in Z(H)^{\circ}(R_{1}^{h})\cap\bigcap_{\pG}Z(H)^{\circ}(R_{\pG}^{h})z_{\pG}$, where the
superscript $h$ denotes normalization in $K^{h}$.
Then we see that $^{zh}\phi_{1}(\Gamma)\subset H(R_{1}^{h})$ and
$^{zh}\phi_{1}(\Gamma)\subset H(R_{\pG}^{h})$ for all $\pG$.
Therefore we have $(^{zh}\phi_{1})(\Gamma) \subset H(R^{h})$ as desired.
\end{proof}

\begin{rem}[Orbits] \label{rem_orbits}
  With the notation of the theorem, the morphism
  $$\uZ^{1}(\Gamma,H)_{R_{0}}\To{\pi}\uH^{1}(\Gamma,H)_{R_{0}}=\{\pi(\phi),\phi\in \Phi_{0}\}$$
  provides a decomposition as a disjoint union of affine $R_{0}$-schemes
$$ \uZ^{1}(\Gamma,H)_{R_{0}}= \bigsqcup_{\phi\in \Phi_{0}} \pi^{-1}(\pi(\phi))
$$
Moreover, the action  $h\mapsto h\cdot\phi$ of $H_{R_{0}}$  provides a surjective morphism
of $R_{0}$-schemes 
$H_{R_{0}} \To{} \pi^{-1}(\pi(\phi))$, which at the level of \'etale
sheaves identifies  $ \pi^{-1}(\pi(\phi))$ with the quotient
$H_{R_{0}}/C_{H}(\phi)$ with $C_{H}(\phi)$ denoting the centralizer of $\phi$. In
particular, we see that this quotient sheaf is representable by an affine
scheme which identifies with the orbit $H\cdot\phi :=
\pi^{-1}(\pi(\phi))$ of $\phi$. To put it in different words, the
natural map $H\To{} H\cdot\phi$, $h\mapsto h\cdot\phi$ is a
$C_{H}(\phi)$-torsor for the \'etale topology.
\end{rem}

\subsection{Centralizers} Our next task is to study the centralizer $C_{H}(\phi)$ of
a cocycle $\phi\in Z^{1}(\Gamma,H(R))$. We have seen in Lemma
\ref{hom_smooth} that this is a smooth group scheme over
$R$. Moreover, by \cite[Thm 2.1]{PY}, its geometric fibers have
reductive neutral components. In other words, the ``neutral''
component $C_{H}(\phi)^{\circ}$ is a reductive group scheme over
$R$. Thus it follows from Prop 3.1.3 of \cite{conrad_luminy} that  the quotient sheaf
$\pi_{0}(C_{H}(\phi)):=C_{H}(\phi)/C_{H}(\phi)^{\circ}$ is
representable by a separated \'etale group scheme over $R$. Our aim
here is to prove that $\pi_{0}(C_{H}(\phi))$ is actually \emph{finite} over
$R$, at least when $R$ is a Dedekind $G$-ring and $\Gamma$ is a solvable group.


Note that $C_{H}(\phi)$ is also the subgroup of $\Gamma$-fixed points in $H$ for the
 ${\rm Ad}_{\phi}$-twisted action of $\Gamma$ on $H$. So, up to changing the action of
 $\Gamma$ on $H$, it suffices to study the finiteness of $\pi_{0}(H^\Gamma)$ as an $R$-scheme. 

 \begin{lemma} As above, assume $\Gamma$ has invertible order in $R$. 
   
i)  Let $H'\To{} H$ be a $\Gamma$-equivariant central isogeny of reductive group schemes
  over $R$. If $\pi_{0}(H^{\prime\Gamma})$ is finite over $R$, then so is $\pi_{0}(H^{\Gamma})$.

ii) Let $\Gamma'$ be a normal subgroup of $\Gamma$. If $\pi_{0}(H^{\Gamma'})$ and
$\pi_{0}((H^{\Gamma',\circ})^{\Gamma/\Gamma'})$ are finite, then so is $\pi_{0}(H^{\Gamma})$.
\end{lemma}
\begin{proof}
i)    Let $Z$ be the kernel of the isogeny, which is a finite central subgroup scheme of $H'$ of
    multiplicative type over $R$.  \emph{We claim that the sheaf $\underline
    H^{1}(\Gamma,Z)$ is representable by  a finite \'etale group scheme over $R$.} 
 Indeed, 
since the category of finite group schemes of multiplicative type over $R$ is abelian (\cite[IX.2.8]{MR0274459}), the sheaves $\underline Z^{1}(\Gamma,Z)$, $\underline B^{1}(\Gamma,Z)$ and, consequently, $\underline
H^{1}(\Gamma,Z)$ are finite group schemes of multiplicative type over $R$. Let us decompose
$Z=\prod_{p}Z_{p}$  into a finite product of its $p$-primary components. Then $\underline
H^{1}(\Gamma,Z)$ decomposes accordingly as a product of  $\underline
H^{1}(\Gamma,Z_{p})$. But $\underline H^{1}(\Gamma,Z_{p})$ is trivial unless $p$ divides
the order of $\Gamma$. Since this order is invertible in $R$, so is the rank of
$\underline H^{1}(\Gamma,Z)$, which is therefore \'etale over $R$.

Let us now look at  the following exact sequence of sheaves  of groups on
    the big \'etale site of $\Spec(R)$.  
  $$1 \To{} Z^{\Gamma} \To{} H^{\prime\Gamma} \To{} H^{\Gamma} \To{}
  \underline H^{1}(\Gamma,Z)\To{} \underline H^{1}(\Gamma, H').$$
In this sequence, we now know that all terms are $R$-schemes. 
Since $\underline H^{1}(\Gamma,Z)$ is  finite \'etale, the morphism $ H^{\Gamma} \To{}
  \underline H^{1}(\Gamma,Z)$ has to be trivial on the reductive subgroups
  $(H^{\Gamma})^{\circ}$, so that we deduce the following exact sequence :
  $$ Z^{\Gamma} \To{} \pi_{0}(H^{\prime\Gamma}) \To{} \pi_{0}(H^{\Gamma}) \To{}
  \underline H^{1}(\Gamma,Z)\To{} \underline H^{1}(\Gamma, H').$$
Now assume that $\pi_{0}(H^{\prime\Gamma})$ is finite over $R$, and therefore finite
\'etale. Since $Z^{\Gamma}$ is finite, its image in $\pi_{0}(H^{\prime\Gamma})$ is closed,
hence is finite \'etale. Therefore $\pi_{0}(H^{\Gamma})$ appears as the middle term of a five terms exact
sequence in which all the four remaining terms are finite \'etale group schemes (the
last one is only a pointed scheme and is \'etale by theorem \ref{h1_finite_etale}).
Going to a finite \'etale covering $R'$ of $R$ over which all these \'etale groups become
constant, we see that $\pi_{0}(H^{\Gamma})$  also becomes constant and finite over $R'$, hence
is already finite over $R$. 

ii) Put $H':=(H^{\Gamma'})^{\circ}$. Applying the $\Gamma/\Gamma'$-invariants functors to the
exact sequence $H'\hookrightarrow H^{\Gamma'}\twoheadrightarrow \pi_{0}(H^{\Gamma'})$, we get an exact sequence
$$ 1\To{} (H')^{\Gamma/\Gamma'} \To{} H^{\Gamma} \To{}
\pi_{0}(H^{\Gamma'})^{\Gamma/\Gamma'}\To{} \underline H^{1}(\Gamma/\Gamma',H').$$
By assumption, $\pi_{0}(H^{\Gamma'})$ is finite \'etale, so the invariant subgroup
$\pi_{0}(H^{\Gamma'})^{\Gamma/\Gamma'}$ is also \'etale and finite since it is closed. Therefore
the map from $H^{\Gamma}$ factors over $\pi_{0}(H^{\Gamma})$. Since
$((H')^{\Gamma})^{\circ}=(H^{\Gamma})^{\circ}$, we thus get an exact sequence
$$ 1\To{} \pi_{0}((H')^{\Gamma/\Gamma'}) \To{} \pi_{0}(H^{\Gamma}) \To{}
\pi_{0}(H^{\Gamma'})^{\Gamma/\Gamma'}\To{} \underline H^{1}(\Gamma/\Gamma',H').$$
All terms but possibly the middle one are finite \'etale (by  Theorem
\ref{h1_finite_etale} for the last one). Therefore, the middle one is also finite
\'etale, as desired.
\end{proof}

\begin{thm}\label{pi0fini}
  Assume that $H$ is reductive over a Dedekind $G$-ring $R$ and is acted upon by  a solvable
  finite group  $\Gamma$ with invertible order in $R$. Then $\pi_{0}(H^{\Gamma})$ is a finite
  \'etale group scheme over $R$.
\end{thm}
\begin{proof}
  As already mentioned in the beginning of this subsection, the problem is to prove
  finiteness. 
Thanks to item ii) of the last lemma, we can use induction to reduce the case of a solvable $\Gamma$ to the
case of an abelian $\Gamma$, and then further reduce to the case of a cyclic $\Gamma$. So
let us assume that $\Gamma$ is cyclic.

By Theorem 5.3.1 of \cite{conrad_luminy}, there is a unique closed semi-simple subgroup scheme
$H_{\rm der}$ of $H$ over $R$ that represents the sheafification of the set-theoretical
derived subgroup and such that the quotient $H/H_{\rm der}$ is a torus.
Then the natural morphism $Z(H)^{\circ}\times H_{\rm der}\To{} H$ is a central isogeny by the
fibrewise criterion, and moreover is $\Gamma$-equivariant (here
$Z(H)^{\circ}$ denotes the maximal central torus of $H$).
Further, by Exercise 6.5.2 of \cite{conrad_luminy}, there is a canonical central isogeny
 $H_{\rm sc}\To{} H_{\rm der}$ over $R$, such that all the geometric fibers of $H_{\rm sc}$ are
 simply connected semi-simple groups. Being canonical, the action of $\Gamma$ on $H_{\rm
   der}$ lifts uniquely to $H_{\rm sc}$.
Let us now consider the $\Gamma$-equivariant central isogeny $Z(H)^{\circ}\times H_{\rm sc}\To{} H$.
By item i) of  the previous lemma, it suffices to prove the finiteness of
$\pi_{0}((Z(H)^{\circ})^{\Gamma})$ and that of $\pi_{0}((H_{\rm sc})^{\Gamma})$. The first
one is clear since $(Z(H)^{\circ})^{\Gamma}$ is smooth and of multiplicative type. For the
second one, we use Steinberg's theorem \cite[Thm 8.2]{steinberg_endo}, which can be
applied here since a generator of $\Gamma$ induces a semisimple automorphism of each
geometric fiber of $H_{\rm sc}$, and which ensures  that
$(H_{\rm sc})^{\Gamma}$ has connected fibers, so that $\pi_{0}((H_{\rm sc})^{\Gamma})$ is
even the trivial group.
\end{proof}

\subsection{Splitting a reductive group scheme over a finite flat
  extension}
A reductive group scheme over any ring $R$ is known to split over a
faithfully \'etale extension of $R$. However, in general it won't split over a
finite \'etale extension. Already over $R=\ZM$, there are
examples where a non-trivial Zariski localization is needed. 
Here we use a similar argument as in the proof of Theorem \ref{representatives} in order to
prove that if $R$ is a localization of a ring of integers, then a
reductive group scheme over $R$ splits over a suitable finite flat extension of
$R$.
\def\pG{{\mathfrak p}}
\begin{prop}\label{split_red_gp}
  Assume that $H$ is reductive over  a normal subring $R$ of a number field
  $K$. Then there is a finite extension $K_{0}$ of $K$ such that $H$
  splits over the normalization $R_{0}$ of $R$ in $K_{0}$. 
\end{prop}
\begin{proof}
Pick a faithfully \'etale $R'$ over $R$
such that $H$ splits over $R'$. Let $R'=\prod_{i=1}^{n}R'_{i}$ be the
decomposition of $R'$ in connected components. Of course, if $n=1$ we
are done, so we assume $n>1$. Replacing $R$ and all $R'_{i}$ by their
normalization in the residue field at 
some generic point of $R'_{1}\otimes_{R}\cdots\otimes_{R}R'_{n}$, we may
assume that each $R'_{i}$ is a localization of $R$
(i.e. $\Spec(R')\To{}\Spec(R)$ is a Zariski cover).
Let $T_{i}\subset H_{R'_{i}}$ be a split maximal torus defined over $R'_{i}$. 
The generic fibers $T_{i,K}$ are split maximal tori in $H_{K}$, hence
are conjugate under $H(K)$. After replacing $K$ by a finite extension,
we may assume that they are conjugate under $H_{\rm sc}(K)$. So there
are elements $h_{i}\in H_{\rm sc}(K)$, $i>1$, such that $^{h_{i}}T_{1,K}=T_{i,K}$.
Put $S:=\Spec(R)\setminus\Spec(R'_{1})$ (a finite set) and for $\pG\in
S$, pick a $i\geq 2$ such that $\pG\in\Spec(R'_{i})$ and put
$h_{\pG}=h_{i}$. Then by the strong approximation theorem, there is
some $h\in H_{\rm sc}(R'_{1})$ such that 
$h\in H_{\rm  sc}(R_{\pG})h_{\pG}$ for all $\pG\in S$. We claim that
the $K$-torus $T_{K}:={^{h}T}_{1,K}$ of $H_{K}$ extends (canonically) to a an
$R$-subtorus of $H$. Indeed, recall that the functor ${\rm Tor}_{H/R}$ which to any $R$-algebra $R'$
associates the set  of maximal subtori of $H_{R'}$ is known to be representable by a
smooth quasi-affine, hence in particular separated, scheme over $R$,
see e.g. \cite[Thm 3.2.6]{conrad_luminy}. By construction, $T_{K}$
comes from a $R'_{1}$-torus $T'_{1}$ of $H_{R'_{1}}$, which is 
unique by separateness of ${\rm Tor}_{H/R}$. Similarly for each $\pG\in S$, there is a
unique extension of $T_{K}$ to a $R_{\pG}$-torus $T_{\pG}$ of
$H_{R_{\pG}}$ and the latter is actually defined over a Zariski open
neighbourhood of $\pG$. This means
that the $K$-section of ${\rm Tor}_{H/R}$ given by $T_{K}$ extends uniquely to a Zariski
covering of $\Spec R$, hence extends to $\Spec R$ itself, whence a maximal torus $T_{R}$ of $H$
extending $T_{K}$. Since $T_{K}$ is split and since tori are known to split over finite
\'etale coverings of the base, $T_{R}$ is split too.

Now, the root subspaces of $T_{R}$ in ${\rm Lie}(H)$ are rank $1$ locally free $R$-modules.
Replacing $K$ by its  Hilbert class field, we may assume that they are actually
free. Since $R$ is connected, this is enough for $H$ to split over $R$,
\emph{cf} the paragraph below Definition 5.1.1 of \cite{conrad_luminy}.
\end{proof}

\begin{rem}
Exercise 7.3.9 of \cite{conrad_luminy} provides another proof that
does not use strong approximation. Namely, start by enlarging $R$ so that
$H_{K}$ splits. So $H_{K}$ contains a Borel subgroup $B_{K}$, which
extends uniquely to a Borel subgroup scheme $B$ of $H$ by  the
properness of the scheme of Borel subgroups. Let $(H',B')$ be the
constant split pair over $R$ that extends $(H_{K},B_{K})$. Then the
functor $\mathcal{I}$ of isomorphisms between $(H',B')$ and $(H,B)$ is a torsor over
the automorphism group $\mathcal{A}=B'_{\rm ad}\rtimes{\rm Out}(H)$ of the pair $(H',B')$. Its class
in $H^{1}_{\et}(\Spec R,\mathcal{A})$ has trivial image in 
$H^{1}_{\et}(\Spec R,{\rm Out}(H))$ since $H$ is split over $K$. On the other hand 
$H^{1}_{\et}(\Spec R, B'_{\rm ad})$ is isomorphic to a sum of copies of 
$\Pic(R)=H^{1}_{\et}(\Spec R, \mathbb{G}_{m})$. 
So let $K_{0}$ be the Hilbert class field of $K$. Since
$\Pic(R)\To{}\Pic(R_{0})$ has trivial image, $\mathcal{I}$ becomes a trivial
$\mathcal{A}$-torsor over $R_{0}$, hence $H$ splits over $R_{0}$.
\end{rem}

\section{Twisted Poincar\'e polynomials}

\def\Sym{{\operatorname{Sym}}}
\def\XX{{\mathbb X}}
 \def\MM{{\mathbb M}}
\def\HH{{\hat H}}
 
 \subsection{Some characteristic polynomials attached to root data}
Let $\Sigma=(\XX,\XX^{\vee},\Delta,\Delta^{\vee})$ be a based root
datum with Weyl group $\Omega$ and group of automorphisms ${\rm Aut}(\Sigma)$.
 Both $\Omega$ and  ${\rm Aut}(\Sigma)$ embed as groups of linear automorphisms of
 $\XX$ and $\XX^{\vee}$, and ${\rm Aut}(\Sigma)$ normalizes $\Omega$. In particular
 ${\rm Aut}(\Sigma)$ acts on the ring of $\Omega$-invariant polynomials
 $\Sym^{\bullet}(\XX)^{\Omega}$ on $\XX^{\vee}$ and on the conormal module
 $\MM^{\bullet}$ of
 $\XX^{\vee}/\Omega$ along the zero section
 $$ \MM^{\bullet}:= \Sym^{\bullet>0}(\XX)^{\Omega} / (\Sym^{\bullet>0}(\XX)^{\Omega})^{2}.$$
 For any $\alpha\in{\rm Aut}(\Sigma)$ we consider its weighted
 characteristic polynomial on $\MM^{\bullet}_{\QQ}$
 \begin{equation}
   \label{eq:defchisigma}
 \chi_{\alpha|\MM^{\bullet}}(T):=\prod_{d>0} \det\left(T^{d}-\alpha\,|\MM^{d}_{\QQ}\right) \,\,\in \ZZ[T].   
 \end{equation}
A priori $\MM^{\bullet}$ may have torsion, but a result of Demazure \cite[Thm  3]{demazure} shows that
$\MM^{\bullet}\otimes \ZZ[\frac 1{|\Omega|}]$ is torsion
 free, so we deduce the following 
 \begin{rem} \label{rk_Demazure}
   If $\ell$ does not divide the order of $\Omega$, the image of
   $\chi_{\alpha|\MM^{\bullet}}$ in $\FF_{\ell}[T]$ is the weighted
   characteristic polynomial of $\alpha$ on
   $\MM^{\bullet}_{\FF_{\ell}}$.
 \end{rem}

 Since  $\Omega$ is a reflection subgroup of  $\Aut(\XX^{\vee})$,
 the $\bar\QQ$-algebra $\Sym^{\bullet}(\XX)^{\Omega}_{\bar\QQ}$ is known to be a
weighted polynomial algebra. More precisely,  any graded  section
$\MM^{\bullet}_{\bar\QQ} \hookrightarrow \Sym^{\bullet>0}(\XX)^{\Omega}_{\bar\QQ}$
induces a graded isomorphism
$ \Sym(\MM^{\bullet}_{\bar\QQ}) \To\sim \Sym^{\bullet}(\XX)^{\Omega}_{\bar\QQ}$ for the
unique ring grading on  $\Sym(\MM^{\bullet}_{\bar\QQ})$ such that $\MM^{d}$ is in degree $d$
for all $d$.
In particular, for $\alpha=1$, we have $\chi_{1|\MM^{\bullet}}(T)=\prod_{i=1}^{r}(T^{d_{i}}-1)$
where $r={\rm rk}_{\ZZ}(\XX)$ and $d_{1}\leq\cdots \leq d_{r}$ are
 the so-called \emph{fundamental degrees} of $\Omega$ acting on $\XX^{\vee}$. 
Here $d_{r}$ is known as the \emph{Coxeter number} of $\Sigma$ and is
the maximal $n\in\NM$ such that $\Phi_{n}(T)$ divides  $\chi_{1|\MM^{\bullet}}(T)$.

More generally, using  an ${\rm Aut}(\Sigma)$-equivariant  section $\MM^{\bullet}_{\bar\QQ}
\hookrightarrow \Sym^{\bullet>0}(\XX)^{\Omega}_{\bar\QQ}$, we see that, at least when $\alpha$ has finite order,
  $\chi_{\alpha|\MM^{\bullet}}(T)=(T^{d_{1}}-\varepsilon_{1,\alpha})\cdots(T^{d_{r}}-\varepsilon_{r,\alpha})$ 
  where the  $\varepsilon_{i,\alpha}$ are as  in Lemma 6.1 of
  \cite{springer-regular}.
Note that in this case, $\chi_{\alpha|\MM^{\bullet}}$ is a product of cyclotomic
polynomials  
and we have 
$\chi_{\alpha|\MM^{\bullet}}=\chi_{\alpha^{-1}|\MM^{\bullet}}$. The maximal $n\in\NM$ such that $\Phi_{n}(T)$ divides
  $\chi_{\alpha|\MM^{\bullet}}(T)$ has been known in the literature as
  the \emph{twisted Coxeter number} associated to $\alpha$. Now, a
  fundamental consequence of Springer's work in this setup is the
  following result.

\begin{prop} \label{prop_Springer}
  $\chi_{\alpha|\MM^{\bullet}}(T)$ is the lowest common multiple in $\bar\QQ[T]$ of the
      characteristic polynomials $\chi_{\omega\alpha|\XX}(T)$ of $\omega\alpha$ on
      $\XX_{\bar\QQ}$, where $\omega$ runs over $\Omega$.
\end{prop}
\begin{proof}
When $\alpha$ has finite order, this is a reformulation of Theorem 6.2 (i) of
\cite{springer-regular}. In general, this follows from the decompositions
$\XX_{\QQ}=\XX_{\QQ}^{\Omega}\oplus \QQ\langle\Delta\rangle$ and
$\MM^{\bullet}_{\QQ}=\XX_{\QQ}^{\Omega}\oplus
\Sym^{\bullet>0}(\QQ\langle\Delta\rangle)^{\Omega} /
(\Sym^{\bullet>0}(\QQ\langle\Delta\rangle)^{\Omega})^{2} $ and the fact that
$\alpha_{|\QQ\langle\Delta\rangle}$ has finite order, while 
$\omega\alpha_{|\XX_{\QQ}^{\Omega}}=\alpha_{|\XX_{\QQ}^{\Omega}}$ for all
$\omega\in\Omega$.
\end{proof}

\subsection{Application to reductive groups} Let $\HG$ be a connected reductive
group  over an algebraically closed field $L$ of characteristic
$\ell$. Attached to $\HG$ is a root datum
$\Sigma=(\XX_{\HG},\XX_{\HG}^{\vee},\Delta_{\HG},\Delta_{\HG}^{\vee})$ as above, that
comes with an 
identification ${\rm Aut}(\Sigma)={\rm Out}(\HG)$. Here $\Sigma$ denotes the limit over all
Borel pairs $(\HB,\HT)$ of $\HG$ of the
root data $(X^{*}(\HT),X_{*}(\HT),\Delta(\HB),\Delta(\HB)^{\vee})$.
Now, let $\beta$ be
an automorphism of $\HG$ with image $\alpha$ in ${\rm Out}(\HG)$. Using the
notation of the last subsection, we put
 \begin{equation}
   \label{eq:defchi}
 \chi_{\HG,\beta}(T):=\chi_{\alpha|\MM^{\bullet}}(T)\in \ZZ[T].   
 \end{equation}
Further, we denote by $h_{\HG,\beta}$ the twisted Coxeter number of
$\Sigma$ associated to $\alpha$  and we put 
 \begin{equation}
   \label{eq:defchistar}
 \chi_{\HG,\beta}^{*}(T):=\prod_{n\leq h_{\HG,\beta}} \Phi_{n}(T) \,\,\in \ZZ[T].   
 \end{equation}
The following result is crucial to track the ``banal'' primes in this paper.

\begin{prop} \label{prop:char_pol}
  Let $\beta$ be an automorphism of $\HG$.
      
      (1) If $\HH$ is a reductive subgroup of $\HG$ stable under $\beta$, then $\chi_{\HH,\beta}$
      divides $\chi_{\HG,\beta}$, $h_{\HH,\beta}\leq h_{\HG,\beta}$ and $\chi_{\HH,\beta}^{*}$
      divides $\chi_{\HG,\beta}^{*}$.
     
      (2) Let $t$ be a semi-simple element of $\HG(L)$ such that
      $\beta(t)=t^{q}$. Then $t$ has finite order, and this order
      divides $\chi_{\HG,\beta}(q)$. 
\end{prop}
\begin{proof}
  (1) As above, we denote by $\alpha$ the image of $\beta$ in ${\rm Out}(\HG)={\rm
    Aut}(\Sigma_{\HG})$, which acts on the ``abstract root lattice'' $\XX_{\HG}$.
  Similarly, we denote by $\alpha_{\HH}$
  the image of $\beta$ in ${\rm Aut}(\Sigma_{\HH})$, which acts on $\XX_{\HH}$.
  Let $(B_{\HH},T_{\HH})$ be a Borel pair of $\HH$, so that we have an identification
  $\Sigma_{\HH}=\Sigma(B_{\HH},T_{\HH})$, and in particular $\XX_{\HH}=X^{*}(T_{\HH})$.
  Through this identification, the action of $\alpha_{\HH}$ on $\XX_{\HH}$ corresponds to the action of
  $\Ad_{h}\circ \beta$ on $X^{*}(T_{\HH})$ for any $h\in \HH$ such that $\Ad_{h}\circ \beta$
  stabilizes the pair $(B_{\HH},T_{\HH})$. More generally, for $\omega_{\HH}\in
  \Omega_{\HH}$ (the ``abstract'' Weyl group of $\HH$), the action of
  $\omega_{\HH}\alpha_{\HH}$ on $\XX_{\HH}$ corresponds to the action of
  $\Ad_{nh}\circ \beta$ on $X^{*}(T_{\HH})$, where $h$ is as above, and  $n\in
  N_{\HH}(T_{\HH})$ is a lift of  $\omega_{\HH}$.
  
  Now, let $(\HB,\HT)$ be a Borel pair in $\HG$ that induces $(B_{\HH},T_{\HH})$ on $\HH$.
  As above, we have an identification  $\Sigma_{\HG}=\Sigma(\HB,\HT)$, and in
  particular $\XX_{\HG}=X^{*}(\HT)$, from which we deduce a surjective morphism
  $\XX_{\HG}\twoheadrightarrow \XX_{\HH}$. 
  With $h$ and $n$ as above, pick also $m\in C_{\HG}(T_{\HH})$ such that $\Ad_{mnh}\circ\beta$
  stabilizes $\HT$. Observe that the action of this automorphism on $X^{*}(\HT)$ induces the
  action of $\Ad_{nh}\circ\beta$ on $X^{*}(T_{\HH})$.
  On the other hand, there is a unique $\omega\in \Omega_{\HG}$ such that, for any
  $n'\in N_{\HG}(\HT)$ above $\omega^{-1}$, the automorphism $\Ad_{n'mnh}\circ\beta$
  stabilizes also $\HB$. Then the action of this automorphism on $\XX_{\HG}$ is $\alpha$.
  Hence it follows that the action of $\omega\alpha$ on $\XX_{\HG}$ induces the action of
  $\omega_{\HH}\alpha_{\HH}$ on $\XX_{\HH}$.
  Therefore the characteristic polynomial
  $\chi_{\omega_{\HH}\alpha_{\HH}|\XX_{\HH}}(T)$ divides the characteristic polynomial
  $\chi_{\omega\alpha|\XX_{\HG}}(T)$. By Proposition \ref{prop_Springer}, we deduce that
  $\chi_{\alpha_{\HH}|\MM_{\HH}^{\bullet}}$ divides $\chi_{\alpha|\MM_{\HG}^{\bullet}}$, as desired.

  (2) The connected centralizer $\HH:=C_{\HG}(t)^{\circ}$ contains $t$ and is stable under
  $\beta$, since $\beta(\HH)=C_{\HG}(t^{q})^{\circ}$ contains $\HH$ and has same dimension as $\HH$. 
  Hence by (1) it suffices to prove the statement when $t$ is central in $\HG$. Then 
  we may compose $\beta$ with some $\Ad_{g}$ so that it fixes a pinning of $\HG$, with
  maximal torus $\HT$.
  Now, consider $t$ as a homomorphism $X^{*}(\HT)\To{} L^{\times}$. Since 
  $\beta(t)=t^{q}$,
  we see that this homomorphism factors over the cokernel of the endomorphism
  $\beta-q$ of $X^{*}(\HT)$. But this cokernel is finite of  order
  $\chi_{\HT,\beta}(q)=\det(q-\beta)$. 
So $t$ has order dividing $\chi_{\HT,\beta}(q)$, hence also dividing $\chi_{\HG,\beta}(q)$. 
\end{proof}

\subsection{The Chevalley-Steinberg formula}

Let now $G$ be a reductive group over $\FF_{q}$.
Let $G^{*}$ be a split form of $G_{\bar\FF_{q}}$ over $\FF_{q}$
and pick an isomorphim
$\psi:\,G_{\bar\FF_{q}}\To\sim G^{*}_{\bar\FF_{q}}$. Then  $\Fr:= {^{\rm Frob}\psi^{-1}}\circ \psi$
is an automorphism of $G_{\bar\FF_{q}}$ (where ${\rm Frob}$ denotes the Frobenius
automorphism of $\bar\FF_{q}$), and we have the following Chevalley-Steinberg
formula for the number of $\FF_{q}$-rational points of $G$.
\begin{thm}[Chevalley-Steinberg] \label{thm:chevallet-steinberg}
$ |G(\FF_{q})| = q^{N}.\chi_{G,\Fr}(q)$, where $N$ is the dimension of a maximal unipotent subgroup of $G_{\bar\FF_{q}}$.
\end{thm}
\begin{proof}
  This formula is stated for absolutely simple adjoint groups in Theorems 25 and 35 of
  \cite{steinberg-lectures}. It is also true for a torus $S$, since we have an isomorphism
  $X_{*}(S)/(q\Fr-1)X_{*}(S)\To\sim S(\FF_{q})$ \cite[(5.2.3)]{DelLu}, from which it
  follows 
  that $|S(\FF_{q})|=|\det(q\Fr-1)|=|\chi_{S,\Fr^{-1}}(q)|=\chi_{S,\Fr}(q)$.

  To prove the formula in general, we first observe that if $G\To{\pi}G'$ is an isogeny, then
  $|G(\FF_{q})|=|G'(\FF_{q})|$. Indeed, the kernel $H:=\ker(\pi)(\bar\FF_{q})$ is a
finite group with an action of the arithmetic Frobenius ${\rm Frob}$ and we have an exact sequence
$$ 1\To{} H^{{\rm Frob}} \To{} G(\FF_{q})\To{} G'(\FF_{q})\To{} H^{1}(\FF_{q}, H)=H_{{\rm Frob}}\To{} 1$$
where the last map is surjective because $H^{1}(\FF_{q},G)=1$. But we also have an exact
sequence $H^{{\rm Frob}}\hookrightarrow H\To{{\rm Frob}-\id}H\twoheadrightarrow H_{{\rm Frob}}$ which shows
that $|H^{{\rm Frob}}|=|H_{{\rm Frob}}|$, so we get $|G(\FF_{q})|=|G'(\FF_{q})|$.

Now we deduce the formula for general $G$ by applying  this observation to the isogeny
$G\To{}G_{\rm ab}\times G_{\rm ad}$ and decomposing $G_{\rm ad}$ as a product of
restriction of scalars of absolutely simple groups.
\end{proof}

\subsection{Kostant's section theorem}\label{sec:kost-sect-theor}
We return to the setting of a reductive group $\HG$ over an algebraically closed field
$L$ and, for simplicity, we assume that $\HG$ is simple adjoint. We also assume that the characteristic
$\ell$ of $L$ does not divide the order of the Weyl group $\Omega_{\HG}$.

Let us fix a pinning $\varepsilon=(\HT, \HB, (X_{\alpha})_{\alpha\in\Delta})$ of $\HG$.
The sum  $E=\sum_{\alpha\in\Delta }X_{\alpha}$ is then a regular nilpotent
  element of $\Lie(\HG)$. The sum $H=\sum_{\beta\in\Phi^{+}}
  \beta^{\vee}\otimes 1\in X_{*}(\HT)\otimes L=\Lie(\HT)$ is a regular semisimple
  element of $\Lie(\HG)$ and the pair $(H,E)$ is part of a unique
  ``principal'' $\mathfrak{s}\mathfrak{l}_{2}$-triple $(F,H,E)$.
  Denote by $\Lie(\HG)_{E}$ the centralizer of $E$ in $\Lie(\HG)$. Under our assumption on $\ell$,  Veldkamp has proved that Kostant' section theorem still
  holds, \cite[Prop 6.3]{veldkamp}.
  This states that the  map
  $$\Lie(\HG)_{E}\To{}  \Lie(\HG)\sslash\HG, \,\,\,\, X\mapsto (F+X)\hbox{ mod }\HG$$ is
  an isomorphism of varieties. Moreover, seeing $\lambda:=\sum_{\beta\in
    \Phi^{+}}\beta^{\vee}$ as a cocharacter of $\HG$, this map is $\mathbb{G}_{m}$-equivariant for
the action $(t,y)\mapsto t\cdot y:= t^{2}\Ad_{\lambda(t)}(y)$ on the LHS and   the 
action $(t,x)\mapsto t^{2}x$ on the RHS.
Composing with the Chevalley isomorphism (which also holds in this context) yields an isomorphism of $\GG_{m}$-varieties
$$\pi:\,\, \Lie(\HG)_{E}\To{\sim}  \XX^{\vee}_{L}/\Omega_{\HG}.$$

Now let  $\Aut(\HG)_{\varepsilon}$ be the group of automorphisms of $\HG$ that preserve the
pinning $\varepsilon$. This group  fixes $E$, so it acts on $\Lie(\HG)_{E}$. It also acts on
 $\Lie(\HG)\sslash\HG$ and $\XX^{\vee}_{L}/\Omega_{\HG}$, and both the Chevalley map and the
 Kostant map are equivariant for these actions.
Identifying ${\rm Out}(\HG)$ with $\Aut(\HG)_{\varepsilon}$, we thus  get on conormal modules at the origin an isomorphism
$$ \MM^{\bullet}_{L} \To\sim  (\Lie(\HG)_{E})^{*}$$
which is ${\rm Out}(\HG)$-equivariant, as well as $\GG_{m}$-equivariant for the (dual) action described above on the RHS
and the action associated with ``twice the $\bullet$-grading'' on the LHS. 
So we deduce the following result.
  
  \begin{prop}\label{prop:kostant}
    For $t\in L^{\times}$ and $\beta\in \Aut(\HG)_{\varepsilon}$ of finite order, we have
    $$
    \det\left(t^{2}\Ad_{\lambda(t)}\Ad_{\beta}-\id\,|\Lie(\HG)_{E}\right) =
     \pm \chi_{\HG,\beta}(t^{2}).
    $$
  \end{prop}
  \begin{proof}
  Indeed, by the foregoing discussion, the LHS equals
  $$\prod_{d}\det\left(t^{2d}\beta^{-1} -\id | \MM^{d}_{L}
  \right)=\det(\beta)^{-1}\prod_{d}\det\left(t^{2d} -\beta | \MM^{d}_{L} \right). $$
  But $\det(\beta)=\pm 1$ since it is a root of unity in $\QQ$, while Remark
  \ref{rk_Demazure} ensures that $\det\left(t^{2d} -\beta | \MM^{d}_{L}
  \right)=\chi_{\HG,\beta}(t^{2})$ in $L$.
  \end{proof}

\bibliographystyle{alpha}
\bibliography{Parameters.bib}

\end{document}